\documentclass[12pt,reqno]{amsart}
\usepackage{amssymb,amsmath,amscd,epsf,accents}
\usepackage[pdftex]{graphicx}
\usepackage[all,knot,curve]{xy}
\usepackage{mathtools}\usepackage{hyperref}

\UseRawInputEncoding

\def\Cal{\mathcal}
\def\N{\mathbb N} 
\def\Z{\mathbb Z} 
 
\def\C{\mathbb C} 
\def\T{\mathbb T}

\def\CCM{{\Cal C\Cal M}_{M}}
\def\CCMtil{{\Cal C\Cal M}_{\wt M}}

\def\V{\Cal V}

\def\CC{\Cal {C}} 

\def\CCL{{\Cal C\Cal M}_{L}}
\def\CCAB{{\Cal C\Cal M}_{AB}}

\def\Ext{\text{Ext\,}}
\def\ExtD{\text{Ext}_{\text{dist}} }

\def\R{\mathbb R}
\def\A{\Cal A}
\def\O{\Cal O}

\def\cl{\text{cl}}

\def\L{L}

\def\M{M}

\def\AB{AB}
\def\FC{\Cal F\Cal C}
\def\E{\Cal E}

\def\nsft{{\em nsft}}
\def\sft{{\em sft}}

\def\B{\Cal  B} 

\def\bpi{\boldsymbol{\pi}}
\def\0{\mathbf 0}    

\def\interior{ \text{ int}}

\def\gcd{{\em gcd}}

\def\ee{\underline e}
\def\dd{\underline d}
\def\aa{\underline a}

\def\bfw{\mathbf w}

\def\u{\mathbf u}
\def\v{\mathbf v}
\def\w{\mathbf w}
\def\z{\mathbf z} 
\def\x{\mathbf x} 
 
\def\1{\mathbf 1}    
\def\a{\mathbf a}   
\def\b{\mathbf b}    

\def\e{ e} 

\def\f{ f} 
 
\def\s{\mathbf s}

\def\then{\implies}

\def\r{{\mathbb R}}

\def\wt{\widetilde}
\def\wh{\widehat}

\def\T{\mathbb T} 
\def\C{\mathbb C}

\def\ep{\varepsilon}

\def\NS{\Cal N \Cal S}
\def\NP{\Cal N \Cal P}
\def\CN{\Cal N}
\def\P{\Cal P}

\theoremstyle{plain}
\newtheorem{theo}{Theorem}[section]
\newtheorem{prop}[theo]{Proposition}

\newtheorem{lem}[theo]{Lemma}
\newtheorem{cor}[theo]{Corollary}

\theoremstyle{definition}
\newtheorem{defi}[theo]{Definition}

\theoremstyle{remark}
\newtheorem{rem}[theo]{Remark}

\newtheorem{exam}[theo]{Example}

\begin{document}
\title[Invariant measures for adic  transformations]{  Finite and infinite  invariant measures for adic transformations}
 \date{January 24, 2026}

\author{Albert M. Fisher} 

\address{Albert M. Fisher,
Dept Mat IME-USP,
Caixa Postal 66281,
CEP 05315-970
S\~ao Paulo, Brazil}
\urladdr{http://ime.usp.br/$\sim$afisher}
\email{afisher@ime.usp.br}

\author {Marina Talet}
\address{Marina Talet
Aix-Marseille University, CNRS, I2M, 
Marseille, 
France}

\email{marina.talet@univ-amu.fr}

\thanks{
A. Fisher and  M. Talet partially 
supported by FAPESP, Franco-Brazil cooperation:  CNPq-CNRS}

\large

\begin{abstract} 
We classify the invariant Borel
 measures for 
 adic transformations, subject to the conditions that 
the alphabets have bounded size and that the measure is 
 finite  on the path space of some
sub-Bratteli diagram.

 To carry this out,   we develop a nonstationary version of 
 the  Frobenius normal form for a reducible matrix, present an  appropriate nonstationary 
notion of   distinguished eigenvector,  and  prove a nonstationary 
Frobenius--Victory theorem.  This parallels the approach
to the
stationary case  developed by 
 Bezuglyi, Kwiatkowski, Medynets and Solomyak in
 \cite{BezuglyiKwiatkowskiMedynetsSolomyak10} 
 where  they classify
 the locally finite invariant measures.
  In 
later work,
they also
address  the nonstationary case.
We extend their work in two ways. Firstly, in both the stationary and
nonstationary settings, we
allow for measures 
which  are  locally infinite, motivating this extension with 
examples. Secondly we give a 
complete classification, presenting
a  necessary
and sufficient condition for  a  measure which is finite on some
subdiagram to be infinite for the
original diagram. As part of our program, we introduce a
related object
called an  {\em adic tower} and a
construction called the {\em canonical cover} of the
subdiagram: if the measure is finite on the subdiagram
but locally infinite on the
original space it will be locally finite on the cover space, though
the path spaces for the original diagram and the cover 
 are measure-theoretically isomorphic.
We further extend our classification
results to adic towers.  Our examples include two new models for the
Integer Cantor Set transformation of \cite{Fisher92}, one locally
finite and one locally infinite, which lead to 
a new class of examples
related to fractal sets of integers:  {\em nested circle rotations}, where one
rotation is embedded in another.
The resulting tower measure inside the original rotation
can  have finite or infinite total mass
as
 determined by our general criterion. 
\end{abstract}

\keywords
{adic transformation,  unique ergodicity, nonstationary subshift
 of finite type
}
\maketitle
\tableofcontents

\section{Introduction}

\subsection{The stationary adic setting}

\smallskip

Adic transformations are of interest in dynamics for a variety of
reasons. They can be used to model such classical examples as
irrational circle rotations, interval exchange transformations,
substitution dynamical systems, and  cutting-and-stacking
constructions.
The formalism involved is elegant and is well adapted to the
analysis of the topological and the measure properties of the
dynamics. Because  such a variety of originally distinct
systems can be treated from this common point of view,
one can move 
ideas and methods from one context to another.

The simplest setting is that of a {\em stationary} adic
transformation, where the state space is the set of paths for a one-sided
subshift of finite type; in the most basic  case this is defined by a
primitive matrix (meaning 
that some power has all entries 
positive).  One then has a pair of dynamical systems,
the left shift map and the  {\em transverse
} dynamics of Vershik's adic transformation defined by placing an (anti)-lexicographic
order on the path space. In this setting, the adic transformation is
both {\em minimal } (every orbit is dense) and  {\em uniquely ergodic}
(there exists a unique invariant probability measure), and moreover it
is
{\em
  deterministic} (in the sense of having zero topological entropy).
This contrasts markedly with the shift map which is {\em hyperbolic} and has
many nontrivial invariant subsets (for example, exponentially  many periodic points)
and many  
invariant measures. There is however one special measure for the shift, the {\em
  Parry measure};
this is the (unique) measure of maximal entropy and has a simple
formula
(due to both Shannon and Parry) using the right and left Perron-Frobenius
eigenvectors of the matrix.

A variant of this {\em Parry formula}, using only the right eigenvector, produces
the unique measure for the adic transformation.

For the simplest example of such a  pair of maps (see Example
\ref{ex:odometer}), we have   the one-sided
Bernoulli shift space $\Sigma^+= \Pi_0^\infty \{0,1\}$, carrying the
product topology which makes it a compact metric space.  The left
shift $\sigma$ is a two-to-one continuous map of $\Sigma^+$; it has as its  measure of maximal entropy the
 $(\frac{1}{2},\frac{1}{2})$  Bernoulli measure $\mu$, representing tosses of
 a fair
 coin. The second map on  $\Sigma^+$ is the
 Kakutani-von Neumann {\em odometer transformation} $T$, which  acts  as a
 homeomorphism, giving  
 the original  example of an adic
transformation; the unique invariant probability measure $\nu$ for $T$
exceptionally for this case is the same: $\nu=\mu$. 
This  pair of maps
satisfies an interesting {\em
  commutation relation} $\sigma\circ T^2 =T\circ \sigma$, 
reminiscent of what happens for the pair (geodesic flow,  stable horocycle
flow)   $=(g_t, h_t)$,  acting on a fiber bundle over a surface:  the unit tangent bundle $F(M)$ of a compact hyperbolic
Riemann surface $M$ of
constant negative curvature. 

This analogy goes quite far:
the geodesic flow, likewise,  has a unique measure of
maximal entropy, while the  stable horocycle flow 
is minimal, uniquely ergodic and of zero entropy.
This flow  $h_t$ preserves the stable foliation of the  geodesic
flow, explaining Vershik's terminology of {\em transverse dynamics}:
the horocycle flow cuts across the geodesic flow orbits. The pair now satisfies the continuous time {commutation relation}
$g_t\circ h_{e^t s}= h_s\circ g_t$,
giving us the  twin diagrams:
$$
\begin{CD}
F(M)     @>h_{e^t s} >>  F(M)\\
@VVg_tV        @VVg_tV\\
F(M)    @>h_s>>  F(M)
\end{CD}
\hspace{2cm}
\begin{CD}
\Sigma^+    @> T^2>>  \Sigma^+ \\
@VV\sigma V        @VV\sigma V\\
\Sigma^+    @>T>>  \Sigma^+ 
\end{CD}
$$

From this diagram, the horocycle flow is
isomorphic to a sped-up version of itself, agreeing with the fact that
its entropy is $0$. Now for the second diagram, since $\sigma$ is
$2$-to-$1$, this is not an isomorphism but a  homomorphism, thus
$T^2$ is homomorphic to $T$, but if we restrict to a $1$-cylinder set
then the first return (or {\em induced}) map on this subset is $T^2$,
and  the shift does give an isomorphism. This is the
{\em renormalization} of $T$ to the induced map 
and the fact that this is an isometry reflects an exact
self-similarity of the geometry and dynamics of the original set with
the subset (much as for the Feigenbaum fixed point). This renormalization is moreover
{\em dynamically realized}  by the dynamics of the shift map.
By analogy one can say that the horocycle flow is  also a renormalization
fixed point as it is 
 renormalized to itself  by the 
geodesic flow, see \cite{Fisher03-1}.

This is the case of the odometer.
In fact, all stationary adic transformations $T$ with their shift map
$\sigma$ can be naturally embedded in a pair of 
flows which satisfy the commutation relations of $(g_t, h_t)$.
The flow space is a
{\em solenoidal } space, meaning that the space is locally of the form
(Cantor set)x($\R^2$). Compare \cite{Sullivan87}.
For now we mention
three further motivating examples from this broader context 
before we
get to the specifics of the current paper.

Let $A$ be an Anosov automorphism of the (two)-torus. Its stable flow (the unit speed
flow along its stable foliation) has a
circle crossection, on which the return map
is an irrational rotation by a quadratic
irrational. The suspension of the
toral map defines a three-manifold which
is acted on by  a
{\em vertical flow}, which takes the role of  the geodesic flow, and
by two {\em horizontal flows}, the stable and unstable flows $h_t^s, h_t^u$,
which preserve the levels of constant
height and are
uniquely ergodic and minimal on each level. The  circle rotation
angles for these two transverse flows are {\em dual} in
that their continued fraction expansions are written in reverse
order, forming one periodic sequence of positive integers when joined
back-to-back. See \cite{ArnouxFisher00} and \cite{ArnouxFisher05}.

Generalizing from genus one (the torus) to higher genus Riemann
surfaces, the analogous picture holds, with the hyperbolic toral automorphism
now replaced by a pseudoAnosov map of the surface, and the quadratic circle rotations
 exchanged for two dual 
 interval exchange transformations of periodic type. Also here
 one can combinatorially build the solenoidal space, recovering the
 three-manifold by a further gluing.

For a second example, analogously
a substitution dynamical system is
transverse to its substitution map. For the third example,  a
stationary cutting-and-stacking map is
transverse to the cutting-and-stacking operation.
In all these cases, further explored in later work, the
transverse  systems
can be modelled by stationary
adic transformations (more precisely, as a factor if the 
substitution is {\em nonrecognizable}
\cite{Mosse92}), transverse to a shift map so the commutation relation
is satisfied for the pair, and the transverse maps
share with circle
rotations the attributes of being
minimal,
uniquely ergodic, and of zero entropy, as contrasted to the 
dynamics of hyperbolicity, with positive entropy, many periodic points
and many invariant measures.

\medskip

\subsection{The nonstationary adic setting} 

\smallskip

When transferred to the adic setting, all of these examples  can be
realized as stationary
adic transformations. The measure theory of these maps 
changes in a fascinating way when the picture is opened up to include
nonstationary adic transformations, that is, when the single matrix
for the \sft \, of the stationary case is replaced by a matrix {\em
  sequence}, and the finite graph of the \sft\, by the locally finite,
infinite graph of the
resulting Bratteli diagram. For the case of interval exchanges
this corresponds to nonperiodic combinatorics, equivalently a
nonperiodic path in the Rauzy graph.
An exchange of two intervals is homeomorphic to a rotation of the
circle, 
and nonperiodic combinatorics means this rotation angle has a
nonperiodic continued fraction expansion, 
equivalently is a nonquadratic irrational number. Now quadratic irrationals
are what one sees for the return maps for 
the stable flows of an Anosov map of the torus $\R^2/\Z^2=
\R/\Z\times \R/\Z$, to one of the circles $\R/\Z$.
 The move to include
 general  irrational angles is described in \cite{ArnouxFisher00} and
\cite{ArnouxFisher05}: the 
 single Anosov  map should now be replaced by a
{\em nonstationary} dynamical system,  defined by a sequence of toral
automorphisms, satisfying a condition of  asymptotic hyperbolicity.

For irrational circle rotations one still has not only minimality but
unique ergodicity; however  for
more general interval exchanges 
this  question  becomes much more subtle.
As discovered by Keane
\cite{Keane77}, for  four or more intervals there exist maps which are
minimal 
but 
not uniquely ergodic, thus giving
a nontrivial simplex of invariant probability measures. See also
Keynes-Newton \cite{KeynesNewton76} and Veech \cite{veech1969strict}.

Such ``Keane counterexamples'' examples are, however,  {\em
  exceptional} since, as shown by Veech \cite{Veech82} and
Masur \cite{Masur82}, almost every interval exchange transformation (this means 
for a.e.~vector in the probability simplex of subinterval lengths)  is
uniquely ergodic. These exceptional maps involve some fascinating
geometry. The interval exchange sits as a crossection to a foliation
of a  Riemann surface; the
proofs involve the Teichm\"uller flow of this surface.
See \cite{yoccoz2010interval} and  \cite{Viana08}, \cite{Viana2016}.

Indeed, the  sequence of matrices defining the Bratteli diagram for
the
adic transformation of the interval exchange 
corresponds to
a random walk path converging to a
point in  the boundary
at infinity 
of the mapping class group of the surface, and  the counterexamples
correspond exactly to points where the Thurston boundary is
different from the Bers boundary \cite{Kerkhoff85} and \cite{masur1982two}.

The study of cutting-and-stacking maps has included nonstationary
construction procedures since the outset; these have been used to
study general questions in  ergodic theory, including the 
construction of interesting counterexamples. See e.g.~\cite{bruin2022topological}.
Further examples of  nonstationary adic transformations come from
nonstationary  
substitution dynamical systems ({\em nsds}) or {\em S-adic systems}.

In the present paper we want to isolate and understand these
 phenomena in the
general context of adic transformations, continuing our work in
\cite{ArnouxFisher00}, 
\cite{Fisher09a} and
\cite{FerencziFisherTalet09}.
The special emphasis here will be on
nonprimitive, nonstationary diagrams (by which we mean equivalently that the matrix sequence
has these properties) and on measures which are
locally infinite but still tractable.

\medskip

\subsection{Overview of the paper}
The step we take here toward this understanding is to present
 a complete classification of the  invariant
 Borel measures
 for Vershik's  adic transformations defined on the one-sided path space of a
Bratteli
diagram of {\em finite rank}, i.e.~we
have  bounded
alphabet size. In  \cite{Fisher09a} we addressed the  nonstationary
primitive case;
now we allow for 
nonprimitive diagrams,  where fascinating new
phenomena can occur.

 We assume in this paper  that the 
the measures are  finite on some
subdiagram. Here a {\em subdiagram} of a given Bratteli diagram is
defined by erasing some edges or vertices, see \S \ref{ss:towers}.
We emphasize that the measures
may for example be
infinite on every nonempty open set for the original diagram.
We further classify the invariant Borel measures  for related
objects called {\em adic towers}, Def.~\ref{d:adictower}.

A key formalism that we need throughout is the notion of
{\em generalized matrix}.  This allows us to conveniently define {\em
  submatrices}  and describe the relation to {subdiagrams}.

 The idea of a generalized  matrix (Def.~ \ref{d:generalizedmatrices})
 is simple but far-reaching: we use unordered sets to index
 the matrix entries. Thus,
 given alphabets $\A,\B$, an $(\A\times\B)$\,--\,
 generalized matrix $M$ with entries in a ring $R$
 is a function from $\A\times\B$ to $R$.
Matrix addition and multiplication are  defined in
the usual way. If the ring is ordered, 
then the matrices  are partially ordered as
follows: if $\A\subseteq \wh \A$, $\B\subseteq \wh \B$  and for all $(a,b)\in  \A\times  \B$ we have
$M_{ab}\leq
\wh M_{ab}$, then we say $M\leq \wh M$. More
generally the entries can be rectangular matrices when the 
multiplication of entries makes sense. 

One then sees that Bratteli diagrams $\mathfrak B,\wh {\mathfrak B}$ defined by
sequences of 
generalized matrices $M, \wh M$ are {\em nested diagrams} (or {\em
  subdiagrams}), written $\mathfrak B\leq \wh {\mathfrak B}$, iff these are
submatrix sequences, thus $M\leq \wh
M$ meaning this holds for all times. See
Def.~\ref{d:subdiagram}.

This concept is of special importance in our paper for multiple reasons: it unifies the definitions of vertex
and edge subdiagrams; it facilitates definitions of canonical
cover, of distinguished eigenvector sequence, of reduced matrix
sequences, of adic towers, and of the block structure of
matrices and hence the Frobenius Decomposition Theorem.

A key notion in ergodic theory and dynamics is that of the
{\em induced} or {\em Poincar\'e map}, the transformation given by
the first return of a measure--preserving transformation to a
subset. When the
subset has
finite positive measure and the total measure is finite, one has the
following hierarchy of results: the Poincar\'e Recurrence
Theorem  states that the return time of a point to the subset is
a.s.~finite, while  Kakutani's
theorem  tells us  that in the invertible, 
ergodic case the tower constructed over
the induced map on the subset with height equal to this return-time function minus
one is
isomorphic to the original map. This leads to an easy ``tower'' proof of Kac's Theorem
which says that if the total mass is one, the expected return time is finite,
equalling 
the
reciprocal of the measure of the subset. Lastly
 the
Birkhoff Ergodic Theorem complements
this with its (much stronger) statement that
``time average equals space average'', so the
frequency of returns of a.e.~point to the subset exists and equals the subset measure.

Now usually  one would never  consider induced maps on a
measure zero subset, since in general none of these theorems
will be true. 
However for adic transformations, as we shall see, there exist naturally
defined measure zero
subsets for which this is interesting indeed, as {\em every} point is
recurrent. In short, one produces from this a recurrent, possibly infinite
measure, ergodic transformation.
To carry this out, we place on the measure-zero subset a
probability
measure which is invariant for the induced dynamics. Following
Kakutani's idea we then build the tower over this base map. 
 The tower height is the return time, and its expected value may,
depending on the situation, be either finite or
infinite, giving   the tower respectively a finite or an infinite
invariant measure and hence (since the tower is an invariant subset)
an interesting new invariant measure for 
the
original map.
Moreover in   some special cases with this measure infinite, a statement
of the form ``time average equals space average'' can still hold, 
first normalizing by the Hausdorff dimension and then applying a log
average: this is  an
{\em order-two ergodic theorem}  \cite{Fisher92}; see also   \cite{MedynetsSolomyak14}.

Let us suppose   the tower base is the path space of  a subdiagram.
As a consequence of the tools developed in this paper, we present  a necessary and sufficient
condition to decide 
whether the resulting  tower measure is finite or infinite. A
concrete
example  is given by  {\em nested circle
  rotations}, see Example \ref{exam:nestedrot}. Now as mentioned, an irrational  circle rotation, $R_\theta: x\mapsto
x+\theta$ on $\T=\R/\Z$  can be realized as the exchange of two
intervals. Precisely, writing the ratio of shortest to longest of
these as
the continued fraction 
$$\alpha= [n]=[n_0\dots n_k \dots] \equiv \cfrac{1}{n_0+
 \cfrac{1}{n_{1}+\dotsb 
}}\in (0,1),$$ the rotation is $R_{\theta}$ where $\theta=
\alpha/(1+\alpha)$
if the shortest is on the left and 
$1/(1+\alpha)$ if on the right.

Suppose we are given two
such continued fractions $[n]$ and $[\wh n]$ with $n_i\leq \wh n_i$
for all $i$.  We call  $(\T, R_{ \theta})$, $(\T, R_{\wh \theta})$ {\em nested
  rotations} for the following reason: there exists  a subset
  $C_{\theta}\subseteq \T$ such that every point of
  $C_{ \theta}$
  is recurrent for $R_{\wh \theta}$, with the induced (i.e.~first
  return) map $T$ on this
  subset measure-isomorphic 
  to the rotation $(\T, R_\theta)$.
  Writing $T_\theta$ for
  the $R_{\wh \theta}\,$-orbit of $C_{\theta}$ inside of $\T$, then
  $(T_\theta, R_{\wh \theta})$
  is
  the 
  tower over the induced map $(C_{\theta},T)$.
Now if  
$n_i< \wh n_i$  infinitely often, then $C_{\theta}  $ is a
Cantor subset of the circle $\T$.
 As we show below this tower transformation
   is {\em (finite or infinite) uniquely ergodic} in the following sense:  it has a unique invariant
   measure,
   up to multiplication by a constant, 
which is positive finite on some open set.
   The  critical point here is that this holds for the intrinsic
   tower
   topology, that is, for the {\em relative} topology on $T_\theta $ as a subset of $\T$. By contrast,
   this is {\em false} for the ambient circle $\T$;
   indeed,   as a consequence of the general theory
   developed here, the tower measure is either  finite, or infinite, on every nonempty
   open subset of this  circle; setting $\lambda_i=
   ([n_in_{i+1}\dots])^{-1}$ and $\lambda_0^n=
   \lambda_0\lambda_1\cdots \lambda_n$, the tower has  finite total mass iff
   $\liminf({\wh \lambda_0^n})/({\lambda_0^n})<\infty.$
   See Example \ref{exam:nestedrot}  below for the measure theory and
later work in preparation
   regarding the geometry of this example.

We note that infinite-measure unique ergodicity of this sort has been previously
 shown to hold for some related examples: the stable horocycle flow of a
Riemann surface of second type \cite{Kenny83}, \cite{Burger90},
and the Integer Cantor Set
example of \cite{Fisher92}, built there as a 
 substitution dynamical system, and described in the next section, and
 for which we present here  two topologically different though
 measure-isomorphic models, showing in a simpler setting 
 the same phenomenon which occurs for the nested rotations.

 \medskip
 
\subsection{The general framework of the paper}
Given a finite  {\em alphabet } $\A$ and an $(\A\times \A)$
    matrix $M$ with                             nonnegative integer
    entries, we define a graph with vertices $\A$ and $M_{ab}$
    directed edges from $a\in \A$ to $b\in\A$. The collection of edges
    is a nonordered
    set $\E$ called the {\em
      edge alphabet}.
    The {\em two-sided} or {\em bilateral (edge) subshift of finite type} (\sft)
    associated to $M$ is the collection of all biinfinite 
    {\em allowed edge paths} $\Sigma_M=\e= (\dots
    e_{-1}.e_0e_1\dots)$ in this diagram (called the
    {\em graph of the \sft}), that is, such that $e_i\in \E$
    and the edges follow each other in the diagram.
    These are directed edges, with edge $e_k\in \E$ an arrow from a
   symbol $a\in \A$ to $b\in \A$, written
    $e_k^-=a$ and $e_k^+=b$; thus  $e_{k+1}$ {follows}
    $e_k$ iff $e_k^+=e_{k+1}^-$.
    Equivalently such a graph defines a
  matrix and hence an \sft. The corresponding {\em one-sided} or
  {\em unilateral} shift space is
  $\Sigma_M^+= (.e_0e_1\dots)$. These are acted on by the {\em (left) shift
    map} $\sigma:  (\dots
    e_{-1}.e_0e_1\dots)\mapsto (\dots
    e_{-1}e_0.e_1\dots)$, respectively $\sigma:  (.e_0e_1\dots)\mapsto (.e_1e_1e_2\dots)$.
The alphabets are given the discrete topology, the product spaces (the
{\em full shifts}) 
$\Pi_{-\infty}^\infty\E$, $\Pi_0^\infty\E$ the associated product
topologies
and the subshifts the relative topologies. These are generated by the
{\em cylinder sets} such as $[.gh]= \{e\in \Sigma_M^+: e_0=g,
e_1=h\}$ which are clopen sets.
    
  The notion of Bratteli diagram extends
  this to the nonstationary setting. In the usual one-sided version (see
  e.g.~ \cite{ArnouxFisher05} for {\em biinfinite}, {\em bilateral}
  or {\em two-sided diagrams}) we have 
 an infinite, locally finite graph, specified 
by a  sequence  $(\A_i)_{i\geq 0}$ of  finite nonempty sets called (vertex) 
{\em alphabets}, a 
sequence of $(\A_i\times \A_{i+1})$ nonnegative integer matrices $M=
(M_i)_{i\geq 0}$
and {\em edge alphabet} sequence  $\E =(\E_i)_{i\geq 0}$, such that there are
$(M_i)_{ab}$ directed edges of $\E_i $ connecting the {\em symbol} or
{\em letter} 
$a\in\A_i$ to the symbol $b\in\A_{i+1}$. We write $\Sigma_{M}^{0,+}$ for the
collection of    infinite edge paths $\e= (.e_0e_1\dots)$ in the
diagram, and define  $\Sigma_{M}^{k,+}$ to be  {\em the paths
starting at time} $k\geq 0$, $\e=
(.e_ke_{k+1}\dots)$.  The disjoint union 
$
\Sigma_{M}^+\equiv \coprod_{k=0}^\infty \Sigma_{M}^{k,+}$ is the 
{\em nonstationary shift of finite type} (\nsft) defined
from the Bratteli diagram; by definition of the disjoint union
topology, each component is a clopen
set. This 
is acted upon by the left shift map  $\sigma$, which sends the edge path
$\e= (.e_0 e_{1}\dots)$ to 
 $\sigma(\e)= (.e_{1} e_{2}\dots)$ and so on, mapping
 the $k^{\text{th}}${\em component} $\Sigma_{\M}^{k,+}$ to the $(k+1)^{\text{st}}$:

 \begin{equation}
   \label{eq:nsft}
   \begin{CD}
 \Sigma_{\M}^{0,+}  @>\sigma>>    \Sigma_{\M}^{1,+}    @>\sigma>>  
\Sigma_{\M}^{2,+} 
 @>\sigma>> \Sigma_{\M}^{3,+}   
\; \cdots\\
\end{CD}
\end{equation}

The {\em stable set} or {\em stable equivalence class} (in the smooth
setting, the {\em stable manifold})
of a point $\e\in \Sigma_{M}^+$ with respect to $\sigma$,
that is, all points in $\Sigma_{M}^+ $ which are forward asymptotic to $\e$,  is 
$W^s(\e)= \{\tilde{\e}\in \Sigma_{M}^{+}:\, \exists j\geq k\text{ with } 
e_i= \tilde e_i \text{ for all } 
i\geq j\}$. We note that if $\e\in \Sigma_{M}^{k,+}$, then  $W^s(\e)\subseteq \Sigma_{M}^{k,+}$.

An {\em order} $\O$ on 
 the Bratteli diagram is a total order on the collection of all edges
 which enter a given symbol at some level $i\geq 0$. The resulting
anti-lexicographic order on 
 the set of edge paths
totally orders each stable set
$W^s(\e)$ for $\e$ in the \nsft. We write $\NS$ for 
the collection of paths with no successor  in the order, and $\NP$
those with no predecessor; these
are finite sets in each component. The   {\em Vershik map} 
$T_\O$ sends each path to its successor, mapping, for each component $k$, 
$\Sigma_{M}^{k,+}\setminus
\NS$ bijectively to $ \Sigma_{M}^{k,+}\setminus\NP$. 
This successor map  can be visualized geometrically:  draw a stable
equivalence class as an upside-down  tree, with
trunk extending infinitely upwards as time in the Bratteli diagram
goes to $+\infty$; 
the order defines a planar embedding of the tree, with
the Vershik map sending one branch to the next from left to right; see
 Figs.~\ref{F:ChaconTreeB} and \ref{F:ICSCurtainTreeNew}. Writing $ \Cal N$ for the countable
set of forward and backward iterates of $\NS$  and $\NP$ (where defined), then  the restriction of  $T_\O$ to the
invariant set $ \Sigma_{M}^{k,+}\setminus\Cal N$ gives a bijective
map. This defines the associated 
 {\em adic transformation}; its domain of definition is, for a given
 component,   the collection of edge
paths for those
 stable trees which branch out infinitely on both sides.

 In the case of a nonstationary shift space, we emphasize that only the
 shift map is a {\em nonstationary dynamical system} or {\em mapping family}
as defined in ~\cite{ArnouxFisher05}, that is,  
  a sequence of
  maps along a sequence of spaces: the components $ \Sigma_{M}^{k,+}$.
  Each component $ \Sigma_{M}^{k,+}$ is
 acted on by the transverse stationary dynamics of the adic
 transformation $T_\O$
 restricted to that  component.

 In the simplest situation, 
the Bratteli diagram is {\em stationary} 
(all the alphabets and matrices are identical, $M_i= M$ for all $i$) and {\em primitive}:
there exists $n$ such that
all entries of $M^n$ are greater than $0$.

In this case the locally finite, infinite graph of the 
  Bratteli diagram countably covers
the (finite) graph of the  \sft \; $\Sigma_M^+$,  defined from
$M$ as an edge shift space, and each component $\Sigma_{M}^{k,+}$
equals $\Sigma_M^+$. 
Applying the Perron-Frobenius theorem, the matrix $M$ has a  largest eigenvalue  $
\lambda>1$ associated to the unique nonnegative
right and left eigenvectors $\w$ and $\v^t, $ normalized so that $||\w||\equiv \sum
|w_i|=1$ and $\v\cdot \w=1$. The shift map 
has a unique measure of maximal entropy $\mu$, with the entropy of
$(\Sigma^+_M, \mu,\sigma)$
equal to
$\log \lambda$.

This measure has a particularly nice algebraic formula, 
given (as shown independently by Shannon in Information Theory as well
as  Parry in Ergodic Theory)  by
the right and left Perron-Frobenius
eigenvectors, see  \S \ref{ss:ParryMeasures}:
$$\mu[x_0\dots x_{n}]= 
\lambda^{-n}v_{x_0}w_{x_n}$$

By contrast any 
adic transformation $T_\O$ defined by choosing  some order $\Cal O$ on the same path space $\Sigma_M^+$ has zero
topological entropy, and  is both minimal and uniquely ergodic, with
unique
invariant probability measure.
This measure $\nu$, known as the {\em Parry eigenmeasure} or in
Vershik's terminology the {\em central
  measure}, also has a nice formula, given by
modifying Parry's formula to use only the right eigenvector, see  \cite{Fisher09a}:
$$\nu[x_0\dots x_{n}]= 
\lambda^{-n}w_{x_n}$$

The uniqueness of the central
measure in this stationary primitive case   is a consequence of Lemma
2.4 of
\cite{BowenMarcus77}. 
In 1977 Vershik's adic transformations had not yet been
defined, and what Bowen and Marcus proved was uniqueness of a
probability measure  invariant for the tail equivalence
relation, that given by the partition into stable sets; this is equivalent to 
 invariance with respect to the action of a
countable group of 
homeomorphisms, the group $\FC$ of
finite coordinate changes and, with the assumption of primitivity, is
equivalent to invariance for any adic transformation acting on the
space. (The generators of $\FC$ are the involutions
$\gamma:[. e_ke_{k+1}\dots
e_n]\mapsto [.f_{k}f_{k+1}\dots f_n] $ where $e_n^+= f_n^+$).  See \S \ref{p:conserg} and \cite{Fisher09a}.

To study  the nonstationary case, we first
generalize 
 the  standard definition,   saying as in \cite{Fisher09a} that  the
 {sequence} $M= (M_i)_{i\geq 0}$
is {\em primitive} if and
only if for every $k$ there exists an
$n>k$ such that $M_k^n\equiv M_kM_{k+1}\dots M_n$ is strictly
positive. (It follows  that this
holds for {any} larger $n$ as well.)
As for a single matrix, this condition guarantees minimality for any adic
transformation defined on the path space $\Sigma_{M}^{0,+}$ of the
Bratteli diagram, and also for the
 tail equivalence relation and, equivalently,
 for the action of 
 $\FC$. The formula for Parry measure extends in a natural
 way to the nonstationary situation: one now has (at least one)
 right and left
 nonnegative 
 {\em eigenvector 
   sequence} and
 a corresponding sequence of Parry measures, developed in \cite{Fisher09a}. This leads,
 moreover, to 
 a necessary and sufficient condition for unique
ergodicity of the adic transformation: that the associated nonstationary dynamical system, the nonhomogeneous Markov chain of
Parry measures, for {\em any} of these Parry measure sequences, be  mixing. The (stationary) mixing of the Parry
measure was a key ingredient in the proof of Bowen
and Marcus, and that proof carries through to the nonstationary setting, given an appropriate 
definition of nonstationary mixing; see Theorems 1.2, 1.3 of \cite{Fisher09a}. 
This has an equivalent algebraic expression,  that there exist
a single normalized nonnegative right eigenvector sequence of eigenvalue
one: the  {\em Perron-Frobenius property}, and an equivalent geometric
version, that positive cones pulled in from near infinity nest down to
a unique direction. 
In contrast to the case of a single matrix, this condition is
no longer implied by primitivity; see
\cite{FerencziFisherTalet09} for some examples where one has minimality but
not unique ergodicity. 

Regarding the left eigenvector sequence, in contrast to the case of a
single matrix, even if the right
sequence is unique there are always many such sequences, as one  can start
with any  nonzero, nonnegative vector at time $0$. See 
\cite{Fisher09a}.

This all  concerns  the primitive case, but
once one removes the simplifying primitivity
assumption,  the way  is opened to a  rich world of new
phenomena, focused on in
\cite{BezuglyiKwiatkowskiMedynetsSolomyak10},
\cite{BezuglyiKwiatkowskiMedynetsSolomyak13} and in the present study.

We illustrate this with some already interesting stationary examples.
For the first, the Chacon adic transformation, the 
associated matrix is $M=\left[ \begin{matrix}
1 & 1\\
0 & 3
\end{matrix}  \right].$ The (stationary) Bratteli diagram has for vertices 
  the alphabet  $\A=\{0,1\}$; the order on the diagram which defines
  the adic transformation is conveniently specified
  by a substitution, $\rho(0)= 0, \rho(1)= 1101$; algebraically this
  defines an  automorphism of the free semigroup generated by $\A$, and the 
the matrix $M$ is the {\em abelianization} of this
substitution and   automorphism. We recall the construction of the related substitution
  dynamical system.
Beginning with a  {\em fixed point}
associated to $\rho$, $ \lim \rho^n(.1)=(.1101\, 1101\, 0\, 1101
...)\equiv (.a_0 a_1\dots)$, first extend this in an arbitrary fashion to
a biinfinite string
 $a=(\dots a_{-1}.a_0 a_1\dots)$;  then take the
$\omega$\,--\,limit set of $a$  within the compact space 
$\Pi_{-\infty}^{+\infty}$ acted on by the left shift $S$, that
is, 
$\Omega_{\rho,a}\equiv \cap_{n\geq 0}(\cl \{
S^k (a) \}_{k\geq n})$ where $\cl$ denotes the closure.    As shown by Ferenczi~\cite{Ferenczi95},
\cite{Ferenczi02}, this system 
has a single invariant probability
measure.

\medskip

\begin{figure}
$$ 
\xymatrix
{0 & 
0\ar[l] |e& 
0\ar[l] |e& 
0 \ar[l] |e \cdots
 \\ 
1 &
1 \ar[lu]|c\ar@<1ex>[l]|a \ar@<-1ex>[l]|d \ar[l]|b&
1 \ar[lu]|c\ar@<1ex>[l]|a \ar@<-1ex>[l]|d \ar[l]|b&
1  \ar[lu]|c\ar@<1ex>[l]|a \ar@<-1ex>[l]|d \ar[l]|b \cdots
}
$$
\caption{Bratteli diagram for the  Chacon adic transformation: the
  substitution $\rho(0)= 0$, $\rho(1)= 1101$; in this figure, arrows
  depict the direction of the substitution map, while  the
   edge path arrows, indicating the direction of time,  point in the
   opposite sense.
   The substitution corresponds to the edge
  order $a<b<c<d$; edge $e$ is ordered trivially as it is the only
  incoming edge to vertex $0$.
  }
\label{F:ChaconOne}
\end{figure}
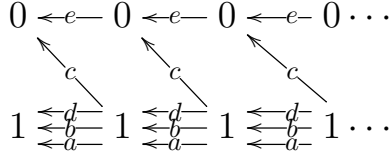

\begin{figure}
\centering
\includegraphics[width=2in]{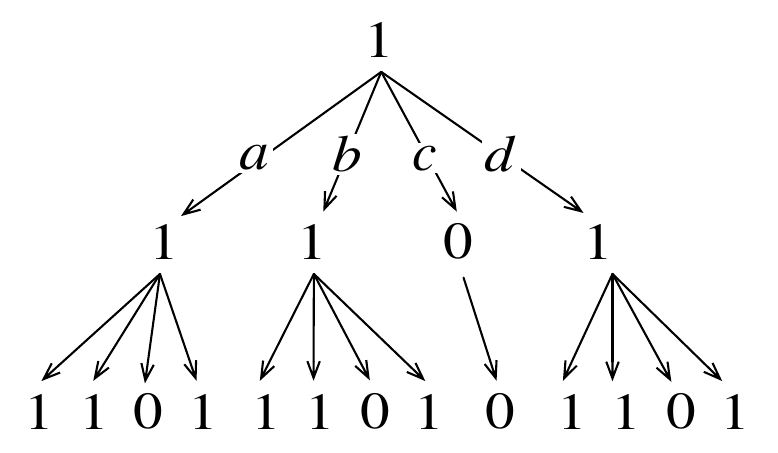}
\hspace{1in}
\caption{Stable tree picture showing simultaneously  points of the Chacon  adic transformation and substitution dynamical
  system, indicating the Livshits factor map:  the substitution maps downwards, vertex and edge
  paths are
  vertical,   while a symbol string  from the
  substitution dynamical system at with time $0$ is horizontal along the
  lowest row.
  }
\label{F:ChaconTreeB}
\end{figure}

A second proof is given in~\cite{FerencziFisherTalet09}; as we showed there, the adic
transformation defined by the substitution is, despite the
nonprimitivity,  
both minimal and uniquely
ergodic. Now
as 
observed by Livshits \cite{Livshits88}, for the primitive case with
an aperiodic string $a$, the adic transformation and the system $(\Omega_{\rho,a},S)$
are naturally isomorphic; this is the {\em recognizable} case
mentioned above. The general statement is that one has a semiconjugacy from the adic
transformation onto the (possibly nonstationary) substitution
dynamical system.
From this canonical {\em Livshits factor map} it follows that unique ergodicity
 holds for the latter, giving a second proof of Ferenczi's result. 

This argument made use of the order on the Bratteli diagram.
A third proof, based on powerful methods developed in
~\cite{BezuglyiKwiatkowskiMedynetsSolomyak10}, is  independent of the
 order: as a consequence of Theorem 2.9 in that paper, see
Theorem \ref{t:basic_thm} below, for the action of
$\FC$ on the adic path space $\Sigma_{M}^{0,+}$,  there are two
ergodic invariant probability measures, one 
 nonatomic and the second (noting that for $\underline
 e\equiv(.eee\dots)$, $W^s(\ee)= \{\ee\})$ is a point mass
on $\ee$, the single fixed point for $\FC$. See Example \ref{ex:Chacon}. 

 But for the adic transformation this
point is removed from the space as it belongs to $\CN$.
To address this last behavior, which can only occur in the nonprimitive
case, we say as in~\cite{FerencziFisherTalet09}
that a group action on a Polish space is {\em minimal} iff every orbit
is dense (note that we do not require the space to be compact) and is 
{\em essentially}
minimal iff it is minimal after removing a countable set (recalling
here that the remaining space will still be Polish, indeed by a
theorem of Alexandrov, any $G_\delta$ subset of a Polish space is Polish) 
and is {\em essentially} uniquely ergodic iff there is a unique nonatomic
invariant probability measure. One can then show (Proposition \ref{p:conserg})
that essential
minimality  
corresponds for an adic transformation and for the action of $\FC$, and
similarly for 
essential unique ergodicity.

Concluding, we have that for the Chacon example  the action of
$\FC$ is  essentially minimal and essentially uniquely
ergodic, while this implies minimality and unique
ergodicity for 
the 
adic transformation $T_\O$ on the noncompact space $ \Sigma_{M}^{+}\setminus\Cal N$, for
any order  $\O$, and hence for the substitution dynamical system.
See Example \ref{ex:Chacon}.

Interestingly, 
the situation changes radically if we slightly alter the 
matrix entries. If, for instance, we replace the matrix by 
$M=\left[ \begin{matrix}
2 & 1\\
0 & 3
\end{matrix}\right],$ then,  again by the methods  of
\cite{BezuglyiKwiatkowskiMedynetsSolomyak10}, 
 there will now be {\em two} nonatomic ergodic invariant
probability measures for $\FC$,  while if we switch the diagonal
entries to have
$M=\left[ \begin{matrix}
3 & 1\\
0 & 2
\end{matrix}  \right]$ then there are still up to rescaling  exactly {two}
locally finite ergodic invariant
measures, except now one of 
them has {\em infinite}
total mass.

A  substitution dynamical system where an infinite measure occurs
naturally 
is the 
Integer Cantor Set map of~\cite{Fisher92}, defined from the
substitution 
  $\rho(1)= 101$, $\rho(0)= 000$
and fixed point $a=\lim \rho^n (0.1)= \dots 000.101\, 000\, 101\, 000000000\, 101\,
000\, 101...$.  See Fig.~\ref{F:ICSCurtainTreeNew}. Now $(\Omega_{\rho,a}, \sigma)$  has a single normalized nonatomic measure
which is positive finite on some open set;  we call this property
{\em infinite measure unique ergodicity}. That the measure is
infinite is related to the fact that the gap lengths in $a$ are so
long as to give an infinite expected return time to the symbol $1$.

\begin{figure}
\centering
\includegraphics[width=5in]{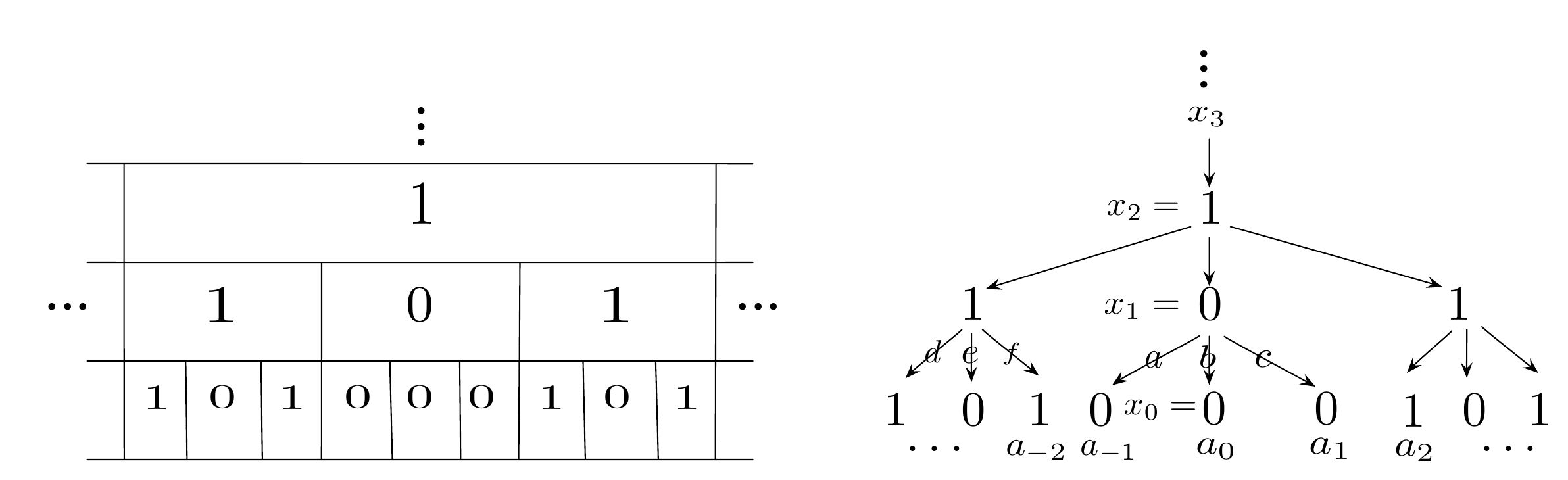}
\caption{A stable equivalence class of the  Integer Cantor Set 
  transformation, depicted in the curtain and stable tree models;
  $(x_i) _{i\in\N}$ is a vertex path 
 for the Bratteli diagram while $(a_i) _{i\in\Z}$ is 
  a
 point in the substitution dynamical system.
 }\label{F:ICSCurtainTreeNew}
\end{figure}

\

The abelianization of this  substitution has the matrix 
$M=\left[ \begin{matrix}
3 & 1\\
0 & 2
\end{matrix}  \right]$ just mentioned, and so one has a
new proof of Theorem 2 of~\cite{Fisher92}, the infinite measure unique
ergodicity:
 after factoring from
 the adic transformation $\Sigma_M^{0,+}$ to the substitution
 dynamical system $\Omega_{\rho,a} $,  the
invariant probability measure for the adic transformation becomes a
 point mass on another substitution fixed point, the string $(\dots 000.000\dots)$ and so is atomic and
 is ruled out of consideration, as this point does not belong to the
 $\omega$\,--\,limit set of $a$.

To study related examples in the nonstationary setting, 
it is useful to introduce a further notion, that of
{\em adic towers}, see Def.~\ref{d:adictower}.
Given two ordered Bratteli diagrams $\mathfrak B\leq \wh {\mathfrak B}$, one nested inside the other by the
removal of edges and/or vertices, then equivalently as noted above (see Def.~\ref{d:subdiagram}) the {generalized matrices}
satisfy $M\leq \wh M$, and 
 we consider the tower over the path space for $M$ inside of the larger
diagram. We are interested in  invariant Borel measures on $\Sigma_M^{0,+}$
which are positive finite on some open set of that path space,  and in the invariant
extensions of these to $\Sigma_{\wh M}^{0,+}$, which
may now 
be 
  infinite on each nonempty open subset of this larger path space
 and yet still be
amenable to study.

One class of examples comes from the nested circle rotations described
above. 
For the most basic example, let us consider a quite different approach to the
Integer Cantor Set map.
Beginning with the $3$\,--\,adic odometer, represented as an edge space by
the stationary $(1\times 1)$ matrix sequence $\wh M=[3]$, consider the 
subdiagram given by $M= [2]$. This last is uniquely ergodic, with its
unique invariant probability measure simply the ``coin-tossing'' 
Bernoulli infinite product measure on the edge path space, which is the
$2$\,--\,adic odometer. 
The tower over this, inside of the $3$\,--\,adic odometer,  is exactly the tower for the Integer Cantor Set
transformation illustrated in
Fig.~4 of~\cite{Fisher92}. It has infinite total mass, and is
infinite-measure uniquely ergodic, as defined above, with respect to
the tower topology,
whereas inside the  $3$\,--\,adic odometer space, the measure   is
infinite on all
nonempty 
open sets, see Lemma \ref{l:finite_on_all}.

We introduce the {\em canonical cover} of 
an adic tower, a path space into which the tower is embedded in a more tractable
way,
though the two towers themselves are measure theoretically  and topologically
isomorphic.
For the present example of $[2]\leq [3]$, the
canonical cover construction yields the matrix 
$\wt M=\left[ \begin{matrix}
3 & 1\\
0 & 2
\end{matrix}  \right]$ just mentioned, and the  tower now embeds in the larger adic transformation as an invariant
open subset. See Figs.~\ref{F:ICSNew} and \ref{F:ICSConstructingCover}.

\begin{figure}
\centering
\includegraphics[width=5.5in]{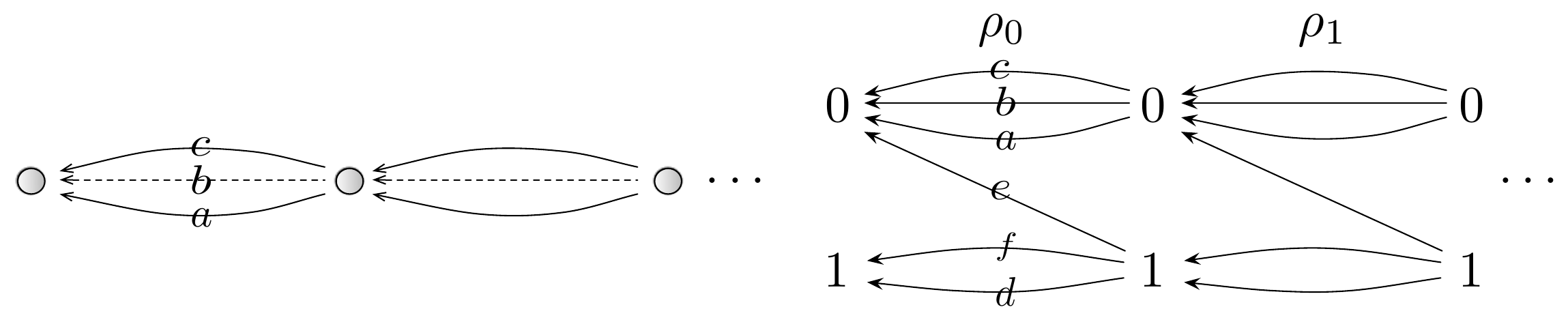}
\hspace{1in}
\caption{The Integer Cantor Set as a tower inside the
3-adic odometer, and  its canonical cover with substitution sequence $\rho_i\equiv\rho$.} 
\label{F:ICSNew}
\end{figure}

In the general case, given generalized matrix 
sequences $M\leq \wh M$ (and without loss of generality with equal
alphabets $\A=\wh \A$), then the canonical cover has matrix
$\wt M=
\left[ \begin{matrix}
\wh M& C \\
0 & M
\end{matrix}  \right] 
$
with $C$ satisfying 
$C_i=\wh M_i- M_i$ for each $i\geq 0$ (Theorem \ref{t:towerring}).
And this is what we have just described for the Integer Cantor Set example,
where $[3]-[2]= [1]$!

\medskip

\noindent
\subsection{Outline of \cite{BezuglyiKwiatkowskiMedynetsSolomyak10}}\label{sss:outline}
We now summarize and outline the remarkable  paper
\cite{BezuglyiKwiatkowskiMedynetsSolomyak10} which provided
key insights which we build on here. Then we describe the present paper
using an analogous outline, so the reader may more easily follow and
contrast both papers. See Remark \ref{r:choices} regarding
differences in notation.

The key result of 
\cite{BezuglyiKwiatkowskiMedynetsSolomyak10} is
a classification of the tail-invariant (equivalently
$\FC$\,--\,invariant) measures which are
 positive finite on
some open set,  for the stationary nonprimitive case, under two technical
assumptions which we do not require, as noted below. Our main objective in this
paper is to extend this in two ways: to 
the nonstationary setting,  part of which has been carried out in the
later paper~\cite{BezuglyiKwiatkowskiMedynetsSolomyak13}, see \S
\ref{ss:BKMS13}, and
to measures which
are finite in their restriction to the path space of some sub-Bratteli
diagram but which may be infinite on each nonempty open set.
The idea of canonical cover will play a key role in this last extension.
To describe our approach, we first 
review and comment on the theory developed in
\cite{BezuglyiKwiatkowskiMedynetsSolomyak10}, which can be
summarized in five steps,
and then sketch our  own versions (and extensions) of those steps and
the results this leads to.

\

\noindent
$(1)$ Given a nonnegative (and not
necessarily primitive) integer matrix $M$, the central measures (the
$\FC$\,--\,invariant probability
measures  
on the edge path space) correspond to 
(strictly) positive nonnegative (right) eigenvector sequences  of eigenvalue
one, normalized so the first vector is a probability vector. The
collection $\V_{\M}^\Delta$ of such sequences is a compact convex set,
see Lemma \ref{l:closedcone} below, whose
extreme points give the ergodic measures. An extreme point corresponds to
a nonnegative right eigenvector  $\w$, with eigenvalue $\lambda$; the associated 
eigenvector sequence with eigenvalue one is simply ${\bf w}_n= \lambda^{-n} \w$.

The correspondence between a central measure $\nu$ and
positive eigenvector sequence $\w$  is direct; considering a  
vertex shift this is
via the formula  $\nu([.x_0\dots x_n]) = ({\bf w}_n)_s$ where
$x_n=s$; the case of  edge shifts is similar. See    Theorems 2.9 and 3.8
of~\cite{BezuglyiKwiatkowskiMedynetsSolomyak10}, and 
 Theorem \ref{t:basic_thm}
below.

We emphasize that although the {\em ergodic} measures correspond to
eigenvector sequences of the particular form ${\bf w}_n= \lambda^{-n}
\w$, and hence to eigenvectors, to represent  the {\em non}ergodic  central
measures we need to consider more general nonnegative eigenvector
 {sequences}, as explained  in \S \ref{s:stationarycase}. This 
foreshadows  the general nonstationary case, for which see 
Theorem \ref{t:basic_thm}.

In \cite{BezuglyiKwiatkowskiMedynetsSolomyak10} there are two 
technical assumptions which our approach does not need. The first is that
the tail equivalence relation
be {\em aperiodic}
(meaning that each equivalence class is infinite).  Aperiodicity guarantees that one doesn't have point masses
for the $\FC$-invariant Borel measures. However point masses can occur
in interesting examples, such as the Chacon and Integer Cantor Set adics.
The second is  that
the diagonal blocks from the Frobenius form of the matrix be
primitive, ruling out the more general irreducible case, regarding which see
Theorem
\ref{t:stationaryirreducible}.

\

\noindent
$(2)$ What is commonly called the  {\em Perron-Frobenius Theorem}
(though this part is actually due to Perron):
that  a primitive nonnegative real
matrix has a
unique normalized nonnegative eigenvector, and this  has a positive
eigenvalue, which is maximal in
modulus. Combined with $(1)$ this gives another proof of 
unique ergodicity for the stationary primitive case considered in~\cite{BowenMarcus77}.

\

\noindent
$(3)$ What we shall call the {\em  Frobenius decomposition theorem},
see \S 3 of \cite{BezuglyiKwiatkowskiMedynetsSolomyak10} and Theorem \ref{t:FrobeniusTheorem} below:
a nonnegative  $(d\times d)$ matrix $M$ can 
be conjugated with
a permutation matrix so as to be put in {\em Frobenius normal form},
an upper triangular block form where the 
diagonal blocks are
 irreducible or 
{\em identically zero} (i.e.~zero for all entries).

For a proof, 
given nonnegative real $M$, we draw a graph with vertices the elements
or {\em symbols}
of the alphabet $\A$ (also called {\em states} of the system) and
with a directed edge from state $a$ to state $b$ 
iff $M_{ab}>0$. We then define $f_M:\A\to \A$ with  $f_M(a)=b$. That is
to say, the graph  defines a discrete dynamical 
system $f_M$ on $\A$. The dynamics generates
a partial order: we say that $a\leq b$ iff $b$ is in the $f_M$- orbit
of $a$,
equivalently there
exists some $m\geq 0$ such that $M^m_{ab}>0$, equivalently  there is a path of directed edges from $a$ to
$b$. In this case we say that 
$a$ {\em communicates
to} $b$.  As usual for partial orders we define $a<b$ iff $a\leq
b$ and $a\neq b$. Since $M^0=I$, $a\leq a$: each state communicates to itself.
See \S \ref{s:Frob}.
The {\em initial states} are the {\em sources} of this dynamical
system, i.e.~ $a$ such that there is no $b<a$,
and the {\em final} states are the {\em sinks}: $b$ such that there is
no $a>b$. 

A collection of states all of whose elements communicate to each
other is called  a {\em communicating class}. These are
the
  {\em basins} of this discrete dynamical system, see \S 4.4 of 
~\cite{LindMarcus95} and  \S \ref{s:stationarycase}. The basins
partition $\A$ and so define an
equivalence relation, with each basin an  
equivalence class of mutually
communicating states. The basins 
inherit the partial
order.

We reorder the alphabet with integers so as to be coherent with this partial
order. The result is that
after a permutation of the alphabet to reflect this grouping into
basins, the
matrix has been conjugated to  upper triangular block form, 
with irreducible or zero blocks on the diagonal. Lastly,  after
 the taking of a power 
to remove periodicity, one is left with  primitive or
zero blocks on the diagonal. Note that each primitive block
corresponds to a basin of communicating states; we say one of these
blocks communicates to another iff that  holds for the basins.
See Theorem \ref{t:FrobeniusTheorem}.
Our conventions, giving upper triangular matrices,  agree with
~\cite{FerencziFisherTalet09} and \cite{Fisher09a},
\cite{LindMarcus95}.

\

\noindent
$(4)$ The 
{\em Frobenius--Victory theorem}~\cite{Victory85}, see also Theorem 6
p.~77 of~\cite{Gantmacher59}, Theorem 4.12 of \cite{Schneider1986}, Theorem 2.1 of
\cite{TamSchneider00} and Theorem 2.1 of \cite{TamSchneider1994core}:
given a nonnegative real matrix in 
(block upper triangular) Frobenius normal form (assuming for
simplicity we have taken a power so the irreducible blocks are
replaced by primitive blocks) we have:

\begin{defi}\label{d:distinguishedStationary}(Distinguished eigenvector, stationary upper diagonal case)
  A nonnegative eigenvector for a primitive subblock (this is unique by the
Perron-Frobenius Theorem) is termed {\em distinguished}  if and only if its
eigenvalue is greater than that for the other subblocks which communicate
to it. 
\end{defi}

The Frobenius--Victory theorem then states two things: the
unique 
nonnegative eigenvector for the primitive subblock $B$
leads to a nonnegative eigenvector for the original matrix $M$ if and only if it is
distinguished; and in that case, the nonnegative eigenvector for $M$ can be reconstructed
from its projection on the subspace of the block via an algorithm. 
In this explicit way, the nonnegative eigenvectors of a nonnegative
 matrix $M$
are in bijective correspondence with the distinguished Perron-Frobenius eigenvectors of the primitive
blocks for the Frobenius normal form.

\

\noindent
$(5)$ 
 The conclusion is given in what we shall call the  {\em
   BKMS-Theorem},
combining point $(1)$ above with the Frobenius--Victory Theorem and with
Lemma 4.2  of~\cite{BezuglyiKwiatkowskiMedynetsSolomyak10}:
conservative 
ergodic Borel
measures which are
positive finite on some open subset of the path space correspond
to Perron-Frobenius eigenvectors of the primitive
blocks, and these have finite or infinite total mass according to
whether or not the eigenvector is distinguished.

This gives the classification noted above, of invariant measures positive finite
on some open subset.

\

Let us review the above
examples in this light.
All three matrices come to us already  in upper
triangular form. For the Chacon adic transformation
$M=\left[ \begin{matrix}
1 & 1\\
0 & 3
\end{matrix}  \right]$,
the nonnegative eigenvector $[1]$ corresponding to the diagonal block $[3]$ is
distinguished, since $3>1$, giving the nonnegative eigenvector 
$\left[ \begin{matrix}
1  \\
2
\end{matrix}  \right]$ for $M$, which defines  one  ergodic invariant 
measure. The second ergodic measure comes from the only other 
nonnegative eigenvector, associated to the upper left block $[1]$; 
this is $\left[ \begin{matrix}
1  \\
0
\end{matrix}  \right]$, and gives the point mass 
on the fixed point for $\FC$ referred to above. 
For $M=\left[ \begin{matrix}
2 & 1\\
0 & 3
\end{matrix}  \right]$, the block $[2]$ yields a finite invariant
measure (the dyadic odometer embedded in the path space) and the block $[3]$ is still
distinguished, so we have two ergodic  finite invariant Borel measures,
corresponding to the two nonnegative eigenvectors $\left[ \begin{matrix}
1  \\
0
\end{matrix}  \right]$, $\left[ \begin{matrix}
1  \\
1
\end{matrix}  \right]$.
For the Integer Cantor Set example,  the eigenvector $[1]$ for the block
$[2]$
of  $M=\left[ \begin{matrix}
3 & 1\\
0 & 2
\end{matrix}  \right]$ 
is not distinguished, since $2\leq 3$. Hence the associated 
measure for the adic transformation is infinite, 
and except for the $3$\,--\,adic measure associated to the block $[3]$ (and
excluded from the substitution dynamical system as explained above),
this is moreover up to multiplication by a constant the only ergodic
invariant Borel measure 
positive finite on some 
open subset, giving a new proof of  the infinite measure unique
ergodicity noted above. (We
calculate that the eigenvector for the eigenvalue $2$ is a multiple of
$(-1,1)$, so there is no nonnegative such eigenvector, agreeing with the
Frobenius-Victory Theorem.)

\subsection{Outline of  this paper}
The aim of this paper is to 
extend the 
dichotomy between finite and infinite measures given in Bezuglyi et al
~\cite{BezuglyiKwiatkowskiMedynetsSolomyak10} to the nonstationary
case, to measures finite on some subdiagram (while possibly infinite on
every nonempty open set), and to adic towers.

We begin, in \S \ref{s:invariant measures}, with an
analysis of the precise relationship of measures which are invariant
in various senses: for an adic transformation, for the tail
 $\sigma$\,--\,algebra, and for $\FC$, the group of finite coordinate changes. 
We also develop some basic tools: the notions of partial
transformations and of generalized matrices, and
the relationship between edge and vertex spaces.

After an initial  treatment of finite invariant Borel measures, described in
$(\wt 1)$ to follow, we move in \S \ref{ss:towers} and \S
\ref{ss:canonicalcover} to develop the ideas of adic
towers and canonical covers. This is what will allow for the reduction
of the general case of measures positive finite on some subdiagram to that of
measures positive finite on some open set.

To treat this more basic situation, we first
extend parts $(1)-(5) $ above to $(\wt 1)$\,--\,$(\wt 5)$ in the nonstationary case, adding
a sixth part $(\wt 6)$ and also an Appendix $(\wt 7)$.

\

\noindent
$(\wt 1)$ For matrix
sequences,   Theorem 2.9 of
~\cite{BezuglyiKwiatkowskiMedynetsSolomyak10} remains valid,
despite the rest of that paper
being focused on the stationary case, so we can use that part  here. See $(i)$ of Theorem \ref{t:basic_thm} for our own proof, included for completeness.

The general version is natural in ~\cite{BezuglyiKwiatkowskiMedynetsSolomyak10} for two reasons. First, as
mentioned in  $(1)$ above, nonnegative eigenvector sequences are necessary for the
representation of nonergodic measures even in the stationary case. And
secondly, both the
statement and the proof for this part  are no
more difficult for a sequence than for a single matrix.

Now the ergodic measures correspond to the extreme points of the set
$\V_{\L}^\Delta$ of normalized positive eigenvector
sequences in both cases, see 
 $(ii)$ of Theorem \ref{t:basic_thm}, 
but for  the stationary  case these are determined by actual nonnegative eigenvectors of the
matrix, which simplifies  the count of the extreme points.
See
\S \ref{ss:Count}.

\

\noindent
$(\wt 2)$ For the nonstationary primitive case, the  Perron-Frobenius Theorem is replaced by the 
study of nested cones and convex collections  of nonnegative eigenvector sequences
~\cite{Fisher09a},
as primitivity no longer implies uniqueness~\cite{FerencziFisherTalet09}.

\

\noindent
$(\wt 3)$ We prove {\em a nonstationary Frobenius decomposition theorem}.
For this, see \S \ref{s:Frob}, we
 decompose the
path space of a nonstationary Bratteli diagram with bounded alphabet
size into {\em primitive streams} together with {\em pool elements}. As a consequence,  after a {\em
  gathering} of the matrix sequence (a taking of partial products,
equivalently a {\em telescoping} of the diagram), and after the
(nonstationary) reordering of the alphabets indicated by the streams, the matrices are now in 
upper triangular block form with primitive or zero blocks (these last
corresponding to the pool elements)  on the
diagonal.

\

\noindent
$(\wt 4)$ In \S \ref{s:Vict} we prove our nonstationary version of the  {\em
  Frobenius--Victory theorem}.
The first task is to extend the standard stationary definition of
Def.~\ref{d:distinguishedStationary}, formulating an appropriate notion of
{\em distinguished eigenvector sequence} of
eigenvalue one for the primitive subblocks.  The guiding idea is to
find a
necessary and sufficient condition for the existence of  nonnegative eigenvector
sequences for the original matrix sequence, with these defined in an 
algorithmic way, which is valid for not only the nonstationary upper
triangular case
but far beyond that, to general submatrices. This is carried out in Def.~\ref{d:dist}.

\

\noindent
$(\wt 5)$ In \S \ref{s:Measures} we give a nonstationary version of
the dichotomy
of Bezuglyi et al in ~\cite{BezuglyiKwiatkowskiMedynetsSolomyak10}:
 distinguished and  non-distinguished extreme points exactly correspond to finite and infinite
ergodic measures.

\

\noindent
$(\wt 6)$ We introduce the idea of
canonical cover for an of adic tower. This allows us to move
beyond the classification given in
~\cite{BezuglyiKwiatkowskiMedynetsSolomyak10} (even in the stationary case) to measures finite
on some subdiagram.
This connects the two cases of  distinguished
eigenvector
 sequences (in the nonstationary case): that for a matrix sequence in  upper triangular 
block form, and that for submatrices of any type.
See   Theorem
\ref{t:inv-meas-subdiag}.

We remark that adic towers and canonical covers are naturally related
to two more familiar ideas: to
 the upper triangular block matrices
which are basic to the nonstationary Frobenius decomposition, and to 
the notion of {\em spacers} in cutting-and-stacking constructions.
In particular, 
given 
a matrix sequence
$$\wt M= \left[ \begin{matrix}
S& C \\
0 & M
\end{matrix}  \right] $$ and considering the adic tower over the
subdiagram for  $\Sigma_{M}^{0,+}\subseteq
\Sigma_{\wt M}^{0,+}$, then  the subblock $S$ corresponds to a generalized notion of
spacers. We explain this more fully in later work.

\

\noindent
$(\wt 7)$ In the Appendix \S \ref{s:stationarycase}, we compare our results and the tools
developed  for the nonstationary
case, specifically the Frobenius Decomposition and  Frobenius-Victory
Theorems and definition of distinguished sequences,   to the classical stationary and periodic cases. This shows
how the nonstationary formulation and viewpoint is not only far more
general, but in fact simplifies our understanding even of the
stationary case in some key respects.

\subsection{Brief summary of main results}\label{ss:Summary}
{\em Basic notions}.
{A {\em reduced} matrix sequence 
has no all-zero columns or rows.
This can always be achieved by removing some letters from the
alphabets
(Remark \ref{r:reduced}). Via the idea of {generalized
  matrices}, see Overview above,
the reduced sequence is conveniently related to the original by $M\leq
\wh M$. See Definitions \ref{d:reduced}
and \ref{d:subdiagram}.
$\FC$ denotes the group of {\em finite coordinate
changes}; a measure on the path space of a Bratteli diagram is
$\FC$-invariant iff it is invariant for the tail equivalence
relation; for nonatomic measures, this is equivalent to invariance for
any adic transformation on the space, see Prop. \ref{p:conserg}.

\

\noindent
{\em (Nonstationary Frobenius decomposition theorem, Theorem \ref{t:FrobDecomp})}

\noindent
A one-sided 
nonnegative integer matrix sequence of bounded size  can be placed
in a canonical upper triangular block form,
after a reordering of the
alphabets. This  is  eventually unique up to  further
permutations.
After a 
gathering of the sequence we can achieve this with square
matrices. The proofs use the geometrical idea of {\em streams} and
the notion of generalized matrices.

\

\noindent
{\em (Importance of eigenvector sequences of eigenvalue one, 
Theorem 2.9
of~\cite{BezuglyiKwiatkowskiMedynetsSolomyak10} and  Theorem
\ref{t:basic_thm})}

\noindent
Given a nonnegative integer $(l_i\times l_{i+1})$ reduced  matrix sequence
$M=(M_i)_{i\geq 0}$, there is a bijective correspondence between the
 ergodic $\FC-$ invariant probability measures and the extreme rays of
 the
 compact convex cone of  such eigenvector sequences.

\
 
\noindent {\em
 (General definition of distinguished eigenvalue sequences (for
 generalized submatrices, hence for general subdiagrams),
 Def.~\ref{d:dist}).}

 \noindent
Given nonnegative real matrix sequences 
$N\leq \wh N$, then an  eigenvector sequence $\w$  of eigenvalue one for
$N$ is $\wh N/N$- {\em distinguished} iff
$\lim_{m\to\infty} \wh N^m (\w)$ exists and is never
$\0$.

\

\noindent {\em
(Nonstationary Frobenius-Victory theorems)}

\noindent
{\bf I} {\em (F-V theorem for upper triangular form, Theorem
\ref{t:FrobVictGeneral})}
Given a nonnegative real matrix sequence $N$  in 
  Frobenius normal form, the extreme rays of
 the
 compact convex cone of   eigenvector sequences of eigenvalue one 
 for $N$ are in bijective correspondence with the
distinguished  eigenvector sequences of eigenvalue one for the
diagonal subblocks.

\

\noindent
{\bf II} {\em  (F-V theorem for subdiagrams, \ref{t:BratFrobVict})}
Given a nonnegative matrix sequence $\wh N$,
each  nonnegative,  never
zero fixed point $\wh \w$ 
for $\wh N$ 
determines, and
is determined by, a $\wh N/N$- distinguished
sequence $\w$ for some primitive generalized
submatrix sequence $N\leq \wh N$.

\

\noindent {\em
(Classification of 
finite and infinite invariant measures)}

\noindent
{\bf I}  (Measures positive on an open
subset, matrix case, Theorem
\ref{t:inv-meas})
Given a nonnegative integer matrix sequence $M$ of finite rank  in 
  Frobenius normal form, the {\em central measures} (the $\FC$-invariant 
  probability measures which are positive on some open subset)  are in bijective correspondence with the
  normalized distinguished  eigenvector sequences of eigenvalue one for the
diagonal subblocks. The {\em infinite} such measures are in 
  bijective correspondence with the
  {\em non}distinguished  such  eigenvector sequences.

 \noindent
{\bf IIa}  (Measures positive on an open
subset, diagram case, Theorem   \ref{t:inv-meas-subdiag})
Given a Bratteli diagram $\wh{\mathfrak B}$ of finite rank with matrix sequence
$\wh M$ and 
ergodic  invariant measure $\nu$, 
suppose that $\nu $ is positive finite on some open subset. Then there exists an
eventually unique 
maximal primitive subdiagram,
with
matrix sequence $M\leq \wh M$,  such that $\nu$ is finite for
$M$. Let $\w$ be the $M$-eigenvector sequence for this restricted
measure. Then $\nu $ is finite for $\wh M$ iff $\w $ is $\wh M/M$- distinguished.

 \noindent
{\bf IIb}  (Measures finite on a subdiagram, or a sub-subdiagram, Theorem   \ref{t:inv-meas-subdiag})
Given $\wh{\mathfrak B}$, $\wh M$ and $\nu$ as above,
suppose that $\nu $ is finite
when restricted to some subdiagram $ \mathfrak B'$ with
matrix sequence $M'\leq \wh M$.   Then by {\bf IIa} there exists a $M'$-eigenvector sequence  $\w'$  for this restricted
measure, and an
eventually unique 
maximal primitive sub-subdiagram
with
matrix sequence $M\leq M'$ and finite positive measure.
Let $\w$ be the $M$-eigenvector sequence for this 
measure. Then $\nu$ is finite for $\wh M$ iff $\w'$ is $\wh M/M'$-
distinguished,
iff $\w$ is $\wh M/M$-
distinguished. The measures infinite for $M'$ can be analyzed in a
similar way.

\

These last proofs make use of the notions of {\em adic tower} and
{\em canonical cover}, see \S \ref {ss:towers} and \S \ref{ss:canonicalcover}. An application is given
to subdiagrams defined by {\em nested circle rotations} and
corresponding to interesting fractal sets, see Example \ref{exam:nestedrot}.

\subsection{Bird's-eye view} Having absorbed
the ideas we develop in this paper, let us now go back
and look at what has been learned.

There are two complementary points of view: given a Bratteli diagram
$\wh {\mathfrak B}$ 
of finite rank, we start with an invariant
measure $\nu$, finite or infinite, and analyze the possibilities. Secondly,  begin
with an interesting subdiagram $\mathfrak B'\leq \wh {\mathfrak B}$, and see what measures can be built which are
finite on some open subset of the path space for this subdiagram,  and
finite, or infinite, on the tower generated by that via the subdiagram.

Now, the diagram is associated to a nonnegative integer matrix
sequence $\wh M$, so  by the Frobenius decomposition theorem we can
permute the alphabets to 
represent this
 in upper triangular form. After a gathering of the sequence (a telescoping of
the diagram) this has been put in  fixed-size form. For any $\FC$-invariant ergodic measure there is a unique primitive 
subblock on the diagonal such that the measure is the tower measure built over
that. The subdiagram $\mathfrak B'$ defined by this subblock  is the unique maximal
primitive subdiagram.

The invariant measure restricted to this
subdiagram 
is either finite or infinite.
We consider first the
finite case, which happens iff the measure is finite on some open
subset of the path space for $\wh {\mathfrak B}$. These measures
correspond to extreme points in the collection of eigenvector
sequences as in \cite{Fisher09a}.

Next, the tower measure (i.e.~the original
measure on the full diagram) also can be either infinite or finite,
being finite exactly when the sequence is $\wh M/M'$-
 distinguished.

 This brings us to measures which are infinite on the subdiagram. Suppose there is a further
 subdiagram with  $M\leq M'$ such that the measure there is
 finite. Then there exists a further maximal primitive reduced $\wt
 M\leq M$  which is analyzed as above.
The full statement is presented in Theorem   \ref{t:inv-meas-subdiag}.

 Conversely, given $\wh {\mathfrak B}$ we can consider some
 subdiagram $\mathfrak B'$,  which is interesting from a geometrical
 or dynamical point of view. We then 
 find the finite invariant measures as above and
 build the tower measure over that, which may be finite or infinite, according to
 whether or not the eigenvector sequence is distinguished. For
 examples of interesting subdiagrams,
 inside of an interval
 exchange transformation, although the original map has a
 finite invariant measure, there can be fractal subsets invariant for
 the induced map, which generate
 invariant towers carrying a measure which can be either finite or
 infinite, again corresponding to whether or not the eigenvector
 sequence is distinguished.
 The example we study explicitly here is for the simplest case: two
 intervals (hence a circle rotation), sketched above.

 These constructions can give interesting examples for
 infinite measure ergodic theory as in \cite{Fisher92} and \cite{MedynetsSolomyak14}.
 We have developed the machinery in this paper with sufficient
 generality and rigor to facilitate future thorough studies of such
 phenomena.
  The notion of canonical cover plays a role here as it
 allows such examples to be re-interpreted as coming from
 nonstationary substitution dynamical systems (aka S-adic systems)
with spacers,  related to fractal geometry.

\subsection{Comparison with
 later papers in the literature}\label{ss:BKMS13}

We have sketched the connections of the present work specifically to 
\cite{Fisher09a},
  \cite{FerencziFisherTalet09}, and especially
  \cite{BezuglyiKwiatkowskiMedynetsSolomyak10}, which was fundamental
  for our approach here. Here we trace connections to some later papers.

  In \cite{BezuglyiKwiatkowskiMedynetsSolomyak13} and later papers, parts of
  \cite{BezuglyiKwiatkowskiMedynetsSolomyak10} are extended to the
  nonstationary setting. We first saw the preprint of \cite{BezuglyiKwiatkowskiMedynetsSolomyak13} after the project
  for the present paper was  largely worked out, so though the work here and in that
  paper was carried out independently, and so is quite different in
  many respects, there is some overlap in the goals
  and hence naturally in some methods and results. We mention that
  those authors first saw our preprints \cite{Fisher09a},
  \cite{FerencziFisherTalet09}
  after \cite{BezuglyiKwiatkowskiMedynetsSolomyak10}  was essentially completed.

In the present paper we give  a
self-contained approach, providing all definitions and proofs for
completeness and also so as to 
avoid the confusion which can arise from mixing notations (see
for example  Remark \ref{r:choices}).

One overlap is in  Therorem 2.6 of
 \cite{BezuglyiKwiatkowskiMedynetsSolomyak13}
 where, making use of Prop 4.6 of
  \cite{bezuglyi2009aperiodic}, 
they prove the Frobenius Decomposition Theorem (Theorem
\ref{t:FrobDecomp} below).

In our treatment of this theorem,
we begin as sketched in point $(3)$ above (see Theorem \ref{t:FrobeniusTheorem})  by giving a
proof for the stationary
case, as this motivates our proof
for the general case, which is likewise geometric and dynamical in
flavor.

In the
course of the 
proof 
we
introduce several notions (streams, pool elements, block structure,
generalized matrices including
virtual matrices and alphabets) which play important roles throughout the
paper.

\begin{rem}\label{r:choices}

  We warn that different conventions are used when  
  defining the matrices in \cite{BezuglyiKwiatkowskiMedynetsSolomyak10},
  \cite{BezuglyiKwiatkowskiMedynetsSolomyak13}  and their subsequent papers, and in the present
  paper.

  The first choice is, given a Bratteli diagram, whether to
  use what one can call a row-  or column- vector convention  for defining the {\em
    incidence matrices} of the diagram. By the row convention we mean
  that if $a$ goes to $b$ in the graph, then entry $ab$ is $1$, or a
  larger integer if there are multiple edges, as explained in point
  $(3)$ above. That is, one looks along
  the row $a$ to see the ``output'' $b$. For the column convention one
  does the opposite. The row convention is 
more usual in Ergodic Theory and
Markov Chain theory.  This is convenient when dealing with Bratteli
diagrams, especially when written from left to right to indicate time
as is done here, since to telescope the diagram, one simply multiplies the
matrices in order. Furthermore, when defining a Markov measure, the integer
matrices $M_k$ are replaced by probability (i.e. ~ {\em
  row-stochastic}: row sum one) matrices $P_k$, and 
measures of cylinder sets are given by the matrix product times an
initial probability row vector, 
cf.~\cite{Billingsley65}, \cite{BrinStuck02}, \cite{AdlerWeiss70},
\cite{RudolphBook}, \cite{LindMarcus95} and see \cite{Fisher09a} for
the nonstationary
case of a 
{\em nonhomogeneous
  Markov chain}.

In \cite{BezuglyiKwiatkowskiMedynetsSolomyak10}  as in
\cite{BezuglyiKwiatkowskiMedynetsSolomyak13},
the incidence matrices are taken to be the transpose of ours, thus
$F_k= M_k^t$. However they then switch to the transposes, denoted
$F_k^T$, which is convenient for considering intersections of cones of
column vectors,
then agreeing with what we do in \cite{Fisher09a}.

The next choice is how to order the alphabet when considering the
Frobenius Decomposition. This is not discussed in
\cite{BezuglyiKwiatkowskiMedynetsSolomyak13}, but in 
\cite{BezuglyiKwiatkowskiMedynetsSolomyak10} they initially make the
same choice as ours (and the opposite of what they had previously
done) which would give upper triangular matrices with no transpose,
but then they make the  opposite
choice to ours as explained in point $(3)$ above, for the partial
order; this results in lower triangular matrices but now written in
the order of time. Then in \S 3.3 of
\cite{BezuglyiKwiatkowskiMedynetsSolomyak10} the transpose is taken.
In \cite{BezuglyiKwiatkowskiMedynetsSolomyak13}, lower triangular
matrices again appear in the Frobenius Decomposition.
Since all this might be confusing,  both for us and for the reader, we thought
it better to give 
a 
fresh and unified 
treatment.
\end{rem}

The idea of extension of a measure from a vertex subdiagram in \S 6 of \cite{BezuglyiKwiatkowskiMedynetsSolomyak10} is much like that
developed here. In \cite{BezuglyiKwiatkowskiMedynetsSolomyak13} edge
subdiagrams are also allowed, but  our notion of subdiagram is more general,  as due to
our use of generalized matrices (as noted above, and see Def.~\ref{d:generalizedmatrices}),
 we can  handle both edge and vertex subdiagrams in a unified way 
 via submatrices  $M_i \leq \wh M_i $.  The fact that the index sets are unordered alphabets
  allows us to include edge alphabets and so to discuss
  state-splitting (Fig.~\ref{F:StateSplit})
  in a convenient way. These definitions are central in all our proofs
  and examples.

In \cite{BezuglyiKwiatkowskiMedynetsSolomyak13}, the question is
addressed as to when the measure on the tower over the subdiagram is
finite or infinite. Several sufficient conditions for finiteness are
proven and examples are given. In the present paper, unlike \cite{BezuglyiKwiatkowskiMedynetsSolomyak13},
we give a necessary and sufficient condition.

We mention that 
for the definition of subdiagrams in \cite{adamska2017subdiagrams},
  vertex as well edges can be erased, however it is not allowed to
  erase {\em all } edges between two vertices. We do not have that
  restriction.  In the present  paper since vertex and edge subdiagrams
are treated in a unified way, the presentation can be
simpler. This is due to the formalism of
generalized matrices. This approach also simplifies the treatment of measures.

Some definitions in \cite{BezuglyiKwiatkowskiMedynetsSolomyak13}  are
more restrictive
than here: a {\em finite measure} is
Borel non-atomic (we do not require non-atomic) and an {\em infinite $\sigma$-finite measure} cannot
be infinite or zero on every clopen set; it must take a finite
positive value on some such set.
In \cite{karpel2012good}, \cite{karpel2012infinite} a measure which is infinite on every nonempty
compact open set is termed {\em defective}; these are excluded from those
studies.

But for us, these unusual measures can provide some of 
the most interesting examples! See \S \ref{ss:examples} for a variety
of such cases which, as noted above, are infinite on {every} open subset of the path space.

We note that there is some overlap between our Theorem \ref{t:inv-meas} and Theorem
3.3 of \cite{BezuglyiKwiatkowskiMedynetsSolomyak13}, which because
of the above restriction is less general than our theorem.

A minor point is that in
other papers
e.g.~\cite{BezuglyiKwiatkowskiMedynetsSolomyak13}, 
the first vertex alphabet is taken to be a single point; for reasons
we explain elsewhere, 
we prefer our  choice of any finite first vertex set. Another minor
point is that we choose
to draw our diagrams horizontally rather than in the vertical form
traditional
in the study of $C^*$-algebras, since our primary interest is Ergodic
Theory, and so the subscript of $M_k$ 
indicates time as in a
Markov chain.

We remark briefly on further work done since
\cite{BezuglyiKwiatkowskiMedynetsSolomyak13}.
 For us the vertex alphabets are bounded, the
  {\em finite rank} condition, though not the edge alphabets (to allow
  for gatherings of the matrix sequence,  i.e.~telescoping of the
  diagram). 
However,  in \cite{adamska2017subdiagrams}  vertex and edge alphabets are assumed finite, but {\em
    not}
  necessarily bounded.

  See e.g.~\cite{bezuglyi2024horizontally}
  and references
there, where the setting is extended beyond finite rank to 
unbounded or countably infinite vertex alphabets; these are
referred to as  {\em generalized} Bratteli digrams.
Basic examples are random walks on the integers, and  the {\em Pascal adic} transformation, see e.g.~
\cite{mela2005dynamical},
\cite{mela2006class}; other examples already occur in the foundational paper
\cite{bratteli1972inductive}.

Regarding the closely related
$S$-adic systems, see
the literature.

\subsection{History and acknowledgements}\label{ss:History}

The origins of this paper go back to several distinct instances where
we had discussed the Integer Cantor Set example of 
~\cite{Fisher92}  with  Jon Aaronson, Pierre Arnoux and  Henk Bruin.

During a conversation with Aaronson in 1989, he remarked that there is
a third  possible approach to the Integer Cantor Set map (the first two,
presented in \cite{Fisher92} being
as a substitution dynamical system, and as a tower over the dyadic
odometer):  the dyadic
odometer is imbedded in a triadic odometer as
a measure zero subset,  see Fig.~\ref{F:ICSNew}; the Integer Cantor Set
map can be viewed as the tower over that base.
 Aaronson recognized this because he had  treated related examples in~\cite{Aaronson1979a}; these
 are as far as we know the first examples of what we are calling here an adic
tower.

After we gave a seminar on the Integer Cantor Set transformation and
 the scenery
flow for Cantor sets, see \cite{Fisher04-2},  Arnoux
came up to talk; he had
been engaged in a deep study of Veech's work on interval exchange transformations and
the Teichm\"uller flow. This was exciting as the same question had occured to both of us:
could the ideas be combined? That is, could one make sense of zooming
as with the  scenery
flow toward small
scales of an interval exchange, and how would this be related to the
Teichm\"uller flow? 
As a case study we took the simplest example of two intervals, which
led to 
~\cite{ArnouxFisher00}. In that paper and the follow-up
~\cite{ArnouxFisher05} we introduced substitution sequences, 
nonstationary dynamical systems, spaces of nested tilings,
 Parry measure sequences, biinfinite Bratteli diagrams and nonstationary subshifts of finite
type.

Lastly, Bruin's question after a talk we gave in 2007 
was the following. We had formulated and proved  the infinite measure unique ergodicity
 of the Integer Cantor Set map, but are there any
other interesting infinite invariant measures? Recalling Aaronson's
remark, we could give an answer, and replied
with Example \ref{exam:bruin}. This 
seemed so interesting as to call for a
deeper investigation.

A second motivation came from the gradual realization while
writing~\cite{Fisher09a} and \cite{FerencziFisherTalet09} that 
nonprimitive adic transformations, for which we only had two examples in mind (the
Integer Cantor Set and the Chacon adic) provided a fascinating and
important general
challenge.  As first steps we proved Theorem \ref{t:basic_thm}, and  began
a study of generalized spacers. At that point we received a preprint from Solomoyak of
~\cite{BezuglyiKwiatkowskiMedynetsSolomyak10} on the stationary
nonprimitive case. Not
only had they anticipated our
result, in  Theorem 2.9  of that paper (since that part includes
the nonstationary case), but they had in the rest of the paper brought into the subject just the right
tools for us to attack the general nonstationary situation, see
$(1)$\,--\,$(5)$ above.

As we began to build on all these ideas and methods, a
second preprint~\cite{BezuglyiKwiatkowskiMedynetsSolomyak13} appeared,
addressing part of the nonstationary case.
Despite some overlap with our two papers- in particular the
nonstationary Frobenius decomposition was independently arrived at by
those authors-we have for completeness and clarity maintained in this paper our own full
proofs, since all of the
completely new
material presented here, such as our definition of distinguished
eigenvector sequences,  our
nonstationary Frobenius--Victory and BKMS theorems, the notion of
canonical cover, and the inclusion of locally infinite measures, depends on our development of this more basic
material, in notation, approach and philosophy. Since however
the basic structure of the present paper was in place before we received
\cite{BezuglyiKwiatkowskiMedynetsSolomyak13},
unlike~\cite{BezuglyiKwiatkowskiMedynetsSolomyak10}
it had little influence on this paper.

We mention that the  ``curtain model''
picture in Fig.~\ref{F:ICSCurtainTreeNew} recalls  what Dennis Sullivan called in lectures and conversations ``Carleson boxes'', connected
with his study of the small-scale structure of doubling maps, and that
Sonin's work on decomposition of nonhomogeneous Markov chains helped
us find an approach to the nonstationary Frobenius theorem; in fact our idea of ``streams'' was partly inspired by Sonin's {\em jets} from
the measure-theoretic context of
\cite{SoninArbitrary92}.

We remark on a common misuse of terminology:  (for a single matrix) what is usually called the
Perron-Frobenius Theorem (the primitive case) is due to Perron; the extension to the
irreducible case, the upper (or lower) triangular normal form (what we are
calling the Frobenius decomposition theorem) and what we are,
following \cite{TamSchneider00}, calling
the Frobenius--Victory theorem are  due to Frobenius, and are all
in Gantmacher's book cited above in $(4)$ of the {\em Outline } of
~\cite{BezuglyiKwiatkowskiMedynetsSolomyak10} and below in \S
\ref{ss:stationarycase}, \ref{ss:Intro2} and \ref{ss:Overviewrob-Vict}.
As far as we
know, the terminology ``distinguished''  eigenvector, which helped to
clarify the importance of this theorem,  originated with Victory. See
the Appendix for a  comparison with the subtleties of the classical (stationary) case.

\section{Invariant Borel measures on Bratteli diagrams and
  towers}\label{s:invariant measures}
In this section, building on the approach and notation developed in~\cite{ArnouxFisher05},~\cite{Fisher09a},
\cite{FerencziFisherTalet09}, we recall basics on Bratteli diagrams and Vershik's
adic transformations. We begin a study of  the ergodic theory of
these maps, demonstrating the basic
result on finite invariant (Borel) measures, Theorem \ref{t:basic_thm}.
After that we develop the fundamentals of sub-Bratteli diagrams and
adic towers and the canonical cover, extending the measure theory to
this more general situation.

For this, we start with the relationship between Bratteli diagrams and 
sequences of matrices, generalized to permit unordered index sets. 
Let $\A= (\A_k)_{k\geq 0}$ be  a sequence of   nonempty finite sets,
called {\em (vertex) alphabets}. This will be the set of vertices of 
{\em level } $k$ in the 
{ Bratteli diagram}.    We  draw
 our diagrams from 
left to right, following 
the
usual ergodic theory or probability theory convention, as in \cite{ArnouxFisher05}; thus levels
correspond to times, see Fig.~\ref{F:StateSplit}. (In the original $C^*$ algebra context
\cite{bratteli1972inductive}, diagrams are  usually drawn
vertically, from top to bottom; the initial alphabet is often taken to be
a singleton, so the diagram begins with a single vertex.)

\subsection{Edge and vertex diagrams, subdiagrams, orders}

We denote by $\E_k$ the
collection of edges connecting  the vertices 
$\A_k$ at level $k$ with those at level $(k+1)$.
This defines the sequence
 $\E= (\E_k)_{k\geq 0}$ 
of {\em edge alphabets}. Thus, for each $k\geq 0$ 
we are given  a function $e\mapsto
(e^-, e^+)$ from $\E_k$ to
$\A_k\times \A_{k+1}$; we draw the edge  $e_k\in \E_k$, oriented to the
right, with   {\em initial symbol} $e_k^-\in
\A_k$ at the tail of the arrow,  and {\em final symbol} $e_k^+\in
\A_{k+1}$ at its head.

\begin{rem}\label{r:emptyfunction}
  For the next definition  we recall how functions are defined in Set
  Theory \cite{Halmos74}. Given sets 
      $A$ and $B$,  a {\em
       relation} $R$ from $X$ to $Y$ is any subset $R\subseteq  X\times
     Y$. That is, the relation is identified with its {\em graph}. 
      A
      {\em function} $f$ from $A$ to $B$ is a 
      relation such that each $a\in A$ is related to some $b\in
      B$, but only one. This is written as
       $f(a)=b$. If at least one of $A$ or $B$  is empty, then
       $A\times B=\emptyset$. Therefore if  $A=\emptyset$ or
       $B=\emptyset$, then there is exactly one relation, hence only
       one function, from $A$ to $B$, the empty set (or empty
       function) $\emptyset$. 
   In what follows, we shall need to allow for  {\em virtual} (empty) alphabets and
   matrices. See Remark \ref{r:virtual}.
\end{rem}

To represent the diagram by matrices, the following abstraction will be useful:
\begin{defi}\label{d:generalizedmatrices}
 Given alphabets $\A,\B$, an $(\A\times\B)$\,--\,
 {\em generalized  matrix} $M$ with entries in a ring $R$
is a function from $\A\times\B$ to $R$; the value at $(a,b)\in
\A\times \B$ is   denoted $M_{ab}$, and the {\em size} of the matrix
is  the {\em index set} $\A\times\B$. The sum of two matrices $M,N$  of the
same size, and the product of $M$ by a scalar in $R$,   are defined as for
functions: $(M+N)_{ab}= M_{ab}+N_{ab}; (r M)_{ab}= r
M_{ab}$. The
{\em product } of an $\A\times \B$ matrix $M$ with a  $\B\times \CC$
matrix $N$ is a matrix $MN$ of size $\A\times\CC$, with
$$(MN)_{ac}= \sum_{b\in\B}M_{ab}N_{bc}.$$

Given an $\A\times \B$ matrix $M$, its {\em transpose} $M^t$ is
  defined to be the $\B\times \A$ matrix $M^t$, with $ba$-entry $M_{ab}$.

As noted above, we allow here  empty alphabets, calling an alphabet
{\em virtual } in this case. Since
given two possibly virtual
alphabets $\A, \B$,
since a  generalized matrix $M$ is  a function 
from $ \A\times\B$ to a ring $R$, then  in the case that one of these is
empty, we have $\A\times \B= \emptyset $, and from Remark
\ref{r:emptyfunction} there is only one such function,
$\emptyset= M\subseteq ( \A\times\B)\times R= \emptyset\times R=\emptyset$. We call this a {\em
  virtual} matrix.

If the
ring is ordered, we partially order the collection of generalized
matrices,  as follows.
Given alphabets $\A, \B$ and $ \wh \A, \wh \B$, let
$M$ and $\wh M$   be generalized nonnegative integer matrices,
respectively of sizes $(\A\times\B)$ and 
$(\wh \A\times \wh\B)$.  We say that 
$M\leq \wh M$ iff $\A\subseteq \wh \A$, $\B\subseteq \wh \B$  and for all $(a,b)\in  \A\times  \B$ we have
$M_{ab}\leq
\wh M_{ab}$.

For where this partial order gets used below see e.g.~ Definitions \ref{d:reduced}
and \ref{d:subdiagram}.
\end{defi}

We note that:
\begin{lem}\label{l:generalizedmatrices}
  For generalized matrices,
  addition and multiplication and transpose satisfy the usual
  properties.

  If the ring $R$ is a field then the collection
$\Cal M_{\A,\B}$ of 
$(\A\times\B)$ matrices is a vector space.
\end{lem}
\begin{proof}
  This is clear when nonvirtual matrices, so we check that for virtual
  matrices,
  thus when $\A$ or $\B$ is empty. Then $\Cal M_{\A,\B}$ has one
  element, $\emptyset$, which is the zero element so this is a vector
  space, of dimension $0$.
\end{proof}

\begin{exam}
  Given $\a$  a fixed alphabet
 and a matrix $M$  all of whose entries are $(\a\times \a)$ matrices
 with entries in say $\R$, 
 then 
 $M$ has entries in the  noncommutative ring $R= \Cal
  M_{\a\times \a}$. So this agrees with the above definition of generalized
  matrix.
  
\end{exam}

\begin{defi}\label{d:blockstructure}
   We extend the definition of generalized matrix
   to include rectangular matrices as entries, as follows.
  For alphabets $\A,  \B$ and an $\A\times \B$ matrix $M$, suppose that
  each symbol, $\a\in \A$, $\b\in \B$ is itself an alphabet, and  that
  each matrix entry $M_{\a\b}$ is an $(\a\times\b)$ matrix with
  entries in some ring.

  Now write $\wh \A=\cup_ {\a\in \A}\a $ and $\wh \B=\cup_ {\b\in \B}\b$
  and
  let $\wh M$ denote the $\wh \A\times \wh \B$ matrix
  with blocks $M_{\a\b}$. In this case we say that $M$ is a {\em block
    matrix}
  for $\wh M$.

  For a concrete example see  Definition \ref{d:fixedsizeFrob} below.

\end{defi}

To show this makes sense, we have:
\begin{lem}\label{l:blockmatrix}
  Matrix multiplication agrees with the block structure.
  That is, given alphabets $\A,\B,\Cal C$,
  then for $ M$ an $ \A\times  \B$ matrix and  $ N$ a
  $ \B\times {\Cal C}$ matrix, and supposing that the symbols in
  each alphabet are themselves alphabets, then
  the $\wh \A\times \wh \B$ matrix  $\wh M$  defined above is
  naturally identified with $M$, and similarly for the
  $\wh \B\times \wh{\Cal C}$ matrix  $\wh N$, and we have that 
  $$\wh M \wh N= \wh{MN}.$$
\end{lem}
\begin{proof}

Each 
  entry  $M_{ab}$ itself an $a\times b$-matrix, and similarly,
 each entry $N_{bc}$ is a $b\times
  c$-matrix, so 
  we check that the multiplication formula
   $(MN)_{ac}= \sum_{b\in\B}M_{ab}N_{bc}$  gives the $ac$- block
   submatrix of the matrix $\wh{MN}$, and equals the $ac$- block of the product
   of $\wh M$ and $\wh N$.
     
\end{proof}

\begin{rem}

  We emphasize that for any finite unordered alphabets (indeed this
  could be extended to {\em any}  index sets, as long as addition of the
  elements is defined)
  the usual rules for matrix addition, multiplication and
  transpose make perfect sense.

  \begin{rem}\label{r:virtual}
  In what follows, all
alphabets and hence matrices {\em will be assumed 
nonvirtual}  (i.e.~ nonempty) unless  expressly
stated otherwise.
Virtual alphabets and matrices will become important
 in \S \ref{s:Frob}  when we discuss the upper triangular block
form.
   
  \end{rem}

For the special case of  nonnegative integer entries, then from a generalized matrix $M$ we
define an {\em edge alphabet } $\E(M)$ to be a set $\E$ together with a
function $e\mapsto (e^-, e^+)$ from $\E$ to $ \A\times \B$, such that
the number of edges with tail (i.e.~ initial symbol) $e^-=a$ and head
(final symbol) $e^+=b$ is $M_{ab}$.
This illustrates one of the reasons we need to allow for unordered
alphabets: 
even in the case where alphabets are ordered, in general there is
 no natural order for the edge alphabets.
\end{rem}

\

\begin{defi}(from generalized matrices to Bratteli subdiagrams)

Given a Bratteli diagram, from the pair $(\E, \A)$ we define an $(\A_k\times \A_{k+1})$ sequence of 
 nonnegative generalized integer matrices
$M= (M_k)_{k\geq 0}$, with the $ab^{\text{th}}$ matrix entry of $M_k$
equal to the number of edges from 
symbol $a$ in $\A_k$ to $b$ in $\A_{k+1}$.

Conversely, given an  $(\A_k\times \A_{k+1})$ sequence of 
 nonnegative generalized integer matrices
$M= (M_k)_{k\geq 0}$ we know the alphabet  sequence $\A$; we define
the   edge 
set sequence 
  $\E$   such that $\E_k$ has $\wh l_k=\sum_{a,b\in \A_k,\A_{k+1}} (M_k)_{ab}$ elements, with
  corresponding heads and tails.
We write  
$\mathfrak B_{ \A, \E, M}$, or simply 
$\mathfrak B_{\A, \E}$ or $\mathfrak B_ M$, for the
Bratteli diagram determined by $(\A, \E)$ and equivalently by the
generalized integer matrix sequence $M$.

We extend this to ring-valued matrices as follows: to an  $(\A_k\times \A_{k+1})$ sequence of 
 ring-valued matrices
$N= (N_k)_{k\geq 0}$, we associate  a $0-1$ matrix sequence $L$ where
the nonzero elements of $N_i$ and $L_i$ correspond. The Bratteli
diagram associated to $N$ is then that for $L$. We label an edge $e_k$
with $a=e_k^-$, $b=e_k^+$ by the $ab$- entry of $N_i$, thinking of
this as a ``weight'' on that edge, as e.g.~ for (row--stochastic) probability matrices. Matrix
multiplication then has the geometrical interpretation of multiplying
followed by  summing
these weights along the edges. In practice, we shall need this notion
for nonnegative real entries.

\end{defi}

Setting $l_k= \#\A_k\geq 1$,  given a choice of orders
  on $\A_k$ we may then take 
$\A_k= \{1,\dots, l_k\}$, and this choice
determines an $(l_k\times l_{k+1})$ matrix sequence in the usual
sense. 
A different sequence  of orders
corresponds to conjugation by a sequence of permutation
matrices.

An {\em (allowed) edge path } in the diagram  $\mathfrak B_M$,
also known as an edge {\em string}, is $\e= (.e_0 e_1\dots)\in
\Pi_0^\infty\E_i$ such that for all $k\geq 0$, $e_k^+= e_{k+1}^-$. 
We denote by  $\Sigma_{M}^{0,+}$
 the collection of all edge paths.

The combinatorial space 
$\Sigma_{\M}^{0,+}$ with the relative topology inherited from the product topology on 
$\Pi_0^\infty\E_i$ (where each $\E_i$ has the discrete topology)
is called a {\em Markov
  compactum}~\cite{Vershik81},~\cite{LivshitsVershik92}. It is
understood that this comes
equipped with
its Borel  $\sigma$\,--\,algebra.

We also use vertex path spaces, defined
as follows.
Given
a sequence $(L_k)_{k\geq 0}$ of   $(\A_k\times \A_{k+1})$ 
 nonnegative generalized matrices with entries in $\{0,1\}$,  
by an {\em (allowed) vertex path } we mean 
$x=(. x_0 x_1\dots )$ with $x_k\in\A_k$  such that the $(x_k x_{k+1})^{\text{th}}$
entry of $L_k$ equals $1$.
An allowed edge path 
$\e=(.e_0 e_1\dots)$ in the Bratteli diagram of $L$ determines the allowed vertex path 
$x=(.x_0 x_1\dots) = (.e_0^- e_1^-\dots) $. For such $0-1$ matrices  there
are no multiple edges, so this correspondence is one-to-one, with
 the vertex and edge path
spaces  canonically identified (technically speaking this
correspondence is given by a bijective {\em two-block code},
$x_kx_{k+1}= e_k^- e_k^+\mapsto e_k$) .
In what follows  we 
use $L$ when  referring to a vertex representation, with
$\Sigma_{\L}^{0,+}$  denoting the collection of vertex paths, and
$\Sigma_{\M}^{0,+}$ for the edge path space even if the entries of  $M$
happen to be $0$ and $1$. See
~\cite{LindMarcus95},
\cite{Kitchens98} for edge and vertex shift spaces in the stationary
case. (From now on unless mentioned explicitly we shall use edge spaces exclusively).

 Recalling  Def. 2.6 of~\cite{ArnouxFisher05}, we have the following notion: 
 \begin{defi}\label{d:gathering}   
 Given $N=(N_k)_{k\geq 0}$ an
$(\A_k\times \A_{k+1})$ matrix
sequence,
then for $k\leq n$, we write $N_k^n$
   for the product $N_kN_{k+1}\dots N_n$, which is $(\A_k\times
 \A_{n+1} )$,  so $N_k^{k}= N_k$. Thus always for $k\leq m\leq n$, $N_k^m N_{m+1}^n=
 N_k^n$. 
  A {\em gathering} of
a sequence of generalized ring-valued matrices is a new sequence give by taking partial
products along a subsequence. Thus, given
 $0=n_0< n_1<\dots$,  we call the $(\A_{n_i}\times
 \A_{n_{k+1}}) $ sequence $\wt N_k=
 N_{n_k}^{n_{k+1}-1}\equiv N_{n_k}N_{n_k+1}\cdots N_{n_{k+1}-1}$ the
  {\em  gathering along the times} $(n_k)_{k\geq 0}$. If $\wt N$ is a
  gathering of $N$ then we say that $N$ is a {\em dispersal} of $\wh
  N$.

For the case of nonnegative integer entries, the edge alphabets for the gathered sequence are naturally labelled 
by the finite edge paths of the original diagram, so $\wt \E_0= \{\wt e_0= (e_0e_1\dots
e_{n_1-1}) \}, \wt \E_1=\{\wt e_1=(e_{n_1}\dots e_{n_2-1})\}$ and so
on; we call $(\wt \E_k)_{k\geq 0}$
the  {\em gathered edge alphabet} sequence.
The gathering of this
matrix sequence 
defines  what is often called a {\em telescoped} Bratteli diagram while 
a dispersal gives a {\em microscoping}
of the diagram. (One can 
think of collapsing a pirate's spyglass, when telescoping; we like to
use the terms gathering and dispersal to avoid  confusion as to which
is which!)
\end{defi}

Given a multiple-edged diagram, and hence equivalently a nonnegative
integer matrix sequence $M$, we can produce a single-edged diagram
in the two canonical ways illustrated in Fig.~\ref{F:StateSplit}; even
if the  alphabets are ordered,
 the formalism of
generalized matrices will  now be necessary, since as noted above the edge alphabets have no
reason to be ordered. 
 Assuming the edge alphabet $\E_k$ is nonempty, so that $\# \E_k\neq 0$, 
we define
$0-1$ generalized matrices $A_k, B_k$ as follows:
$A_k$ is $(\A_k\times \E_k)$ with
the $ae_k-$ entry of $A_k$  equal to $1$ iff for $e_k\in \E_k$ we have 
$e_k^-=a$, while $B_k$ is $(\E_k\times \A_{k+1})$ with  the $e_kb-$ entry equal to $1$ iff 
$e_k^+= b$; we then have  $M_k= A_kB_k$. This procedure, called {\em
  state-splitting} or {\em
  symbol splitting}, 
 factors 
 each matrix as a product of 
 $0-1$ generalized matrices;  defining $L_k= B_k A_{k+1}$, we have a second sequence of
$0-1$ generalized matrices, now $(\E_k\times \E_{k+1})$, which 
give the allowed transitions from edges in $\E_k$ to those in
$\E_{k+1}$, where the edge alphabets are vertex sets of a new graph. 

So the sequence $(A_0B_0A_1B_1\dots)$ is a dispersal of both $M$  and of $L$.

\begin{figure}
\centering
\includegraphics[width=4.5in]{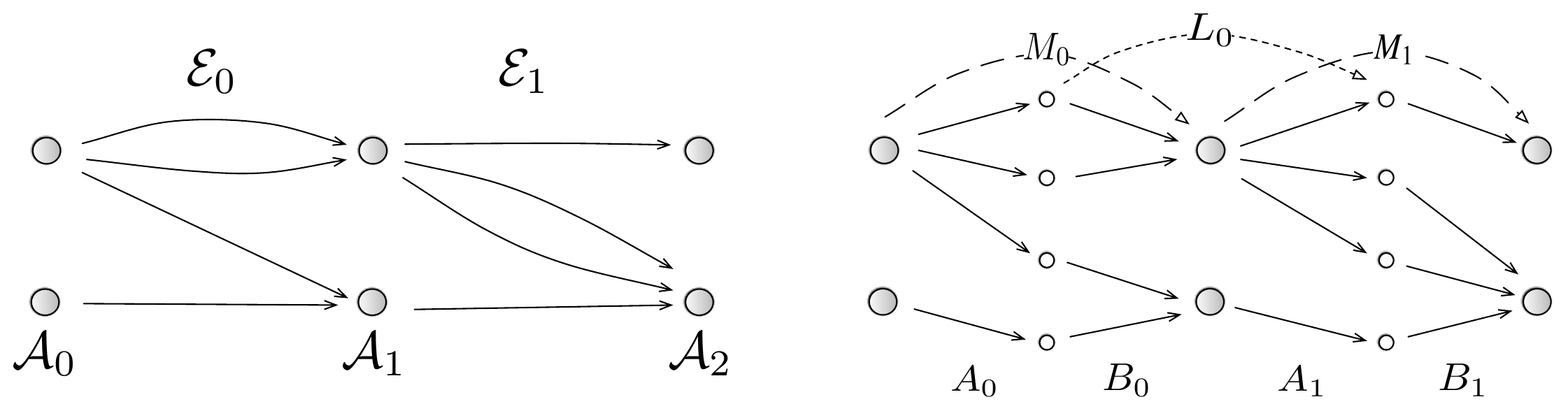}
\hspace{1in}
\caption{Symbol splitting: $A_k B_k = M_k$  and $ B_k
  A_{k+1}=L_k$} 
\label{F:StateSplit}
\end{figure}

To introduce a shift map, 
we write $\sigma M= (M_1, M_2,\dots)$ for the shifted matrix sequence  and 
 define $\Sigma_{\M}^{k,+}= \Sigma_{\sigma^k M}^{0,+}$ for all $k\geq 0$. 
This is the $k^{\text{th}}$ {\em component} of the disjoint union 
$\Sigma_{\M}^+\equiv 
\coprod_{k=0}^\infty \Sigma_{M}^{k,+}$, the 
(one-sided) {\em nonstationary shift of finite type} (\nsft) defined from the 
matrix sequence, see the Introduction.
We give each component the product topology 
inherited from $\Pi_k^{+\infty}\E_i$, so $ \Sigma_{\M}^{k,+}$ is a
compact topological space; each component is 
declared to be open in $\Sigma_{\M}^+$ (and so is clopen, since the
union of the remaining components is open). 
We define the {\em  word metric}  
$d(\cdot, \cdot)$ on $\Sigma_{\M}^+$ as in~\cite{ArnouxFisher05},
 which gives this same topology, as follows.
Beginning with 
 the $0^{\text{th}}$ component,
for $0\leq j\leq k$, we write
$w(j,k)$ for the number of allowed edge paths from
$j$ to $k$ 
(so by definition $w(j,j)=l_j$). Then, 
given $\e\neq \f$ in    $\Sigma_{\M}^{0,+}$,
we  define $d(x,y)=1 $ if $e_0\neq f_0$; otherwise, 
\begin{equation}\label{eq:distance}
d(\e,\f)= \{(w(0,m))^{-1}\}
\end{equation}
where $m$ is the largest nonnegative integer such that
$e_i=f_i$ for $0\leq \ i\leq m$. We  extend this metric to the
$k^{\text{th}}$ component by the same formula applied to the shifted
matrix sequence. Lastly, points in
different components are  declared to be distance $1$ apart.

The left shift map  $\sigma$  sends the edge path
$\e= (.e_k e_{k+1}\dots)\in \Sigma_{\M}^{k,+}$ to 
 $\sigma(\e)= (.e_{k+1} e_{k+2}\dots)\in \Sigma_{\M}^{k+1,+}$ (the
 ``decimal point'' is placed to
the left of the zeroth coordinate in each component).
This defines the nonstationary shift dynamics of the {\em mapping
  family}
 $(\Sigma_{\M}^+, d(\cdot,\cdot),
\sigma)$,
a sequence of maps along the sequence of components of the \nsft, as
in ~\cite{ArnouxFisher05}; see
the diagram  $\eqref{eq:nsft}$.

 We
need this notation: 
given an allowed edge string $\e\in \Sigma_{\M}^{0,+}$ and $m\geq 0$, we write
\begin{equation}
  \label{eq:cylinder}
  [.e_0\dots e_{m}]=
\{\tilde \e\in \Sigma_{\M}^{0,+}:\, \wt e_0=e_0, \dots, \wt
e_{m}=e_{m}\};
\end{equation}
we call a nonempty 
subset of $\Sigma_{\M}^{0,+}$ of this 
form a  {\em thin (edge) cylinder set}. 
For a vertex path space $\Sigma_{\L}^{0,+}$, with $(L_i)_{i\geq 0}$ an 
 $(\A_k\times \A_{k+1})$ sequence of 
generalized $0-1$ matrices, 
a {\em thin (vertex) cylinder set} is
$[.x_0\dots x_{m}]=
\{\tilde x\in \Sigma_{\L}^{0,+}:\, \wt x_0=x_0, \dots, \wt
x_{m}=x_{m}\}\subseteq \Sigma_{\L}^{0,+}$.
Cylinders for other components are defined similarly.

We define general cylinder sets to be nonempty sets with restrictions in other slots, for
example $[.* e_1* e_3]\subseteq \Sigma_{\M}^{0,+}$ equals  $ \{( .f_0f_1f_2f_3\dots):\; f_1=e_1, f_3=e_3\}$;
the $*$ indicates an arbitrary allowed entry. Note that the cylinder subset $[.*\dots * e_k\dots e_{m}]\subseteq  \Sigma_{\M}^{0,+}$   can be uniquely
decomposed as a union of thin
cylinders of the form
$[.f_0\dots f_{k-1}e_k\dots e_{m}],$ where
these are the allowed left-continuations.
The same holds for vertex
paths. Of course the reason for the name ``cylinder set'' is that
these are products of subsets of the ``axes''  $\E_0, \E_1, \dots$
like an actual cylinder in $\R^3$.

\begin{defi}\label{d:reduced}
  We say
 an $(\A_i\times \A_{i+1})$  generalized real  matrix sequence $(N_i)_{i\geq 0} $ is {\em column--reduced} 
iff these matrices have no all-zero 
columns,  {\em row--reduced} if it has no all-zero rows.

We note that the property of a row or column being 
all-zero, and hence the notion of a reduced matrix
sequence, makes sense for generalized matrices.

\end{defi}

We have immediately:
\begin{lem}\label{l:reduced}
\item{(i)}  For nonnegative integer matrices, being row--reduced is equivalent to that any admissible finite edge path
$e_k\dots e_{m}$  for $k>0$ can be
continued infinitely to the right;  being column--reduced is equivalent to that any
allowed string can be
continued to the left.
\item{(ii)}  
If the matrix sequence is both
row--and-column--reduced then 
any allowed string  $e_k\dots e_{m}$ defines a (nonempty) general
cylinder 
 set
$[.*\dots * e_k\dots e_{m}]\subseteq  \Sigma_{\M}^{0,+}$.
\end{lem}\ \ \qed

\begin{rem}\label{r:reduced}
For an example, consider the stationary vertex shift with 
$L=\left[ \begin{matrix}
1& 1 \\
0 & 0
\end{matrix}  \right] $ and with alphabet $\{a,b\}$; the word $ab$ is
allowed but $[.ab]$ as defined by equation \eqref{eq:cylinder} is empty, hence
not a cylinder.
On the other hand, for the matrix $\left[ \begin{matrix}
1& 0 \\
1 & 0
\end{matrix}  \right] $ the word $ba$ is allowed and $[.ba]$ is
nonempty, hence a cylinder, but the shifted set 
$[.*ba]$ is empty as there exists no $c$ such that $c b a $ is an allowed string. 
Viewed as edge shifts, let for the first example $e$ denote the edge
with $e^-=a, e^+=b$; then $e_0= e$ is an allowed finite string but
$[.e_0]=\emptyset$.
For the second, with $e$ denoting the edge
with $e^-=b, e^+=a$; then $e_0= e$ is an allowed finite string and
$[.e]$ is nonempty, but $[.*e]=\emptyset$.

Requiring that cylinder sets be nonempty avoids technical issues,
 for instance in the definition of central measures.
Fortunately one can always produce an essentially equivalent matrix
sequence that is reduced. Precisely,
 there exists
 a canonical 
reduced sequence 
$\breve N\leq N$ in the sense of
Definitions \ref{d:generalizedmatrices} and \ref{d:subdiagram},
with possibly smaller alphabets. 
For $M$ with nonnegative integer entries,
this canonical reduced sequence $\breve M$ has two special properties:

\item{(i)} it has the same allowed (infinite) edge paths as the original non-reduced sequence, 
i.e.~$\Sigma_{\breve M}^{0,+}= \Sigma_{\M}^{0,+}  $, and

\item{(ii)} as noted above, every allowed finite symbol string defines a (nonempty) cylinder
set.

See also Proposition 2.2.10 of
\cite{LindMarcus95},
where a reduced graph for a (stationary)  \sft \, is termed {\em essential}.

Our proof here follows that of 
Lemma 2.2 of \cite{Fisher09a}. In that proof we work with $0-1$
matrices, and assume the the shift space is
nonempty, whence the matrix sequence and also the reduced sequence 
have no virtual matrices.

Here we treat the general case, allowing for virtual  (i.e.~empty) matrices as well as virtual vertex and edge alphabets.
We
note that if a matrix $M_k$ is all-zero, then the edge alphabet $\E_k$
is empty.

As we shall see, 
the reduced matrix sequence still can be defined but will have some
virtual alphabets and matrices. For example, if a matrix is all-zero, then in the reduced matrix sequence we will replace
this by a virtual matrix.

\end{rem}

\begin{lem}
  Given an alphabet sequence $\A= (\A_k)_{k\geq 0}$ and an
$(\A_i\times \A_{i+1})$  generalized nonnegative real matrix
 sequence $M=(\M_i)_{i\geq 0} $, and corresponding edge alphabet
 sequence $\E= (\E_k)_{k\geq 0}$, 
then there exists a unique sequence $\wt M\leq M$ with alphabets $\wt
\A\subseteq \A$, $\wt\E\subseteq \E$ such that:
\item{(i)} $\wt M$ is reduced;
  \item{(ii)} if the matrices have nonnegative integer entries, then $\Sigma^+_{\wt M}= \Sigma^+_{M}$, i.e.~ the
    corresponding nested Bratteli diagrams $\wt{\mathfrak B}\subseteq \mathfrak B$ have
    exactly the same infinite edge paths.

\end{lem}
\begin{proof}
We follow the proof 
of  Lemma 2.2 in \cite{Fisher09a},
making the needed changes. We define a $0-1$ matrix sequence $L=(\L_i)_{i\geq 0} $,
of the same dimensions to have a $0$ iff $M$ does. We shall produce a
reduced version $\wt L$ of $L$ and then define $\wt M_k$ to have the
same size as $\wt L_k$ except with the entry copied from $M_k$. This
matrix sequence 
is then also reduced. Moreover, it is clear that (for the integer
entry case) since the edge paths
of $L$ and $\wt L$ correspond, so do the edge paths of $M$ and $\wt
M$. 

To give the proof for $L$ then, the only new ingredients added to the
case studied in \cite{Fisher09a} are 
generalized matrices, including the
possibility of virtual alphabets or matrices.

First
we note that the property of a row or column being 
all-zero, and hence the notion of a reduced matrix
sequence, makes sense for generalized matrices. Now for the proof, we consider the
finite sequence $L_0,L_1,\dots , L_n$ and produce a row-reduced version.
We list all the vertex alphabets and matrix elements of this finite sequence, giving a
finite set $X$, and form the product space $\Cal F=\{\0,\1\}^X$, which is the
collection of all functions from $X$ to $\{\0,\1\}$. So the matrix
sequences correspond bijectively to this function space $\Cal
F$. Changing a  symbol from $\1$ to 
$\0$ corresponds to removing a given letter or matrix.
The collection $\Cal F$ is partially ordered coordinate-wise: one
function is $\leq$ a second iff that holds for each coordinate.
We
define a decreasing  operator $\Cal R_n$ on $\Cal F$,  which removes all zero rows from
$L_n$, giving $\wt L_n$, also removing the corresponding letters from $\A_n$, giving
$\wt A_n$. We also remove the corresponding columns from $L_{n-1}
$ giving $\wt L_{n-1}$ . This results in 
 the vertex alphabet sequence $\A_0, \A_1,\dots, \A_{n-1},
\wt\A_n$ and matrix sequence $\L_0, \L_1,\dots, \L_{n-2},\wt \L_{n-1},
\wt\L_n$, which are of compatible sizes. If $L_n$ is all zero, then $\wt L_n$ is
virtual as all rows have been removed, and then so is $\wt A_n$.
Continuing to the next step, we define an operator $\Cal R_{n-1}$
which removes all zero rows from $\wt L_{n-1}$ and corresponding letters
from $\A_{n-1}$ and columns from $L_{n-2}$. 
In this way we get operators $\Cal R_n, \Cal R_{n-1}, \dots \Cal R_0.$ This
last operator removes all zero rows from $L_{0}$ and corresponding
letters from $\A_0$.

Relabeling the matrix sequence as $L_0,L_1,\dots , L_n$, we note that
the matrices $L_0,L_1,\dots , L_n, L_{n+1}$ and alphabets $\A_0, \A_1,\dots, \A_{n},\A_{n+1}$ 
are still compatible, hence so are 
$L_0,L_1,\dots , L_m$ with $m>n$  with corresponding alphabets; we
apply the above procedure to this sequence and note that this is
consistent with the previous steps since  the results are nonincreasing for the order on $\Cal
F$.
This yields  an infinite row-reduced sequence,  relabelled as
$(L_i)_{i\geq 0}$, with corresponding alphabets.

Next we remove zero columns. We define an operator $\Cal C_0$ on $\Cal
F$ which removes the all-zero columns from $L_0$, and the
corresponding letters from $\A_1$. Next,  $\Cal C_1$ removes the all-zero columns from $L_1$, and the
corresponding letters from $\A_2$, and also corresponding rows from
$L_2$. Applying these in the order $\Cal C_0$, $\Cal C_1,\dots$ results
in an alphabet sequence $\wt \A$ for matrix sequence   $\wt  L$ which is now both row-and column-reduced.

Lastly, if the reduced sequence $\wt  L$ has a virtual alphabet
$\wt\A_k$
then both $\wt L_{k-1}$ and $\wt L_k$ are virtual matrices.
We note that if $\wt L_k$ is virtual, that does {\em not} force $\wt \A_k$
or
$\wt \A_{k+1}$ to be virtual, since the empty function is a subset of
any set (in particular of $\wt \A_k \times \wt \A_{k+1}$) (think of
the Bratteli diagram, with no edges in one spot).
If 
$\wt  L$ has no virtual matrices, then the edge shift spaces for
$L$  and $\wt  L$ are identical, and hence also for $M$  and $\wt  M$,
and these are nonempty as there exist infinite allowed edge paths. If
$\wt  M$ has a virtual matrix $\wt M_k$, then the edge space is empty
as is  the edge space for $M$.  Thus in all cases the edge space for
$M$ and the reduced sequence are the same.

\end{proof}

\begin{rem}
  We remark that the same proof works for biinfinite matrix sequences,
  like those encountered in our forthcoming work.
 \end{rem}

Now we return to focus on nonnegative integer  matrix sequences  and their Bratteli diagrams.
The  {\em stable set} of an edge path $\e\in \Sigma_{\M}^{k,+}  $ is
the collection  of points in the \nsft \,  which are forward asymptotic to
$\e$, 
$W^s(\e)= \{ \wt \e= (.\wt e_k \wt e_{k+1}\dots):\, \exists m\geq   0 \text{ with }
\wt e_i= e_i \;\forall i\geq m\}$. 
These are the equivalence classes for the {\em stable} or {\em tail}
 {\em equivalence relation} on  the \nsft \, $\Sigma_{M}^+$. 
Transversal to the   shift dynamics is the action  of $\FC_k$, the group of finite
coordinate changes on the component
$\Sigma_{\M}^{k,+}  $,  generated by the involutions $\gamma$ which
interchange two thin cylinder sets: $\gamma([. e_ke_{k+1}\dots
e_n])=[.f_{k}f_{k+1}\dots f_n] $ where $e_n^+= f_n^+$, the tails being
left unchanged by $\gamma$. Note that the $\FC_k$\,--\,orbit of $\e$ is its
stable set $W^s(\e)$.

This brings us to
Vershik's construction of a  single map which has the same
orbits as this group, after the removal of a countable set.
We shall use the following notation:
for  $b\in \A_{k+1}$  we set $\E_k^+(b)= \{ e_k\in \E_k:\, e_k^+= b\}$.
Thus the collection of sets $\{\E_k^+(b):\, b\in \A_{k+1}\}$ partitions
$\E_k$ according to the final symbol.

We  define 
  a {\em stable order} $\O=\O^s$ on $\E$ to be  a sequence
of partial orders $\O_k$ on 
$\E_k$, which restrict to a
linear
(i.e.~total) order $\O_k(b)$ on
each $\E_k^+(b)$. Thus,
edges of $\E_k$ are comparable  for $\O_k$ if and only if they have the same
final symbol.
By an {\em ordered Bratteli diagram} $\mathfrak B_{\A, \E,\O}$ 
one means the diagram $\mathfrak B_{\A, \E}$  together with a stable order $\O$.

This linear order on the sets of edge alphabets for a given time
and final symbol  then passes to a
partial order on the collection
of edge paths.
Thus
for $\e, \wt \e$ in the same stable equivalence class, 
 supposing 
$n$
  is the
 least  
index such that $\wt e_i= e_i \; \forall i\geq  n$,
 then for $k= n-1$, $ e_k^+=\wt e_k^+\equiv b$; we declare that
$\e< \wt \e$ iff
$e_k<\wt e_k$ in the order $\O_k$ on $\E_k^+(b)$. This  {\em
  anti-lexicographic order }defines a partial
order on each component, which is a linear
order when restricted to any stable equivalence class $W^s(\e)$. One can picture this order
geometrically,
in the {\em stable tree model}; see Fig. \ref{F:ChaconTreeB}.

The map $T_\O^{(k)}$ is  defined 
to send a string
 $\e$ in the $k^{\text{th}}$ component  to its successor in this
 order. 
This does not quite define a transformation in the ordinary sense, that is, 
a map $T:X\to X$ where the domain and range spaces agree. One has
however the following  suitable notion, see e.g.~Def. 2.2 of~\cite{FeldmanMoore77a}:

\begin{defi}\label{d:partialtf}
  Given a set $X$, a {\em partial transformation} of $X$ is a
  bijection $T: A\to B$ for some  $A, B\subseteq X$, with {\em inverse}
  $T^{-1}:B\to A$. The $T$\,--\,{\em orbit} of a point $x$ is
$\{T^n(x):\, n\in \Z \text{ such that  this is defined, where
  $T^0(x)=x$ for all $x$}\}.$

Where defined, we call $T(x)$  the {\em successor} of $x$,  $T^{-1} (x)$ its
{\em predecessor}; the set of points with no
successor is  $\NS= X\setminus A$, those with no predecessor is  $\NP= X\setminus B$.
Setting $\CN\equiv  (\cup_{i\leq
  0} T^i \NS)\cup (\cup_{i\geq 
  0} T^i \NP)$, then $T$ is a bijection on $X\setminus \CN$.

That is, given a partial transformation $T:A\to B$, we can extract a
bijective transformation 
in a canonical way, simply by restricting attention to the  full  two-sided
orbits; this is
$(X\setminus \CN, T)$.
\end{defi}

So given a stable order $\O$, and focusing on the $0^{\text{th}}$ component $\Sigma_{\M}^{0,+}$, we
define 
$\NS$ and $\NP$ as above, 
 and will see shortly that these sets are at most
 countable (and are finite for bounded alphabet size). We call  the partial transformation $T_\O= T_\O^{(0)}:\, 
\Sigma_{\M}^{0,+}\setminus \NS\to \Sigma_{\M}^{0,+}\setminus \NP$  the
{\em Vershik map}; this Borel measurable map has the same orbits (the stable equivalence classes) as  the group of
 homeomorphisms $\FC= \FC_0$ on the compact space $\Sigma_{\M}^{0,+}$.

For the particular case of the Vershik map, we call the bijective map 
$T_\O$ on the noncompact  space $\Sigma_{\M}^{0,+}\setminus \CN$ 
the {\em adic transformation}
defined by the order $\O$.

\begin{exam}\label{ex:odometer}
 Instead of producing a transformation by restriction, one might try to
extend the partial
transformation to the whole space. The easiest example is  the most
classical example of adic transformation, the {\em Kakutani-von Neumann
  dyadic odometer}  (or  {\em adding machine}).

Recall that given a $0-1$ matrix sequence $L$ and alphabet sequence $\A$,  the
vertex shift
space is
$\Sigma_L^{+,0}= \{\x=(.x_0x_1\dots)\}: x_k\in \A_k\}$. To define a  {\em vertex
  order} on this space we begin with   a linear order $\Cal O_k$ on each alphabet
$\A_k$. For
vertex paths 
 $\x, \wt \x$ in the same stable equivalence class, 
 supposing 
$n$
  is the
 least  
index such that $\wt x_i= x_i \; \forall i\geq  n+1$,
 then we define 
$\x< \wt \x$ iff
$x_n<\wt x_n$ in the order $\O_n$ on $\A_n$.  This induces an
edge order which only depends on the initial symbols of the edges
entering $b\in \A_{n+1}$.  See
\cite{Fisher09a}. 

We recall the vertex and
edge shift models for the odometer: 
taking alphabet $\A=\{0,1\}$ and $L=\left[ \begin{matrix}
1 & 1\\
1 & 1
\end{matrix}  \right]$,  then 
the vertex
shift space $\Sigma^+_L$ is $\Pi_0^\infty \A$ with the vertex
  order $0<1$, so the vertex map is
$T: (.000\dots)\mapsto  (.100\dots) \mapsto  (.0100\dots) \mapsto
(.1100\dots)
\mapsto  (.00100\dots) $ and so on.
For the edge shift model, the vertex alphabet is a single point and
the edge alphabet is $\E=\{0,1\}$ with  matrix $M=[2]$. 
In both cases,
$\NS= \{.111\dots\}$, $\NP= \{.000\dots\}$.
Intuitively, thinking of $\Sigma^+_L$  as a car odometer written
backwards, when the odometer gets all the way to $.111\dots$ the next
step is to ``turn over'' to $.000\dots$. 
But this is just the unique continuous
extension of the map to all of $\Sigma^+_L$.

In general however, finding an extension
which is nice, e.g.~in the
sense of being bijective or continuous, may  not be possible. Thus,
given some function $f:\NS\to \NP$, we define $T_{\O, f} (\e)$ to be 
$T_\O(\e)$ for $\e\notin \NS$, and to be $f(\e)$ on $\NS$, calling
this the $f-${\em extension} of $T_\O$.  However  there
are examples
where $\# \NS\neq\# \NP$ so no bijective extension  exists.  This fact
depends only on the matrix sequence, and there are other examples where depending on the
choice of order $\O$,
a continuous extension does or does not exist. 
See
Examples 3,6
in~\cite{Fisher09a}, and   the example after
Proposition 5 in \cite{Medynets2006}.
 
\end{exam}

\begin{exam}\label{ex:Chacon} As we showed in Remark 5.1 of
  \cite{FerencziFisherTalet09}, for the Chacon adic there exists no
  continuous extension of the Vershik map on the noncompact space
  $ \Sigma_{M}^{0,+}\setminus\Cal N$
  to   $\Sigma_{M}^{0,+}$.
Briefly, in our current notation: as in Figs.~ \ref{F:ChaconOne}, \ref{F:ChaconTreeB},
we label the edges in the Bratteli diagram by $a,b,c,d,e$. These are ordered
 $a<b<c<d$; $e$ is ordered trivially.
Then $T: (.dd\dots db*)\mapsto (.ee\dots ec*)$ so, writing $\ee=(.ee\dots)$ and so on, the continuous
extension should be $T  (\dd)=\ee$; however
$T: (.dd\dots da*)\mapsto (.aa\dots ab*)$  so we should have
instead $T (\dd)=\aa$, and there is no 
continuous extension.

We   note that $\ee$ has no predecessor and no
   successor.
Now $\dd$ also has no successor and $\aa$ no
predecessor.
So $\NS=\{\dd,\ee\}$ and  $\NP=\{\aa,\ee\}$. Thus $\CN$ is the
countably infinite set $\{\ee, T^n(\aa),T^{-n}(\dd)\}_{n\in \N}.$

   \end{exam}

We now show, as promised, that the set $\CN$ is countable.
The elements of
$\NP, \NS$ are the
minimal and maximal elements with respect to  the partial order defined by $\O$ on this path space.

\begin{prop}\label{p:numberminimal}
  Let $(M_i)_{i\geq 0}$ be a  sequence of  $(\A_i\times \A_{i+1})$ 
nonnegative integer matrices, with $l_i= \#\A_i$, 
and let $\O$ be an order on the associated Bratteli diagram. 
\item{(i)} Assume first that $M$ is primitive. Then for each $\e\in \Sigma_{M}^{0,+} $, 
 $W^s(\e)$ is  finite iff $\liminf l_i=1$, otherwise  is countably infinite.
  \item{(ii)}
The number of minimal and maximal elements, $\# \NP$, $\#\NS$, are bounded above by $\liminf l_i$. 
  \item{(iii)}
    There is always at least one  minimal and one maximal element:
$\#\NP\geq 1$, $\#\NS\geq 1$. There exist examples with $l_i=N$ for all
$i$ but with $\#\NP=\#\NS=1$. 
\end{prop}
\begin{proof}
For $(i)$, given an edge path $\e$ let us say it {\em branches (to the
left)
at time $(k+1)$} iff there exists an edge path $f$ with $f_k\neq e_k$ and
$e_k^+= f_k^+$. If $\liminf l_i=1$ then $\e$
branches to the left only finitely many times so  there is
a largest such time, hence  $\Sigma_{M}^{0,+} $ is finite (and so $W^s(\e)$ is finite).

Next suppose $l_i>1$
infinitely often. Given an edge path $\e$, we shall show $W^s(\e)$ 
 is countably infinite.
This consists of those  paths $\f$ such that $\f$
branches to the left from $\e$ at some maximal time $k$ and from then on
agrees with $\e$. 

There exists $n\geq 0$  such that $l_n\geq
2$. Since $M$ is primitive, there exists such an $n>0$ such that also 
$M_0^{n-1} >0$.  Hence $a= e_0^-$ is connected to $b= e_n^+$ by the
path $e$, and there is some other path $f$ with $a= f_0^-$ and $c=
f_n^+$
with $c\neq b$. Since $M$ is primitive, there exists $m>n$ such that
$M_n^{m-1} >0$. Hence there is a finite allowed edge path connecting
$c\in \A_n$ to $d=e_m^+$. Hence we can define $\wt f$ to equal $f$ from
$0$ to $n$ and then continue in this way to agree with $e$ at time
$m$ and for all larger times. Thus $f\neq e$ and $f$ is in $W^s(\e)$. We repeat the
argument from $m$ to a larger time; doing this $k$ times  we get at
least $2^k$
paths in $W^s(\e)$, which is therefore infinite.

  For $(ii)$, now considering the order, suppose $\e,\f$ are minimal elements. If for some $n, e_n^+= f_n^+$, then
  certainly $e_k= f_k$ for all $0\leq k\leq n$. Therefore, if $\e\neq
  \f$, there exists some least $m\geq 0$ such that $\e_m\neq
  f_m$. Thus 
   $\e_m^-=  f_m^-$ and for each $k\geq m$, $e_n^+\neq
   f_n^+$. That is,  if $\e\neq\f$ are both minimal elements, there
   exists some least $m\geq 0$ such that before that time the paths
   are identical and after it, have always different vertex paths. So
   given any finite collection $F$ of minimal elements, there is a
   $n\geq 0$ such that from that time on, all have different vertex
   paths, so $\#F\leq l_m $ for all $m\geq n$, whence $\#F\leq \liminf
   l_k $ and so $\# \NP\leq \liminf
   l_k $, and similarly for $\NS$,  proving $(ii)$.

$(iii)$ Let $l_i=m=\liminf l_i$ and let $j<i $ with  $l_j=m$. Starting at vertex $a_m\in \A_m$, there exists a
unique least edge $e_m$ with $e_m^+=b_m$. Calling $a_{m-1} \equiv
e_m^-$, similarly there exists a
unique least edge $e_{m-1}$ with $e_{m-1}^+=a_{m-1}$. We do this for
each $b_m\in \A_m$,  producing edge paths from some $b_j\in \A_j$ to
$b_m\in \A_m$.
We note that existence of a unique least edge only holds to the left, not to the
right. For an example, the diagram could begin
with a single vertex with $k$ edges, and then continue with $k$
constant edge paths, each being both minimal and maximal. One can also
have an edge path $e_0,\dots , e_m $ which
is least at each vertex $e_k^+$ but cannot be further extended to the
right while maintaining minimality.

To produce an infinite minimal path, we instead make  use of the
compactness of $\Sigma_{M}^{0,+} $. Thus, for $m_k$
increasing, let
$\wh e_k$ be a finite minimal edge string from time $0$ to time
$m_k$. By compactness here exists a convergent subsequence $\wh e_k\to \wh e$.
This is an infnite minimal path, because out to an arbitrarily large
index, it agrees with one of the paths $\wh e_k$.
This proves that there is at least one such path.

That we have at best an upper bound in $(ii)$ is shown by the  example of
the $N$-adic odometer, with
 constant alphabet of size $N$ and with $\#\NP=\#\NS=1$.

\end{proof}

\subsection{Group actions, partial transformations, orbit equivalence relation}

We next examine how the basic ergodic theory for these  
are related: the
Vershik map, its associated adic transformation,  the
action of the 
countable group $\FC$,  and the stable equivalence
relation.  See~\cite{FeldmanMoore77a},
\cite{KechrisMiller04} and~\cite{Aaronson97} for additional
background. 

One can summarize by saying that 
 the idea of countable equivalence relation subsumes all: 
 the notions of
invariant set and measure, wandering set, ergodic, and conservative
for a  map
 or a group action
correspond to that for the orbit equivalence relation.
This fact leads
us to a natural definition of all these notions for a partial transformation such as the
Vershik map.

A subtle point is that this is not quite the case for measurability, which
needs to be defined
separately for maps. A (Borel) measurable map does generate a measurable 
equivalence
relation (Corollary \ref{c:measgraph}), however 
A. Kechris \cite{Kechris2015}  has shown us an example of an
invertible map  of a Polish space which is non-Borel-measurable yet whose orbit equivalence
relation is Borel, in fact such that there exists a measurable map with the same orbits.

As we note below, conservativity  has a natural
  definition for countable equivalence relations, and
thereby for group actions;
see Proposition
\ref{p:conservative_equivreln} below.

\begin{defi}\label{d:wandering}
  Let $X$ be a Polish space (a topological space with an
  equivalent metric  which makes it a complete separable metric space)
  with Borel $\sigma$\,--\,algebra  $\B$ and Borel measure $\nu$; throughout, ``measurable'' will
  mean Borel measurable. We say $T:X\to X$ is a {\em
    measurable map} or {\em transformation} iff $T^{-1}(\B)\subseteq \B$. It is a {\em
     measurable 
    invertible transformation} of $X$ iff $T$ is invertible and both $T$ and its
  inverse are measurable. The  measure $\nu$ is {\em
    preserved} by a  measurable map $T$ on $X$
iff for each measurable set $A$, $\nu(T^{-1}(A))= \nu (A)$.
A measurable subset $A\subseteq X$ is {\em invariant}
 iff $
T^{-1}(A)=A$. 
It  is {\em wandering} iff $\{T^{-i}(A)\}_{i\in\N}$  (here $\N=\{0,1,2,\dots\}$) are
disjoint sets, and is   {\em two-sided wandering} if the map is
invertible and this holds for
$\Z$ in place of $\N$.

Let $G$ be a countably infinite  group which acts 
measurably on  $X$. That is, we are given a map $\Phi :
G\times X\to X $ such
that, writing  $g(x)=\Phi(g,x) $, then  $g: X\to X$ is an invertible 
measurable map on $X$, 
satisfying $(g_1 g_2)(x)= g_1(g_2(x))$. The action is
{\em measure-preserving} if that holds for each element of $G$ simultaneously, and
a set  is
{\em invariant} iff it is invariant  for each  element.
A set $A\subseteq X$ is {\em wandering} for the action iff $\{ \gamma (A)\}_{ \gamma\in G}$ is an  infinite
disjoint family of sets, equivalently iff $\gamma (A)\cap A=
\emptyset$ for every $\gamma\neq e$.

A {\em Borel relation} $R$ is a Borel subset of $X\times
X$. Writing $x\sim_R y$ ($x$ is {\em related to} $y$) iff $(x,y)\in
R$,  this is a {\em Borel
  equivalence relation} iff  $x\sim_R y$  is reflexive, symmetric, and transitive. It  is 
{\em countable} ({\em countably infinite})
iff 
that is true for
each equivalence class. The  {\em
  saturation} $R(A)$ of a  set $A\subseteq X$ is
the union of equivalence classes which meet it.  
 The set $A$  is
{\em invariant} for $R$ iff $R(A)= A$, and
is {\em wandering} iff $R$ is countably infinite for points in $A$ and
$A$ meets an 
equivalence class in at most one
point. A {\em bijective equivalence} is a bimeasurable partial transformation
 $f$ of $X$ such that $x\sim_R f(x)$.  $R$ is {\em
  measure-preserving}
 iff two measurable sets 
which are bijectively equivalent have the same measure.

A measurable set $A\subseteq X$ is {\em trivial} iff  either
$\nu(A)=0$ or $\nu(A^c)=0$.
A measurable transformation,  group action or   equivalence relation is {\em ergodic} iff
 any invariant set $A$ is trivial; it is {\em conservative}
iff any wandering
set has measure zero.
A transformation is {\em recurrent} iff for any set $A$ of positive
  measure, for a.e.~$x\in A$ there exists  $n>  0$ such
  that $T^n (x)\in A$.

The {\em orbit equivalence relation} of an invertible transformation,
a partial
transformation (Def.~
\ref{d:partialtf}), or a group action is the relation whose equivalence classes consist of
orbits. 

We say a partial transformation of $X$ is measurable iff its domain
and range are Borel subsets  and  it is a Borel map there.
For  a measurable partial transformation, we define the  notions of
measure-preserving, ergodic and conservative via
its orbit equivalence relation. This definition is consistent with
that for maps 
by part
$(i)$ of the Proposition to follow.
\end{defi}

To understand how  these various points of view are related, we note
first that by
Theorem 14.12 of~\cite{Kechris1995}:
\begin{theo}
 Let $X,Y$ be Polish spaces, and $f: X\to Y$. Then $f$ is
  Borel measurable iff the graph of $f$ is a Borel subset of $X\times X$. 
\ \qed\end{theo}
\begin{cor}\label{c:measgraph}
 If $f:X\to X$ is  a Borel measurable function and is invertible, then so is $f^{-1}$. Thus $f$
  defines a Borel $\Z$\,--\,action, and
  moreover 
   its orbit equivalence relation is Borel.
\end{cor}
\begin{proof}
 The orbit equivalence relation
 is a countable union of Borel subsets of $X\times X$, as it is the union of the
 graphs of all iterates $f^n$ for $n\in\Z$.
\end{proof}

\begin{prop}\label{p:conservative_equivreln}
\item{(i)} Let $X$ be a Polish space, and 
  consider either 
  $T$ an invertible Borel transformation or  $G $ a countable group
  with Borel
  action on $X$, with $R$ the  corresponding orbit equivalence relation.

The  properties of being measure-preserving, ergodic  or conservative,
and the notions of a  set being invariant or
wandering, 
hold for $T$ or the $G$\,--\,action
 iff this   holds for the corresponding orbit equivalence relation $R$.
\item{(ii)} A (possibly noninvertible) measure-preserving
  transformation is
  conservative iff it is recurrent.
\item{(iii)} When the measure of $X$ is finite,  a measure-preserving 
 transformation, group action or equivalence relation is conservative.
\end{prop}
\begin{proof}
 \item{(i)} 
 Knowing  the orbit equivalence relation is 
  measure-preserving a fortiori  implies it
  for the corresponding transformation or group action. For the converse, see $(c)\then (d)$ of  Proposition 2.1 of
~\cite{KechrisMiller04}.

If $T^{-1}A=A$, then since $T$ is
  invertible, $A= T(T^{-1}A)=T(A)$, whence
  $A$ is invariant for the $\Z$\,--\,action and so for the orbit equivalence relation.
That  a set $A$   is two-sided wandering for $T$ is equivalent to $A$ being wandering for $R$. We claim that  if $A$ is wandering
 then it is two-sided wandering. We are to show that for $n<m$, then
 $T^n (A)\cap T^m (A)=\emptyset$. But  $T^{-n}(T^n (A)\cap T^m (A))=
 A\cap T^{m-n} (A)=\emptyset$
so that is true. It follows that the transformation is ergodic, or conservative, iff that
holds for its orbit
 equivalence relation. 
 
For a $G$\,--\,action, it is clear that a set is invariant, or wandering, 
iff that holds for the corresponding orbit
 equivalence relation.
Hence the notions of ergodicity and conservativity correspond  there as well.

\

\noindent
\item{(ii)}Assuming a map $T$ is
conservative, let $A$ be a subset of positive measure, and 
 define $B\subseteq A$ to be the set of all 
$x\in A$ which never return, i.e.~there does not exist $y\neq x$ in
$A$ and  $n>0$ such that
$T^n(x)= y$. 
Then in particular this holds for $x$ in $B$ itself. Hence $B\cap
T^{-n} B=\emptyset$ and so $T^{-k}B\cap T^{-(n+k)}(B)=\emptyset$ whence
$B$ is wandering, and so
by  conservativity $B$ must have
measure zero. Thus  $T$ is recurrent.
Conversely, assuming $T$ is recurrent, suppose that $A$ is a wandering set;
then $A$ must have  measure zero, as otherwise 
there would be 
 a contradiction.

\

\noindent
\item{(iii)} Given an
equivalence relation,
then 
if there is a wandering set of positive measure, by
 measure preservation  the total measure of the space would be
infinite, giving  a contradiction. By part $(i)$ this
implies the claim for group actions and invertible transformations.
For a noninvertible map, we use part $(ii)$ and then apply the same reasoning.
\end{proof}

\begin{rem}\label{r:ergodicmeasures} Combining  $(iii)$  with $(ii)$ gives the Poincar\'e recurrence
theorem for maps, while 
$(iii)$ can be thought of as  the natural version of this  for
actions of countably
infinite groups and  countably
infinite equivalence relations, and so also for partial
transformations.
\end{rem}

\subsection{Central measures, eigenvector sequences and cones}

\medskip

For $M$ as in Proposition \ref{p:numberminimal} we are interested, first of all, in the probability measures
on $\Sigma_{\M}^{0,+}$ which are invariant for 
the action of the group $\FC$, and so equivalently, by $(i)$ of Proposition \ref{p:conservative_equivreln}, invariant for the stable
equivalence relation and the Vershik partial transformation. A 
simple  equivalent condition (studied by Bowen and Marcus in the
stationary case, see  Lemma
2.4 of
\cite{BowenMarcus77}) is given in $(i)$ of the next
proposition.

\begin{defi}\label{d:ICCM}
  We denote by $\CCM$ the collection of
  $\FC$-invariant Borel probability measures on $\Sigma_{\M}^{0,+}$.
 These are 
the {\em
  central  measures}. 
We are also interested in two infinite analogues of these: the $\sigma$\,--\,finite infinite conservative measures which are
positive finite on some nonempty open subset, and those which are
positive infinite on every nonempty open subset. We call these
respectively the 
  {\em infinite central measures} and 
{\em locally infinite central measures}.
  \end{defi}

  \begin{rem}
    Vershik in~\cite{Vershik81} introduced the term {\em central
      measure} for an invariant Borel probability measure for an
adic transformation on the
(noncompact) 
space $\Sigma_{\M}^{0,+}\setminus \Cal N$.
Given an $\FC$\,--\,invariant Borel measure, to have a central measure we 
may need to 
remove a finite number of point masses (on fixed points) and normalize to have a
probability measure,
see $(ii)$ below. 
  \end{rem}

  \begin{rem}\label{r:infiniteon open}
   Regarding the locally infinite measures, mentioned in the
   Introduction, see Lemma \ref{l:finite_on_all},
Theorem \ref{t:covermeasure}, Remark
\ref{r:innerregular}. Illustrative examples
include the Integer Cantor Set in its ternary model,  Example \ref{rem:Bruin}
and certain
nested circle rotations, see Example \ref{exam:nestedrot}.
  \end{rem}

\begin{lem}\label{l:orbitequivrel}
  Given an ordered Bratteli diagram with edge path space
  $\Sigma_{\M}^{0,+}$,
 the orbit equivalence relation
  for the partial transformation $T_\O$ and for $\FC$ are the same: the
  equivalence classes are the stable sets $W^s(\e)$. For the adic
  transformation on $\Sigma_{\M}^{0,+}\setminus \CN$, this is true
  if we restrict the action of $\FC$ to that subset.
\end{lem}
\begin{proof}
  If for $\e\in \Sigma_{\M}^{0,+}$ we have $T_\O(\e)= \f$, then there
  exists $k>0$ such that $e_k^+= f_k^+$ and $e_n=f_n$ for all $n>k$. There is
  $\gamma\in\FC$ with $\gamma([. e_ke_{k+1}\dots
e_n])=[.f_{k}f_{k+1}\dots f_n] $ which fixes the tails, whence $\gamma(\e)=\f$.
Conversely, if $\gamma\in \FC$ satisfies $\gamma([. e_ke_{k+1}\dots
e_n])=[.f_{k}f_{k+1}\dots f_n] $ where $e_k^+= f_k^+$ while fixing
the tails, then for $f= (.f_0\f_1\dots f_k f_{k+1}\dots)$ and $\e=
(.e_0\e_1\dots e_k f_{k+1}\dots)$, 
 we have $\f\in W^s(\e)$, whence $T_\O^m(\e)= \f$ for some
$m\in\Z$. Since any $\eta\in \FC$ can be written as a finite product of
generators,  the same holds for $\eta$. The second statement follows.
\end{proof}

  \begin{prop}\label{p:conserg}
Let $(\A_i)_{i\geq 0}$ be an alphabet   sequence with $l_i= \#\A_i>0$. 

 \item{(i)}Given $L= (L_i)_{i\geq 0}$ an 
 $(\A_i\times \A_{i+1})$ $0-1$ matrix sequence, then    a measure
 $\nu$  on the  vertex space $\Sigma_{\L}^{0,+}$ is a {central measure}
iff it is a probability measure such that  for each thin cylinder set
we have
\begin{equation}
  \label{eq:7}
  \nu([.\wt x_0\dots \wt x_n]) =\nu([.x_0\dots x_n])
\text{ whenever } \wt x_n=  x_n=s\in \A_n.
\end{equation}

\noindent
For a nonnegative integer  matrix sequence $M$, $\nu$ on 
 the edge path space  $\Sigma_{\M}^{0,+}$ is a central measure iff it is a probability
 measure such that

\begin{equation}
  \label{eq:8}
 \nu([.\wt e_0\dots \wt e_{n-1}]) =\nu([.e_0\dots e_{n-1}])
\text{ whenever } \wt e_{n-1}^+= e_{n-1}^+=s
. 
\end{equation}

\item{(ii)} Given some order $\O$ on the Bratteli diagram, a (finite
  or infinite)  invariant Borel
  measure for the adic transformation $T_\O$ is 
invariant for any extension $T_{\O,f}$ and  for the action of $\FC$. A
nonatomic $T_{\O,f}$ or $\FC$\,--\, invariant Borel measure is invariant for
$T_\O$. 
\item{(iii)}
If a nonatomic invariant Borel measure is ergodic, or conservative, for one of $T_\O$, $\FC$ or
$T_{\O,f}$ then that holds for the other two.
\end{prop}

\begin{proof}We give the proofs for an edge shift.
For $(i)$,  considering the generators of $\FC$,  $\gamma$
interchanges the two thin cylinder sets $[. e_0e_1\dots
e_n]$ and $[.f_0f_1\dots f_n] $ where $e_n^+= f_n^+$,  and this is clear.

 For $(ii)$, if $\nu$ on
 $\Sigma_{\M}^{0,+}\setminus \CN$  is $T_\O$\,--\,
invariant, then  the extension of $\nu$ given by assigning $\CN$
measure zero is invariant for the extended map $T_{\O,f}$ on
$\Sigma_{\M}^{0,+}$. 
 Next, consider a generator
$\gamma$ of $\FC$ which interchanges  $[. e_0e_1\dots
e_n]$ and $[.f_0f_1\dots f_n] $ with $e_n^+= f_n^+$. In the order
defined by $\O$ one of these strings is least, say $(.e_0e_1\dots
e_n)$, so there exists $m\geq 1$ with $T_\O^{-m} ([.f_0f_1\dots f_n])= [. e_0e_1\dots
e_n]$. Moreover, the restriction of $T_\O^{-m} $ to $[.f_0f_1\dots
f_n] $ is $\gamma$. Thus the measure of any Borel subset  $E\subseteq
[.f_0f_1\dots f_n]$ is preserved by $T_\O^{-m} $ hence by $\gamma$. 
(Note: it is important to allow for subsets here, rather than just
approximating a Borel set by cylinders, as the cylinders
themselves may all have infinite measure. ) Thus $\nu$ is  $\FC-$ invariant.

For the converse, let  $\nu$ on
 $\Sigma_{\M}^{0,+}$ be nonatomic. Then $\CN$ has
 $\nu-$measure zero since by Prop.~\ref{p:numberminimal}  that set is
 countable thus any nonzero measure on it would be atomic. Hence if
 $\nu $ is  $T_{\O,f}$\,--\, invariant then that holds for 
 $T_\O$.  Lastly, suppose that $\nu$ is nonatomic and
 $\FC$\,--\,invariant.  Then given a cylinder set $[. e_0e_1\dots
e_n]$,
 if $e_n$ has a successor $f_n$ for $\O$  in the collection of edges
which enter the vertex $e_n^+$, then, say, $T_\O^{-1}  ([.f_0f_1\dots f_n] )=[. e_0e_1\dots
e_n]$. But there is some $\gamma$ which
interchanges these two cylinder sets,  $\gamma([. f_0f_1\dots
f_n])=[.e_0e_1\dots e_n] $, and moreover $\gamma$ equals the
restriction of $T_\O^{-1}$ to that cylinder. Hence as before, the
measure of a Borel subset is preserved for $T_\O^{-1} $.

We claim that a general thin cylinder set is a countable union of 
such cylinders, plus points in $\NS$. 
To prove this consider the tree of possible extensions of the string $(.e_0e_1\dots
e_n)$, stopping a branch of this tree at a finite stage if there is a
successor for the last edge added.  The branches which continue
infinitely are exactly the points in $[. e_0e_1\dots
e_n]\cap \NS$, proving the claim. Since this has measure zero, we are done.

Part $(iii)$ follows from Lemma \ref{l:orbitequivrel} together with $(i)$ of
Proposition \ref{p:conservative_equivreln}.
\end{proof}

\

We next show that central measures correspond to
sequences of nonnegative eigenvectors of eigenvalue one, which form a
finite-dimensional simplex, for which the ergodic central
measures are the extreme points.
First we need:
\begin{lem}\label{lem:ergodicmeasures}Given a countable group or semigroup $G$ acting continuously
  on  a compact
  metric space $X$, the ergodic invariant Borel probability measures
are exactly the extreme points of the convex compact set of invariant
probability 
Borel measures on $X$.
  \end{lem}
\begin{proof}
 Since  $X $ is  a compact metric space, the collection $\Cal M_G$
 of invariant  Borel probability measures  is
compact convex with the weak* topology,  by the Banach-Alaoglu
theorem. Given a continuous transformation $T$ on $X$, the
extreme points  are identified as being  the ergodic measures,
 see  e.g.~
Proposition 3.4 of~\cite{Furstenberg81}. This proof goes through
without change for the action of a countably generated semigroup.
\end{proof}

The special
importance of 
the ergodic measures comes from Choquet's
strengthening \cite{Phelps2001Choquet} of the Krein-Milman Theorem
(pp.~66-70,~\cite{Rudin73}): the Krein-Milman theorem tells us that
each invariant probability measure is represented as the barycenter 
of a probability measure on the closure of the extreme points, while from
Choquet's theorem this measure is supported on the extreme points
themselves, and moreover is unique. So in this sense each invariant
measure has an 
{\em ergodic decomposition}, as an (integral) convex combination of
the ergodic measures.  

That is the general abstract framework. However in the present setting
things are much simpler:   we show directly that the  $\FC$-invariant
Borel probability 
measures on  
$\Sigma_{\M}^{0,+}$
form a finite-dimensional simplex. That implies the
Banach-Alaoglu theorem for this case, and also directly gives the
ergodic decomposition, with 
the ergodic probability measures corresponding to the
extreme points. This gives us  (see   Proposition \ref{p:extremepts}) a first upper
bound, $\lim\inf\# \A_n$, for their
number.  A second, better bound 
follows later from this same basic idea together with the Frobenius
decomposition, see Corollary \ref{c:secondcount}, and this approach
also leads to a bound on the number of ergodic infinite central measures.

The simplex will be described using 
 intersections of nested cones of column vectors.

We recall that
a {\em cone} $C$ in a real vector space $V$ is a nonempty subset such that
$\alpha C\subseteq C$ for each $\alpha\geq 0$; it is a {\em convex}
cone iff $C+C\subseteq C$, iff $C$ is a convex set, and is a {\em
  positive} cone iff $C\cap -C= \{\0\}$. The collection $\r^{d+}$ of nonnegative 
vectors  is a positive convex cone, termed the  {\em standard positive
  cone} of $V=\R^d$. Note that $\{\0\}$ is a cone, the {\em trivial
  cone}. If the alphabet $\A$ is
empty, then  $d=0$ and $\r^{\A}=V= \{\0\}$ (as for Lemma \ref{l:generalizedmatrices}). 
For any cone $C$, the point 
$\0$ is its {\em vertex} and if $C$ is convex, is  its only extreme point.
Denoting by $C_i$ the standard positive cones of  column  
vectors in $V_i= \R^{\A_i}$, 
a nonnegative real $(\A_i\times \A_{i+1})$ matrix sequence
$N=(N_i)_{i\geq 0}$ maps each  into the next as follows:

\begin{equation}
\begin{CD}\label{CD:cones}
     C_0    @<N_0<<  C_1 @<N_1<< C_2 
 @<N_2<<   C_3
\cdots\\
\end{CD}
\end{equation}

\

\noindent
A global point of view is useful here:

\begin{defi}\label{d:eigenvectorsequence}
  Given
    alphabets $\A_i$, and  $(N_i)_{i\geq 0}$ an
$(\A_i\times \A_{i+1})$ sequence of
real matrices, we make the following definitions. We write 
$V_\A\equiv  \Pi_{i\geq 0}
V_i$, where $V_i= \R^{\A_i}$,  with the product topology, for this topological vector space. 
Its  zero element is
  $\0= (\0_0, \0_1,\dots)$, where  
$\0_n\in \r^{\A_n}$. 
We define $N:V_\A\to V_\A$,
sending
$(\v_0, \v_1,\v_2,\dots)$ to $(N_0\v_1,
N_1\v_2,\dots)$. This is a continuous linear map. Note that  $N= \Pi_0^\infty N_i$.

By an {\em eigenvector sequence} $\w=(\w_0, \w_1,\dots)$ with
 {\em  eigenvalues}  $\lambda=
(\lambda_i)_{i\geq 0}$ for $\lambda_i\in
  \C$ we mean that

\noindent
\item{(i)} $\w_i\in V_i\setminus \0_i$ and

\noindent
\item{(ii)}
$\lambda_i{\bf
  w}_i=  N_i{\bf w}_{i+1},   \text{ for all  $i$.}$

We call  condition $(i)$ {\em never zero}. We note that if the
eigenvalue sequence happens to be constant, $\lambda_i=\lambda$ for
all $i$, then an eigenvector
sequence is an eigenvector for the map $N$ of
the vector space
$ V_\A$, as it is never zero hence nonzero, but the never zero
condition is stronger, as $\bfw$ is nonzero
for each time projection $\bfw_i\in V_i$. Given an eigenvector
sequence, then in the special case where  $\lambda_i\neq 0$ for
all $i$,  we can normalize, setting $\bfw'_0=\bfw_0$,
$\bfw'_1=\bfw_1/\lambda_0$,
$\bfw'_2=\bfw_2/(\lambda_0\lambda_1)$,  $\dots,$
$\bfw'_{k}=\bfw_k/\lambda_0^{k-1}$ for $\lambda_0^{k-1} =\lambda_0\cdots
\lambda_{k-1}$ to get an
eigenvector sequence $\bfw'$ with eigenvalue one.

\

We define the {\em positive cone of  $ V_\A$} to be  $C_\A\equiv  \Pi_{i\geq 0}
C_i$.

\noindent
Assume that the $N_i$ are nonnegative, as in \eqref{CD:cones};
equivalently,
$$N: C_\A\to C_\A.$$

\noindent
We denote the collection of fixed points for this map by $\V_{N}^\0.$ That
is, 
$$\V_{N}^\0\equiv \{ \w= ( {\bf w}_0 {\bf w}_1\dots)
\text{ such that }{\bf w}_i\in C_i\,  \text{ and } {\bf
  w}_i=  N_i{\bf w}_{i+1}  \text{ for all } i\}.$$

\noindent
$\V_{N}\subseteq \V_{N}^\0$ denotes
$\w$ such that ${\bf w}_i\in C_i\setminus \{\0_i\}$ for all $i$;
these are the { positive
(right) eigenvector
  sequences with eigenvalue one}. 

We write $\wt \0= \V_{N}^\0\setminus \V_{N}$.
If $\w\in \wt \0$ is not the {\em  identically zero} sequence 
$\0 $ (which is   the vertex of the cone $ \V_{N}^\0$), we
call it a
{\em partially zero} sequence.

On the space $\r^d$, we use the $L^1$\,--\,norm
$||{\bf u}||=\sum_{i=1}^d |u_i|;$
with this choice of norm
the map $ {\bf u} \mapsto {{\bf u}}/{||{\bf u}||}
$
projects  the  positive cone $C= \r^{d+}$ of column vectors  minus its vertex ${\bf 0}$, 
onto the $d$\,--\,simplex $\Delta$.

We write $\Delta_i\subseteq C_i$ for this norm-$1$
simplex,
and $\V_{N}^\Delta \equiv\{ \w \in \V_{N}:\, 
\w_0\in \Delta_0\}$.

Normalizing the sequence
$\w$ so its first
element is in $\Delta_0\subseteq R^{\A_0} $, we define the projection
 $\w\mapsto \w/ ||\w_0||$ whenever $\w_0\neq \0$ from $\V_{N}$, the collection
 of all positive right eigenvector sequences with
 eigenvalue one, 
to $\V_{N}^\Delta \equiv\{ \w \in \V_{N}:\, 
\w_0\in \Delta_0\}$.

\end{defi}

\begin{lem}\label{l:closedcone}
  Given
    alphabets
  $\A=(\A_i)_{i\geq 0}$ 
  and  $(N_i)_{i\geq 0}$ a sequence of 
$(\A_i\times \A_{i+1})$ 
nonnegative
real matrices, then:
  \item{(i)}
  $C_\A$ is a closed convex cone in $V_\A
  $. So   is $\V_N^\0$.
  $\V_{N}^\Delta \equiv\{ \w \in \V_{N}:\, 
\w_0\in \Delta_0\}$ is a compact  convex subset of $V_\A
$.

\item{(ii)} If $N$ is column--reduced, then
 $\wt \0=\{\0\}$, and so $\V_N^\0= \V_N\cup\{\0\}$.
\item{(iii)}  If $N$ is 
 primitive and row-reduced, then for $\v\in \V_N^\0$, if $\v_k\neq\0_k$ for some $k$,
 $\v_i\neq\0_i$ for all $i$.
\end{lem}
 \begin{proof}
   \item{(i)}
That these are convex cones is clear, and $C_\A$
  is a  closed subset of $V_\A$ (with respect to  the product
  topology).

   We prove that $\V_{N}^\0$ is closed.
Let $\v^{(k)}\in \V_{N}^\0$, 
and suppose that $\v^{(k)}\to \v \in C_\A$. 
We are given that each  $\v^{(k)}$ is a fixed point for $N$. Thus by continuity of this map
we have 
$N(\v)= N(\lim_{k\to\infty} \v^{(k)})= \lim_{k\to\infty} N(\v^{(k)})= \lim_{k\to\infty} \v^{(k)}=
\v$. 
Since $N\v=\v$, we have that $\v\in \V_{N}^\0$, as claimed.

To show that $\V_{N}^\Delta $ is  compact, let $\v^{(k)}\in
\V_{N}^\Delta$; and suppose that $\v^{(k)}\to \v \in C_\A$. Then as
above $N\v=\v$. Since $(\v^{(k)})_0\in \Delta_0$, which is compact,
so is $\v_0$;
hence each element of the sequence $\v$ is nonzero, and thus $\v\in
\V_{N}^\Delta $
as well.

\item{(ii)} Let  $\v\geq 0$ and $N\v=\v$. Suppose $\v_k= \0_k$; we claim
that then this is true for all $k$. Certainly  $\v_i= \0_i$
for all $i<k$. Now suppose
$\v_{k+1}\neq  \0_{k+1}$; that is, there exists
$s\in \A_{k+1}$ with $(\v_{k+1})_s>0$. Then since $\v_{k}=\0_k$, the $s-$ column
of $N_k$ is all zero, contradicting that $N$ is column--reduced.

\item{(iii)} Let  $\v\geq 0$ with $N\v=\v$; suppose $N$ is primitive and
$\v_k\neq  \0_k$. We know there exists $m_0>k$ such that for any
$m\geq m_0$, all entries of 
$N_k^m$ (recalling Def.~\ref{d:gathering}) are greater than zero. Certainly $\v_{m+1}\neq  \0_{m+1}$, since
otherwise we would have $\v_k= \0_k$. Now using the fact  that $N_k^m$ is strictly
positive, we have that in fact $\v_k=
N_k^m \v_{m+1}$ is strictly positive. Since $N_j$ is row-reduced for
$j= k-1$, the same holds for $\v_j$ and inductively
for all $j<k$. We have shown that if some $\v_k\neq  \0_k$ then in
fact it is strictly positive and  moreover that is true for all lesser
indices as well. But no larger index $\v_l$ can be $\0_l$ since that
would imply $\v_k=  \0_k$, a contradiction. This shows $\v_l\neq
\0_l$
 for all $l\in\N$ whence all are strictly positive.
  \end{proof}

It will be useful, here and below in  \S \ref{s:Vict}, to have available two different notations for
iteration. 

\begin{defi}\label{d:partialproduct}
  As in Def.~\ref{d:gathering},  we denote the product of the matrices from $i$ to $n$
  by $N_i^n= N_iN_{i+1}\cdots N_n$. Thus  on column vectors as in \eqref{CD:cones}, $N_i^n: V_{n+1}\to V_i$.
This is 
iteration from time $i$ to time $n$ along the
sequence of maps.
  
Secondly, 
$N^m$ denotes the $m^{\text{th}}$ iteration of the  total map, that
is, of the linear operator $N$
on $V_\A$.

These are related as follows: 
defining  
 $(N^m)_i=N^{m+i-1}_i$ for all $i\geq 0$ and $m\geq1$,
then $(N^m)_i: V_{i+m}\to V_i$ and 
$N^m$ is the product of maps
$N^m=\Pi_{i=0}^\infty (N^m)_i$.
  \end{defi}

For  $0\leq k\leq n$ we define:
$$C_N^{(k,n)}=
N_k^n C_{n+1}.$$
Note that
$$C_{k}\supseteq N_kC_{k+1}\supseteq N_kN_{k+2}C_{k+2} \supseteq
\cdots$$
that is, 
$$C_{k}\supseteq C_N^{(k,k)}\supseteq C_N^{(k,k+1)}\supseteq \cdots$$ 
We set $ C_N^{(k,+\infty)}\equiv \cap_{n= 0}^{+\infty} C_N^{(k,
  k+n)}$ and note that for each $k$,
  $N_k C_N^{(k+1,+\infty)}=C_N^{(k,+\infty)}$.

Each of $C_N^{(k,n )}$, $C_N^{(k,+\infty)}$, $\V_{N}^\0$
is a closed convex cone. Since $\V_{N}^\Delta $ is 
compact convex,   $\V_{N}\setminus \0$ is compact modulo 
projective equivalence. 
We let  $\Ext
  C_N^{(k,n)}$, $\Ext
  C_N^{(k,+\infty)}$, $\Ext \V_{N}$ denote
the union of the extreme rays.
By the Krein-Milman Theorem (for finite dimension), see after 
Lemma \ref{lem:ergodicmeasures}, every point in $\V_{N}$ is a convex
combination of points in $\Ext \V_{N}$.

\begin{lem}\label{l:inductive eigenvectors}
  Given
    alphabets
  $\A=(\A_i)_{i\geq 0}$ 
  and  $(N_i)_{i\geq 0}$ a sequence of 
$(\A_i\times \A_{i+1})$ 
nonnegative
real matrices, then:
\item{(i)} 
All sequences $\w\in\V_{N}^\Delta$ can be built inductively as follows:

  $(1)$
Choose ${\bf w}_0$ in  $\Delta_0\cap C_N^{(0,+\infty)}.$

$(2)$
Choose ${\bf w}_{i+1}\in C_N^{(i+1,+\infty)}$ to be a preimage of ${ \bf w}_{i}$.

\item{(ii)}
$ C_N^{(0,+\infty)}\equiv
\cap_{n=0}^{+\infty} N_0^n C_{n+1}\neq \{\0\}$ iff
$\V_{N}^\Delta\neq \emptyset.$
\item{(iii)} $\w\in \Ext \V_{N}^\Delta$ iff $\w\in \V_{N}^\Delta$ and
  for each $k$,  $\w_k\in \Ext  C_N^{(k,+\infty)}$.
\item{(iv)}
$N$ is a primitive sequence $\iff$ $\V_{N}^\Delta\neq\emptyset$ and $\forall \w\in \V_{N}^\Delta,\w_i>0 $ for all
$i$.
\item{(v)} For $M=N$ as above but with integer entries,
$\Sigma_{M}^{0,+}\neq \emptyset $ iff 
$ C_M^{(0,+\infty)}\neq \{\0\}.$
\end{lem}
\begin{proof}
 \item{$(i)$} Without loss of generality we can assume $\V_{N}^\Delta$ is nonempty,
  since otherwise this is trivially true. (That is the case if the
  alphabets are nonempty).

  Now there exist preimages since this is the definition of $
  C_N^{(0,+\infty)}$, so part $(i)$ is clear.

  \item{$(i)$} Part $(ii)$ follows from this.

\item{(iii)}:
For  $\w\in\V_{N}^\Delta$ then if this is not extreme there exist distinct sequences  $\w^1,\w^2\in
\V_{N}^\Delta$ and $p,q>0$,
$p+q=1$ such that $\w= p\w^1+ q\w^2$. But then there exists some time
$k$ such that they are different: $\w^1_k\neq \w^2_k$, and still $\w_k= p\w^1_k+
q\w^2_k$ so $\w_k^1, \w_k^2$ are not extreme in $
C_N^{(k,+\infty)}$.  Conversely if $\w$ is extreme in $\V_{N}^\Delta$ then this
cannot happen for any $k$.

\item{(iv)}:
  As above we can assume $\V_{N}^\0$ is nonempty.
If $N$ is primitive, then choosing $i\geq 0$, there exists $n>i$ such
that $N_i^n>0$, so 
$N_i^n (C_{n+1}\setminus \{\0\})\subseteq \interior \, C_i $, 
and hence for all $\w\in \V_{N}$, 
$\w_i$ is nonzero. Therefore $\V_{N}$, and hence $\V_{N}^\Delta$, are
nonempty,
and for all $ \w\in \V_{N}^\Delta$, each $\w_i>0 $.
For the converse, if $N$ is not primitive, 
then there exists $i\geq 0$ such that for each $n>i$
some column of $N_i^n$
contains a zero. Let $\w_{n+1}^n$  be the corresponding
standard basis vector, so $N_n \w_{n+1}^n$ is that matrix column and
so has a zero entry. Defining   $N_j \w_{j+1}^n =
\w_{j}^n$, giving a finite nonnegative sequence
$\w_i^n, \w_{i+1}^n, \dots, \w_{n+1}^n$ such that each
$\w_{i}^n$ has at least one zero entry.  
Then by a compactness argument there exists $\w\in\V_N$ such that $\w_i$ is not nonzero.

\item{(v)}:
By 
compactness, $\Sigma_{\M}^{0,+} \neq \emptyset $ iff  
there exists an infinite sequence
of nested decreasing (nonempty) thin cylinder sets. This holds iff for all $n$
there exists an allowed edge path of length $n$. 
Now the number of allowed finite edge paths beginning with a
  symbol $i\in \A_0$ and ending with $j\in \A_{n+1}$  is equal to 
${\mathbf e }_i^tM_0^n {\mathbf e }_j$ where ${\mathbf e}_i$ is the standard basis vector
and ${\mathbf e}_i^t$ its transpose, so  
the collection of allowed edge paths of length $n$ is nonempty iff
$M_0^n$
has some non-zero entry, iff $M_0^nC_{n+1}$ contains some
nonzero vector. 
And $C_M^{(0,+\infty)}\neq \{\0\}$ iff for all $n$, $C_M^{(0,n)}\setminus\{\0\}\neq \emptyset
$; again we use compactness, of 
the intersection of the cone $C_M^{(0,n)} $ with the closed unit
sphere.

\end{proof}

 In particular, if  the matrices are invertible then
the sequence 
in part $(i)$ of the lemma is 
determined by choice of its first element ${\bf w}_0$.
Regarding part $(iv)$ see  Lemma 4.2 of ~\cite{Fisher09a}.

\begin{defi}\label{d:Radon}
We recall: a Borel measure is {\em inner regular} iff the measure of a set is the
sup of the measures of its compact subsets. It is  {\em outer regular} iff
the measure of a set is the
inf of the measures of the open sets containing it, and is {\em
  regular } if it is both outer and inner  regular.

We shall say a measure is {\em locally finite} if there exists a neigborhood
of each point with finite measure, and {\em positive} locally finite
if it is in addition strictly positive on each nonempty open subset.

 We recall: a {\em Radon} measure is a Borel
measure which is both inner regular and locally finite. 
\end{defi}

\begin{rem}\label{r:innerregular}
  Inner regularity will hold for all the measures considered in this paper, and outer
regularity for the finite measures, as noted in
the proof below. 
Indeed, a countable sum of inner regular Borel measures is inner
regular, whence the tower measures constructed later on are inner
regular, though they may be locally infinite and hence not Radon.
\end{rem}

The next theorem gives the basic information about finite $\FC$-invariant
Borel measures on an edge  space $\Sigma_{M}^{0,+}  $, that is the elements of
$\CCM$, for the nonprimitive, nonstationary case. (For simplicity of
the statement and proofs we begin
a vertex space $\Sigma_{L}^{0,+}  $). Part $(i)$ is Theorem 2.9 of
~\cite{BezuglyiKwiatkowskiMedynetsSolomyak10}.

\begin{theo}\label{t:basic_thm}
  Given nonempty alphabets
  $\A=(\A_i)_{i\geq 0}$,
let 
$L= (L_i)_{i\geq 0}$  be an $(\A_i\times \A_{i+1})$ 
 $0-1$ reduced matrix sequence. 
\item{(i)} There is an  affine
  homeomorphism
 $\Phi:\, \CCL\to \V_{\L}^\Delta$ defined
by
$\Phi(\nu)= \w$ with
\begin{equation}
  \label{eq:basictheorem}
  ({\bf w}_n)_s= \nu([.x_0\dots x_n]) \text{ where } s=x_n,\;\;
\end{equation} 
where $[.x_0\dots x_n]$ is a thin cylinder set.
\item{(ii)} $\Phi$ takes  the ergodic $\FC-$ invariant probability measures on
  $\Sigma_{L}^{0,+}  $, which is 
$\Ext\CCL$, bijectively to $
  \Ext\V_{L}^\Delta$.  
\item{(iii)} $L$ is primitive $\iff $ each $ \nu\in\CCL$ is 
  strictly positive and finite
on each
cylinder set.

All of the above holds for edge spaces defined from a nonnegative integer
matrix sequence $M$, with the map
$\Phi$ from $\CCM$ to
$\V_{\M}^\Delta$ defined 
by
$\Phi(\nu)= \w$ with
\begin{equation}
  \label{eq:centralmeasure}
  ({\bf w}_n)_s= \nu([.e_0\dots e_{n-1}]) \text{ where } s=e_{n-1}^+.
\end{equation}
These measures are regular and hence Radon.
\end{theo} 

\begin{proof}
    We address $(i)$.

  First we consider the inverse of the map $\Phi$. Thus, assume we are given
  an eigenvector sequence $\w$ of eigenvalue one, with $\w_0$ a
  probability vector; that is, 
$\w\in\V_{\L}^\Delta$. Considering  a  thin cylinder set $[.x_0\dots x_n]$,
we define its measure   to be
\begin{equation}
  \label{eq:12}
  \nu([.x_0\dots x_n=s]) =({\bf w}_n)_s.
\end{equation}
 We don't yet know this will give us a measure on the
 $\sigma$-algebra; at this point $\nu$  is a nonnegative real-valued function defined on the collection of
thin cylinders. We shall show this
has a unique extension to a Borel probability measure on 
$\Sigma_{\L}^{0,+}$. We write $\Cal B$ for the Borel $\sigma$\,--\,algebra and 
$\Cal B_0$ for the algebra generated by the thin cylinders. One checks
that elements
of 
$\Cal B_0$ are exactly the  finite unions of 
 thin cylinders. We extend the definition to this algebra, defining  $\nu(A)$ for $A\in \Cal
 B_0$  simply to be the sum of the measures of
these sets. However there are many ways to decompose $A$, so to show
this is well-defined we need 
to check these all give the same number.

Suppose first that $A$ itself is a thin cylinder set $[.x_0\dots x_n] $. 
We then rewrite it as a union of thin cylinders of length $k$ for some 
$k>n$. If $k= n+1$, then we have for any $a\in \A_n$,
$[.x_0\dots x_{n-1} a] =\cup_{b\in \A_{n+1}}\{ [.x_0\dots x_{n-1} a b]:\,  (L_n)_{ab}= 1\}$.
We have
 \begin{align}
&\sum_{\{b:(L_n)_{ab}= 1\}}\nu([.x_0\dots x_{n-1}  ab])
   \nonumber\\
 & = \sum_{\{b:(L_n)_{ab}= 1\}}({\bf w}_{n+1})_b \label{eq:5c}
 \end{align}

 Now by the definition of matrix multiplication,

  \begin{align}
    &{\bf w}_n= L_n{\bf w}_{n+1}\label{eq:5d}\\
    &{\text{iff}}\nonumber\\
    &({\bf w}_n)_i = \sum_{j\in \A_{n+1}} (L_n)_{ij} (\w_{n+1})(j) \label{eq:5f}\\
  \end{align}
   so since $\w$   is an eigenvector sequence of eigenvalue one, which is \eqref{eq:5d},
   \eqref{eq:5c} is 
 \begin{align*}
 &  \sum_{\{b:(L_n)_{ab}= 1\}} ({\bf w}_{n+1})_b = (L_n {\bf
   w}_{n+1})_a= ({\bf w}_n)_a= \nu([.x_0\dots x_{n-1} a])
 \end{align*}

  which shows that
 \begin{align*}
&\sum_{\{b:(L_n)_{ab}= 1\}}\nu([.x_0\dots x_{n-1}  ab])
  =\nu([.x_0\dots x_{n-1} a]).
 \end{align*}

 That is to say, adding up those thin cylinder measures in the two decompositions of the thin cylinder set
 $[.x_0\dots x_n] $ as itself,  $[.x_0\dots x_{n-1} a]$
and as $\cup_{b\in \A_{n+1}}\{ [.x_0\dots x_{n-1} a b]:\,
 (L_n)_{ab}= 1\}$, 
 gives the same number.
Inductively this is true for any $k>n$,
 so we are done in the case where 
$A$ is a thin cylinder and the cylinders in its decomposition have equal length.

Now suppose $A\in \B_0$ is written in two different ways as a
union of collections of thin cylinders $\Cal S_1$ and $\Cal S_2$, so 
$A=  (\cup\Cal S_1)=  (\cup\Cal S_2)$. We claim these give the same
number: that 
$\sum _{B\in \Cal S_1} \nu B= \sum _{B\in \Cal S_2} \nu B$.
To show  this we let $k$ be the maximum length of the elements of 
$\Cal S_1 \cup\Cal S_2$; then we decompose each $B\in \Cal S_i$ as a
union of thin cylinders of length $k$. By the previous step the sums of their measures agrees 
with $\sum _{B\in \Cal S_i} \nu B$, and as this decomposition is now unique, 
the sums for $i=1,2$ are equal. Hence $\nu(A)$ 
does not depend on the way in which $A$ is decomposed into thin
cylinder sets, i.e.~$\nu$ is well-defined on $\B_0$.

It follows that $\nu$ is additive on $\Cal B_0$: taking $A, B$ disjoint in 
$\Cal B_0$, then $A\cup B$ is a union of thin cylinders composing $A$ and $B$,
so what we have just shown demonstrates that $\nu(A\cup B)=\nu(A)+
\nu(B)$.
Note that $\nu$ is regular (Def.~ \ref{d:Radon}) on the algebra  $\Cal B_0$ since, as we
noted above, elements of $\Cal B_0$ are unions of thin cylinders.

Now since the space $\Sigma_{\L}^{0,+}$ is compact and $\nu$ is finite, then by Alexandroff's Theorem 
\cite{DunfordSchwartz57} p.~138, Theorem 13, it has a unique regular extension from the
algebra $\B_0$ to a $\sigma$\,--\,additive measure on all of $\Cal B$;
by Theorem 14 there, this extension is regular. Furthermore, since ${\bf
  w}_0$ is in the unit simplex, the total mass is
$\nu(\Sigma_{\L}^{0,+})=\sum_{s\in \A_0} \nu([.s])=\sum_{s\in \A_0}
(\w_0)_s=1$.

 Now we return to the statement of $(i)$, proving the other direction.
We are given $\nu\in \CCL$; in particular $\nu$
    is defined for each thin
  cylinder set $[.x_0\dots x_n]$. Recall that for this to be a thin cylinder
   (see above Definition \ref{d:reduced})
 the string $x_0\dots
  x_n$ is allowed, and there exists an infinite continuation of this string to the
  right (equivalently $[.x_0\dots x_n]$ is nonempty). This follows from the hypothesis that $L$ is row--reduced.

Turning around \eqref{eq:12}, for each $n\geq 0$ and for any for any $s\in
\A_n$, we define a vector $\w_n$ by
\begin{equation}
  \label{eq:13}
   ({\bf w}_n)_s= \nu([.x_0\dots x_n=s])
 \end{equation}
 This is well-defined:   $L$ is column-- reduced, so  by $(i)$
of Lemma \ref{l:reduced}. there exists  $x_0\dots
 x_n$ such that $[.x_0\dots x_n=s]$ is an allowed string.  
 This is indeed a cylinder set (by definition, with an allowed string and nonempty)
 since 
 the fact that 
$L$ is row--reduced implies it is nonempty. 
Furthermore, this number only depends on $n$ and $s$
as $\nu$ is a central measure. Also, $\w_0\in \Delta_0$ since
$\sum_{s\in \A_0} (\w_0)_s= \sum_{s\in \A_0} \nu([.s])=\nu(\Sigma_{\L}^{0,+})
=1$. Thus we have a vector sequence $\w$; we are to show that 
${\bf w}_n= L_n{\bf w}_{n+1}$.
By \eqref{eq:5d}-
\eqref{eq:5f}, it is equivalent to show that
$$({\bf w}_n)_i = \sum_{j\in \A_{n+1}} (L_n)_{ij} (\w_{n+1})(j) $$

  We have
 $$
   ({\bf w}_n)_i = [.x_0\dots x_{n-1} i] =\cup_{j\in \A_{n+1}}\{ [.x_0\dots x_{n-1} i j]\}
 $$
  Since $\nu$ is a measure, by additivity
 \begin{equation}
    \label{eq:5}
  \nu[.x_0\dots x_{n-1} i] =\sum_{j\in\A_{n+1}}\nu[.x_0\dots x_{n-1}i j]
 \end{equation}

Now as before, by the definition of matrix multiplication, 
  
  \begin{align}
    &{\bf w}_n= L_n{\bf w}_{n+1}\nonumber\\
    &{\text{iff}}\nonumber\\
    &({\bf w}_n)_i = \sum_{j\in \A_{n+1}} (L_n)_{ij} (\w_{n+1})(j) \label{eq:5a}\\
  \end{align}

We have from \eqref{eq:13} and \eqref{eq:5}
\begin{align}
  \label{eq:10}
    & ({\bf w}_n)_i= \nu[.x_0\dots x_{n-1} i] =
      \sum_{j\in\A_{n+1}}\nu[.x_0\dots x_{n-1}i j]\\
     & = \sum_{j\in \A_{n+1}} (L_n)_{ij} (\w_{n+1})(j)\\
\end{align}
  
 because the $i^{\text{th}}$ row sum of $L$ counts how many cylinders
   of that type there are.

 This proves \eqref{eq:5a}.

\

We have shown that $\Phi$ maps $\CCL$ to
$\V_{\L}^\Delta$, and that this map
is invertible.

  To finish the proof of $(i)$, from the definitions,  $\CCL$, $\V_{L}^\Delta$ are compact and convex, 
and  $\Phi$ is an affine map. Since weak* convergence of measures
is equivalent to convergence of the measures of each thin cylinder
set, and since this correspondence is clearly bijective, this is a homeomorphism, completing the proof of $(i)$.

Next, 
 since from
part $(i)$ $\Phi$ is affine, and using Lemma \ref{lem:ergodicmeasures}, we have $(ii)$. 
Part $(iii)$ follows from $(iv)$ of Lemma \ref{l:inductive
  eigenvectors} via part $(i)$.

We next show how to derive the same results for edge spaces. Taking the state-splitting factorization 
$(A_iB_i)_{i\geq 0}$ of 
$ (M_i)_{i\geq 0}$, so 
 $AB=(A_iB_i)_{i\geq 0}= (A_0, B_0, A_1, B_1,\dots)$ are $0-1$
 matrices,
and with $
\wt l_i=\#\E_i$, 
 then
denoting by $C_i^\E$ the cone of nonnegative  column  
vectors in $\wt l_i$\,--\,dimensional Euclidean space, the diagram of
\eqref{CD:cones}
 extends to

\[
\begin{CD}\label{CD:conesAB}
     C_0    @<A_0<<   C_0^\E    @<B_0<< C_1 @<A_1<< C_1^\E @<B_1<< C_2 
 \cdots
\\
\end{CD}
\]

A nonnegative eigenvector sequence $(\w_i)_{i\geq 0}$  with eigenvalue one for the original sequence
$ (M_i)_{i\geq 0}$ 
extends uniquely to a sequence 
$
( {\bf w}_0, {\bf w}_0^\E, {\bf w}_1, {\bf w}_1^\E\dots)$ for the
dispersed matrix sequence $(AB)$, where
${\bf w}_i^\E\equiv B_i{\bf w}_{i+1}$ and so ${\bf w}_i= A_i{\bf w}_{i}^\E$.
For the $0-1$ matrices $L_i\equiv B_i A_{i+1}$, then
$(\w_i^\E)_{i\geq 0}$ is a nonnegative eigenvector sequence with eigenvalue one for 
$(L_i)_{i\geq 0}$. 

There is a natural bijective correspondence between these finite
allowed strings: $(.e_0\dots e_{n-1})$
 and  $(.x_0e_0x_1\dots x_{n-1}e_{n-1}x_n)$ for $\Sigma_{\M}^{0,+}$ 
and  $\Sigma_{AB}^{0,+}$, where $x_0= e_0^-$ and for $i>0$, $x_i=
e_i^-= e_{i-1}^+$.
This induces a bijection 
 from a cylinder set 
 $[.e_0\dots e_{n-1}]$  of  the edge space $\Sigma_{\M}^{0,+}$ to the cylinder set
 $[.x_0e_0x_1\dots x_{n-1}e_{n-1}x_n]$ of the vertex space $\Sigma_{AB}^{0,+}$.

Given $\nu_M$ on $\Sigma_{\M}^{0,+}$ we define $\nu_{AB} $ on
$\Sigma_{AB}^{0,+}$ via
 this correspondence between cylinder sets, setting
$\nu_{AB}([.x_0e_0x_1\dots x_{n-1}e_{n-1}x_n])\equiv\nu_M([.e_0\dots e_{n-1}]). $
This gives  a bijection from $\CCM$ to $\CCAB$.
From part $(i)$ we have the bijection $\Phi_{\AB}:\, \CCAB\to \V_{\AB}^\Delta$.
Thus $\nu_{AB}$ determines a nonnegative eigenvector  sequence $( {\bf w}_i {\bf w}_i^\E)_{i\geq 0}\in \V_{\AB}^\Delta$,
which in turn specifies  $(\w_i)_{i\geq 0}$, a nonnegative eigenvector sequence
with eigenvalue one for $(M_i)_{i\geq 0}$, this correspondence also
being a bijection.

We claim that
  ${\bf w}_0^\E\in \Delta_0^\E$. This will show that the composition
  of the three correspondences defines  
a bijective map $\Phi_M:\, \CCM\to \V_{\M}^\Delta$.

We know from $(i)$ that  $( {\bf w}_i {\bf w}_i^\E)_{i\geq 0}\in
\V_{\AB}^\Delta$. So
${\bf w}_0\in \Delta_0$. 
Each column of the matrix $A_0$ has a
single $1$ entry with the rest $0$ (the edge $\e_0$ begins at a
well-defined symbol $\e_0^-$). Thus $1=
||\w_0||=||A_0\w_0^\E|| =||\w_0^\E|| $ 
and so 
${\bf w^\E}_0\in \Delta_0^\E$. This completes the proof.
\end{proof}

\begin{defi}\label{d:measure}
Given an alphabet sequence $\A$, a $0-1$ matrix sequence $L$ as above,
then for the vertex shift space $\Sigma_{L}^{0,+}  $, given 
$\w\in \V_{L}$, we write $\nu_\w$ for the
  measure on $\Sigma_{L}^{0,+}  $ defined by \eqref{eq:basictheorem}.
  (If we wish to emphasize  the matrix sequence involved, we write  e.g.~
  $\nu^L_\w$).
 We make the similar definitions for the edge
space  $\Sigma_{M}^{0,+}  $.

\end{defi}

We examine some effects of the  operation 
of gathering:

\begin{prop}\label{p:gathermeasure}
Let $(\A_i)_{i\geq 0}$ be an alphabet   sequence with $l_i= \#\A_i\geq
1$,  
and
let $M= (M_i)_{i\geq 0}$ be an $(\A_i\times \A_{i+1})$ 
 nonnegative integer  matrix sequence.
 For $0=n_0< n_1<\dots$, let
 $\wt M=(\wt M_i)_{i\geq 0}$ 
denote the gathered matrix sequence, with $(\wt \A_i) _{i\geq
  0}$, $(\wt \E_i) _{i\geq
  0}$ 
the
gathered alphabet and edge alphabet sequences. Write $\mathfrak B_M$,
 $\mathfrak B_{\wt M}$
for the corresponding Bratteli diagrams.

Then if $M$ is reduced  so is
$\wt M$.

The map $\Phi$ taking an edge path to the corresponding gathered edge path is
a homeomorphism 
from $\Sigma_{\M}^{0,+} $ to $\Sigma_{\wt
  M}^{0,+} $. This conjugates the actions of $\FC_M$ and $\FC_{\wt
  M}$ and preserves the stable equivalence
relations. If $\O$ is an order on  $\mathfrak B_M$ then writing $\wt \O$
for the order  induced on $\mathfrak B_{\wt M}$ by this map, the
Vershik maps and adic transformations
$T_\O, T_{\wt\O}$
are conjugate.
In this way, the orders on $\mathfrak B_M$ embed (in general, not surjectively) in 
those on  $\mathfrak B_{\wt M}$.

$\Phi$ defines an affine map 
$\Psi:\, \V_{\M}^\Delta \to
\V_{\wt M}^\Delta $. 

Writing $\Phi_*$ for the map  induced by the homeomorphism $\Phi$ on the collection of
all Borel measures, this is an affine homeomorphism  of 
measures, which preserves the $\FC$\,--\, invariant, ergodic, and conservative measures.

All the above holds for vertex in place of edge spaces. 
\end{prop}
\begin{proof}First, being reduced clearly passes from $M$ to its
  gathering (though not vice-versa, by easy counterexamples). 
We note that the number of elements of the gathered edge alphabet $\wt \E_i$ which
begin at a vertex $a\in \A_{n_i}$ and end at $b\in \A_{n_{i+1}}$ is
$(M_{n_i}^{n_{i+1}-1})_{ab}= (\wt M_i) _{ab}$.

The map $\Phi$ is defined to take an edge path 
$\e= (.e_0e_1\dots
)$ to the corresponding gathered edge path $\wt e=  (.\wt e_0\wt e_1\dots
) $; thus $\wt e_0= (e_0\dots e_{n_1-1})$,  
$\wt e_1=(e_{n_1}\dots e_{n_2-1})$ and so on. This is clearly a
homeomorphism: it is a bijection; the inverse image of a thin cylinder
set is a thin cylinder set and the image of a cylinder set is a union
of thin cylinders. Thus a generator of $\FC_{\wt M}$ is taken by $\Phi^{-1}$ to
a generator of $\FC_{ M}$, so  $\FC_{\wt M}$ embeds in $\FC_{ M}$, while
a generator of $\FC_{M}$ is taken by $\Phi$ to
an element of  $\FC_{\wt M}$ which is a finite product of generators (one for each
cylinder of that union). Thus the actions are conjugate. That the
Vershik maps are congugate is clear from the orders. 
It follows that the map induced on measures takes invariant Borel measures
to invariant Borel measures, and preserves ergodicity and conservativity. 

We define $\Psi:  \V_{\M}^\Delta \to
\V_{\wt M}^\Delta $  sending $\w= (\w_0\w_1\dots)$ to the subsequence at
times $n_i$; this is an affine bijection. Since a central measure
$\nu\in \CCM$ is
defined on a thin cylinder set by the final vertex, $\nu([.e_0\dots e_j: e_j^+=s]= \w_i(s)$,
considering $j= n_i$ defines a natural map 
from $\CCM$ to
$\CCMtil$. This affine bijection  is the restriction to probability measures 
of the map $\Phi_*$.

While it is true that an  order $\O$ on the Bratteli diagram of $\Sigma_{\M}^{0,+} $
determines an order on the edge paths of $\Sigma_{\wt M}^{0,+} $ which
enter a given vertex, there are more possibilities for 
the gathered diagram: as easy examples show, not every order on $\wt \E_i$ 
is induced from a anti-lexicographic order on the  edge paths from time
$n_i$ to $n_{i+1}$.
\end{proof}

\medskip

\subsection{Counting the ergodic central measures}\label{ss:Count}
According to
\cite{BezuglyiKwiatkowskiMedynetsSolomyak13}, the upper bound proved
here is a ``folklore
theorem''; a proof making use of
Theorem 2.1 
of Pullman~\cite{Pullman71}  is given in Proposition 2.13 of~\cite{BezuglyiKwiatkowskiMedynetsSolomyak13}
 (though Pullman's result is essentially non-stationary  he only
 applies his argument to the
case of a single matrix). The proof we present is self-contained.

Recall that we are using the $L^1$\,--\,norm on  $\r^d$, with
$||{\bf w}||=\sum_{i=1}^d |w_i|,$ so that 
$ {\bf w} \mapsto {{\bf w}}/{||{\bf w}||}
$
projects  $\r^{d+}\setminus \{{\bf 0}\}\to \Delta$. Given a sequence $(M_i)_{i\geq 0}$
of  
$(\A_i\times \A_{i+1})$ 
 nonnegative integer matrices with $l_i= \#\A_i$,
we  define 
$f^M_i:\Delta_{i+1}\to \Delta_i$
by 
$$
f^M_i({\bf v})=\frac{M_i {\bf v}}{||M_i{\bf v}||},$$
and write $f^M_{(k,n)} =f^M_{k} \circ f^M_{{k+1}} \circ\cdots \circ f^M_{{n-1}} $.
For  $0\leq k\leq n$ we define:
$$\Delta_M^{(k,n)}=
f^M_{(k,n)} \Delta_n,$$
and so
$$\Delta_M^{(k,n)}= C_M^{(k,n)}\cap \Delta_M^k.
$$
Taking the  intersection of these  nested simplices, we define
$$\Delta_M^{(k,+\infty)}=\cap_{n\geq k} \Delta_M^{(k,n)}=  C_M^{(k,\infty)}\cap \Delta_k.$$
 
\begin{lem}\label{l:simplex to simplex}
For any finite set of points 
$\{{\bf v}_1,\dots {\bf v}_j\}\subseteq \Delta_{k+1}
$,  the image of the 
convex hull is the convex hull of the image: writing $[ {\bf v}_1,\dots {\bf v}_j]$ 
for the   convex hull of these points, then 
  $[ \{f^M_{k}({\bf v}_i)\}_{1\leq i\leq j}]= f^M_{k}([{\bf v}_{1\leq i\leq j}])$.
\end{lem}
\begin{proof}
For two points, the statement is that 
the image of the segment $[{\bf v},{\bf w}]$ with endpoints ${\bf v},{\bf w}$
is the segment (possibly a point)
 $[f^M_{k}{\bf v}), f^M_{k}({\bf w})]$.
Indeed,
since $M_k$
is linear, the image of a line segment in the positive cone 
$C_{k+1}=\r^{l_{k+1},+}$ 
is a line segment in $C_k$,
and when normalized to 
$\Delta_k$ this gives either a line segment or a point, with those
extreme points. 

It follows from this statement  that the image by $f_k^M$ of a convex
set is convex,
but that is not yet enough as we need to show that it is generated by
the image of the extreme points. We prove this by induction on the
number $j$ of extreme points, illustrating the inductive step
$(j\implies j+1)$ for the case $j=2$. Thus we
consider ${\bf v}, {\bf w}, {\bf z}\in\Delta_{k+1}$;
given a point 
 ${\bf x}=a {\bf v}+b {\bf w}+ c {\bf z}$ 
 where $a+b+c =1$ and $a,b,c\geq 0$,
there is a point $\tilde {\bf x} $ on the segment 
$[{\bf v},{\bf w}]$ such that ${\bf x}$ lies on the segment $[\tilde {\bf x} , {\bf z}]$.
Indeed, take 
$\tilde {\bf x}=({a}/{a+b}){\bf v}+({b}/{a+b}){\bf w}$. 
 By the previous argument each of these segments is 
mapped to a segment, the result follows.
The general induction step is similar.
\ \ \end{proof}

\begin{lem}\label{l:ConeNumberofextremepts}
\item{(i)} 
The sets
$\Delta_M^{(k,n)}$
and
$\Delta_M^{(k,\infty)}$ are   compact, convex and nonempty.
$\Delta_M^{(k,n)}$ has  at most $l_n$  extreme points.
\item{(ii)}  The map $f^M_{k}$ sends
$\Delta_M^{(k+1,\infty)}$ onto $\Delta_M^{(k,\infty)}$, and maps
the  set of
extreme points $\Ext (\Delta_M^{(k+1,n)})$ onto $\Ext(\Delta_M^{(k,n)})$, 
and similarly for $n=\infty$.  $\Delta_M^{(k,\infty)}$ has at most
$\inf_{n\geq k} l_n$ extreme points.
\item{(iii)} 
Via the projection 
$ {\bf w} \mapsto {{\bf w}}/{||{\bf w}||}
$
from $C_k\setminus \{\bf 0\}$ onto  
$\Delta_k$, the collection of extreme rays of the convex cone
$\Ext(C_M^{(k,\infty)})$ corresponds bijectively to $\Ext(\Delta_M^{(k,\infty)})$.
\end{lem}
\begin{proof}
For $(i)$ 
 we have a sequence of onto, continuous maps 

\[
\begin{CD}
     \Delta_M^{(k,n)}   @<f^M_k<<  \Delta_M^{(k+1,n)}  @<f^M_{k+1}<<\Delta_M^{(k+2,n)}  \cdots
      \Delta_M^{(n-1,n)} @<f^M_{n-1}<<   \Delta_n
%>>
%this corrects an AUCtex problem--it was looking for a "close quote"
\\
\end{CD}
\]
so from Lemma \ref{l:simplex to simplex}, 
$\Delta_M^{(k,n)}$ is a compact convex nonempty set, hence so is the nested intersection
$$\Delta_M^{(k,\infty)}=\cap_{n=0}^{\infty}\Delta_M^{(k,n)}$$.
Also from Lemma \ref{l:simplex to simplex}, 
$$\Ext(\Delta_M^{(k,n)})\subseteq f^M_{(k,n)}(\Ext\Delta_n)$$
whence 
$\#\Ext(\Delta_M^{(k,n)})\leq \#\Ext(\Delta_n)= l_n$.

For $(ii)$, to show that
$f^M_k(\Delta_M^{(k+1, \infty)})=\Delta_M^{(k, \infty)}$, we prove
 the following more general 

\noindent
{\bf CLAIM:}
Let $(K_i)_{i\geq 0}$ be compact sets with  continuous maps $f_i: K_{i+1}\to K_i$.
Write $f_{(k,k)}= f_k$; $f_{(k,n)}= f_k\circ f_{k+1}\circ \dots \circ
f_n$ for $n> k$.
Then for $K_{(k,n)}\equiv f_{(k,n)}(K_{n+1})$ 
and $K_{(k,\infty)}\equiv \cap_{n=k}^\infty K_{(k,n)}$, we have: 
$$K_{(k,n+1)}  \subseteq K_{(k,n)} \text{ and }
f_k(K_{(k+1,\infty)})=K_{(k,\infty)}.$$

\noindent
{\bf Proof of Claim:} Since 
for each $n$, 
$f_{n+1}(K_{n+2})\subseteq K_{n+1},$ 
applying $f_{(k,n)} $ gives 
$$K_{(k,n+1)} =f_{(k,n+1)}(K_{n+2}) \subseteq K_{(k,n)}.$$

Next,
$f_k(K_{(k+1,n)})= f_{(k,n)} (K_{n+1})= K_{(k,n)}$,
so 
$$f_k( K_{(k+1,\infty)})=  f_k(\cap_{n=k+1}^\infty K_{(k+1,n)})\subseteq
\cap_{n=k+1}^\infty f_k( K_{(k+1,n)})=  \cap_{n=k}^\infty
K_{(k,n)}=K_{(k,\infty)}.
$$

To show this is onto,
let $x\in K_{(k,\infty)}$; we shall find 
$w\in  K_{(k+1,\infty)}$ such that
$f_k(w)=x$.
Since $x\in  K_{(k,n)}$ for each $n$,
there exists 
$y_{n+1}\in K_{n+1}$ with  $f_{(k,n)}(y_{n+1})=x$.
Define 
$w_n= f_{(k+1,n)}(y_{n+1})$; thus 
$ f_k(w_n)= x$.

Since $w_m\in K_{(k+1,n)}$ for all $m\geq n$ (by the first part of the
Claim) and this set is compact,
there exists a subsequence 
$w_{m_l}$ and point $w$ with $w_{m_l}\to w$. Therefore 
$w\in  K_{(k+1,m_l)}$  for each $l$ and  hence $w\in  K_{(k+1,\infty)}$. 
 By continuity of $f_k$,  $f_k(w)=x$ as well.
\ \ \qed

\

Next we examine the extreme points of  $\Delta_M^{(k,\infty)}$.  
For $x \in \Delta_M^{(k,\infty)}$,
for each $n>k$, since $\Delta_M^{(k,\infty)}\subseteq \Delta_M^{(k,n)}$,
there are real numbers
$\lambda_i^{(n)}$, $1\leq i\leq j(k,n)$,
such that 
$$
x=\sum_{i=1}^{j(k,n)}\lambda_i^{(n)}{\bf z_i}^{(n)}
$$   
where $\{ {\bf z}_1^{(n)},\dots,{\bf z}_{j(k,n)}^{(n)}\}=
\Ext(\Delta_M^{(k,n)})$.
Let us write  
$m=\min_{n\geq k} \{j(k,n)=\#\Ext(\Delta_M^{(k,n)})\}.$
Thus  there exists $J$ such that for every $t\geq J$, 
$\#\Ext(\Delta_M^{(k,t)})=m.$
We claim  that
$\#\Ext(\Delta_M^{(k,\infty)})\leq m.$

Now by compactness of 
$\Delta_M^{(k,n)}$
and $\Delta_M^{(k,\infty)}$,
there exists for each $i$
a subsequence
of $ ({\bf z}_i^{(j(k,n))})_{n=J}^\infty$
which converges to some point ${\bf z}_i\in\Delta_M^{(k,\infty)}$.
We claim that 
$\Ext(\Delta_M^{(k,\infty)})\subseteq \{ {\bf z}_i\}_{i=1}^m$.
(Here the order on each set 
$\Ext(\Delta_M^{(k,n)})$ is fixed but otherwise 
is of no importance.)
Indeed, any point
$x\in\Delta_M^{(k,\infty)}$ can be written as a convex combination 
$$
x=\sum_{i=1}^{m}\lambda_i^{(j(k,n))}{\bf z}_i^{(j(k,n))}
$$   
for each  $n\geq J$;
by compactness of $[0,1]$, for each $i$
there exists a 
subsequence
of $\lambda_i^{(j(k,n))}$
converging to $\lambda_i$ such that
$
x=\sum_{i=1}^{m}\lambda_i{\bf z_i}
$;
hence  $\Ext(\Delta_M^{(k,\infty)})\subseteq \{ {\bf z}_i\}_{i=1}^m$,
so indeed $\#\Ext(\Delta_M^{(k,\infty)})\leq m.$

Finally we show 
 $f^M_{k}(\Ext(\Delta_M^{(k+1,\infty)}))= \Ext(\Delta_M^{(k,\infty)})$.

We claim
 that given ${\bf b}\in \Ext(\Delta_M^{(k,\infty)})$, there exists some
${\bf a}\in \Ext(\Delta_M^{(k+1,\infty)})$ which maps to ${\bf b}$.
Indeed, since the map $f^M_k$ is onto, there exists some preimage ${\bf c}
\in \Delta_M^{(k+1,\infty)};$
if ${\bf c}$ is not extreme, it is a nontrivial 
convex combination of the extreme points; 
but by the previous lemma, 
its image is a (generally different, since $f^M_k$ may not be linear)
 convex combination 
of the images of these points. Now we use 
Lemma \ref{l:simplex to simplex}:   ${\bf b}= f^M_k({\bf c})$
is extreme, hence  this convex combination is trivial, 
either because the coefficients are all zero except for one or because 
the points coincide, and in either case one of the extreme  points
must map to ${\bf b}$.

The proof of $(iii)$  is obvious. 
\ \ \end{proof}

\begin{prop}\label{p:extremepts}
The number of  ergodic $\FC_M-$invariant
 probability  measures is $\#\Ext\V_{M}^\Delta=\liminf \# \Ext \Delta_M^{(k,\infty)}$
 and this is at most $\liminf l_n$.
\end{prop}  
\begin{proof}  
As explained above (from $(ii)$ of Theorem \ref{t:basic_thm}) we wish to count $\Ext\V_{M}^\Delta$. 
 Now $\w= (\w_0, \w_1,\dots) \in \V_{M}^\Delta$ iff
$\w_0\in \Delta_0$ and for all $i$,
$\w_i= M_i \w_{i+1}$;   this is an extreme point iff, by $(iii)$ of
Lemma \ref{l:inductive eigenvectors}, $\w_i\in \Ext  C_M^{(i,+\infty)}$. 
The number of finite sequences 
$(\w_0\dots \w_k )$ satisfying this for all $i\leq k$ is nondecreasing
in $k$ and so is equal to
$\#\Ext(C_M^{(k,\infty)})$, whence $  \#\Ext\V_{M}^\Delta=\sup_{k\geq 0} \#\Ext(C_M^{(k,\infty)})$.

From
$(iii)$ then $(ii)$ of Lemma 
\ref{l:ConeNumberofextremepts}, 
$\#\Ext(C_M^{(k,\infty)})= \#\Ext(\Delta_M^{(k,\infty)})\leq
\inf_{n\geq k} l_n$.
Concluding, $
\#\Ext\CCM=\sup\inf \#\Ext(\Delta_M^{(k,\infty)})= \liminf \#\Ext(\Delta_M^{(k,\infty)})
\leq\liminf l_n$.
\end{proof}

\medskip

\

\subsection{Nested diagrams and adic towers}\label{ss:towers}
We begin by recalling from  Definition
\ref{d:generalizedmatrices} the partial  order on individual generalized
matrices. We then extend this to matrix  sequences and equivalently to 
Bratteli diagrams.
\begin{defi}\label{d:subdiagram}
  
  Given alphabets $\A, \B, \wh \A, \wh \B$, then
for $M$, $\wh M$   generalized nonnegative integer matrices of size $(\A\times\B)$ and 
$(\wh \A\times \wh\B)$ respectively, then as in  Definition
\ref{d:generalizedmatrices}, we say that 
$M\leq \wh M$ iff $\A\subseteq \wh \A$, $\B\subseteq \wh \B$  and for all $(a,b)\in  \A\times  \B$ we have
$M_{ab}\leq
\wh M_{ab}$.

Given  Bratteli diagrams
 $\mathfrak B_{\A,\E}, $$ \mathfrak
 B_{\wh \A,\wh\E}$ with alphabet and edge sequences $(\A, \E)$
and $(\wh \A, \wh \E) $, suppose that  $\A\leq \wh
\A, $ $\E\leq \wh \E$, by which we mean that $\A_k\subseteq \wh \A_k$,
 $\E_k\subseteq \wh \E_k$ for all $k\geq 0$. 
 Equivalently, $M\leq \wh
M$, i.e.~ $M_k\leq
\wh M_k$ for all $k$.
We say in this situation that the matrix sequences, and diagrams, are {\em nested} with
 the first a {\em subdiagram} of the second, written  $\mathfrak
 B_{\A,\E, M} \leq \mathfrak
 B_{\wh \A,\wh\E, \wh M}$.

Given orders $\O, \wh\O$ on the nested diagrams, we say $\O\leq \wh \O$ iff
   ($e<f$ in $\O_k$)$\implies$ ($e<f$ in $\wh \O_k$). 
That is,  for each $k\geq 0$, $\O_k\subseteq \wh \O_k$  as  relations (i.e.~as subsets of
$\wh \E_k\times \wh \E_k$). We then write $\mathfrak B_{\A,\E,M, \O}\leq \mathfrak
 B_{\wh \A,\wh\E, \wh M, \wh \O}$; this is a 
 partial order on the class of all ordered Bratteli
diagrams.
\end{defi}

\begin{rem}\label{r:gen_matrices}
\item{(i)} In Remark \ref{r:reduced} we encountered a special case
  of this: given a
nonnegative integer matrix $\wh M$,
 let $M$ denote the reduced matrix sequence guaranteed by Lemma
 2.2 of~\cite{Fisher09a}. In producing
the reduced diagram, one has removed symbols from the alphabets in the
alphabet sequence, resulting in a new  sequence $
 \A_k\subseteq \wh\A_k$, whence $ M\leq \wh M$, where we have  eliminated the all-zero rows and columns.

\item{(ii)} One can assume here that $\A= \wh \A$.
That is,  when producing a subdiagram of $ B_{\wh \A,\wh\E, \wh
  \O}$, instead of
erasing both edges and symbols,  one can retain the alphabet
sequence $\wh A$ and erase only
edges.  The resulting sequence $M$  will have some
all-zero
rows and  columns, corresponding to the erased symbols (and so may not
be reduced) but 
 $M_k$ will be  the same size as $\wh M_k$ for all
$k$, with  $( M_k)_{ab}\leq( \wh M_k)_{ab}$ for all $(a,b)\in \wh\A_k\times
\wh\A_{k+1}$.

\end{rem}

 

\begin{prop}\label{p:nested_diagrams}

If $\mathfrak B_{\A,\E}\leq \mathfrak
 B_{\wh \A,\wh\E}$ then the collection of allowed edge paths for 
the first is a subset of that for the second. That is, 
$\Sigma_{\M}^{0,+}\subseteq \Sigma_{\wh M}^{0,+}$ where 
$\M, \wh M$ are the corresponding matrix sequences. 
The stable equivalence
classes are  nested: for all $\e\in \Sigma_{\M}^{0,+}$, 
 $W^s_M(\e)\subseteq W^s_{\wh M}(\e).$
$\FC_M$ naturally embeds as a subgroup of $\FC_{\wh M}$.
Let now $\O\leq  
\wh \O$. Then
 $\e<\f$ in $(\Sigma_{\M}^{0,+},\O)$ implies that $\e<\f$  in 
$(\Sigma_{\wh M}^{0,+},\wh\O)$, and  $ \Sigma_{\M}^{0,+}\setminus\Cal N_M$
is contained in $ \Sigma_{\wh M}^{0,+}\setminus{\Cal N_{\wh M}}$.
\end{prop}

\begin{proof}
 Given a generator $\gamma$ of
  $\FC_M$,
this extends to a map of  $\Sigma_{\wh M}^{0,+}$, as a generator of
$\FC_{\wh M}$. Therefore $\FC_M$ embeds in $\FC_{\wh M}$.  The
other statements are clear.
\end{proof}

\begin{defi} 
In the above situation, 
we say that
 $T_\O$ on $ \Sigma_{\M}^{0,+}\setminus\Cal N_M$
is a {\em sub-adic transformation } of
$T_{\wh \O}$ on $ \Sigma_{\wh M}^{0,+}\setminus{\Cal N_{\wh M}}$.
\end{defi}

Given  nested diagrams $\mathfrak B_{\A,\E,\O}\leq \mathfrak
 B_{\wh \A,\wh\E, \wh \O}$,
we define for each fixed $m$  sequences 
\begin{align*}
 \E^{(m)}_i&=\wh \E_i , \text{ for } i\leq m & 
\A^{(m)}_i&=\wh \A_i , \text{ for } i\leq m+1& 
\O^{(m)}_i&=\wh \O_i , \text{ for } i\leq m & 
\\
&=\E_i, \text{ for } i> m,
 & &=\A_i, \text{ for } i> m+1,
&  &=\O_i, \text{ for } i> m.
\end{align*}
We write $
M^{(m)}
$ for the corresponding generalized matrix sequence (which may not be
reduced, even if $\wh M$ and $M$ are).
Then $\Sigma _{M^{(m)}}^{0,+}$ denotes
the edge path space defined from the matrix,
alphabet, edge alphabet and order sequences
$\bigl(M^{(m)},\A^{(m)},  \E^{(m)} ,  \O^{(m)}\bigr)$, with
$\FC_{M^{(m)}}$ the corresponding  group of finite coordinate
changes.

We
recall that, in the category of sets, taking the connecting morphisms which
define the directed index set to be inclusion of sets, then the {\em direct limit} 
of a nested increasing sequence of sets is 
simply their union; indeed, if the connecting morphisms are all
injections, then any direct limit can be thought of in this way,
by identifying a set with its image in the limiting space. In the
category 
of topological spaces 
(where the connecting morphisms are
continuous maps) then the direct limit topology is defined to be 
the {\em final topology} on the direct limit set, i.e.~the smallest
topology to make the maps continuous; thus the image of each space is
open, and on an increasing union of open sets the direct limit
topology  is simply the union of
the topologies.

\begin{defi}\label{d:adictower} 
The {\em adic tower  space} of $\wh M$ over $\M$
is
the direct limit set
$\Sigma_{\wh M/M}^{0,+}\equiv \varinjlim \Sigma _{M^{(m)}}^{ 0,+} $, together with the direct
limit topology. 
 The  {\em base } of the tower is $\Sigma_{M}^{0,+}$.

Note: at one point below (in the proof of Theorem
\ref{t:towerring}) we will need to include the ambient space in the
notation, and then we will write
${}^{\wh M} \!  M^{(m)} $ 
 for  $M^{(m)}$, and so $\Sigma _{{}^{\wh M} \!  M^{(m)} }^{0,+}$ for $\Sigma _{M^{(m)}}^{ 0,+} $.
\end{defi}

 The terminology comes from a connection
 with Kakutani towers, see Proposition \ref{p:towermap}.

\begin{prop}\label{p:towerspace}
Given   nested Bratteli diagrams 
$\mathfrak B_{\A, \E}\leq \mathfrak B_{\wh A,\wh\E}
,$ and assuming that $M$ is column--reduced, the set $\Sigma_{\wh M/M}^{0,+}$ is the union of all the stable
sets of  elements of $\Sigma_{\M}^{0,+}\subseteq \Sigma_{\wh
  M}^{0,+}$, and  
 the tower space   $\Sigma_{\wh M/M}^{0,+}$ is  the smallest 
$\FC_{\wh  M}$\,--\,invariant set containing $\Sigma_{\M}^{0,+}$:
$$\Sigma_{\wh M/M}^{0,+}=
W^s_{\wh M}(\Sigma_{\M}^{0,+})\equiv \cup\{ W^s_{\wh M}(\e):\; \e\in
\Sigma_{\M}^{0,+}\}=\FC_{\wh M}(\Sigma_{\M}^{0,+})
.$$ The direct limit topology on $\Sigma_{\wh M/M}^{0,+}$ is equal to
its relative topology as a subset of 
$\Sigma_{\wh M}^{0,+}$.

\end{prop}

\begin{proof}
As noted above the direct limit space of the nested spaces
$\Sigma _{\wh M^{(0)}}^{0,+} \subseteq \Sigma _{\wh M^{(1)}}^{
  0,+}\subseteq \dots$  is 
the union  $\varinjlim
\Sigma _{M^{(m)}}^{ 0,+} =
\cup_{m\geq 0}\Sigma _{M^{(m)}}^{ 0,+} .$
Given $\e\in \Sigma_{\M}^{0,+}$, then 
$\f\in 
W^s_{\wh M}(\e)$ iff there exists $m$ such that 
  $f_k= e_k\; $ for all $k> m$. In this case $f\in \Sigma
  _{M^{(m)}}^{0,+}$ and hence is in the nested union. 
Thus 
$ W^s_{\wh M}(\Sigma_{\M}^{0,+})\subseteq \Sigma_{\wh M/M}^{0,+}$.

The  direct limit topology is 
the union of the  relative topologies on
 each  of the nested spaces $\Sigma
  _{M^{(m)}}^{0,+}$ as a subset of  $\Sigma_{\wh M}^{0,+}$; we show this is
equal to the relative topology on the tower space. A relatively open set in $\Sigma_{\wh
  M/M}^{0,+}$ is of the form $\Cal U\cap \Sigma_{\wh
  M/M}^{0,+}$ for $\Cal U$ open in $ \Sigma_{\wh M}^{0,+}$. But
$\Cal U\cap \Sigma_{\wh
  M/M}^{0,+}= \Cal U\cap (\cup_{m=0}^\infty \Sigma
  _{M^{(m)}}^{0,+})=\cup_{m=0}^\infty  (\Cal U\cap \Sigma
  _{M^{(m)}}^{0,+})$ which is open in the direct limit topology.
Conversely, any open set in $ \Sigma_{\wh M}^{0,+}$ is a countable
disjoint union of thin cylinder sets, since the collection $\{\Cal
V_i\}_{i\geq 0}$ of thin cylinders is countable and they generate the topology. An open set
for the direct limit has the form   $\cup_{m=0}^\infty(\Cal U_m\cap \Sigma
  _{M^{(m)}}^{0,+})$ where $\Cal U_m$  is open in $ \Sigma_{\wh
    M}^{0,+}$. Since the  $\Sigma
  _{M^{(m)}}^{0,+} $ are nested, this can be rewritten as  $\cup_{i=0}^\infty\cup_{m=0}^\infty (\Cal V_{k_i}\cap \Sigma
  _{M^{(m)}}^{0,+})= (\cup_{i=0}^\infty \Cal V_{k_i})\cap \Sigma
  _{\wh M/M}^{0,+}$
which is open in the relative topology.
\end{proof}

Now since the tower is an $\FC_{\wh M}$-invariant subset, both  $\FC_{\wh M}$ and the Versik
map $ T_{\wh \O}$  act on  $\Sigma_{\wh M/M}^{0,+}$
by restriction. We next examine invariant Borel measures, recalling Def.~\ref{d:wandering}.

\begin{theo}\label{t:towermeasure}
  Given   nested Bratteli diagrams 
$\mathfrak B_{\A, \E}\leq \mathfrak B_{\wh A,\wh\E}
,$ with $M$  column--reduced, and assuming that for $l_i=\#\A_i$ we
have $\liminf
l_i>1$, then:
 \item{(i)}$\FC_{\wh M}$\,--\and  $\FC_M$\,--\,invariant Borel
   measures on the tower $\Sigma_{\wh M/M}^{0,+} $ and base $B\equiv \Sigma_{M}^{0,+}$ correspond bijectively,
via invariant extension and restriction; the same holds for invariant subsets.
Furthermore, writing  $\wh \nu$, $\nu$ for corresponding measures:
\item{(ii)}Wandering sets in the base 
 and tower correspond, as
  follows. If $E\subseteq B$ is wandering for $\FC_M$ then it is
  wandering for $\FC_{\wh M}$. If $\wh E$ is wandering for $\FC_{\wh
    M}$ then $E\equiv B\cap \FC_{\wh
    M}(\wh E)$  is
  wandering for  $\FC_M$, and $\wh\nu(\wh E)=\nu(E)$.

The $\FC_M$ action on the base 
$(\Sigma_{M}^{0,+},\nu)$
is conservative,  respectively
   ergodic, iff the $\FC_{\wh M}$\,--\,action on the tower
$(\Sigma_{\wh M/M}^{0,+},\wh \nu)$
 is. If $\nu$ is finite, $\wh\nu$ is conservative. 
\item{(iii)}Given an order $\wh \O$ on  $\Sigma_{\wh M}^{0,+}$,  the corresponding
statements  hold for the action of
 $T_{\wh \O}$ on the invariant parts of these spaces.
\end{theo}
\begin{proof}For the first part of $(i)$, more precisely, 
we show an $\FC_M$\,--\,invariant Borel measure  $\nu$ on the base $\Sigma_{\M}^{0,+}$
has a unique $\FC_{\wh M}$\,--\, invariant extension $\wh \nu$ to the 
adic tower
$\Sigma_{\wh M/M}^{0,+}\subseteq \Sigma_{\wh M}^{0,+}$,  that the 
restriction $\nu$ of an invariant Borel measure $\wh \nu$ to the base is
$\FC_M$\,--\,invariant and that the extension
of $\nu$
is again $\wh \nu$. 

To get started, we extend the given   $\FC_M$\,--\,invariant Borel measure $\nu$ from
$\Sigma_{\M}^{0,+}$
 to $\Sigma_{\wh
  M}^{0,+}$ by assigning mass zero to the complement.
Then, since 
$\Sigma_{\wh M/M}^{0,+}= \cup_{m\geq 0}  \Sigma _{\wh
  M^{(m)}}^{ 0,+} $ is a nested union, we  shall
define for each $m\geq 0$ a measure $\wh \nu_m$ on $\Sigma_{\wh M}^{0,+}$, with support on 
$\Sigma _{M^{(m)}}^{ 0,+}$, and  shall then take
 $\wh \nu$ to be the limit of $\wh \nu_m$. (Here one can think of $\wh \nu_{-1}=\nu$.)

Let $[.f_0\dots f_m]$ be a thin  cylinder of $\Sigma_{\wh
  M}^{0,+}$. Suppose that this cylinder
meets $\Sigma _{M^{(m)}}^{ 0,+}$, so
there exists 
$f=(.f_0f_1\dots)\in [.f_0\dots f_m]\cap \Sigma _{M^{(m)}}^{ 0,+}$.
Then 
 $f_k\in \E_k$ for all $k>m$. Since $M$ is column--reduced, we can 
extend to the left to a path $\e=(.\e_0 \e_1\dots e_m f_{m+1}\dots)
\in \Sigma_{\M}^{0,+}$  with $e_m^+= f_m^+$.
Hence for the cylinder $[.e_0\dots e_m]$ of
$\Sigma_{\wh \M}^{0,+}$, there exists 
$\gamma\in \FC_{\wh M}$ with $\gamma([.f_0\dots f_m])= [.e_0\dots
e_m]$.

Given now a Borel subset $E\subseteq [.f_0\dots f_m]\cap \Sigma _{\wh
  M^{(m)}}^{ 0,+}$, then since $\gamma$ preserves the stable
equivalence relation it preserves $\Sigma _{\wh  M^{(m)}}^{ 0,+} $, so
$\gamma(E)\subseteq  [.e_0\dots e_m]\cap
\Sigma _{\wh  M^{(m)}}^{ 0,+} =  [.e_0\dots e_m]\cap
\Sigma _M^{ 0,+}\equiv  [.e_0\dots e_m]_M$
 and we define 
$\wh \nu_m(E)
=
\nu(\gamma(E)) $. This number
does not depend on choice of the cylinder
$[.e_0\dots e_m]$: suppose there is another path
$\wt \e=(.\wt \e_0 \wt \e_1\dots \wt e_m f_{m+1}\dots)
\in \Sigma_{\M}^{0,+}$; then
 there exists
$\eta\in \FC_{ M}$ with $\eta([.\wt e_0\dots \wt e_m]_M)= [.e_0\dots
e_m]_M$. Therefore
indeed, $\nu(\eta^{-1}\circ \gamma(E))= \nu(\gamma(E))$ by the $\FC_M$-invariance of $\nu$.

This defines $\wh \nu_m$ on the Borel subsets of a cylinder
$[.f_0\dots f_m]$; we extend to the Borel  $\sigma$\,--\,algebra of 
$\Sigma_{\wh M}^{0,+}$ by additivity.

We then define  $\wh \nu= \lim_{m\to\infty} \wh
\nu_m$; that is, its value on a
Borel subset of  $\Sigma_{\wh M/M}^{0,+}$ or $\Sigma_{\wh M}^{0,+}$ is the limit of the increasing sequence of numbers given
by $\wh \nu_m$ as $m\to\infty$; $\sigma$\,-- additivity is
preserved in the limit.

We claim that $\wh \nu$ defined in this way is $\FC_{\wh M}$\,--\,invariant. Consider $\gamma$ a  generator of $\FC_{\wh M}$ which
interchanges two cylinders
$[.f_0\dots f_k]$, $ [.g_0\dots g_k]$, so $f_k^+=  g_k^+$.
 It will be enough to show
that for any $m>k$, and for any Borel subset $E\subseteq [.f_0\dots f_k]$, $\wh \nu_m(E)=
\wh \nu_m(\gamma(E) )$. We decompose $[.f_0\dots f_k]$, $[.g_0\dots g_k]$ into thin  cylinders of length
$m$, for example $[.f_0\dots f_kh_{k+1}\dots h_m] $ and $[.g_0\dots
g_kh_{k+1}\dots h_m]$.  Then since $M$ is column--reduced there
exists  a cylinder $[.e_0\dots e_ke_{k+1}\dots
e_m]$ of   $\Sigma _{M^{(m)}}^{ 0,+}$ such that $e_m^+=
h_m^+$, and so there exist 
 $\eta, \xi\in \FC_{\wh M}$ such that
$\eta([.f_0\dots f_kh_{k+1}\dots h_m])= [.e_0\dots e_ke_{k+1}\dots
e_m]= \xi([.g_0\dots g_kh_{k+1}\dots h_m]).$ Note that $\xi^{-1}\circ \eta=\gamma$.
Now by definition,  $\wh \nu_m( E\cap [.f_0\dots f_kh_{k+1}\dots
h_m])= \nu( \eta (E)\cap [.e_0\dots e_ke_{k+1}\dots
e_m])= \wh \nu_m( \xi^{-1}\circ \eta(E)\cap [.g_0\dots
g_kh_{k+1}\dots h_m])= \wh \nu_m(\gamma(E)\cap [.g_0\dots
g_kh_{k+1}\dots h_m])$ whence indeed $E=E\cap [.f_0\dots f_k]$ and 
$\gamma(E)=\gamma(E)\cap [.g_0\dots g_k] $ have the same $\wh \nu_m $-measure.

Now suppose that $\wh \nu$ on $\Sigma_{\wh M/M}^{0,+} $ is $\FC_{\wh
  M}$\,--\,invariant. We define $\nu$ to be the restriction of $\wh \nu$ to 
$\Sigma_{M}^{0,+}$, and claim that this is $\FC_{ M}$\,--\,invariant. 
Let $\gamma\in \FC_{ M}$ be such that 
for two cylinders $[.e_0\dots e_k]_M, [.\wt e_0\dots \wt e_k]_M$ of
$\Sigma_{M}^{0,+} $  we have 
$\gamma([.e_0\dots e_k]_M)=[.\wt e_0\dots \wt e_k]_M.$
To show invariance, it will be sufficient to show that for a Borel
subset $E$ of $[.e_0\dots e_k]_M$, then $\nu(E)=\nu(\gamma(E))$.
For $[.e_0\dots e_k] $   the corresponding cylinder set of $\Sigma_{\wh
  M}^{0,+} $ we note that $[.e_0\dots e_k]_M=[.e_0\dots e_k] \cap \Sigma_{\wh M/M}^{0,+}$ 
can be written as a nested intersection of finite unions of  $\wh
M$\,--\,cylinders, 
$\cap_{m>k} (\cup_{ \{e_{k+1}'\dots
  e_{m}'\}}[.e_0\dots e_k e_{k+1}'\dots  e_{m}'])$,
where the union is taken over all possible extensions within the
allowed strings of 
$\Sigma_{M}^{0,+}$.
The same holds for $[.\wt e_0\dots \wt e_k].$ Now by 
 $\FC_{\wh  M}$\,--\,invariance of $\wh\nu$, for each such
extension the $\wh M$\,--\,cylinders $[.e_0\dots e_k e_{k+1}'\dots  e_{m}']$ and
$[.\wt e_0\dots
\wt e_k e_{k+1}'\dots  e_{m}']$ have the
same $\wh \nu$-measure, whence, taking the limit,  $\nu([.e_0\dots e_k]_M)= \nu( [.\wt
e_0\dots \wt e_k]_M)$.
And moreover, since  $\gamma\in \FC_{ M}$ extends naturally to
$\gamma\in \FC_{\wh M}$, for any Borel subset $E$ of $[.e_0\dots
e_k]_M$ we have $\wh \nu(E\cap [.e_0\dots e_k e_{k+1}'\dots
e_{m}'])=\wh \nu(\gamma(E)\cap [.\wt e_0\dots
\wt e_k e_{k+1}'\dots  e_{m}'] )$, whence indeed $\nu(E)=\nu(\gamma(E))$.

Next
we show  that  invariant sets in the tower
$\Cal T\equiv \Sigma_{\wh M/M}^{0,+} $ 
and in the base 
$B\equiv \Sigma_{M}^{0,+}$ 
correspond.  Writing $\sim_{M} $ 
and $\sim_{\wh M} $ for the orbit equivalence relations
of the actions of $\FC_{M}$ and  $\FC_{\wh M}$, then $\sim_{M} $ is  the
restriction of $\sim_{\wh M} $ to $B\times B\subseteq \Cal T\times
\Cal T$, since for $e, e'\in B$ then $e\sim_{M} e'$ iff  the tails of $e$ and $e'$ are
equal after some $k$, iff $e\sim_{\wh M} e'$.

 Hence if $\wh E\subseteq \Cal T $ is invariant for $\FC_{\wh M}$,
 then  a fortiori its
restriction
to the base, $E= \wh E\cap B $, is  invariant for $\FC_{M}$.

We note that moreover, $\wh E= \FC_{\wh M}(E)$. This is
because, by
Proposition \ref{p:towerspace}, every  orbit of the tower meets the
base; that is, $\Cal T=
\FC_{\wh M} (B)$, and so, $\wh E= \FC_{\wh M}(B\cap \wh E)$.

Conversely, we claim that if  $E\subseteq
B$ is $\FC_{M}$-invariant, then its invariant extension to the tower
meets the base in $E$; that is, $\Sigma_{M}^{0,+}\cap \FC_{\wh M} (E)=E$.
Let us take $E\subseteq
B$ such that $\FC_{M}(E)= E$, and define $\wh E=\FC_{\wh M}(E)$. We
claim that $B\cap\FC_{\wh M}(E)= E$. For $e'\in  B\cap \FC_{\wh
  M}(E)$,  we shall show that $e'\in E$. Now there exists $e\in E$
with
$e\sim_{\wh M} e'$. But since $e'\in B$, also $e\sim_{M} e'$
and so by $\FC_{M}$--invariance  of $E$, $e'\in
E$.

This proves  that invariant sets in the base and in the tower
correspond bijectively, via the dual operations of restriction and of extension by
the action of $\FC_{\wh M}$, completing part $(i)$.

For $(ii)$,  let $\wh \nu$  be an ergodic $\FC_{\wh M}$--invariant
measure on $\Cal T$. 
 Then by $(i)$, its restriction $\nu$ to $B$ is
$\FC_{M}$--invariant. Let $E$ be an $\FC_{M}$--invariant
subset of $B$ of positive measure. Then $\wh E=\FC_{\wh
  M}(E)$ has positive measure and is  invariant hence by ergodicity of $\wh
\nu$ it has measure zero complement in $\Cal T$. But then $E$ must have 
measure zero complement in $B$ since otherwise its complement would 
generate a disjoint positive measure invariant subset of $\Cal T\setminus \wh E$.
Conversely, let   $\nu$  be an ergodic $\FC_{M}$--invariant
measure on $B$, with $\wh \nu $ denote its invariant extension to $\Cal T$ as in
$(i)$, and suppose 
 $\wh
E\subseteq \Cal T$ is an invariant positive measure subset of $\Cal T$. 
Then from $(i)$, the restriction  $E\equiv \wh E\cap B$ is
$\FC_{M}$--invariant,  and $\wh E= \FC_{\wh M}(E)$, whence 
$E$ has positive measure and so by ergodicity of $\nu$ is all of
$B$. Thus $\wh E= \Cal T$ up to a null set, and  
ergodicity of the $\FC$\,--\,actions 
corresponds. 

For the first part of $(ii)$ it seems easiest to work with the orbit
equivalence relation, which by  Proposition
\ref{p:conservative_equivreln} will be equivalent. Now 
$E\subseteq B$ is a wandering set for $\sim_M$
iff the orbit of each 
$x\in E$ is
infinite and meets $E$ in at most one point; thus, $(x\sim_M y)\then
x=y.$ 
Let $E$ be wandering for $\sim_M$ and suppose that for $x,y\in E$ we
have $x\sim_{\wh M} y$. But   $\sim_M$ is the restriction of
$\sim_{\wh M} $ to $B$, so $x\sim_M y$ whence $x=y$.
Thus  $E$ is wandering for $\sim_{\wh M} $ as well.

Now suppose $\wh E\subseteq \Cal T$ is wandering for $\sim_{\wh M} $,
and set $E\equiv B\cap \FC_{\wh
    M}(\wh E)$. Given $x,y\in E$ with $x\sim_M y$, we claim that
  $x=y$. Now since by definition $\Cal T= \FC_{\wh M}(B)$, then 
$\FC_{\wh M}(E)= \wh E$ so there exist
  $\wh x, \wh y\in \wh E$ with $\wh x\sim_{\wh M} x$ and $\wh
  y\sim_{\wh M} y$;
since $\wh E$ is wandering these are the unique such points. Then by
transitivity
of $\sim_{\wh M} $, $\wh x\sim_{\wh M} \wh y$ whence $\wh x=\wh y$,
since  $\wh E$ is wandering. But then  $x=y$, as claimed. 
Now by Proposition \ref{p:numberminimal}, since by assumption $\liminf
l_i>1$,
the orbit equivalence classes of $\sim_M$ are countably infinite, so
$E$ is indeed a wandering set.

We note also
that $ \nu (E)=\wh \nu (E)= \wh \nu (\wh E)$  by Proposition
\ref{p:conservative_equivreln}, since these sets are bijectively
equivalent.

It follows that ergodicity and conservativity correspond, for the base
and the tower. 

For part $(iii)$, that invariant sets and wandering sets correspond
for the transformations $(\Sigma_{\wh M/M}^{0,+}\setminus \CN_{\wh M},
T_{\wh \O})$ and $(\Sigma_{M}^{0,+}\setminus \CN_{M}, T_{\O})$
follows  by restriction from the corresponding facts proved above, since 
$ \CN_{\wh M}$ is $\sim_{\wh M} $\,--\,invariant, and therefore 
ergodicity and conservativity correspond. 
\end{proof}

\begin{rem}\label{r:towers infinite}
  We note that in the   proof of $(i)$, although the cylinder sets generate the
 $\sigma$\,--\,algebra, it has been necessary throughout to  consider  Borel
subsets of 
cylinders,   as it is quite possible that the measure of every cylinder
set
is infinite! Indeed we encounter explicit such examples below.

\end{rem}

We recall these basic ergodic theory notions introduced by Kakutani,
rewritten for partial transformations:
\begin{defi}\label{d:inducedmap}
  Given a partial transformation $T$ of a set $X$ and a subset $A\subseteq X$, the 
{\em  first return-time function} $r: A\to \N\cup\{\infty\}$ is 
\[
r(x)= \inf\{n\geq 1: T^n(x)\in A\}\]
(since by definition the inf of the empty set is
$+\infty$, 
 $r(\e)= \infty$ if the point never returns). Let us write 
 $\NS_A$ for this collection of points, and $\NP_A$ for the points in
 $A$ which never return for $T^{-1}$. The {\em first-return map} is the partial transformation
$T_A: A\setminus \NS_A\to A\setminus \NP_A$ defined by
$T_A(x)= T^{r(x)}(x)$; setting 
$\CN_A=  (\cup_{i\leq
  0} T_A^i \NS_A)\cup (\cup_{i\geq 
  0} T_A^i \NP_A)$ and $B= A\setminus \CN_A$,  then 
the map {\em induced } by $T$ on $A$ is the restriction of the first
return map $T_A$ to
the bijection $T_B: B\to B$.
The {\em (external)
  tower} or  {\em  Kakutani skyscraper} of height $r(x)-1$ is the space
$\Cal T_B\equiv \{ (x, k):\, x\in B, 0\leq k< r(x)\}$, acted on by the 
transformation $T_r$ defined by
$T_r(x,k)= (x, k+1)$ for $k<r(x)-1$,  $T_r(x,r(x)-1)= (T_B(x),
0)$. We call $B\times \{0\}\subseteq \Cal T_B$  the {\em base} of the tower.

We write $X_B\equiv \cup_{n\in\Z} T^n(B)$ for the $T$-orbit of $B$; we
call this the  {\em
  internal tower} with base $B$. Then  $T$ acts as a
bijection on this space by
restriction.
\end{defi}

\begin{prop}\label{p:inducedmaptower}
 Given $X, T, A$ as above, the 
internal and external tower transformations  
$(X_B,  T)$ and $(\Cal T, T_r)$ are  (set-theoretically) isomorphic: begin with the natural 
identification  $B\to B\times \{0\} $, and extend via the
dynamics. If we are given a  $\sigma$\,--\,algebra $\Cal B$ on $X$ and $A$
is a
measurable subset, then taking the restricted  $\sigma$\,--\,algebra $\Cal B_B$, the
return-time function $r$ is measurable (where $\N\cup\{\infty\}$ is given
the discrete  $\sigma$\,--\,algebra), and  the
induced map $(B, T_B)$ is measurable.  
Given an  invariant Borel measure $\nu$ on $X$,  and if  $\nu(B)>0$, then
this is a measure-theoretic isomorphism. \ \ \qed
\end{prop}

The next proposition explains the choice of the term  ``adic tower''.

\begin{prop}\label{p:towermap} Given nested ordered Bratteli diagrams
  $\mathfrak B_{\A,\E,M, \O}\leq \mathfrak
 B_{\wh \A,\wh\E, \wh M, \wh \O}$, 
assume the generalized matrix sequence 
$M=(M_i) _{i\geq  0}$ is column--reduced. For $A\equiv \Sigma_{\M}^{0,+}$, let   $r$ denote the 
 first return-time function of the map $T_{\wh O}$ on $\Sigma_{\wh \M}^{0,+}$ to  $A$.
We define subsets $B_n\subseteq A$ by
\[B_n=
\{\e\in A:\, \text{ for some }
  1\leq k\leq n, e_k \text{ is not maximal in } \O_k\}.
\]

\noindent
Then:
\item{(i)}the sets $B_n$ increase to $\Sigma_{\M}^{0,+}\setminus\NS_M$;

\item{(ii)} 
$r$ is finite and continuous on $\Sigma_{\M}^{0,+}\setminus\NS_M$, and
the  induced transformation on the subset
$\Sigma_{\M}^{0,+}$ for the partial transformation $T_{\wh \O}$ is
$T_{\O}$ on $ B=\Sigma_{\M}^{0,+}\setminus \Cal N_M$.

\item{(iii)}the restriction of the adic transformation $T_{\wh \O}$ to the
invariant subset
$\Sigma_{\wh M/M}^{0,+}\setminus{\Cal N_{\wh M}}$
is the internal tower over the adic transformation
$T_{\O}$ on the base $B$,
with return-time function
$r$.
\item{(iv)} If $M$ is primitive, 
then $\Sigma_{\wh M/M}^{0,+}=  \Sigma _{\wh M}^{0,+}$
if and only if at most finitely many changes have been made, 
that is, when for
some $m$ we have
$\Sigma _{M^{(m)}}^{0,+}= \Sigma_{\wh M}^{0,+}$. 
There exist nonprimitive examples where this is not true.
\end{prop}
\begin{proof}
Let $\e$ be in $\Sigma_{\M}^{0,+}\setminus\NS_M$. Then there exists  
$\f\in \Sigma_{\M}^{0,+}$ such that $\e<\f$; hence for some $m>0$, 
$f_k=
e_k\; \forall k> m$ while 
$e_m< f_m$ in
   $ \O_m$, whence
$\e\in B_m$. Conversely, if  $\e\in B_m$ for some $m$, then
there exists for some $k$ with $1\leq k\leq m$ an edge $f_k\in \E_k$  with $e_k< f_k$; because the
matrix sequence is column--reduced, we can continue this in some way to
$f_i$ 
for $0\leq i<k$;  defining  $f_i= e_i$ for all $i>k$  produces 
   an infinite string $\f\in \Sigma_{\M}^{0,+}$ with
$\e<\f$. This proves $(i)$. 

For $(ii)$, with $\e\in \Sigma_{\M}^{0,+}\setminus \NS_M$, let $m$ be
the least integer such that $\e\in B_m$. Taking $\f=
T_\O (\e)$,  then $e_m<f_m$ in $\O_m$ while $f_k= e_k$ for all
$k>m$.
 Consider all 
the edge paths $g$ in $\Sigma_{\wh M}^{0,+}$ with
$\e<g<\f$. Then  $r(\e)$ is equal to the number of such paths plus one, 
and since $g_m^+= e^+_m= f^+_m$ this is finite. This also shows that
$\NS_M$ is the set $\NS_A$ of Definition \ref{d:inducedmap} for $A=\Sigma_{\M}^{0,+}$
(the points where $r(\e)=\infty$) and that 
the induced map on $A$ is indeed $(\Sigma_{\M}^{0,+}\setminus \Cal
N_M, T_\O)$.

To show continuity, suppose that for $\e^{(n)}, \e\in
\Sigma_{\M}^{0,+}\setminus \NS_M$ we
have that
$d(\e^{(n)}, \e)\to 0$. Let $m$ be
the least integer such that $\e\in B_m$.
The function $r$ is constant on the cylinder set 
$[.e_0\dots e_m]$; for sufficiently large $n, \e^{(n)}$ is in this
cylinder,
so in fact $r( \e^{(n)})=r(\e)$. 

By Proposition \ref{p:inducedmaptower}, part $(iii)$  follows from $(ii)$. To prove $(iv)$, $\Sigma_{\wh
  M/M}^{0,+}=  \Sigma _{\wh M}^{0,+}$ if and only if there are no edge
paths in $\Sigma _{\wh M}^{0,+}$ which contain infinitely many ``new''
edges (elements from 
$\wh \E\setminus \E$).
Now if
 $\wh M$ is primitive and $\# \{n:\, \wh \E_n\setminus \E_n\neq\emptyset\}$
is infinite, then there exists an edge path $\e$ with infinitely many
new edges:  beginning with an edge $e_k$ in 
$\wh \E_k\setminus \E_k$ for some $k$, we  wait until time $n$ when $M_k^n$ is strictly positive;
there is a path connecting $e_k$ to $e_m\in \wh \E_m\setminus \E_m$ for some $m>n$. 
Continuing in this way produces $\e$.

For a nonprimitive example where this is no longer true, consider 
$M_i=\left[ \begin{matrix}
1& 1 \\
0 & 1
\end{matrix}  \right] $ 
and
$\wh M_i=\left[ \begin{matrix}
1& 2 \\
0 & 1
\end{matrix}  \right] $: each path in $\Sigma_{\wh M}^{0,+}$
contains at most one new edge, and so is in $\Sigma _{\wh
  M^{(m)}}^{0,+}$ for some $m$. Thus  the tower space $\Sigma_{\wh
  M/M}^{0,+}$ equals the Markov compactum $\Sigma _{\wh M}^{0,+}$
despite the fact that for no $m$ do we have $\Sigma _{M^{(m)}}^{0,+}= \Sigma_{\wh M/M}^{0,+}$. 
\end{proof}

\

\begin{rem}\label{r:towers}
  Summarizing, we always have 
\[\Sigma_{\M}^{0,+}\subseteq \Sigma _{M^{(m)}}^{0,+}\subseteq \Sigma_{\wh M/M}^{0,+}\subseteq \cl (\Sigma_{\wh M/M}^{0,+})\subseteq \Sigma_{\wh M}^{0,+},\] 
with (if we are given an order $\wh \O$) the map $T_{\wh \O}$ 
acting on the last 
three 
 of these and its induced maps  
$T_{ \O}$, $T_{ \O}^{(m)}$ acting on the first two (as  sub-adic
transformations).
 In
general, however, the adic tower space is not itself a Markov
compactum and the tower map is not an adic transformation. 
The tower closures are important for identifying the other  compact invariant
subsets of the $\FC$-action; see  Theorem
\ref{t:invariantsubsets}. 
  \end{rem}

Here is a special way adic towers can come about:

\begin{prop}\label{p:componentastower} 
Let $M=(M_i)_{i\geq 0}$ be a column--reduced 
$(\A_i\times \A_{i+1})$ sequence of 
 nonnegative integer matrices.
For all $k\geq 1$, 
$(\Sigma_{\M}^{0,+},T_\O) $ is isomorphic to 
a
bounded height adic tower over
$(\Sigma_{\M}^{k,+},T_\O^{(k)}) $. Indeed, there exists $\wt
 M\leq M$ such that $(\Sigma_{\M}^{0,+}, T_\O)$ is isomorphic to 
$(\Sigma_{M/\wt M}^{0,+}, T_\O) $, which is a bounded height adic tower over
$(\Sigma_{\M}^{k,+},T_\O^{(k)}) $.
\end{prop}
\begin{proof}The proof will be inductive; 
let us first suppose that $k=1$. 
For each element $a\in \A_1$ we choose an edge $e\in \E_{0}$
such that $e^+=a$; we can do this since the matrix sequence is
column--reduced. We call this collection  of  edges $\wt\E_{0}\subseteq
\E_{0}$, and denote by $\wt \A_{0}$  the collection
of symbols in $\A_0$ which are the initial symbols of these edges.
We write 
  $\wt M_{0}$ for the associated matrix, and let
$\wt\A, \wt \E, \wt M$ be the alphabet, edge and
matrix sequences 
equal to  $\wt \A_{0}, \wt \E_{0}, \wt M_{0}$ at level $0$ 
and to $\A_i, \E_i, M_i$ for all $i\geq 1$. 
Then  each edge path $\e= (.e_1e_2\dots)\in \Sigma_{\M}^{1,+}$ has a unique extension
to a path $(.e_0e_1e_2\dots) $ in $\Sigma_{\wt M}^{0,+}$, and the Vershik map $T_\O$ on $\Sigma_{\wt M}^{0,+}$
is isomorphic to $T_\O^{(1)}$ on  $\Sigma_{\M}^{1,+}$ via the conjugacy
given by this extension. 

Now 
$\mathfrak B_{\wt M}\leq \mathfrak B_{M}
$ with only finitely many edges removed, so by $(iv)$ of Proposition \ref{p:towermap}, 
$(\Sigma_{\M}^{0,+} ,T_\O)$ is isomorphic to the tower map
$(\Sigma_{M/\wt M}^{0,+},T_\O)$, but the 
base map $(\Sigma_{\wt M}^{0,+},T_\O)$ is  isomorphic to 
$(\Sigma_{\M}^{1,+},T_\O^{(1)})$, so we are done for this case.

For the general case we can either use the fact that a tower over a tower is a
tower, or proceed as follows: the above proof takes us from level $k$
to level $k-1$, and we then begin with the new sequences
 $\wt\A, \wt \E, \wt M$ starting at level $k-1$ and  apply
the same proof to pass to level $k-2$, choosing a unique extension at
each stage from the remaining part of the alphabet and continuing
until we reach level $0$; these choices determine  
$\wt M$ and  the base we take for the tower.
\end{proof}

Combining Theorem \ref{t:towermeasure}  with Proposition \ref{p:componentastower} we have in particular:
\begin{cor}\label{c:zeroand_k}
  The collections of finite and infinite, conservative, and ergodic,
  $\FC$\,--\,invariant Borel measures
  of $\Sigma_{\M}^{0,+} $  and of $\Sigma_{\M}^{k,+}$ correspond bijectively.
\end{cor}

Thus only the tail of the matrix sequence matters for classifying
invariant Borel measures. This complements Proposition \ref{p:gathermeasure}
regarding gatherings. 

\medskip

\subsection{The canonical cover}\label{ss:canonicalcover}
Given nonnegative integer matrix sequences
 $M\leq \wh M$,  
the adic tower $\Sigma_{\wh M/M}^{0,+} $ may not sit topologically so
nicely inside of 
$\Sigma_{\wh M}^{0,+}$: it may be dense with empty interior, see
$(i)$ of Theorem \ref{t:towerring}, 
and may have  an invariant Borel measure  which is locally finite for the tower 
but  infinite on every nonempty open
subset of $\Sigma_{\wh M}^{0,+}$,  see
$(iii)$ of
Theorem
\ref{t:covermeasure}, and see the example of the  Integer Cantor Set:
Example \ref{exam:Cantor}, Fig.~ \ref{F:ICSConstructingCover}, 
 and the discussion in the Introduction.

Here we describe the construction of $\Sigma_{\wt M}^{0,+}$,  a  topological cover of $\Sigma_{\wh M}^{0,+}$
which remedies this situation, as the tower 
sits 
inside of the cover as an open set. The two towers (over the same base
$\Sigma_{M}^{0,+} $, but inside different spaces) are topologically and
measure-theoretically isomorphic. This construction will later
prove useful for the
classification of
invariant Borel measures on the original space $\Sigma_{\wh M}^{0,+}$. 

We have $\E\leq \wh \E$ and for convenience we now
assume that $\wh \A= \A$ (if not, we would simply extend $\A$ to $\wh \A$).
For each $i\geq 0$, set   $ \overline \E_i=  \wh \E_i\setminus \E_i$.
Let $\A'_i$ be a disjoint copy of $\A_i$ and let 
$\E_i', \overline\E'_i$ and $\wh\E'_i,$ be    disjoint copies  of $\E_i,
\overline\E_i$, and $ \wh \E_i$ for that new vertex alphabet.

If $e_i\in  \E_i$
with $(e_i^-, e_i^+)= (a,b)\in \A_i\times \A_{i+1}$ then 
$ e_i'$ will denote the
corresponding edge  in $\E_i'$
with initial and final symbols $(a',b')\in \A_i'\times \A_{i+1}'$. Similarly
given $ \overline e_i\in  \overline  \E_i$,  then $\overline e'_i$
denotes the corresponding 
element of
$\overline  \E_i'$.

We then define  some new edges as follows. 
Given 
 $\overline e_i\in \overline\E_i$ with $(\overline e_i^-, \overline e_i^+)= (a,b)$,  we define
 $\overline e_i^\circ$ to have 
$
(\overline e_i^\circ)^-= a'$ and $ (\overline e_i^\circ)^+=b$. Thus $
(\overline e_i^\circ)^-= (\overline e_i')^-$ and $ (\overline
e_i^\circ)^+=\overline e_i^+$. 
This defines a new edge alphabet $\overline\E^\circ_i$.

We then  form the disjoint unions
  \begin{equation}
    \label{eq:canonicalcover}
    \wt\A_i= \A_i'\cup\A_i,\;\;\text{and}
\end{equation}
\[
\wt\E_i= \E_i\cup \overline\E^\circ_i\cup \wh\E_i', \]
 and write 
$\wt M_i$ for the corresponding $(\wt \A_i\times \wt \A_{i+1})$ matrix
sequence; we have doubled both index sets.
Given an order $\wh \O$ on $\mathfrak B_{\wh M}$,   we define an order
$\wt\O$ on 
 $\mathfrak B_{\wt M}$ 
as follows:
$\wh \E'_i$ is given the order inherited from $\wh
\O_i$ on $\wh \E_i$, while  $\E_i
\cup \overline\E^\circ_i$ is also given the order  inherited from $\wh
\O_i$ , since  these are the edges in $\wt \E_i$ with final symbol in
$\A_{i+1}$.  $\wt\O$ denotes the union of these two disjoint
orders.

There is a natural {\em cover map}  $\phi$ from the graph $\mathfrak B_{\wt M}$ to $\mathfrak
B_{\wh M}$; thus,  $\phi:\wt \A\cup\wt \E\to \A\cup\E$  sends $\wt
\A_i$  to $\A_i$ and  $\wt
\E_i$  to $\E_i$; it is two-to-one on both edges and vertices. The vertices $a'_i$,
$a_i$ are mapped to
 $a_i$; the edges $e_i$ and $e_i'$
are sent to $e_i\in \wh \E_i$, while $\overline e'_i$ and $\overline e^0_i$ are
sent to $\overline e_i$. See
Figs. \ref{F:ICSNew} and \ref{F:ICSConstructingCover}.
Note that the order $\wt \O$ projects to $\wh \O$ via this map; that is, if
$\wt\e_i\leq \wt f_i$ then $\phi(\wt e_i)\leq \phi( \wt f_i)$.
We call $\Sigma_{\wt M}^{0,+}$  the {\em canonical cover } of $\wh M$
over $M$. 
The cover map $\phi$ induces a map on edge paths, $\Phi:\Sigma_{\wt M}^{0,+}
\to  \Sigma_{\wh M}^{0,+}$, with $\Phi(.\wt e_0 \wt e_1\dots)=
(.\phi(\wt e_0)
\phi(\wt e_1)\dots )$, called the {\em covering map}, and the lexicographic order on 
$\Sigma_{\wt M}^{0,+}$  projects via $\Phi$ 
to  the lexicographic order on
 $\Sigma_{\wh M}^{0,+}$; that is,  if
$\wt\e\leq \wt f$ then $\Phi(\wt e)\leq \Phi( \wt f)$.

\begin{figure}
\centering
\includegraphics[width=5.5in]{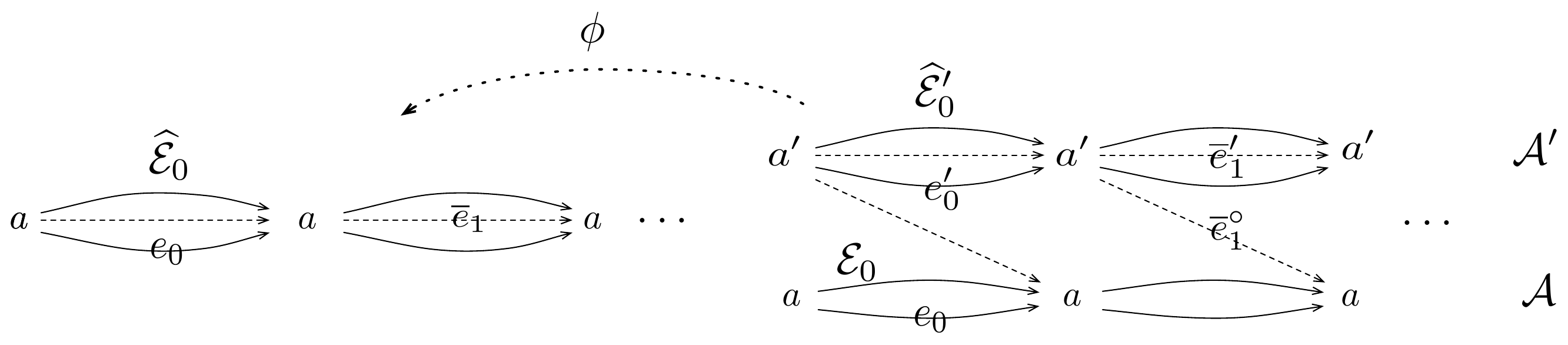}
\hspace{1in}
\caption{Constructing the canonical cover for the Integer Cantor Set example} 
\label{F:ICSConstructingCover}
\end{figure}

\begin{rem}\label{r:permutationsequence}
  In part $(ii)$ of the next theorem we switch from generalized matrices
  to usual matrices, so as to more conveniently visualize the
  canonical cover matrix. This does not affect the topology
  on the edge path space $\Sigma_M^{0,+}$, or the $\FC$-action, since those only depend on the
  (edge) cylinder sets. Thus,
  suppose  we have a Bratteli diagram with alphabet and
  generalized matrix sequence $\A$, and choose an order
  on each alphabet. With $ l_i=\# \A_i$, we have an $(l_i\times
   l_{i+1})$ matrix
  sequence $M=(M_i)_{i\geq 0}$. Now if we choose a second order on
  each alphabet, giving a matrix sequence $M'$, then these are
  conjugate via an $(l_i\times l_i)$ sequence $P= (P_i)_{i\geq 0}$ of
 permutation matrices (that is, the entries are $0,1$ with exactly one
 $1$ in each row and column).  The permutation sequence
 induces a map of the vertex cylinder sets, which can be extended to the
 edge cylinders and  hence to the edge path spaces
$\Sigma_M^{0,+}$ and 
 $\Sigma_{M'}^{0,+}$.   So these path spaces are topologically conjugate, as are the $\FC$-
 actions, as claimed.

  \end{rem}

\begin{theo}\label{t:towerring}
 Given $\A=\wh \A  $ and
 $\E\leq \wh \E$, so
$M\leq \wh M$, then:
\item{(i)} 
If $\wh M$ is primitive and $\Sigma_{\wh
    M/M}^{0,+}\neq \Sigma_{\wh M}^{0,+}$  then 
  $\Sigma_{\wh M/M}^{0,+}$  is dense with empty
  interior
 in $\Sigma_{\wh
    M}^{0,+}$.
\item{(ii)} For any $M\leq \wh M$, ordering the alphabets so that we
  can use usual (not generalized) matrices, and ordering
  $\wt \A$ so that the elements of $\A'$ come before
those of   $\A$, 
 the {\em canonical cover matrix sequence} $\wt M$ has (rectangular) block form
$
\left[ \begin{matrix}
\wh M& C \\
0 & M
\end{matrix}  \right] 
$
where 
$C_i\equiv \wh M_i- M_i$ for each $i\geq 0$.

\item{(iii)} If  $M\leq \wh M$ and  $\Sigma_{M}^{0,+}$ is an open subset
  of  $\Sigma_{\wh M}^{0,+}$, then the  tower $\Sigma_{\wh M/M}^{0,+}$ 
is an open subset of the cover space $\Sigma_{\wh M}^{0,+}$,  indeed is a countable
union of disjoint clopen sets. In particular, the cover tower $\Sigma_{\wt M/M}^{0,+}$ 
is an open subset of the cover space $\Sigma_{\wt M}^{0,+}$.

\item{(iv)}  The 
covering map
$\Phi:\Sigma_{\wt M}^{0,+} \to  \Sigma_{\wh M}^{0,+}$ is continuous
and surjective; its restriction 
to $\Sigma_{\wt M/M}^{0,+}$ is a  homeomorphism to 
$\Sigma_{\wh M/M}^{0,+}$.

Given an order $\wh \O$ on $\mathfrak B_{\wh M}$, this map conjugates the transformations
$T_{\wt \O}$ and $T_{\wh \O}$ on the two towers.  The action of $\FC_{\wh M}$ on $\Sigma_{\wh M/M}^{0,+}$
lifts to an action of a subgroup of $\FC_{\wt M}$ on $\Sigma_{\wt M/M}^{0,+}$,
which has the same orbit equivalence relation as $\FC_{\wt M}$.
\end{theo}
\begin{proof}
For   $(i)$,   primitivity of $\wh M$ implies
minimality of the $\FC_{\wh M}\, $--action (see Theorem 2.13 of \cite{Fisher09a}), so the orbit of any point 
is a dense subset, whence the
tower is. 
Now suppose that there exists a nonempty open subset 
$\wh{\Cal U}$ of 
$\Sigma_{\wh
  M}^{0,+}$ such that $\wh{\Cal U}\subseteq \Sigma_{\wh M/M}^{0,+}$.
We have assumed that
there exists 
$e\in \Sigma_{\wh M}^{0,+}\setminus \Sigma_{\wh M/M}^{0,+}$; this
 has dense orbit, and 
so meets 
$\wh {\Cal U}\subseteq \Sigma_{\wh M/M}^{0,+}$, 
but that gives a contradiction since the tower is $\FC_{\wh
  M}-$invariant yet does not contain $e$.

For $(ii)$, the $a'b^{\text {th}}$ entry of the  matrix $\wt M_i$
is the number of edges from $a'\in \A_i' $ to
$b\in\A_{i+1}$, which is the number of edges in 
$\overline \E_i $ from $a$ to $b$, and this is 
$(\wh M_i)_{ab}- (M_i) _{ab}= (C_i) _{ab}$, verifying 
$(ii)$. Note that from this equation, the sum of entries in $\wt M_i$ is twice that
of $\wh M_i$, agreeing with the  observation before the Proposition
that the cover map is two-to-one on edges.

For $(iii)$, 
to show $\Sigma_{\wt M/M}^{0,+} $ is open in $\Sigma_{\wt M}^{0,+}$,
note that 
a thin cylinder set $[.e_0\dots e_n]$ of $\Sigma_{M}^{0,+}$ is also a
thin cylinder for $\Sigma_{\wt M}^{0,+}$, since all extensions of this
word must
remain within $\E$;
 thus it is a clopen subset of $\Sigma_{\wt
  M}^{0,+}$ and so $\Sigma_{M}^{0,+} $ is an open subset of
$\Sigma_{\wt M}^{0,+} $. Now 
for each $m$, $\Sigma_{\wt M^{(m)}}^{0,+}$ is a finite union of
sets of the form  $[.f_0\dots
f_m]$ where $f_m^+= e_m^+$, which is 
homeomorphic to  such a cylinder $[.e_0\dots e_m]$.
Therefore the tower $\Sigma_{\wt M/M}^{0,+} $ is a countable union of 
clopen sets of $\Sigma_{\wt M}^{0,+}$ and hence is an open subset of
$\Sigma_{\wt M}^{0,+}$.

Supposing next for $M\leq \wh M$  that  $\Sigma_{M}^{0,+}$ is an open subet
  of  $\Sigma_{\wh M}^{0,+}$, then
given a cylinder set
$[.e_0\dots e_m]_M$ of $\Sigma_{M}^{0,+}$ let us write $[.e_0\dots e_m]_{\wh
  M}$ for the corresponding cylinder set of $\Sigma_{\wh M}^{0,+}$. 
Since  $\Sigma_{M}^{0,+}$ is  open in  $\Sigma_{\wh M}^{0,+}$ then 
$[.e_0\dots e_m]_M= [.e_0\dots e_m]_{\wh
  M}\cap \Sigma_{M}^{0,+}$ is also open. The simplest case is when
$[.e_0\dots e_m]_M=  [.e_0\dots e_m]_{\wh
  M}$, as then   the proof is like that just given, for $\wt M$:
the $\FC$-orbit of $[.e_0\dots e_m]_M$ is a
subset of $\Sigma_{\wt M^{(m)}}^{0,+}$ and is a finite union of open
sets in $\Sigma_{\wh M}^{0,+}$ hence is open. In general, since 
$[.e_0\dots e_m]_M$ is open it is a a union (perhaps
countably infinite) of cylinders of  $\Sigma_{\wh M}^{0,+}$, of the
form $[.e_0\dots e_m\dots e_n]_{\wh M}$, so $[.e_0\dots e_k]_M=  [.e_0\dots e_k]_{\wh
  M}$, and the 
 $\FC$-orbit of each of these is open, whence so is the  $\FC$-orbit
 of $[.e_0\dots e_m]_M$. It follows that the tower $\Sigma_{\wh
   M/M}^{0,+} $ is an open set, indeed it is a countable union of
 clopen subsets of $\Sigma_{\wt M}^{0,+}$. See also Props.~
 \ref{p:opensets}, \ref{p:opengeneral}.

For $(iv)$, the map $
\Phi: \Sigma_{\wt M}^{0,+}\to \Sigma_{\wh M}^{0,+}$ defined above 
 is clearly onto and continuous. From the structure of the edge alphabet $\wt\E_i$, there are
  three possibilities for an allowed edge path $\e$: either it is
  entirely within $\wh \E'$, so e.g.~$\e=(. e_0' \overline e_1'\dots\e_k'\dots)$,  entirely within $\E$, so $\e=(.e_0 e_1\dots)$,
or the first part is within $\wh \E'$ and the second within $\E$, that
is, for some $m>0$, e.g.~$\e=(. e_0' \overline e_1'\dots  e_{m-1}' \overline e^\circ_m
e_{m+1}\dots )$. 
A path in $\Sigma_{\wh M}^{0,+}$ with infinitely many  edges 
in $ 
\overline \E_i $ has a single preimage
in $\Sigma_{\wt M}^{0,+}$, which stays forever in $\wh \E'$. All
others have two preimages: if the path has all edges in $\E$ then it
has one preimage always in $\E'\subseteq \wh
\E'\subseteq \wt \E$    and one always in $\E\subseteq \wt \E$, and if it has a positive finite number of edges
in $\overline \E_i $ then 
it has one preimage which  stays forever in $\wh \E'$ and a second of
the form e.g.~$\e=(. e_0' \overline e_1'\dots  e_{m-1}' \overline e^\circ_m
e_{m+1}\dots )$
as above. Note that in both cases the second is in the tower
$\Sigma_{\wt M/M}^{0,+}$, while the first is not. Therefore the restriction of $\Phi$ from
$\Sigma_{\wt M/M}^{0,+}$ to
$\Sigma_{\wh M/M}^{0,+}$ is  bijective. Recalling Definition
\ref{d:adictower}, since we now have two ambient spaces, 
we  write  $
{}^{\wh M} \!  M^{(m)} 
$ and $
{}^{\wt M} \!  M^{(m)} 
$ rather than simply $M^{(m)}$;
the image of the subset 
$\Sigma _{{}^{\wt M} \!  M^{(m)} }^{0,+}$ is $\Sigma _{{}^{\wh M} \!  M^{(m)} }^{0,+}$, for
each $m$. A
cylinder $[.f_0\dots f_k]$ of $\Sigma _{{}^{\wt M} \!  M^{(m)}
}^{0,+}$
for $k\geq m$
corresponds via $\Phi$ to a unique cylinder $[.g_0\dots g_k]$ of
$\Sigma _{{}^{\wh M} \!  M^{(m)} }^{0,+}$
(warning: this is not true for $k<m$). It follows that $\Phi^{-1}$ is
continuous from
$\Sigma _{{}^{\wt M} \!  M^{(m)} }^{0,+}$
to
$\Sigma _{{}^{\wh M} \!  M^{(m)} }^{0,+}$
, and so $\Phi$ is a homeomorphism there.
Hence taking the direct limits, 
$
\Phi: \Sigma_{\wt M/M}^{0,+}\to \Sigma_{\wh M/M}^{0,+}$ is a
homeomorphism. 

Next, since the orders are preserved by $\Phi$, the tower maps are conjugate.

Since $\Phi$ is a homeomorphism from $\Sigma_{\wt M/M}^{0,+}$ to
$\Sigma_{\wh M/M}^{0,+}$,
the group $\FC_{\wh M}$ acting on $\Sigma_{\wh M/M}^{0,+}$ lifts via $\Phi$ to a subgroup of the group of
homeomorphisms of $\Sigma_{\wt M/M}^{0,+}$. We wish to show that this
is a subgroup of $ \FC_{\wt M}$ as it acts on $\Sigma_{\wt
  M/M}^{0,+}$, and to show that we have to look more closely.

Each thin cylinder $[e_0\dots e_m]$ of  $\Sigma_{\wh M}^{0,+}$ has as its
$\Phi$-inverse image two disjoint thin cylinders of 
$\Sigma_{\wt M}^{0,+}$. Indeed there are two possibilities: either
each $e_i\in \E_i$, in which case the two preimages are $[e_0\dots
e_m]$ and $[e_0'\dots e_m']$, or some edges are in 
$\overline \E_i $. Let $\overline e_k$ with $0\leq k\leq m$ be the
last such edge; then the cylinder is, say, $[e_0\dots \overline
e_k\dots  e_m]$
and its two preimages are $[e_0'\dots e_{k-1}'\overline e_k'e_{k+1}\dots  e_m']$ and 
$[e_0'\dots  e_{k-1}'\overline e_k^0e_{k+1}\dots  e_m]$. Now we claim
that
a generator $\gamma$  of $ \FC_{\wh M}$ in fact lifts to an element
(which is {\em not} a generator!) of
$ \FC_{\wt M}$. Suppose $\gamma$ interchanges $[e_0\dots e_m]$ and
$[f_0\dots f_m]$. 
Now note that the two cylinders in the preimage by $\Phi$ of $[e_0\dots e_m]$
end in different vertices, one in $\A_{m+1}'$ and one in
$\A_{m+1}$, whence they are disjoint. Therefore $\gamma\in \FC_{\wh M}$ lifts to $\eta\circ
\zeta=
\zeta\circ \eta$ where $\eta, \zeta$ are generators of $\FC_{\wt M}$
defined in the only way possible given the ending vertices: in the first case above, 
$\eta$ interchanges $[e_0\dots
e_m]$ and  $[f_0\dots
f_m]$, while $\zeta $ interchanges $[e_0'\dots e_m']$ and $[f_0'\dots
f_m']$;
in the second case, $\eta$ interchanges
$[e_0'\dots e_{k-1}'\overline e_k'e_{k+1}\dots  e_m']$ and 
$[f_0'\dots f_{k-1}'\overline f_k'f_{k+1}\dots  f_m']$, while $\zeta $ interchanges 
$[e_0'\dots  e_{k-1}'\overline e_k^0e_{k+1}\dots  e_m]$
and
$[f_0'\dots  f_{k-1}'\overline f_k^0f_{k+1}\dots  f_m]$.

What is actually going on is that $\FC_{\wt M}$ is a product of two
normal (commuting) subgroups, one generated by the interchanges of
thin cylinders which end on a vertex of $\A_m$ for some $m\geq 0$,
the other by those which end on a vertex of $\A_m'$; we have found
that (despite the fact that $\Phi$ is not bijective) $\gamma$ lifts to a homeomorphism of $\Sigma_{\wt
  M}^{0,+}$ which is  product of these. When restricted to the towers
$\Phi$ is a homeomorphism, which conjugates these actions.

Lastly, this subgroup of $\FC_{\wt M}$ has the same orbit equivalence
relation (the stable equivalence relation) as $\FC_{\wt M}$  for its
action on all of  $\Sigma_{\wt M}^{0,+}$, so that is true on the tower
as well.
\end{proof}

Given nested diagrams $\mathfrak B_{\A,\E}\leq \mathfrak
 B_{\wh \A,\wh\E}$, with matrix sequences $M\leq \wh M$, 
for the purpose of forming the adic tower
some of the vertices and edges may be irrelevant. That is, there may be
$\wh M'< \wh M$  such that $\Sigma_{\wh
  M'/M}^{0,+} = \Sigma_{\wh M/M}^{0,+} $. There is a canonical way to
find the least such diagram:

\begin{prop}
  Given $M\leq \wh M$, there exists a least $\wh M'\leq \wh M$  such
  that $\Sigma_{\wh
  M'/M}^{0,+} = \Sigma_{\wh M/M}^{0,+} $. 
\end{prop}
\begin{proof}
 Starting at $k=0$ and proceeding inductively to $(k+1)$, remove all
 elements of  $\wh \A_k$ and $\wh \E_k$ which do not connect to some
 element of
 $\E_m$ for some $m>k$, via some edge path. 
\end{proof}

\begin{defi}
  We call the resulting
  least $\wh M'$ the {\em distillation} of $\wh M$ with respect to $M$.
\end{defi}

The next result helps justify the name ``canonical cover'', as it
shows
you can't keep going:
\begin{prop} The operation of taking the canonical cover (after
  distillation) is
  idempotent. That is, given $M\leq \wh M$, with $\wt M$ the canonical
  cover of $\wh M$ over $M$, then 
  the canonical cover of $\wt M$ over $M$ is equal to $\wt M$ after both are
  distilled.

\end{prop}
\begin{proof} To begin the proof we extend $\A$ to $\wh \A$ so as to have equal
  alphabets; this 
 replaces $M$ by a possibly nonprimitive sequence. 
  By $(ii)$ of Theorem \ref{t:towerring} the matrix for the cover
  of $\wh M$ over $M$ is then
$$\wt M=
\left[ \begin{matrix}
\wh M& C \\
0 & M
\end{matrix}  \right] 
$$ with $C= \wh M-M$. To construct  the canonical cover of $\wt M$ over $ M$ we first extend the
alphabet for $M$ to $\wt \A$.  This replaces the matrix sequence $M$ by
$M'=
\left[ \begin{matrix}
0& 0\\
0 & M
\end{matrix}  \right] 
.$ Note that 
$\wt M- M'=
\left[ \begin{matrix}
\wh M& C \\
0 & 0
\end{matrix}  \right]\equiv C'.$
Then we form 
$$
\left[ \begin{matrix}
\wt M& C' \\
0 & M'
\end{matrix}  \right] 
= 
\left[ \begin{matrix}
\wh M& C &\wh M& C\\
0&M&0&0\\
0&0&0&0\\
0 &0& 0&M\\
\end{matrix}  \right] 
 $$
There are four subaphabet sequences, associated to this (rectangular)
block structure. We distill this matrix sequence, first removing edges which do not
eventually connect to the fourth of these, associated to 
the matrix sequence $M$ in the lower right corner. Since the third row
is all zero, we can make the third column all zero. Since the second
alphabet only connects to itself, we can make the second row all
zero, whence also the second column. 
This gives 
$$\left[ \begin{matrix}
\wh M& 0 &0 & C\\
0&0&0&0\\
0&0&0&0\\
0 &0& 0&M\\
\end{matrix}  \right] 
 .$$ So after removing unnecessary alphabets,
we end up with the distilled form
$$\left[ \begin{matrix}
\wh M& C\\
0&M\\
\end{matrix}  \right] 
 $$ which is exactly $\wt M$.
\end{proof}

Recalling Definition \ref{d:Radon}, we see next a condition we shall
encounter frequently, which will
guarantee that a locally finite measure is in fact positive locally finite.

\begin{lem}\label{l:finite_on_all}
 Let $(M_i)_{i\geq 0}$ be a primitive sequence   of  nonnegative integer
matrices. Let $\nu$ be an $\FC$\,--\, invariant Borel measure on $\Sigma^{0,+}_M$
which is positive
finite, respectively infinite,
on some open set. Then $\nu$ is positive finite, respectively infinite,
on all nonempty open sets.
\end{lem}
\begin{proof}
First we show
the statement for  
thin cylinder sets.
Given $[e_0\dots e_n]$ and $[f_0\dots f_m]$ such that 
$\nu ([e_0\dots e_n])>0$, we shall show the same is true for  
$[f_0\dots f_m]$. By primitivity there exists $k>n,m$ such that
$M_n^k$ and $M_m^k$ have all entries $>0$.
Since $[e_0\dots e_n]=\cup [e_0\dots e_ng_{n+1}\dots g_k]$, with the
union taken over all allowed extensions of that word, 
at least one of these has positive measure, say 
$[e_0\dots e_ne_{n+1}\dots e_k]$. Due to the primitivity there exists 
$[f_0\dots f_mf_{m+1}\dots f_k]$ with $f_k^+= e_k^+=a$; by invariance
these have equal measure. Now if $\nu ([e_0\dots e_n])=\infty$,
then by the same argument, the same  holds for $[f_0\dots f_m]$.
It follows that if $\nu ([e_0\dots e_n])<\infty$,
then $\nu ([f_0\dots f_n])<\infty$ as well, since otherwise we would have a contradiction.

Now if some open set $\Cal U$ has infinite measure then $\Sigma^{0,+}_M$ does, and
since that is a finite union of $0$\,--\, cylinders, one of those does, so
by the above every thin cylinder does; since every open set $\Cal V$
is a countable union of thin cylinders, the same holds for $\Cal
V$. If some open $\Cal U$ has positive finite measure, then since it is a countable
union of thin cylinders, one of these has positive finite measure, so
by the above all thin cylinders do. And the same must hold for  any
open set $\Cal
V$,  as if it had infinite measure, then so would $\Cal U$, as just shown.
\end{proof}

\begin{theo}\label{t:covermeasure}
Given nonnegative integer matrices $M\leq \wh M$, let $\nu$ be an $\FC_M$\,--\,invariant Borel measure on $\Sigma_{\M}^{0,+}$.
Write $\wh \nu$ for its extension to the adic tower
$\Sigma_{\wh M/M}^{0,+}\subseteq \Sigma_{\wh M}^{0,+}$ and $\wt \nu$ for its extension to
the adic cover tower $\Sigma_{\wt M/M}^{0,+}\subseteq \Sigma_{\wt
  M}^{0,+}$.
We have defined the covering map 
$\Phi:\Sigma_{\wt M}^{0,+} \to  \Sigma_{\wh M}^{0,+}$ in Theorem \ref{t:towerring}.
We have the following:
\item{(i)} Via the map $\Phi$, then given an order $\wh \O$ on  $\Sigma_{\wh M}^{0,+}$,  the transformations $(\Sigma_{\wh M/M}^{0,+},\wh \nu, T_{\wh \O})$ and  $(\Sigma_{\wt M/M}^{0,+},\wt
\nu, T_{\wt \O})$ are
measure-theoretically 
isomorphic. The action of $\FC_{\wh M}$ on 
$(\Sigma_{\wh M/M}^{0,+},\wh \nu)$ is 
isomorphic to the action of a subgroup of $\FC_{\wt M}$ on $(\Sigma_{\wt M/M}^{0,+},\wt
\nu)$. 

\noindent
\item{(ii)}
The map $\Phi$ induces a bijection $\wt \nu\mapsto \wh \nu$
between the collections of those conservative ergodic measures on
$\Sigma_{\wt M}^{0,+}$ and $\Sigma_{\wh M}^{0,+}$ which give positive
mass to $\Sigma_{M}^{0,+}$.

\item{(iii)} 
The measure $\nu$ is positive on some open set of $\Sigma_{M}^{0,+}$
iff $\wh \nu$ is positive on some open set of 
$\Sigma_{\wh
  M/M}^{0,+}$ iff $\wt \nu$ is positive on some open set of  
$\Sigma_{\wt M/M}^{0,+}$.
Moreover $\nu$ is
(positive) locally finite on $\Sigma_{\M}^{0,+}$
iff $\wh \nu$ is (positive) locally
finite on $\Sigma_{\wh M/M}^{0,+}$ iff $\wt \nu$ is (positive) locally
finite on $\Sigma_{\wt M/M}^{0,+}$, and in this case $\wt \nu$ is locally
finite on $\Sigma_{\wt M}^{0,+}$.
However it may not be locally finite on $\Sigma_{\wh M}^{0,+}$: if $\wh M$ is
primitive then an infinite invariant Borel measure on $\Sigma_{\wh
  M/M}^{0,+}$ is not locally finite on $\Sigma_{\wh M}^{0,+}$, indeed
it is  infinite
on every nonempty open subset. In this case $\wt \nu$
is a Radon measure on $\Sigma_{\wt M}^{0,+}$ while $\wh \nu$ is inner
regular but not
Radon on $\Sigma_{\wh  M}^{0,+}$.
\end{theo}
\begin{proof}

  \item{(i)} We recall from $(iv)$ of Theorem \ref{t:towerring}
that $\Phi$ is continuous
and surjective, and that its retriction to the towers is a
topological isomorphism. 
By  part $(i)$ of Theorem \ref{t:towermeasure} an invariant Borel measure $\nu$ on the
base $\Sigma_{\M}^{0,+}$ extends to a unique invariant Borel measure 
$\wh\nu$ on $\Sigma_{\wh M/M}^{0,+}$. Now since $\Phi$ is a
topological conjugacy, $\check\nu \equiv \nu\circ \Phi$ defines an
invariant Borel measure on $\Sigma_{\wt M/M}^{0,+}$. And since the
restriction of $\Phi$ to $\Sigma_{\M}^{0,+}$ is the
identity map, the restriction of the measure $\check \nu$ to
$\Sigma_{\M}^{0,+}$  is $\nu$. Now, again by Theorem
\ref{t:towermeasure}, $\wt\nu$ is the unique invariant extension of
$\nu$ to $\Sigma_{\wt M/M}^{0,+}$, whence $\check \nu=\wt\nu$. This
proves that $\Phi$ is a measure-theoretic isomorphism, proving $(i)$,  and
simultaneously, together with part $(ii)$ of Theorem
\ref{t:towermeasure}, 
proves statement $(ii)$.

\noindent
\item{(iii)}
The  topology on $\Sigma_{\wh M/M}^{0,+}$ is generated by the
collection of thin  cylinder sets of $\Sigma_{\wh
  M}^{0,+}$ which 
meet $\Sigma _{M^{(m)}}^{ 0,+}$ for some $m$. 
Let $[.f_0\dots f_m]$ be a such a cylinder set.
Then there exists a 
cylinder $[.e_0\dots e_m]$ of
$\Sigma_{\wt \M}^{0,+}$ such that $f_m^+= e_m^+$ 
and such that $[.e_0\dots e_m]$ in fact is a subset of 
$\Sigma _M^{ 0,+}$, so  $[.e_0\dots e_m]= [.e_0\dots e_m]_M$. The two
cylinders have the same $\wt \nu$-measure, and this agrees with 
$\nu[.e_0\dots e_m]$, proving the first statement. It also follows that if  $\nu$ is locally  finite on $\Sigma_{\M}^{0,+}$,
then 
$\wt \nu$ is  locally finite 
 on $\Sigma_{\wt M/M}^{0,+}$. Since the topology on $\Sigma_{\wt
   M}^{0,+}$ is generated by the thin cylinders, $\wt \nu$ is also
 locally finite on $\Sigma_{\wt
   M}^{0,+}$. Since $\Sigma_{\wh M/M}^{0,+}$ is homeomorphic to
 $\Sigma_{\wt M/M}^{0,+}$,
with $\wh \nu $ mapped to $\wt \nu$, we have that $\wh \nu$
is then locally finite 
 on $\Sigma_{\wh M/M}^{0,+}$ iff that holds for $\wt \nu$
 on $\Sigma_{\wt M/M}^{0,+}$.

 If the measure is positive on each open set of one, this also then passes to the others.

Next, suppose  $\wh M$ is
primitive with  $\wh \nu$ an infinite invariant Borel measure on $\Sigma_{\wh
  M/M}^{0,+}$.  Then it is infinite invariant on one open set of
$\Sigma_{\wh M}^{0,+}$ (the whole space) so it is infinite on every
nonempty open set, by
Lemma \ref{l:finite_on_all}.
Here inner regularity (recall Definition \ref{d:Radon})
holds for both measures, but as we have just
seen, local finiteness fails for $\wh \nu$ on $\Sigma_{\wh
  M}^{0,+}$. 
Therefore indeed $\wt \nu$
is a Radon measure on $\Sigma_{\wt M}^{0,+}$ while $\wh \nu$ is inner
regular since it is a countable sum of inner regular measures (see
Remark \ref{r:innerregular}) but not locally finite hence not Radon on $\Sigma_{\wh  M}^{0,+}$.
\end{proof}

\begin{rem}
  Next we examine the topology of subshifts for upper
triangular block form which occurs in the canonical cover. 
In Proposition \ref{p:opengeneral}
we use the nonstationary Frobenius Decomposition
  Theorem
  to strengthen this result: in fact one does not need the assumption
  in part $(ii)$ that $A$ be primitive.
  See also 
Cor.~\ref{c:uppertriangularopen}.

\end{rem}

\begin{prop}\label{p:opensets}
 Consider a nonnegative integer matrix sequence $M=\left[
\begin{matrix}
A& C \\
0 & B
\end{matrix}  \right] $.
 \item $(i)$ Then $\Sigma_A^+$,  $\Sigma_B^+$ are closed subsets of
$\Sigma_M^+$, and $\Sigma_B^+$ is open.
\item $(ii)$
Suppose also  that $B$ is reduced.
  If $A$ is primitive then  $\Sigma_A^+$ is open iff
the sequence $C= (C_i)_{i\geq 0}$ is zero except for finitely many
$i$.
\end{prop}
\begin{proof}
  We prove they are closed: let $K_k=\cup[e_0\e_1\dots e_k] $ where
  the union is taken over all thin cylinder sets of length $k$ of $\Sigma_B^+$. Then
  $\cap_{k\geq 0} K_k$ is an intersection of clopen sets, hence
  closed, which equals $\Sigma_B^+$. The same argument works for
  $\Sigma_A^+$.

  Next we show $\Sigma_B^+$ is open. Let $\e= (e_0 e_1\dots)\in
  \Sigma_B^+$. Set $U= [e_0]$. Then for any $\f\in U$, since $M$ is upper triangular, $\f\in \Sigma_B^+$. Thus $\e\in U\subseteq \Sigma_B^+$ so $\Sigma_B^+$ is open.

  For $(ii)$, suppose $A$ is a primitive sequence. We first take the
  hypothesis that 
 $C_j$ has some positive entry for infinitely many $j$, and wish to
 then show
  $\Sigma_A^+$ is not an open subset of $\Sigma_M^+$. For this it
  suffices to find a point $\e= (e_0 e_1\dots)\in
  \Sigma_A^+$ such that any open set containing $\e$ meets the
  complement of $\Sigma_A^+$. We claim that in fact any point in $
  \Sigma_A^+$ will serve this purpose. Now 
for $U_k= [e_0\e_1\dots e_k]$, we have that
  $\{\e\}= \cap_{k\geq 0} U_k$. Since the thin cylinder sets are a base
  for the topology,  it will suffice to show each $U_k$ meets the
  complement.

  By primitivity, for a chosen $k$ there exists $N$ such that
  $A_kA_{k+1}\dots A_n$ is strictly positive for any $n\geq N$. By the
  hypothesis, there
  exists $j>N$ such that 
 $C_j$ has some positive entry, $(C_j)_{ab}$. Now $A_kA_{k+1}\dots
 A_{j-1}$ is strictly positive, whence there exists a
  path $\f= (f_0 f_1\dots)$ with $f_0= e_0, \dots, f_k=
  e_k$ and $f_j^+=a, f_{j+1}^+=b$, with $f_i^+\in \B$ (the subalphabet
  sequence 
  for $B$) for all times $i\geq j+1$, using the fact that $B$ is
  reduced so the path $(f_0 f_1\dots f_{j+1})$  can be continued to the
  right. But $\A_i\cap\B_i=\emptyset$ so $\f\notin \Sigma_A^+$ and we are done.

  On the other hand,  if $C_i= 0$ for $i\geq k$, let
  $\e= (e_0 e_1\dots)\in
  \Sigma_A^+$ and now take $U= [e_0\dots e_k]$. Then for any $\f\in
  U$, since $C_i$ is $0$, $\f\in \Sigma_A^+$. Thus $\e\in U\subseteq \Sigma_A^+$ so $\Sigma_A^+$ is open.

\end{proof}

\begin{exam}\label{exam:Cantor}
(the Integer Cantor Set inside the triadic  odometer)
We start with the stationary Bratteli diagram determined by matrix
sequence $M= (M_i)_{i\geq 0}$ with $\wh M_i= [3]$ for all $i$, with edge alphabet $\wh \E_i= 
\{a,b,c\} $,  and with order $\O$ that of the triadic odometer, that
is, 
with $a<b<c$.  We consider the subdiagram for $M\leq \wh M$ with
$M_i=[2]$ and $\E_i= \{a,c\}$.  
The matrix for the canonical cover space is 
$\wt M=\left[
\begin{matrix}
\wh M& C \\
0 & M
\end{matrix}  \right] 
=\left[
\begin{matrix}
3& 1 \\
0 & 2
\end{matrix}  \right] 
,$ since $C= [3]-[2]= [1]$. This describes the relationship between two different adic  models for the Integer Cantor Set,
as adic towers; for the first the tower embeds as a dense set with
empty interior in $\Sigma_{\wh M}^{0,+}$ giving infinite measure to
each nonempty open set of that space, while in the second,
the tower embeds as an open dense set of $\Sigma_{\wt M}^{0,+}$ and gives
finite mass to those clopen cylinder sets  which correspond to levels
of the Kakutani tower $\Sigma_{\wt M/M}^{0,+}$ over the base
$\Sigma_{M}^{0,+}$. See Fig.~\ref{F:ICSNew}.  The base is  the dyadic odometer, with Bernoulli
$(\frac{1}{2},\frac{1}{2})$ measure and infinite
expected return time, proving infinite measure unique ergodicity for
the tower map. 

We study the invariant
measures for related examples by means of a general criterion in \S
\ref{ss:measuresfinite};
see Examples \ref{exam:Cantor_b}, \ref{exam:bruin}.

Note that from Proposition \ref{p:opensets}, $\Sigma_{M}^{0,+}$ is
a clopen subset of 
$\Sigma_{\wt M}^{0,+}$, while $\Sigma_{\wh M}^{0,+}$ is a closed but not
an open subset.

We return to this example in  Examples \ref{ex:nestedodometer} and \ref{exam:Cantor_b}.
\end{exam}

Lastly we note that given $M\leq \wh M$, the constructions of the adic
tower  and canonical cover are
respected by the operation of gathering, extending the results of
Proposition \ref{p:gathermeasure} and of Corollary \ref{c:zeroand_k}
to towers and covers; the proof is a corollary of those results:
\begin{prop}\label{p:gateringtower}
  Given $M\leq \wh M$, with adic
tower  $\Sigma_{\wh M/M}^{0,+}$ and canonical cover $\wt M$,  and
given a subsequence $0=n_0< n_1<\dots$, let 
$A, \wh A, \wt A$ denote the gatherings of these matrix sequences along
the times $(n_i)_{i\geq 0}$. Then the natural isomorphism from 
$\Sigma_{\wh M}^{0,+}$
to $\Sigma_{\wh A}^{0,+}$  (with respect to the $\FC$-actions) restricts to an isomorphism from
$\Sigma_{\wh M/M}^{0,+}$ to $\Sigma_{\wh A/A}^{0,+}$, and the
canonical cover spaces $\Sigma_{\wt M}^{0,+}$, $\Sigma_{\wt A}^{0,+}$ are
naturally isomorphic. These are topological isomorphisms; the
invariant Borel measures correspond, giving measure isomorphisms. 

The invariant Borel measures correspond, moreover, for $\Sigma_{\wh
  M/M}^{0,+}$ and $\Sigma_{\wh M/M}^{k,+}$ for any $k>0$, and hence
also for $\Sigma_{\wh A/A}^{j,+}$,  for any $j\geq 0$.
 \ \ \qed\end{prop}

\section{A nonstationary Frobenius theorem}\label{s:Frob}
In this section we prove a nonstationary version of the classical
Frobenius decomposition theorem; we follow this in \S \ref{s:Vict} with a nonstationary Frobenius--Victory theorem. Both of these
are results in linear algebra, and as in the
classical stationary case, we will work with matrices with nonnegative
real entries, although in our
applications to the ergodic theory of adic transformations this will restrict to integer
entries.

\subsection{The stationary case}\label{ss:stationarycase}

We begin with the stationary case. See e.g.~\cite{Gantmacher59} \S
XIII.4,
and further references in the Appendix below.
Given a $(d\times d)$ nonnegative real matrix $N$, we say state $i\in
\A= \{1,\dots,d\}$ 
{\em communicates to} state $j$ iff for some $n\geq 0$ we have
$N^n_{ij}>0$; here $N^0=I$, the identity matrix, so {\em every state
communicates to itself}. We say $i$ {\em strictly} communicates to $j$ iff
this holds for some $n>0$ and 
we say $i$ {\em immediately} communicates to $j$ iff $N_{ij}>0$. 
The matrix $N$ defines a discrete dynamical system, the map $f_N:\A\to
\A$ with $f_N(a)= b$ iff $N_{ab}>0$; then the states to which $a$
communicates is exactly the orbit $\{f_N^n(a): \, n\geq 0\}$.

A maximal collection of states all of whose elements communicate to each
other is called  a {\em communicating class} or {\em basin}. 
The  basins partition $\A$ and so define an
equivalence relation. 
We shall call a state $i$ such that
$N^n_{ii}=0$ for all $n>0$ a {\em pool state} (the idea for the name being that
non-pool basins may be linked together by passing through pool states). 
We note that there are two types of singleton equivalence classes: those
such that $i$ strictly communicates to itself and those such that it doesn't
(the pool singletons).
The matrix  is termed  {\em irreducible} iff there is a
single equivalence  class:
every state communicates
to every other state, {\em reducible} otherwise. 
A {\em primitive} matrix has the stronger
property that this happens for one time simultaneously: there exists
an $n>0$ with $N^n_{ij}>0$ for all  $i,j$. By these definitions, the $(1\times 1)$
matrix $[0]$ is irreducible but 
not primitive;  the basic nontrivial  example of
irreducible but 
not primitive 
is the
matrix of some cyclic permutation.

As above Definition \ref{d:gathering}, we associate to the nonnegative real matrix $N$ a $0-1$ matrix $L$,
 replacing each nonzero entry by
a $1$; whether or not two states communicate is not altered by this,
so $N$ is irreducible or primitive iff $L$ is. In terms of
the graph of the subshift of finite type for $L$,  for the 
basins there exists a path in each direction between any two
of its  elements.  We say an equivalence 
class $\alpha$ communicates 
to a class $\beta$ iff some (hence any) element of $\alpha$ communicates 
to some (hence any) element of   $\beta$; we say $\alpha$ immediately communicates 
to $\beta$ iff some element of $\alpha$ immediately communicates 
to some element of   $\beta$. 
We  define an order on  the basins, writing  $\alpha \leq 
\beta$ iff $\alpha$ communicates 
to $\beta$; this is transitive, reflexive ($\alpha\leq\alpha$),
and also is antisymmetric: if $\alpha \leq 
\beta$ and $
\beta \leq \alpha$, then $\alpha=\beta$, so is a
partial order.
An equivalence class $\alpha$ is termed  an {\em initial class} if
$\beta \leq \alpha\then \beta=\alpha$, a {\em final class} if
$\alpha \leq \beta \then \beta=\alpha$.
If an equivalence class  has no communications at all outside of
itself, it is by definition both initial and final.

We define a {\em class graph} whose vertices are the equivalence
classes. For this, draw a directed edge from $\alpha$ to $\beta$ iff
$\alpha$ immediately
communicates to
 $ \beta$.  The initial
and final classes are, respectively, the  repelling and
attracting  fixed points of this graph, which may have more than one component. 
 There are no directed
loops other than self-loops, so there are maximal elements. 
We write $\alpha < \beta$ iff $\alpha\leq \beta$ and $\alpha\neq
\beta$. We define $\text{level}(\alpha)=0$ iff $\alpha$
is maximal; $\text{level}(\alpha)=-n$ iff the longest path
$\alpha= \alpha_{-n}<\alpha_{-n+1}<\dots <\alpha_0$ from $\alpha$
to a maximal element $\alpha_0$ is $n$ steps. 

From now on we assume $N$ is
 reduced  (i.e.~it has no all-zero rows or columns); equivalently, the
 graph has no  {\em isolated points}, by which we mean vertices with no incoming
 or outgoing edges (a vertex with a self-edge is {\em not} isolated).

Now we 
draw the
graph (embedded in $\R^2$; edges may cross) so all the
maximal elements are on the top level, and so on, for levels $-1,
-2,\dots, -m$. Note that $\alpha\leq \beta \then
\text{level}(\alpha)\leq \text{level}(\beta)$. 

Next we linearly order the classes by $\preccurlyeq$ so as to respect levels; that is, 
so $\text{level}(\alpha)< \text{level}(\beta)\then \alpha \preccurlyeq
\beta$. For a geometical proof that this can be done, tip the embedded  class graph slightly and
 order by height; or, count in some way along levels, respecting
 levels; or formally, prove by induction. Lastly, linearly order each
 class in some (arbitrary) way, and combine this with $\preccurlyeq$. The result is 
a new linear order, also written $\preccurlyeq$, on $\A$
 (this corresponds to  conjugating $N$ by a permutation
matrix) so that the elements of equivalence  classes are grouped
together,  while respecting $\leq$; that is, for $a\in \alpha$ and
$b\in \beta$, 
 $\alpha\leq \beta\implies
 \alpha \preccurlyeq \beta\implies a  \preccurlyeq  b$.
This reordering
 puts the matrix in 
{\em  Frobenius normal form}, so that for the block
structure corresponding to the equivalence classes:

\

\noindent
$(1)$ there are square blocks along the diagonal which are
irreducible, including possibly $(1\times 1)$ zero matrices;

\noindent
$(2)$ the matrix is  upper triangular  with respect to these
blocks.

\noindent
$(3)$ the blocks corresponding to initial classes occur first, with
rows indexed
in the matrix as usual
from top to bottom,
while the final classes occur last.

\noindent
$(4)$ each pool state, which corresponds to a  zero block on the
diagonal,
occurs just before the first equivalence class to which it communicates.

\noindent
By taking a power $N^n$, one can also achieve:

\noindent
$(5)$ the nonzero diagonal blocks are primitive.

We have proved:
\begin{theo}(Stationary Frobenius Theorem)\label{t:FrobeniusTheorem}
 Given a  $(d\times d)$ nonnegative real matrix $N$ with alphabet
 $\A$, the alphabet can be permuted (i.e.~ $N$ can be conjugated with
 a permutation matrix) so as to put it in Frobenius form satisfying
 $(1)-(5)$.
 This is unique up to a further permutation: of the initial states,  of the
 pool states which occur just before before a given symbol, and of the
 states within the alphabet for a diagonal block.
\end{theo}

\begin{rem}
If a class is both initial and final, we arbitrarily choose one of those.
  
 We note that the usual way of indexing the matrix rows and columns agrees with our convention
 for  
drawing  Bratteli diagams horizontally and from left to right.
  
  If
one choses to instead use lower triangular form, the initial classes would
occur last. Both choices occur in the literature.

In part $(5)$ if the irreducible block is e.g.~ a $(n\times n)$ permutation
matrix, then upon taking the $n^{\text{th}}$ power this  becomes the
identity matrix, with $(1\times 1)$ diagonal
blocks.
\end{rem}

In the figure below, $0$ and $O$ indicate zero blocks, with $O$ on the
diagonal; $A,B$ and $F,G$ correspond to initial and final
equivalence classes respectively.

\begin{equation}
  \label{eq:4}
 \left[ \begin{matrix}
A &0 &*&*&\dots&*  &*&*&*\\
0 &B  & *&*&\dots&*&*&*&*\\
0 &0  & C&*&\dots&*&*&*&*\\
0 &0  & 0&D& \dots&*&*&*&*\\
\vdots& & & &\ddots   &  & & &\vdots\\
0& \dots& & &\dots  &O &*&* &*\\
0 &\dots  & & &\dots & 0 &E& *&*\\
0 &\dots  & & &\dots  &0 &0& F& 0\\
0 &\dots  & & &\dots   &0 &0& 0& G\\
\end{matrix}  \right]
\hspace{0.5cm}
\left[ \begin{matrix}
1 &0 &*&*&\dots&*  &*&*&*\\
0 &1  & *&*&\dots&*&*&*&*\\
0 &0  & 1&*&\dots&*&*&*&*\\
0 &0  & 0&1& \dots&*&*&*&*\\
\vdots& & & &\ddots   &  & & &\vdots\\
0& \dots& & &\dots  &0 &*&* &*\\
0 &\dots  & & &\dots & 0 &1& *&*\\
0 &\dots  & & &\dots  &0 &0& 1& 0\\
0 &\dots  & & &\dots   &0 &0& 0& 1\\
\end{matrix}  \right] 
\end{equation}

\

The matrix for the subshift of finite type of the corresponding class
graph is to the right of the block matrix. Note that there 
is an identity matrix on the ends of the diagonal, corresponding to the
initial and final  classes, and that there is at least one final and
one initial class (these being equal iff $N$ is irreducible).

\subsection{The nonstationary case}
Our main
result in this section,  Theorem \ref{t:FrobDecomp}, will be that for a
one-sided matrix sequence,   one can
always find a nonstationary reordering of the alphabets so as to put the
matrices in  an analogous upper triangular block form, called
(nonstationary) Frobenius normal form; moreover, after  a gathering   this can be put
in {\em fixed--size} Frobenius normal form.

Here are the precise definitions:

\begin{defi}\label{d:fixedsizeFrob}
  We recall from Definition \ref{d:generalizedmatrices}
  an empty alphabet is termed a {\em virtual} alphabet, and that
  these are permitted for the index sets of generalized matrices,
  giving {\em
  virtual matrices} (equal to the empty function).

This notion will facilitate the definition of upper triangular block form
for matrix sequences. See also Definition \ref{d:blockstructure}.

Given an alphabet sequence $\A=(\A_k)_{k\geq 0}$, assume we are given,
for each $k$ fixed, a partition $\{\A^i_k\}_{i= 1}^{\wh l_k} $ of
$\A_k$ into $\wh l_k$ possibly empty sets, called the {\em block
  alphabet partition}. We write $\wh A_k= \{1,\dots,\wh l_k\}$,
calling this 
the {\em block alphabet}.

Choice of a  block
  alphabet partition sequence defines a {\em block form} for the
  matrix sequence $(N_k)_{k\geq
    0}$ as follows. Given a label $(i,j)$ in $(\wh A_k\times \wh A_{k+1})$,
  the $(i,j)^{\text{th}}$  {\em
    block} of the matrix $N_k$ 
  is  the (possibly virtual) submatrix $(N_k)_{\A_k^i \A_{k+1}^j}$.

  We define an associated  $(\wh A_k\times \wh A_{k+1})$ {\em block matrix} $B_k$
with entries in $\{0,1\}$, such
that $(B_k)_{ij}= 1 $ iff some entry of $(N_k)_{\A_k^i \A_{k+1}^j}$ is
nonzero. Thus, we assign the entry $0$ to either an all- zero
or a virtual block.

Note that the block alphabets are ordered, with
$\wh A_k=\{1,\dots, \wh l_k\}$. 
We   call $(B_k)_{ii}$  the {\em diagonal elements } of $B_k$ and the
corresponding blocks   $(N_k)_{\A_k^i \A_{k+1}^i}$ the {\em diagonal blocks} of $N_k$.

Note that when the matrices are multiplied, so are the corresponding
diagonal blocks.
That is, $$(N_kN_{k+1})_{\A_k^i \A_{k+2}^i}= (N_k)_{\A_k^i \A_{k+1}^i}(N_{k+1})_{\A_{k+1} \A_{k+2}^i}.$$

There are special elements of $\wh \A_k$  we call {\em primitive
  elements}; the complement will be {\em pool elements}. 
We indicate these as follows: we write the $k^{\text{th}}$ block alphabet as
$\wh \A_k= (1,\dots, \wh l_k)= (1, P_2, 2, P_3, \dots , \wt l_k).$
We define a
 {\em pool index} to be a $P_j$ in this list.
The collection of pool indices is written $\wh\P_k$.
The corresponding blocks of $N_k$ are called the {\em pool blocks} and
the corresponding alphabets the {\em pool alphabets}; these may be virtual. The elements of $\P_k\equiv\cup_{i\in \wh\P_k}\A^i_k$ are termed {\em pool elements}
of $\A_k$.

\

\noindent
We say 
$(N_k)_{k\geq 0}$ with nonnegative real entries is in (upper triangular)
{\em Frobenius normal form}, with respect to the block alphabet
sequence $\wh\A$, iff:

\

\noindent

\noindent
$(1)$ the primitive diagonal blocks are reduced and primitive matrix
sequences;  the non-primitive diagonal blocks are
 pool blocks, for which  the product from time $k\geq 0$ to
$n>k$ is, for $n$ sufficently large, either zero or
the virtual matrix.

\noindent
$(2)$ for all $k\geq 0$,  for the block matrix sequence, $(B_k)_{ij}=
0$ for all $j<i$.

\

\noindent
We  say the sequence is in 
 {\em fixed--size   Frobenius normal form }
if the alphabet size is bounded, and in addition to $(1), (2)$ we have:

\

\noindent
$(3) $ those columns of $(B_k)_{k\geq 0}$ which have for all $k$ a
single $1$ in entry $(ii)$ occur first, and the rows which have a
single $1$ in entry $(ii)$ occur last.  Then $(B_k)_{ii}$ is called an
initial, respectively final, block, and the corresponding 
symbols $a\in \A_k^i $ 
are called {\em initial elements} and 
 {\em final elements}. If a block is both initial and final,
that is to say, if  there are no front-or back-connections to other
streams, then it is chosen arbitarily to be one of these.

\

\noindent
$(4)$ 
for all $k\geq 0$,  both $N_k$ and  $B_k$  are square, of
 dimensions  $(l\times l)$ 
and $(\wh
 l\times \wh l)$, and 
with  square diagonal blocks. Furthermore, all virtual blocks have been
eliminated so the block alphabet is now a subset of $(1, P_1, 2,P_2,\dots, d)$
with $P_j$ removed from this list if it is virtual. The diagonal
block sequences  are either reduced primitive or identically zero
(i.e.~ zero for all times).

\begin{rem}\label{r:uppertriangular}
 Given these properties,  we can then choose an order for  the alphabets $\A_k$, writing 
with $\A_k=\{1,\dots, l_k\}$, which is {\em compatible} with that of
the block alphabets $\wh A_k=\{1,\dots,\wh l_k\}$, by which we mean
that the block alphabet partition $\{\A^i_k\}_{i= 1}^{\wh l_k}
$  divides the
  alphabet $\A_k$
  into $\wh l_k$ consecutive  (possibly empty) subsets.

  The generalized matrices $(N_k)_{k\geq 0}$ are then matrices in the
  ususal sense, as is each block submatrix. This is an upper
  triangular matrix sequence with respect to the order and block
structure, as we have zero blocks below the diagonal.

 Note that, letting $s\geq 0$ denote the number of initial
elements of $B_k$ for the  fixed--size  form and 
$t\geq 0$ the number of final  elements for $k\geq 0$, then  $B_k$ has  
an $(s\times s)$ identity matrix  in the upper left  and a $(t\times t)$
identity matrix in the lower right corner. 
\end{rem}

\end{defi}

The main preparation for the proof of in Theorem \ref{t:FrobDecomp} is
geometrical in nature, 
given
in Theorem \ref{t:StreamDecomp}; 
 the primitive diagonal blocks will correspond to 
primitive {\em streams} of symbols in the Bratteli diagram
(replacing what we called  {\em basins} in the stationary setting), with
the connections between
streams either direct or passing through nonprimitive streams called
{\em pools}, an extension of the notion of pool states in the
stationary case. Here are the definitions:

\begin{defi}
Beginning with an $(l_i\times l_{i+1})$ nonnegative real matrix sequence $N=(N_i)_{i\geq
  0}$  with alphabets $(\A_i)_{i\geq 0}$, as above Definition
\ref{d:gathering} we 
define a $0-1$ sequence $L=(L_i)_{i\geq
  0}$  by replacing each nonzero entry with a
one; by the
Bratteli diagram associated to $N$ we mean that for $L$. 
We define
a  {\em stream} $\alpha$ to be a sequence $(\alpha_i)_{i\geq 0}$ of
 (possibly empty) sub-alphabets $\alpha_i\subseteq \A_i$. Set
operations 
(union, intersection, complement,
 difference) are taken componentwise. Thus we define  the difference
 $\alpha\setminus \beta$  of two
 streams by $(\alpha\setminus
 \beta)_i= \alpha_i\setminus \beta_i$ for all $i\geq 0$;
streams are partially ordered by inclusion: $\alpha\subseteq \beta$ iff
for each $i\geq 0$ we have $\alpha_i\subseteq \beta_i$;
two streams $\alpha,\beta$
are {\em disjoint} iff $\alpha_i$, $\beta_i$ are disjoint sets for all $i$.  

A stream is   {\em proper}  iff there is $k\geq 0$, called the 
 {\em starting time},  such that 
the $\alpha_i$  are empty for all $i<k$ and nonempty for all 
$i\geq k$. From a proper stream we define  a
 sequence  of generalized matrices $(N^\alpha_i)_{i\geq k}$ by
 $N^\alpha_i= N_{\alpha_i\alpha_{i+1}}$  for
$i\geq k$.

A {\em singlet stream} is a proper stream consisting of
a single letter  at each time greater than or equal to its starting time $k$.

A {\em vertex path} is an infinite or finite sequence $ x =
(x_k x_{k+1}\dots)$ or  $(x_k
x_{k+1}\dots x_{k+n})$ for $x_i\in \A_i$, 
which is {\em allowed} in that there  exists an edge path in
the Bratteli diagram associated to $N$ with those vertices,
that is, satisfying $(N_n)_{x_n x_{n+1}}>0$ for $n\geq k$. An infinite
vertex path $ x$ determines a singlet stream $\alpha_{ x}$ 
consisting of its entries, thus $(\alpha_{ x})_i= \{x_i\}$;
an infinite vertex path  $ x$  is  {\em disjoint} from a stream $\alpha$ iff the streams
$\alpha_{ x}$, $\alpha$ are disjoint.

We say a symbol $i\in \A_k$ {\em communicates to} $j\in \A_n$ iff
$k<n$ and there exists an edge path from $i$ to $j$, that is, iff the
$ij^{\text{th}}$ entry of $N_k N_{k+1}\cdots N_{n-1}$ is positive.

Given two streams $\alpha$, $\beta$, 
we call  a vertex path $x_k, y_{k+1},\dots y_{n-1}, x_n$, for $n\geq k+1
$,
such that
$x_k\in\alpha$, $x_n\in\beta$ and 
$y_i\notin \alpha\cup \beta$ 
a {\em
  front-connection} from $\alpha$ to $\beta$, with  {\em connection pair} $(k,n)$. 
We say 
$\alpha$ {\em front-connects} to $\beta$ at time $k$ when there 
 exists a front connection from $\alpha$ to $\beta$ for some pair $(k,
 n)$, 
and $\alpha$ {\em infinitely front-connects} connects to $\beta$
iff this happens for arbitrarily large times $k$.
We say
$\beta$ {\em back-connects}  or {\em infinitely back-connects} to $\alpha$ iff
$\alpha$ {front-connects} respectively infinitely front-connects to $\beta$.

We say an infinite vertex path $ x$ front-connects to a stream
$\beta$ iff that holds for its singlet stream $\alpha_{ x}$.
(So in particular an infinite vertex path  infinitely front- and back-connects to itself).

 We  say that a stream $\alpha $ is {\em
   primitive}  iff it is a proper stream with starting time $k$ 
such that the  sequence  $(N^\alpha_i)_{i\geq k}$ is primitive. 
A {\em reduced stream}   $\alpha$  is proper stream which is a
union of infinite
vertex paths, not necessarily
disjoint,  from its starting time $k$ on. Equivalently, $(N^\alpha_i)_{i\geq k}$ is a reduced
generalized matrix sequence. We say that $\alpha$ is {\em reduced
  beyond time } $m$ iff  $(N^\alpha_i)_{i\geq m}$ is a reduced sequence.

A {\em special stream}  is a stream which can be written as a
disjoint union of infinite
vertex paths, possibly with different starting times; it is  of {\em size } $d\geq 1$
iff it eventually consists of $d$ disjoint infinite vertex paths. (Necessarily,
 this is reduced after the maximum of the starting times.)

Regarding edge paths, we say an edge path $\e$ {\em accompanies} a stream $\alpha$ iff 
$e_i^-\in \alpha$ for all $i\geq 0$, that is, iff its vertex path
belongs to $\alpha$, and that an edge path front-connects to an edge path, a vertex path, or
a stream iff that holds for its vertex path.
\end{defi}

Beginning with a Bratteli diagram with unordered alphabets, 
we shall first prove a stream decomposition theorem (Theorem \ref{t:StreamDecomp}); this will enable us to
order the alphabets, while defining a compatible block alphabet partition which
puts the matrices in upper triangular form.

For an example of the role to be played by streams,  consider a sequence of nonnegative   matrices $(N_i)_{i\geq 0}$

\[
\left[ \begin{matrix}
1 &1
\end{matrix}\right]
\left[ \begin{array}{ccc|c}
1 &0& 0&0\\
0&1 & 1&1
\end{array}\right]
\left[ \begin{array}{c|c|ccc}
1 &0& 0 &0\\
1&0 & 1 & 0\\
1&0 & 0 & 1\\
\hline
0&1 & 1 & 0\\
\end{array}\right]
\left[ \begin{array}{c|c|c}
1 &0&0\\
\hline
0&1 &0\\
\hline
0&0&1\\
0&0&1\\
\end{array}\right]
\left[ \begin{array}{c|c|cc}
1 &0&1&0 \\
\hline
0&1&0&1 \\
\hline
0&0&1&1 \\
\end{array}\right]
\left[ \begin{array}{c|c|cc}
1 &0&1&0\\
\hline
0&1 &0&0\\
\hline
0&0&1&1\\
0&0&0&1\\
\end{array}\right]
\left[ \begin{array}{c|c|cc}
1 &1&0&0\\
\hline
0&1 &1&0\\
\hline
0&0&1&0\\
0&0&1&1\\
\end{array}\right]
\]

\noindent
with
 corresponding  block matrices $(B_i)_{i\geq 0}$:

\[
\left[ \begin{matrix}
1 \end{matrix}\right],
\left[ \begin{matrix}
1 &1
\end{matrix}\right],
\left[ \begin{matrix}
1 &0& 1\\
0&1 & 1
\end{matrix}\right],
\left[ \begin{matrix}
1 &0& 0\\
0&1 & 0\\
0&0 & 1\\
\end{matrix}\right],
\left[ \begin{matrix}
1 &0& 1\\
0&1 & 1\\
0&0 & 1\\
\end{matrix}\right],
\left[ \begin{matrix}
1 &0& 1\\
0&1 & 0\\
0&0 & 1\\
\end{matrix}\right],
\left[ \begin{matrix}
1 &1& 0\\
0&1 & 1\\
0&0 & 1\\
\end{matrix}\right]
\]

The corresponding stream descriptions are seen in 
Fig.~\ref{F:Frobenius}.

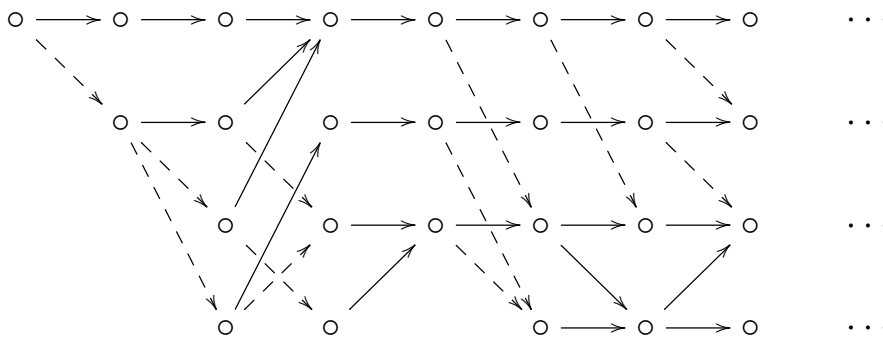
\begin{figure}
$$ \xymatrix
{
\circ \ar[r]  \ar@{-->}[rd]&
\circ \ar[r]   &
\circ \ar[r]  &
\circ \ar[r]  &
\circ \ar[r]   \ar@{-->}[rdd]&
\circ \ar[r] \ar@{-->}[rdd]&
\circ \ar[r] \ar@{-->}[rd]&\circ&\cdots &
\\
& 
\circ \ar[r]  \ar@{-->}[rd] \ar@{-->}[rdd] &
\circ \ar[ru]  \ar@{-->}[rd]  &
\circ \ar[r]  &
\circ \ar[r]  \ar@{-->}[rdd] &
\circ \ar[r]   &\circ \ar[r] \ar@{-->}[rd]&\circ&\cdots &
\\
& 
& \circ \ar@{-->}[rd] \ar[ruu] 
&\circ  \ar[r] 
&\circ  \ar[r] \ar@{-->}[rd] 
&\circ  \ar[r]\ar[rd] 
&\circ\ar[r]
&\circ&\cdots&
\\
& 
&\circ \ar[ruu] \ar@{-->}[ru] 
&\circ \ar[ru]
& 
&\circ\ar[r] 
&\circ\ar[r]\ar[ru]
&\circ &\cdots
}
$$
\caption{A Bratteli diagram for matrices $N_i$ in 
  Frobenius normal form, with three primitive streams $\wh \alpha\prec\wh\beta\prec\wh\gamma$, from top to
bottom,
 indicated by the solid arrows.}\label{F:Frobenius}
 \end{figure}

\begin{figure}
$$ \xymatrix
{
\circ \ar[r]  
&\circ \ar[r]  \ar @{-->}[rd]
&\circ \ar[r]  \ar@{-->}[rdd]
&\circ \ar[r]  
&\circ \ar[r]   \ar @{-->}[rdd]
&\circ \ar[r] \ar @{-->}[rdd]
&\circ \ar[r]  \ar @{-->}[rd]&\circ&\cdots 
&
\\
& &
\circ \ar[r]  \ar @{-->}[rd]  &
\circ \ar[r]  &
\circ \ar[r]  \ar @{-->}[rd]&
\circ \ar[r]   &\circ \ar[r]  \ar @{-->}[rd]&\circ&\cdots &
\\
& & &
\circ \ar[r]&
\circ  \ar[r] &
\circ  \ar[r] &
\circ\ar[r]&\circ&\cdots&
}
$$
\caption{Diagram for the corresponding matrices $B_i$; the streams
  have coalesced into singlet streams, and  as in Theorem
  \ref{t:StreamDecomp}, only  front connections
  remain.}\label{F:Frobeniusb}
 \end{figure}
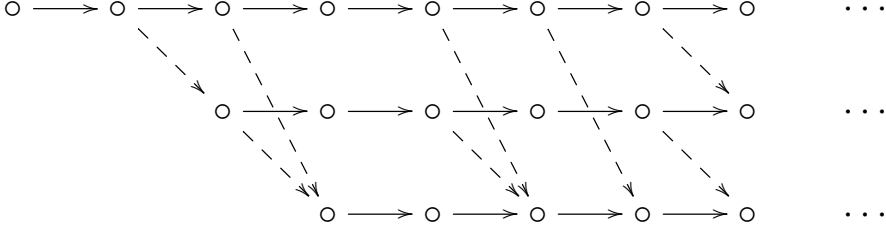

\begin{lem}\label{l:getting one more}
 Let 
$(N_i)_{i\geq 0}$ be a reduced sequence  of $(l_i\times l_{i+1})$  
nonnegative real
  matrices with alphabets $\A= (\A_i)_{i\geq 0}$. If $\A$ is not
 itself a primitive stream, then given a reduced primitive  stream $\alpha\subseteq\A$,
there exists an infinite vertex path  $ x=(x_i)_{i\geq k}$ for some $k\geq 0$ such that 
the singlet stream  $\alpha_{ x}$ with starting  time  $k$ is disjoint from
 $\alpha$. 
\end{lem}
\begin{proof}
The proof will be by contradiction: we show that if there does not exist such a
vertex path $ x$
 then 
$\A$ is  primitive.
Thus, choosing $k\geq 0$,
we shall show that there exists $n>k$ such that 
for the matrix $ N_kN_{k+1}\cdots N_{n-1}$,  for every $a\in \A_k$ and  every $b\in \A_n$, the
$ab- $ entry is positive.
Supposing that $a\in \A_k$, 
we continue $a= x_k$ to the right with a   vertex path (this
exists since $(N_i)_{i\geq 0}$ is reduced) until the maximum time possible, 
when by hypothesis it is forced to join the stream $\alpha$.  Since
$\alpha$ is reduced, there is
a  time $t$ such that each $a\in \A_k\setminus \alpha_k$ 
connects to some element of $\alpha$ at that time.
By the primitivity of $\alpha$, there exists $s>t$ such that $N^\alpha_t
N^\alpha_{t+1}\cdots N^\alpha_{s-1}$ has all entries positive. Thus
 for
every  $a\in \A_k$, there exists a vertex path connecting $a$ with all
elements of $\alpha_s$.

Letting $d= \# \alpha_s$, suppose $(l_s-d)> 0$, as otherwise we are
done. Now begin $(l_s-d)$ vertex
 paths, starting at the elements of $\A_s\setminus \alpha_s$.
There is a least  time, say $r>s$, by which they all must have  joined $\alpha$, for otherwise there would exist an infinite  disjoint vertex path beginning at time $s$.
Next consider backwards vertex paths starting at the elements of 
$\A_r\setminus \alpha_r$ and remaining outside of $\alpha_k $ for $k<r$.
 They cannot continue all the way to time $k=s$, 
otherwise it would contradict the definition of    $r$. Therefore, choosing $b\in 
\A_r\setminus \alpha_r$, there exists $q>s$ and a  vertex path $w_qw_{q+1}\dots
w_r=b$ with $w_{q-1}\in \alpha_{q-1}$ and with the rest of this path segment disjoint from $\alpha$.

Hence there is a vertex path $a=x_0 x_1\dots x_t\dots  w_{q-1}\dots  w_{r-1}w_r=b$, with its  middle portion 
$x_t\dots w_{q-1}$
 in $\alpha$. We have shown that
 $\A$ is itself primitive, completing the proof.
 \end{proof}

 \begin{lem}\label{l:alphacupbeta}
Let $(N_i)_{i\geq 0}$ be a  reduced sequence of  $(l_i\times l_{i+1})$
nonnegative real matrices, with bounded alphabet size.
If two reduced primitive streams 
$\alpha$ and $\beta$ infinitely both front- and back-connect then 
$\alpha\cup\beta$ is reduced primitive.
 \end{lem}
 \begin{proof}
   Choosing $k\geq 0$, we are to find 
$n> k$ such that for 
 $a\in  \alpha_k\cup \beta_k$ and $b\in \alpha_n\cup \beta_n$,
then the $ab$\,--\, entry of $N^{\alpha\cup\beta}_k\cdots N^{\alpha\cup\beta}_{n-1}$ is positive;
that is there exists a vertex path $x_k\dots x_n$ starting at $a= x_k$ and ending at $b=x_n$ 
with $x_i\in \alpha_i\cup\beta_i$ for all $k\leq i\leq n$.

To carry this out, supposing that $a\in \alpha_k$, 
we 
wait until time $t_0$ 
such that (using the fact that $\alpha $ is reduced)
$N^{\alpha}_k\cdots N^{\alpha}_{t}$ is strictly positive
for all $t\geq t_0$. Then we wait until such a 
time $t$ at which some
 element $c$ of $\alpha_t$  front-connects to some 
element $d$ of $\beta_{s}$, where $s>t$. Lastly we wait until time $r_0$
when the 
matrix product $N^{\beta}_s\cdots N^{\beta}_r$ is strictly positive, for
all $r\geq r_0$ (using that $\beta$ is reduced).
Beginning a vertex path $(x_i)_{i\geq k}$ at $x_k=a$ we  
continue within the stream $\alpha$ until $x_t= c$, then front-connect 
to the stream $\beta$ at 
$x_s= d$, continuing  on within $\beta$ to any chosen $x_r= b$.
This has produced a vertex path within $\alpha\cup\beta$ from 
$a\in \alpha_k$ to any $b\in \beta_r$, for any $r\geq r_0$.   Similarly since $\beta$ front-connects infinitely with $\alpha$,
we could instead have started in $\beta$. We let $n$ be the max of
these two times $r$, hence  $N^{\alpha\cup\beta}_k\cdots
N^{\alpha\cup\beta}_{n-1}>0$. Since the union of two
reduced streams is reduced, we have shown that $\alpha\cup\beta$ is
reduced primitive, as claimed.
 \end{proof}

Before moving on to the main results of this section,  we apply these
two lemmas in the restricted setting of integer entries to 
examine more closely the relationship between 
minimality and primitivity. Recall from Theorem 2.13 of
\cite{Fisher09a} that primitivity implies
minimality for adic transformations (we repeat that argument in $(ii)$
to follow), while (see
\cite{FerencziFisherTalet09})  the 
Chacon adic transformation illustrates that the reverse is false in general.
By contrast, for the action of 
$\FC$, one does
have an equivalence, as  seen 
from part $(iii)$ of the following.

 \begin{prop}\label{p:min_iff_prim}
Let $M=(M_i)_{i\geq 0}$ be a  sequence of  $(l_i\times l_{i+1})$
nonnegative integer matrices, with alphabets $\A= (\A_i)_{i\geq 0}$.
\item{(i)}
For $e, f$ edge paths,  $f$ is in the $\FC$\,--\,orbit closure of $e$
iff 
$e$ infinitely back-connects to  $f$.
\item{(ii)} If $M$ is primitive then the action of $\FC$ on the edge path space is  minimal.
\item{(iii)} 
Assume the alphabets are 
of bounded
 size. Then if the action of $\FC$ is minimal, $M$ is primitive.
  \end{prop}
 \begin{proof}To prove $(i)$, $f$ is in the $\FC$\,--\,orbit closure of $e$
   iff
for any  edge cylinder set $[f_0\dots f_k]$, there exists $n>k$ and
edge path $f_0\dots f_k g_{k+1}\dots g_n$
that $g_k^+= e_k^+$, but this is so iff $e$ infinitely back-connects to  $f$.

For $(ii)$, let $\e\in \Sigma_M^+$; we show its
orbit is dense. Let $[f_0\dots f_k]$ be a (nonempty) cylinder set. By
primitivity there exists an
$n>k$ such that $M_k^n\equiv M_kM_{k+1}\dots M_n$ is strictly
positive. Thus there is a finite allowed edge path $(.f_0 f_1\dots f_k
f_{k+1}\dots f_n) $ such that $f_n^+= e_n^+$. Letting $\gamma$ denote the
generator of $\FC$ which switches   $[. e_0e_{1}\dots
e_n]$ and $[.f_{0}f_{1}\dots f_n] $, then $\gamma (\e)$ is in the cylinder.

For $(iii)$ we show that if $M$ is not primitive then the
   action of $\FC$ is not minimal. Without loss of generality (by
   Lemma 2.2 of~\cite{Fisher09a}) we can
   assume $M$ is reduced. Let
$ w= (w_i)_{i\geq 0}$ be an allowed vertex path. Let $\alpha$ be a reduced primitive
special stream containing $ w$, of maximal size. (That
exists 
since the alphabet size is bounded).  By 
Lemma \ref{l:getting one more} since $\A$ is not primitive, there
exists for some $k\geq 0$ a vertex path $ x= (x_i)_{i\geq k}$ disjoint from $\alpha$. The singlet
stream $\alpha_{ x }$ is primitive, hence by 
Lemma \ref{l:alphacupbeta} it cannot infinitely both front-and
back-connect with $\alpha$, as otherwise $\alpha\cup \alpha_{ x }$ would be
strictly larger in size and still primitive reduced and special.

We know that after some $n>k$, $ x$ either no longer front-connects or no
longer back-connects to $ w$. Suppose the latter. Then
consider an edge cylinder set $[f_0\dots f_n]$ with $f_0^-= w_0,\dots,
f_n^-= w_n$. The $\FC$\,--\,
orbit of an edge path $ e= (e_i)_{i\geq 0}$ with $
e_i^-= x_i$ for all $i\geq k$ (this extension to the left exists since
$M$ is reduced)  does not meet $[f_0\dots f_n]$, since otherwise $ x$ would
back-connect to $ w$ after time $n$. Suppose the former. Then consider
the cylinder set $[e_0\dots e_n]$. In the same way, the $\FC$\,--\,orbit 
of $ f$ does not meet $[e_0\dots e_n]$. In either case, minimality
is contradicted.
 \end{proof}

\begin{theo}(Stream Decomposition Theorem) \label{t:StreamDecomp}
Let $N=(N_i)_{i\geq 0}$ be a  reduced sequence of  $(l_i\times l_{i+1})$
nonnegative real matrices with alphabet sequence $\A= (\A_i)_{i\geq
  0}$ and $\#\A_i=l_i>0$, with $l_i$ bounded.
Then there exists a finite collection 
$\alpha(i)$, $i=1,2,\dots ,d$ of reduced primitive streams, together with
a possibly nonproper stream $\P= (\P_k)_{k\geq 0}$ (called the {\em pool stream}) 
such that:
\item{$(i)$} $\{\alpha(1)\dots, \alpha(d),\P\} $ partition
$\A$;  after the maximum of the starting times  each 
$\alpha(i)$  is nonempty; 
\item{$(ii)$} there are front-connections from $\alpha(i)$  to 
$\alpha(j)$ only if $i<j$;
\item{$(iii)$} the pool stream $\P$ contains no infinite
vertex paths; for each $p\in \P$  there exists a front connection to
some stream $\alpha(j)$, with $j>1$ (with no front connection to $\alpha(1)$), and  there exist
front-connections from a given  stream $\alpha(i)$ to $p\in \P$
and from $p$ to  $\alpha(j)$ only
if $i<j$.
\item{$(iv)$} This decomposition is unique after some finite time, 
    up to a renumbering of the streams which preserves property
    $(ii)$.
    \item{$(v)$} If $\alpha$ is a reduced primitive stream which is
      maximal in the sense of containment, then there exists a stream
      decomposition satisfying $(i)$-$(iii)$ which includes $\alpha$.
    \end{theo}

\begin{proof}
Since the matrix sequence $(N_i)_{i\geq 0}$ is reduced, there exists
at least one  infinite vertex path $ x = (x_0 x_1\dots)$. 
Let  $\alpha $ be a  primitive special stream containing
$\alpha_{ x}$ of maximal 
size.

Now consider the stream $\Cal S$ of all symbols  which  front-connect to 
$\alpha$ (so $\Cal S\supseteq\alpha$). Suppose that  $\Cal S$ is
not all of $\A$; then (since $\A$ is reduced) there exists an infinite vertex path $ y = (y_m y_{m+1}\dots)$ disjoint from
$\alpha$ (if some $y_i\in \alpha$, then $y_m$ is in $\Cal S$, a
contradiction). 
 Let $\beta $ be a  
be a 
primitive special stream containing this singlet stream denoted
$\beta_{ y}$
and
disjoint from $\alpha$, 
of maximal 
size.

Next consider the stream $\Cal S$ of all symbols  which  front-connect to 
$\alpha$ and  $\beta$ or both (so $\Cal S\supseteq\alpha\cup\beta$). Suppose that  $ \Cal S$ is
not all of $\A$. Then, as before, there exists an infinite vertex path $ w = (w_l w_{l+1}\dots)$ disjoint from
$\alpha\cup\beta$; let $\gamma $ 
be a 
primitive special stream containing this singlet stream $\gamma_{ w}$
and
disjoint from $\alpha\cup \beta$, 
of maximal 
size.
Continuing in this way as far as possible, we have
produced a finite collection (since alphabet size is bounded) of disjoint primitive special streams, such
that their complement in $\A$ consists (since $\A$ is reduced) of symbols which front-connect
to one of these
streams.

Next, we   partially order  this
collection.
First, for general streams, we write
$\alpha\lesssim   \beta$ iff $\alpha$  infinitely front-connects to 
 $\beta$.
This relation is transitive.   If, as in our collection, 
$\alpha, \beta$ are reduced, then
$\alpha\lesssim \alpha$ (reflexivity).  If $\alpha\lesssim   \beta$  and $\beta\lesssim   \alpha$ with
$\alpha, \beta$ reduced primitive, then $\alpha\cup \beta$ is 
also reduced primitive, from Lemma \ref{l:alphacupbeta}. 
Since the streams of our
collection are reduced primitive and also
were chosen at each stage to have 
maximal size, that is not possible, and so
 if $\alpha\lesssim   \beta$ and $\beta\lesssim   \alpha$ then 
$\alpha=\beta$, proving symmetry. Thus 
$\lesssim $  defines a partial order on our collection.

Now we define
 a linear  order $\preccurlyeq $ which
  respects the partial order, that is, so that $\alpha\lesssim \beta\implies
 \alpha \preccurlyeq \beta$.
 One can always do this: as noted above regarding the stationary case, 
one can visualize the partial order $\lesssim$ as a
 tree; tilting the graph slightly and viewing it from the side projects
 to a compatible linear order. (Disconnected parts of the tree can be
 ordered arbitrarily.) See Figs.~\ref{F:Frobenius} and  \ref{F:Frobeniusb}.

Now we modify these  streams to achieve what we want.
For notational simplicity let us assume that we have 
three  primitive special streams $\alpha\prec \beta \prec \gamma$.
(Here $\alpha\prec \beta$ means $ \alpha \preccurlyeq \beta$ but $ \alpha \neq \beta$.)
We shall define 
new proper streams $\wh \alpha, \wh \beta,  \wh \gamma$, plus a
possibly nonproper stream $\P$,  giving these the inherited linear order,
which satisfy $(i), (ii), (iii)$ of the theorem.

We begin by defining $\wh\alpha\supseteq\alpha$ by 
$$
\wh \alpha_j= \{ a\in \A_j:\, a \text{ communicates to } \alpha\}
$$ 
for all $j\geq 0$.
We claim that $\wh \alpha$ is still primitive. 
Let $a\in \wh\alpha_k\setminus \alpha_k$ for some chosen $k\geq 0$. We are to show that for 
some $p>k$, $a$ communicates to all of $\wh\alpha_p$.

Since $\alpha$ is less than or equal to all the other special streams
in the linear order $\preccurlyeq $, 
there exists $s>0 $ such that  for any  connection
pair $(i,n)$ from $\beta$ or $\gamma$  to $\alpha$ then  $i<s$.

Since $a\in \wh\alpha_k$, there exists some least $\wt m>k$ such that  $a$
communicates to some $b\in \alpha_{\wt m}$. We  take $s$ to also be greater than $\wt m$. 

 Now let $t>s$ be 
such that
$N^{\alpha}_s\cdots N^{\alpha}_{t-1}$ is strictly positive.
Thus, since $\alpha$ is reduced, 
$a$ communicates to all of $\alpha_i$ for any $i\geq t$.

Define $R$ to be the supremum of the  times $l $ such that there
exists a vertex path 
 $w_t\dots w_l$ such that for all $i$ with $t\leq i\leq l$, then $w_i\notin \alpha\cup\beta\cup\gamma$.
 This is
 finite, since otherwise (by compactness of the path space) there
 exists an infinite path with this property, i.e.~we could form a new
primitive special stream.

Now for $p>R$, we claim that $a$ communicates to 
any chosen $c\in \wh\alpha_p$. This is true for 
$\alpha_p$, so we assume that $c\in \wh \alpha_p\setminus \alpha_p$.

Since $\A$ is reduced, there is a vertex path $\wt w_t\dots \wt w_p=c$. Each
$\wt w_i\in \wh\alpha$, since it front-connects to $c$ hence to
$\alpha$. And  $\wt w_i$ cannot be in $\beta_i\cup\gamma_i$,
since $t>s$, and those streams do not have forward connections to $\alpha$
after time $s$. Thus since $p>R$, there is some $q$ with $t\leq q\leq
p$ such that
$\wt w_q\in\alpha_q$. Hence there is 
a vertex path $a=x_k\dots b=x_{\wt m} \dots x_{q}=\wt w_q\dots  \wt w_p=c$, as
claimed. And so
$\wh \alpha$ is indeed primitive.

We next define $\wt\beta= \beta\setminus\wh\alpha$ and claim this
stream is still primitive. 
But if for some $k\geq 0$, $a$ is an element of $\beta_k\setminus\wh\alpha_k$, then it can
be continued infinitely within $\beta$ since $\beta $ is reduced. And
this continuation remains within $\wt\beta$ since $a$ does not communicate with
$\alpha$. Thus in particular $\wt\beta$ is proper, but moreover  since symbols have been removed from $\beta $ only up to
some finite time, 
the primitivity of $\beta$ beyond that point  will pass over to
$\beta\setminus\wh\alpha$. 

In the same way, defining $\wt\gamma= \gamma\setminus\wh\alpha$, this
is still primitive.

We note that if $a\in\A_k\setminus(\wh
\alpha_k\cup\wt\beta_k\cup\wt\gamma_k)$, then $a$ does not
front-connect to $\wh\alpha$, hence it must front-connect to
$\wt\beta\cup\wt\gamma$,  as otherwise we could produce a
new primitive special stream.

Next we define the stream $\wh\beta$ by, for all $j\geq 0$,
\[
\wh \beta_j= \{ a\in \A_j\setminus\wh\alpha_j: \, a \text{
  front-connects to }\wt\beta, \text{ and back-connects to
} \wt \beta \text{ or } \wt\gamma\}
.\]

The reason for requiring that ``$a  \text{ back-connects to
} \wt \beta \text{ or } \wt\gamma$'' is to have the right definition
of pool stream: if, for example, $a$ back-connects both to $\wh
\alpha$ and $\wt\beta$ then it will be included in $\wh\beta$, while
if it back-connects  only to 
$\wh \alpha$ and/or $\P$ then it will be included in $\P$, as we see below.

It is clear that
$\wh\beta\supseteq \wt\beta$ and  that it  is  a proper stream, i.e.~that 
if $\wh\beta_k\neq\emptyset$ then also $\wh\beta_{k+1}\neq\emptyset$.

Now $\wh\beta$ is  primitive. The argument is similar to that 
 for $\wh\alpha$; thus, take
$a\in \wh\beta_k\setminus \wt\beta_k$ for some $k\geq 0$. This symbol connects
to $b\in \wt\beta_{\wt m}$ for some least time $\wt m$.   There is 
$m'$ such that $\wt\gamma$ does not connect to $\wt\beta$ after time $m'$;
let
$s$ be greater than $\wt m$ and $m'$. By
primitivity of $\wt\beta$ there exists 
$t>s$ 
such that
$N^{\wt\beta}_s\cdots N^{\wt\beta}_{t-1}$ is positive. Let 
$R$  be the supremum of   times $l $ such that there
exists a vertex path 
 $w_t\dots w_l$ such that for all $i$ with $t\leq i\leq l$, then
 $w_i\notin \wh\alpha\cup\wt\beta\cup\wt\gamma$. Then
 $R$ is finite (as above), since this implies
$w_i\notin \alpha\cup\beta\cup\gamma$.

For $p>R$, and $c\in \wh\beta_p\setminus \wt\beta_p$, there
exists a vertex path $\wt w_t\dots \wt w_p=c$.
We claim that for some time $q$ with $t\leq q\leq
p$ we have
$\wt w_q\in\wt\beta_q$.

Now $\wt w_i\notin \wt\gamma_i$,
since $t>s$, and $\wt w_i\notin
\wh\alpha$ since $\wt w_i\in
\wh \beta$ (as it front-connects to $c\in\wh \beta$) while
$\wh\alpha\cap\wh\beta=\emptyset$.
We cannot have $\wt w_i\notin
\wt \beta$ for all times, since then $\wt w_i\notin
\wh\alpha\cup\wt\beta\cup\wt\gamma  $, violating the definition of $R$.
This proves the claim. 

Hence there is a vertex path 
$a=x_k\dots b=x_{\wt m} \dots x_{q}=\wt w_q\dots  \wt w_p=c$,
proving primitivity of $\wh \beta$.

Redefining
$\wt\gamma=\gamma\setminus(\wh\alpha\cup\wh\beta)$,
we verify as for $\wt\beta$ that  this is still primitive, the proof
now being easier since in fact  there are no front-connections from $\wt \gamma$ to $\wh\alpha\cup\wh\beta$.
Then
we define the stream $\wh\gamma$ by, for all $j\geq 0$:
\[
\wh \gamma_j= \{ a\in \A_j\setminus(\wh\alpha_j\cup\wh\beta_j): \, a
\text{  front-connects to } \wt\gamma,   \text{ and back-connects to
} \wt\gamma\}
.\]

As for $\wh\beta$, the stream $\wh\gamma$ is proper
and   primitive.

We have shown that the streams  $\wh\alpha,\wh\beta$, and $\wh\gamma$ are proper and primitive; by
construction, they  are disjoint,
with
no front-connections from $\wh \beta$ to $\wh\alpha$. There are also
no front-connections
from $\wh \gamma$ to $\wh\alpha\cup\wh\beta$: this is true for $a\in
\wt \gamma$, and for $a\in
\wh \gamma\setminus \wt\gamma$, then $a$ back-connects to $\wt
\gamma$,
so if it front-connects to $\wh \alpha$ then it is already in
$\wh\alpha$, and if it front-connects to $\wh \beta $ then it is already in
$\wh\beta$, a contradiction.

We now consider the complement, $ \A\setminus (\wh\alpha\cup\wh\beta\cup\wh\gamma)$,
which we define to be $\P$, the {\em pool stream}.
If $a\in \P$ front-connects to $\wh\beta$ then it cannot
back-connect to $\wh\beta$ or $\wh\gamma$ (since then it would be in 
$\wh\beta$, a contradiction). Therefore it back-connects to
$\wh\alpha$
or $\P$, or to nothing if $a\in \P_0$.
If $a\in \P$ front-connects to $\wh\gamma$ then it cannot
back-connect to $\wh\gamma$ (since then it would be in 
$\wh\gamma$). Therefore it back-connects to one or more of
$\wh\alpha$, $\wh \beta$ and $\P$, or to nothing.

In summary, the  elements of $\P$ are 
those symbols which have a front-connection to $\wh\beta$ or
$\wh\gamma$ or both, and which if they have a 
back-connection from one of the streams $\wh\alpha, \wh\beta, \wh\gamma$ and a front-connection to a different
one, then the first is less than the second in the linear
order.

Next we prove $(iv)$, uniqueness.
Suppose we are given two partitions of $\A$, $\{\alpha(1)\dots, \alpha(d),\P\}
$ and $\{\alpha'(1)\dots, \alpha'(d'),\P'\} $,
satisfying $(i)$-$(iii)$. We claim they are eventually
identical. Now each non-pool stream is a union of
infinite vertex paths, i.e.~ singlet streams. Take one, $\alpha'_x\in
\alpha'(i)$.  It cannot be a subset of $\P$ since the pool stream
contains no infinite singlet streams.  So it must meet some
$\alpha(j)$ infinitely often. If it also meets another  stream
$\alpha(k)$
infinitely often then $j\leq k$, so eventually it only meets one such
stream, say $\alpha(j)$. Now for any other singlet stream $\alpha'_y\in
\alpha'(i)$ this also holds, and it has to be the same stream
$\alpha(j)$, since by primitivity of $\alpha'(i)$ the two singlet
streams $\alpha'_x $ and $\alpha'_y$ infinitely front-and
back-connect. This proves that $\alpha'(i)\subseteq 
\alpha(j)$ eventually. Reversing the argument, $\alpha(j)\subseteq 
\alpha'(j')$ for some $j'$. Thus the two collections of primitive
streams
are eventually identical, whence so are their complements, $\P$ and $\P'$.

Lastly we prove $(v)$: if $\alpha $ is a primitive reduced stream, then it
contains an infinite vertex path $ x = (x_0
x_1\dots)$.
Beginning the proof with this path proves the statement, since then
by the construction
$\alpha(1)\supseteq \alpha$.
\end{proof}

\begin{defi}
Given a stream decomposition satisfying $(i)-(iii)$ of the theorem,  
we define an {\em initial stream} to be $\alpha(i)$ such that 
 such that no other stream infinitely front-connects to it, and we define a
{\em final stream} to be $\alpha(i)$ such that  no other stream infinitely
back-connects to it,
i.e.~ such that it has no infinite front-connections to the other streams.
\end{defi}

\begin{theo}(nonstationary Frobenius Decomposition Theorem) \label{t:FrobDecomp}
Given a reduced sequence $N=(N_i)_{i\geq 0}$ of $(l_i\times l_{i+1})$ 
 nonnegative real matrices with bounded alphabet size,
then there exists a reordering of the alphabets $\A=(\A_i) _{i\geq 0}$
such that the new matrix sequence is in
Frobenius normal form.

After some finite time, this form is unique up to a (nonstationary)
permutation of the alphabets.

Moreover we can choose this order so as to place the initial streams first and
the final last. Furthermore, there exists 
a gathering such that the new matrix sequence is in
{\em fixed--size}   Frobenius normal form for times $\geq 1$. That is,

\begin{equation}\label{eq:block_form}
N=
\left[ \begin{matrix}
A_1 &C_{12} &\dots  &C_{1\wh l}\\
0_{21} &A_2 & &\vdots\\
\vdots & &\ddots  &\\
0_{\wh l 1} &\dots  &  &A_{\wh l}
\end{matrix}  \right]
\end{equation}
\end{theo}

\begin{rem}\label{r:fixed--size_notunique}

 As in Remark \ref{r:permutationsequence} the choice of
  order does not affect the topology of the edge path space, as a change
  of the order
  induces  a
 topological  conjugacy (vertex and hence edge cylinder sets correspond). Further, 
 a gathering induces a topological conjugacy of the spaces and of the
 $\FC$-actions (though not quite for adic transformations, as there
 are more orders on the gathered diagram): see Proposition \ref{p:gathermeasure}.

  See Remark \ref{r:normalform} below for why in the  statement we
  begin with time one, not time zero.
  
 We note  that  the fixed--size form is not unique, because a further
 gathering will change the form. Indeed,  the natural equivalence relation on Bratteli
diagrams is defined by gathering together with dispersal--the ``full
orbit'' of this operation--not just by the
``forward dynamics'' of gathering,     as shown by the
following example.

Define a periodic alphabet sequence
  $\A_i$ with
  $\#\A_i=1$ for $i$ even,
 $\#\A_i=2$ for $i$ odd, and define a  matrix sequence  $M_i$ by
  $M_i= \left[ \begin{matrix}
1 &1\\
\end{matrix}\right]$ for  $i$ even,
$M_i= \left[ \begin{matrix}
    1\\
    1\\
\end{matrix}\right]$  for $i$ odd. We gather first along the even
times, giving the constant matrix sequence 
$M_i'= \left[ \begin{matrix}
2\\
\end{matrix}\right]$ and then along odd times, giving
$M_i''= \left[ \begin{matrix}
    1 &1\\
    1 &1\\
\end{matrix}\right]$. Thus the way we gather affects the fixed-size
form; here, 
neither of these Bratteli diagrams can be reached from the other by a further
gathering.

This observation is related to Remark 3.2 of \cite{BezuglyiKwiatkowskiMedynetsSolomyak13}.

We note that the sequence $M$ is primitive and periodic hence uniquely ergodic,
and that $M', M''$ are naturally isomorphic, being
two representations for the dyadic odometer, the first as an edge and
the second as a vertex shift.

The (nonstationary) permutations of alphabets which can be allowed in
the statement about uniqueness include a
permutation of the initial and final streams,  permutations within
the primitive streams, and a reordering of the primitive streams as
long as this preserves upper triangularity (for example, if
stream $(1)$ and $(2)$ only communicate to stream $(3)$, then $(1)$ and $(2)$
could be permuted).

\end{rem}
\begin{proof} 
From Theorem \ref{t:StreamDecomp}, we have reduced
  primitive streams
 $\alpha(1),\alpha(2), \dots, \alpha(d)$. We
now extend this list to include pool elements, by partitioning the
pool stream $\P$ into
substreams $\P(i) $ for $2\leq i\leq  d$, as follows.

We
know that a pool element $x_k\in\P_k\subseteq \A_k$ front-connects to at least one
non-pool stream. Letting $i$ be the least integer such that $x_k$
front-connects to $\alpha(i)$,
we place $x_k $ in $\P(i)$. We then linearly order the streams and sub-streams,
defining 
$\alpha(1)< \P(2)<\alpha(2)<\P(3)<\dots<\P(d)<\alpha(d)$.
 We recall that the primitive streams are proper, hence empty before their
starting times, and that
a pool sub-stream $\P(i)$ may be empty at any given time.

Note that if $x_k\in \alpha(i)$ front-connects to a pool element $x_{k+1}\in
\P(j)$,
then necessarily $i< j$, 
since the stream $\alpha(i)$
then front-connects to $\alpha(j)$.

Rows and columns
of  a sequence of matrices $B_k$ with $0-1$ entries
will be indexed by this ordered list of streams and
pool sub-streams, now numbered by $1\leq i\leq (2d-1)$. For $k\geq
0$, we define
a $(2d-1)\times
 (2d-1)$ $0-1$ matrix $B_k$ 
to have a $1$ in the $ij^{\text{th}}$ place iff
the stream numbered $i$ front-connects to stream $j$ at time $k$, with
connection pair $(k, k+1)$.
 
Since there are 
only front-connections from  stream $i$ to stream $j$ if $i\leq j$,
 the matrix $B_k$ is upper triangular.

The streams 
define the
(possibly virtual)
subalphabets $\{\A^i_k\}_{i= 1}^{\wh l_k} $ for each time $k$, where
$\wh l_k= 2d-1$ for all $k\geq 0$, with $\A^1_k\equiv \alpha(1)_k$,
$\A^2_k\equiv \P(2)_k$ and so on.
We   linearly order the elements of each of these subalphabets for each
time $k$ in a way compatible with the order on streams. This fixes the order on 
the alphabets $\A_k$ and thereby  defines the
matrices $N_k$. 
Since $B_k$ is upper triangular, as noted in Remark \ref{r:uppertriangular}.
$N_k$ satisfies our
definition of upper triangular block form (recalling that virtual
alphabets are allowed in our definition),   so we have achieved the desired
Frobenius normal form for the sequence $N$.

The stated uniqueness of the form follows directly from the uniqueness
proved in $(iv)$ of Theorem
\ref{t:StreamDecomp}.

We next describe how to perform a gathering in order to arrive at
the fixed--size form, which will also eliminate the virtual alphabets.

From the statement of the Theorem,
 the fixed--size for
 block sequence is to begin at time $1$; we achieve this by  gathering
 all the irregularities into the first matrix $N_0$.

Our first step will be to take care of the initial and final streams. 
If $i<j$, we know there may be front-connections from $\alpha(i)$ to
$\alpha(j)$,  finite or infinite in number. We perform a
first gathering, from time $0= n_0$ to time $n_1$, with $n_k= n_1+k-1$
for $k\geq 1$, so as to concentrate all finite front-connections
before time $1$. Then
we reorder the streams so that 
the initial streams are listed first, the final
 streams  last. Thus, in the new list, if
$\alpha(i)$ has any front-connections to $\alpha(j)$ then $i<j$.
This first gathering guarantees that the  matrices from time $1$
on will  be upper triangular.

In doing this we carry the pool streams along with the streams that
immediately follow them. This maintains  property $(iii)$ of Theorem
\ref{t:FrobDecomp}.

Now the 
 number of nonempty primitive streams is  constantly $d$  after time $k=t_0$, the maximum of
their starting times,
and as a second  step 
we gather the matrix sequence along the subsequence $0, t_0, t_0+1,
t_0+2,\dots$,  replacing
the first matrix
$N_0$ by $N_0 N_1\cdots N_{t_0}$. 

We next  gather along the subsequence $0= n_0, n_1,\dots$ such that 
$\#\P(2)_{n_i}= \limsup_{i\geq 1} \#\P(2)_i$ for
$i\geq 1$. Then we 
gather along the similar sub-subsequence for $\P(3)$, and so on. 
The numbers $\#\P(i)_k$ are
now constant in $k$ for each $2\leq  i \leq d$. Next we remove 
any  sub-streams $\P(i)$ from the list 
~such that $\#\P(i) =0$.

In the resulting list, all streams are nonempty for all times $\geq
1$.  We define a new 
 $0-1$ matrix sequence $(B_k)_{k\geq 0}$ 
as before.
 Letting  $\wh l$ denote the number of streams after time $1$, then $B_k$ is 
now $(\wh l\times \wh l)$ 
with   the
 matrix $B_k$  upper triangular, for all $k\geq 1$.

In the next step we want to guarantee that there are  $0$'s on the 
diagonal of the block matrices $B_i$ for the rows corresponding to the
pool elements.
For this we gather again, so
as to eliminate front-connections from pool elements to other
pool elements in
the same substream. We know that for any $2\leq  i \leq d$ there is no
infinite vertex path within $\P(i)$ (otherwise we would have a new
primitive special stream). Hence for all $i$ and for each  $k\geq 0$, 
 there is a least time $r>k$ such that there is no finite vertex
 path within $\P(i)$ from time $k$ to time $r$. Let $R(k)$ be the max of $r+1$ 
over all $2\leq  i \leq d$. Then  if $N_{ii}(k)$ denotes the 
block of $N_k$ corresponding to a null diagonal entry $B_{ii}(k)$,
then $N_{ii}(k) N_{ii}(k+1)\cdots N_{ii}(R)$ is all--zero. We
set $n_0=0, $ then inductively $n_{i+1}= R(n_i)$.
The gathering along the sequence $(n_i)_{i\geq 0}$ eliminates all  front-connections within 
$\P(i)$ for all $2\leq  i \leq d$. Hence the corresponding matrices 
$(N_k)_{ii}$ are all--zero, and for the block matrices $B_k$
the diagonal element $(B_k) _{ii}$ is zero.

Next, considering the size of the diagonal blocks $(B_k) _{ii}$ for
fixed $i$, 
some value must repeat infinitely, so if we gather along those times 
these blocks are square. We do this successively for $i=1,\dots , \wh
l$ taking  sub-subsequences. The reduced, primitive diagonal blocks were
already square, completing the construction.

\end{proof}

\begin{rem}
If we begin with a nonnegative real $(d\times d)$ matrix $N$, then the theorem
gives the usual Frobenius upper-triangular form for $N$, with one change, as noted above:
if there is an irreducible but nonprimitive 
class, then this will have been further decomposed into its
primitive cycles. 
\end{rem}

\subsection{Invariant subsets and minimal components for adic transformations}

Now we specialize to nonnegative integer matrices and so return to the
world of adic transformations. 
For a nonnegative integer  sequence $M=(M_i)_{i\geq 0}$ with bounded
alphabet size,
Proposition \ref{p:min_iff_prim}  showed the equivalence of
primitivity of $M$ and
minimality for the action of $\FC$. The Frobenius
normal form together with the notion of adic towers  leads us here to a
complete description of orbit closures of edge paths, including the 
identification of  the 
minimal invariant subsets, for   general $M$.
First we need:

\begin{defi}
  Given a stream $\alpha$ with starting time $m>0$, we extend this 
to the {\em augmented stream} $\wh \alpha$ defined 
by: $\wh \alpha_i= \alpha_i$ for $i\geq m$; for $0\leq i<m$, we define 
$\alpha_i\subseteq \A_i$ to be all vertices which front-connect to
$\alpha_m$.

Note that given a stream decomposition of the theorem,  the augmented
streams are in general no longer disjoint up until the maximum of the
starting times.
\end{defi} 

In what follows we will replace the primitive streams  from the Frobenius
decomposition as $\alpha(l)$ for $1\leq l\leq d$ by the augmented (also primitive, but at the beginning no
longer disjoint) streams denoted $\overline\alpha(l)$. The reason for using the augmented streams is so
the adic towers will make sense, as the submatrix sequences  then begin at time $0$.

\begin{theo}{\em (Invariant subsets and minimal components for nonstationary adic transformations)}\label{t:invariantsubsets}
Let $(M_i)_{i\geq 0}$ be a reduced sequence   of  nonnegative integer
matrices with bounded alphabet size. We  list the primitive streams  from the Frobenius
decomposition as $\alpha(1),\alpha(2), \dots,\alpha(d)$.
Let $m$ denote the maximum of the starting times of the
streams. We order
the alphabets so that
$(M_i)_{i\geq 0}$ is in  Frobenius normal form, with 
diagonal primitive blocks $A(1),A(2), \dots,
A(d)$ corresponding to
these streams. Each matrix sequence $A(l)$ is nonvirtual from time
$m$ onwards.

We   construct the augmented streams $\overline
\alpha(l)$. We write $\overline A(l)= (\overline A(l)_i)_{i\geq 0}$ for the corresponding
matrix sequence.

Then: 

\item{(i)} 
For  $\e\in \Sigma_{M}^{0,+} $, let $k$ be the (unique)
  integer such that $e$ eventually accompanies a primitive stream
  $\alpha(k)$.
 Then $\e$ is in  the (unique) adic tower $\Sigma_{M/\overline A(k)}^{0,+}
 $, and its $\FC$\,--\,orbit  is dense
 in this tower.

\item{(ii)} The orbit closure of $e$ is the union of the towers $\Sigma_{M/\overline A(i)}^{0,+}$ such that
$\alpha(i)$ infinitely front-connects to $\alpha(k)$.

\item{(iii)} 
The compact invariant minimal subsets are the
tower spaces $\Sigma_{M/\overline A(i)}^{0,+}$ such that $\alpha(i)$ is an
initial stream.
\end{theo}
\begin{proof}
As in the proof of Theorem \ref{t:FrobDecomp}, we order the (possibly
virtual) primitive and pool streams as
$\alpha(1)< \P(2)<\alpha(2)<\P(3)<\dots<\P(d)<\alpha(d)$. (The initial streams
do not necessarily  come first.)

Each edge path $\e\in \Sigma_{M}^{0,+} $ eventually accompanies  some unique primitive stream
$\alpha(k)$ in the list (i.e.~its vertex path eventually belongs to
that stream); equivalently it eventually accompanies the augmented
stream $\overline
\alpha(k)$. It may at first accompany one or more  primitive streams
$\alpha(i)$ or null 
streams $\P(i)$, in which case 
$i<k$ or $i\leq k$ respectively.

Now $\overline\alpha(k)$ has starting time $0$, so
its matrix subsequence $\overline A(k)$ is nonvirtual for all times
$\geq 0$.  
The space of edge paths  $\Sigma_{\overline A(k)}^{0,+} $ is
the collection of 
paths $f$
which accompany the stream $\overline\alpha(k)$. The  $\FC_M$--orbit
of $\Sigma_{\overline A(k)}^{0,+} $
consists of all 
edge paths $\e$ which eventually equal some such $f$, equivalently those
which eventually accompany  $\alpha(k)$. From Proposition
\ref{p:towerspace}, the $\FC_M$-orbit of
$\Sigma_{\overline A(k)}^{0,+} $ is the tower space $\Sigma_{M/\overline A(k)}^{0,+} $.

Thus every edge path $\e\in \Sigma_{M}^{0,+} $  is in a unique adic tower space
$\Sigma_{M/\overline A(k)}^{0,+} $, and 
is in the orbit of some $f\in \Sigma_{\overline A(k)}^{0,+} $.
Since $\overline A (k)$ is primitive,  by
  $(ii)$ of Proposition \ref{p:min_iff_prim} the
 orbit (for $\FC_M$ or $\FC_{\overline A (k)}$) of $\f$  is dense in $\Sigma_{\overline A(k)}^{0,+} $. Hence
  $\FC_M(\f)= \FC_M(\e)$ 
  is dense in the tower
  $\Sigma_{M/\overline A(k)}^{0,+} $, by the definition of the tower
  topology, proving $(i)$.

To prove $(ii)$, by part
$(i)$,  any edge  path
$e\in \Sigma_{M}^{0,+} $ belongs to a unique tower space $\Sigma_{M/\overline A(k)}^{0,+}
$ and equivalently $e$ eventually accompanies $\alpha(k)$. 
From part $(i)$ of
Proposition \ref{p:min_iff_prim}, the orbit closure of $\e$ consists
of all $g$ such that $g$ infinitely front-connects to $\e$. This path
$g$ also belongs to a unique tower space $\Sigma_{M/\overline A(i)}^{0,+}$.
Hence, 
$\alpha(i)$ infinitely front-connects to  $\alpha(k)$. We claim that
this holds for
any  $f$ in
$\Sigma_{M/\overline A(i)}^{0,+} $. It is enough to check
this for
$f$ in  $\Sigma_{\overline A(i)}^{0,+} $.  Given 
$l>0$ we show that $f$ front-connects to  $\alpha(k)$ after time
$l$. Since $\alpha(i)$ is primitive there is $n>l$ such that $f_l$
front-connects to all of $\alpha(i)_n$, and since $\alpha(i)$
infinitely front-connects to  $\alpha(k)$ this is true, proving $(ii)$.

For $(iii)$, if  $\alpha(i)$ is initial, there are no other streams
which infinitely front-connect to it. Hence by $(ii)$  there
are also no edge paths not in  $\Sigma_{M/\overline A(i)}^{0,+}$ which
infinitely front-connect to it, and this cannot contain any
$\Sigma_{\overline A(j)}^{0,+}$ 
for $j\neq i$. After some finite time $m$, no stream $\alpha(j)$
front-connects to $\Sigma_{\overline A(i)}^{0,+}$, so
$\Sigma_{M/\overline A(i)}^{0,+}$ in fact equals  $\Sigma _{(\overline
  A(i))^{(m)}}^{ 0,+} $, which is a compact 
 $\FC-$ invariant
subset.
Since $\alpha(i)$ is primitive, the matrix sequence $\overline A(i)$ is 
primitive, hence  by 
  $(ii)$ of Proposition \ref{p:min_iff_prim} 
$\Sigma_{M/\overline A(i)}^{0,+}$ is minimal for the action of
$\FC_{\overline A (i)}$ hence for $\FC_M$.

Conversely, let $X$ be a nonempty compact invariant minimal subset,
and let $\e\in X$. Then $\e$
eventually accompanies some
primitive stream
$\alpha(i)$, so its orbit closure includes
$\Sigma_{M/\overline A(i)}^{0,+}$ whence $\Sigma_{M/\overline
  A(i)}^{0,+}\subseteq X$. We claim that  $\alpha(i)$ is  initial. If not, there is
some other stream $\alpha(j)$ which infinitely front-connects to
$\alpha(i)$, hence by $(ii)$, $\Sigma_{\overline A(j)}^{0,+}$ is contained in the
orbit closure of $\Sigma_{\overline A(i)}^{0,+}$. But then the orbit
closure of $\Sigma_{\overline A(j)}^{0,+} $ must be strictly smaller
than $X$ (a contradiction), as otherwise 
$\alpha(i)$
will infinitely front-connect to
$\alpha(j)$, giving $\alpha(i)= \alpha(j)$.
\end{proof}

\begin{rem}
We note  that if, in the above theorem, there are no front-connections to
the initial states (for instance after gathering from time $0$ to time
$m$), then $\Sigma_{\overline A(j)}^{0,+} =
\Sigma_{M/\overline A(j)}^{0,+} $ so this is the minimal component.

  We mention that, due to the choice of conventions, 
 the partial order for dynamics on the graph in terms of
  communication of states is the
  opposite for the dynamics on the path space; thus e.g.~in the  stationary case,
  the initial states are the repelling fixed points for the graph of
  the \sft, while the corresponding components of the path space are
  the attracting fixed points (and thus the minimal components) for
  the action of an adic transformation.
\end{rem}

\medskip

Next as promised we prove a stronger version of Proposition
\ref{p:opensets}, where we no longer  assume the sequence $A$ is primitive.

\begin{prop}\label{p:opengeneral}
 Consider a nonnegative integer matrix sequence $M=\left[
\begin{matrix}
A& C \\
0 & B
\end{matrix}  \right] $, and write $\A$, $\B$ for the streams
(subalphabets) 
associated to the matrix sequences $A$ and $B$. 
 \item $(i)$ Then $\Sigma_A^+$,  $\Sigma_B^+$ are closed subsets of
$\Sigma_M^+$, and $\Sigma_B^+$ is open.
 \item $(ii)$  Suppose   that $B$ is reduced. Then $\Sigma_A^+$ is
open in $\Sigma_M^+$ iff there are at
most finitely many front connections from $\A$ to $\B$, iff
the sequence $C= (C_i)_{i\geq 0}$ is zero except for finitely many
$i$.
\end{prop}
\begin{proof}
Part $(i)$ is included from Proposition
\ref{p:opensets}.
For part $(ii)$,
applying the Frobenius Decomposition Theorem to 
the sequence $A$, after a gathering and nonstationary permutation of
the 
alphabets, $A$ is in upper triangular form with diagonal primitive or
zero blocks $A_1,\dots, A_{\wh l}$
\begin{equation*}
A=
\left[ \begin{matrix}
A_1 &C_{12} &\dots  &C_{1\wh l}\\
0_{21} &A_2 & &\vdots\\
\vdots & &\ddots  &\\
0_{\wh l 1} &\dots  &  &A_{\wh l}
\end{matrix}  \right]
\end{equation*}
with
associated subalphabet sequences $\A_1,\dots, \A_ {\wh l}$. Supposing that there are infinitely many front connections from $\A$ to
$\B$, we are to show that  $\Sigma_{A}^+$ is not open.
Then there exists $k$ with $1\leq k\leq \wh l$ such that there
are infinitely many front connections from $\A_k$ to
$\B$. Furthermore we can assume  $(A_k)$ is primitive (not $0$), since
pool states must connect some earlier index alphabet to $B$.
We shall find a path $\e= (e_0 e_1\dots)\in
\Sigma_{A_k}^+$ such that any neighborhood meets the complement of
$\Sigma_A^+$. Since $(A_k)$ is primitive,
for a chosen $K$ there exists $N$ such that
  $A_KA_{K+1}\dots A_n$ is strictly positive for any $n\geq N$. By the
  hypothesis, there are infinitely many front connections from $\A_k$ to
$\B$, which means that there 
  exists $j>N$ such that the $k^{\text{th}}$ row of $C$, 
 has some block  $C_{ks}$  which has some  positive entry at time $j$. Writing
 $C_{ks}^{(j)} $ for that matrix, then   $(C_{ks}^{(j)} )_{ab}>0$.
  Now $A_KA_{K+1}\dots
 A_{j-1}$ is strictly positive, whence there exists a
  path $\f= (f_0 f_1\dots)$ with $f_0= e_0, \dots, f_K=
  e_K$ and $f_j^+=a, f_{j+1}^+=b$, with $f_i^+\in \B_i$ (the subalphabet
  sequence 
  for $B$) for all times $i\geq j+1$, using the fact that $B$ is
  reduced so the path $(f_0 f_1\dots f_{j+1})$  can be continued to the
  right. But $\A_i\cap\B_i=\emptyset$ so $\f\notin \Sigma_A^+$ and we are done.

Hence  $\Sigma_{A}^+$ is not open.

The final statement  is proved as before.

\end{proof}

\begin{cor}
  (of the proof)\label{c:uppertriangularopen}
 Consider a matrix sequence $N$ in  upper triangular block form as in \eqref{eq:block_form}.

 Assume the diagonal
blocks $A_k$ are reduced and primitive. Then for all $k$,
 $\Sigma_{A_k}^+$ is a closed subset of
$\Sigma_N^+$, and  $\Sigma_{A_{k}}^+$  is open iff 
the stream $\alpha_k$ has only finitely many front-connections.
In particular this holds for $A_{\wh l}$.
\ \ \qed\end{cor}

\begin{rem}\label{r:normalform}(On beginning with time one)
  
  In Theorem  \ref{t:FrobDecomp} we proved that after reordering of
  the alphabets and
  gathering we can achieve a fixed--size form beginning with time one.
  This is because,
  before the gathering, 
  the upper triangular form for the original matrix
  sequence only begins after some time $k$; 
   we then gather anomalies in the times $0,\dots, k-1$ into the first matrix $N_0$,
  with $N_i$ upper-triangular for $i\geq1$. We make two remarks about
  this:
\item{(i)} whether we begin at time $0$ or $1$ will not affect our
  principal result (the classification of $\FC$-invariant Borel measures),
  since by Remark
  \ref{r:fixedsize}, and the proof of Theorem \ref{t:invariantsubsets}, the path space for times $\geq 0$ is a
 tower
 of that for times $\geq 1$, whence the measures correspond
 bijectively;
 \item{(ii)} this indicates why there is no hope of proving a bilateral Frobenius
   theorem, because to get the upper triangular form for positive
   times, we might need to gather to $-\infty$, which doesn't make
   sense. Instead, the right- and left- sided shift apaces are treated
   separately, 
   with separate Frobenius forms and separate measures, constructed
   from right and left nonnegative eigenvector sequences respectively. The product
   measures
   give and invariant measure sequence for the two-sided \nsft,  much like the way Parry measure is constructed for an
\sft.

   For a concrete example, consider the  nonprimitive
   sequence
   \[ M_i=
\left[ \begin{matrix}
    1 &1\\
    0 &1
  \end{matrix}\right]\]
for all $i\geq 0$, and 
\[ M_i=
\left[ \begin{matrix}
    1 &1\\
    1&1
  \end{matrix}\right]\]
for all $i<0.$
Then $M$ is in upper triangular form for $i\geq 0$ and for $i<0$
separately, but this cannot be achieved for all times simultaneously.
However e.g.~ for times $\geq -4,$ we can do it, simply by  beginning with
the matrix
\[N_0=\left[ \begin{matrix}
    1 &1\\
    1&1
  \end{matrix}\right]^4.\]

\end{rem}

\section{Distinguished eigenvector sequences and a 
nonstationary Frobenius--Victory theorem}\label{s:Vict}

\subsection{Introduction to Frobenius-Victory}\label{ss:Intro2}

The key to everything in the next sections lies with a careful
development of the  $(2\times 2)$
block case:

\begin{equation}\label{eq:block_form_two}
 N=\left[ \begin{matrix}
A& C\\
0& B
\end{matrix}  \right]
\end{equation}

There the matrix sequence is $N= (N_i)_{i\geq 0}$ with alphabets
$\Cal A$ for the block $A$, $\Cal B$ for the block $B$, thus
corresponding to 
the Bratteli diagram of Fig.~\ref{F:2times2}. We need to
understand  in particular  how eigenvector
sequences with eigenvalue $1$ for $N$ are derived from those for the
subblocks  $A$
and $B$, via the notion of {distinguished} eigenvector
sequences, and how those relate to both finite and infinite
invariant measures.

To move beyond this basic case we build on three key ideas.
First, the general upper triangular block form is treated by
using the $(2\times 2)$
block case as the inductive step. Second, given a general Bratteli
diagram, by the Frobenius stream decomposition theorem, we can achieve
this upper triangular block form; by the operation of gathering, this
can be put into a fixed size square matrix form. Third, 
given two
nested Bratteli diagrams, the second derived from the first by erasing
symbols or edges, we build the canonical cover matrix and diagram,
which is then in $(2\times 2)$
block form, and apply the previous analysis. 

 We
describe the relationship to the classification of invariant Borel measures for the $(2\times 2)$
block case. All 
finite  invariant
measures for the full diagram of $N$  correspond to
eigenvector
sequences with eigenvalue $1$. These in turn come from either
eigenvector
sequences for the subblock $A$, extended trivially to the full
alphabet, or eigenvector
sequences for $B$ which are distinguished, and which
produce an eigenvector
sequence for $N$ via a limiting procedure. The {\em non--}~
distinguished sequences for $B$ also play an important  role,
as they give the infinite invariant Borel measures which are finite for that
subdiagram.

All of this extends to general diagrams and subdiagrams via the twin procedures of
the stream decomposition and the canonical cover construction.

The key to the $(2\times 2)$
block case is understanding the properties of distinguished
eigenvector sequences, which requires us to develop some
linear algebra, interesting in its own right.

\begin{figure}
$$ \xymatrix{
  \Cal A \ar[r]^{A}
\ar[rd]^C&  \Cal A  \cdots
  \\
  \Cal B  \ar[r]^{ B}
  &
 \Cal B  
 \cdots}
$$
\caption{ Bratteli diagram for the $(2\times 2)$ block case.}
\label{F:2times2}
\end{figure}
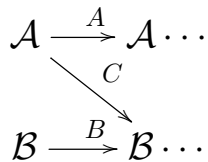

Thus the end result is to
identify the $\FC$--
invariant Borel measures, finite and infinite, which are finite on some
subdiagram, including as a special case those which are finite 
on some open subset (by $(iii)$ of Theorem
\ref{t:inv-meas-subdiag}).

\subsection{Overview of Frobenius-Victory}\label{ss:Overviewrob-Vict}

Now suppose we are given a nonstationary Bratteli diagram. 
By the Stream Decomposition, Theorem \ref{t:StreamDecomp},
 the diagram consists of  reduced primitive
streams plus pool streams; this  decomposition is eventually unique
up to reordering. As mirrored by the use  of the word ``primitive''
or ``prime''  in other parts of mathematics, the idea is to 
address each primitive stream
separately, and then put that information together to analyze
ergodic measures on the full diagram.

For a primitive diagram,
we  know by Lemma \ref{l:finite_on_all}  that 
  an invariant  measure which
is positive finite on some open set is positive finite 
on all nonempty
open subsets. Thus it is 
{positive locally finite}
(see Definition \ref{d:Radon}) and in addition
has finite total measure.

For a not necessarily primitive 
diagram, as we have seen in Theorem \ref{t:basic_thm},  see Theorem 2.9 of
 ~\cite{BezuglyiKwiatkowskiMedynetsSolomyak10}, the ergodic probability measures are in bijective
correspondence with the extreme nonnegative eigenvector sequences of
eigenvalue one.

Now there is an interesting difference here between the stationary
primitive and irreducible cases, which foreshadows the general
nonstationary case. For an irreducible nonnegative matrix $M$,
the Perron-Frobenius
theorem still 
guarantees a single  nonnegative eigenvector (up to normalization).
(This extension beyond the primitive case is due to Frobenius).
That means in our terms that 
one  has the  single  eigenvector
sequence, $\w_n=\lambda^{-n}\w$, where $\lambda $ is the eigenvalue.

Then for the general stationary case, the 
Frobenius--Victory Theorem (apparently actually already due to Frobenius) identifies the nonnegative eigenvectors
for $M$,
stating
 that they correspond to 
certain eigenvectors from the irreducible diagonal blocks.
The nonnegative eigenvectors for these blocks 
 come in two distinct types,
according to whether or not they are
 {\em distingushed} (so termed by Victory, Def.~
 \ref{d:distinguishedStationary}), those for which  the 
eigenvalue of the block is
greater than for any block to which it communicates.
The statement of the Theorem is that the distinguished eigenvectors are exactly those
which
determine an eigenvector for
the full matrix  $M$. This eigenvector is, moreover,  generated by an
algorithmic process; see Proposition 1 of \cite{Victory85}, 
see Theorem 3.7 of \cite{Schneider1986}, 
Theorem 3.3 of \cite{TamSchneider00} and the related Theorem 6 p.~77
of~\cite{Gantmacher59}. 

For the application of these ideas to stationary adic transformations,
a key insight of \cite{BezuglyiKwiatkowskiMedynetsSolomyak10} is
that while the distinguished eigenvectors will give the finite 
invariant Borel measures, the nondistingushed eigenvectors are also
important, as these will give the locally finite  infinite measures.

There is an interesting and informative subtlety which can be seen
already here in the stationary
case. Let us suppose $M$ is 
irreducible and
periodic  of period $p$; see Def.~\ref{d:period}. Then   the (unique)  Perron-Frobenius
eigenvector produces the
eigenvector sequence of eigenvalue one   $\w_n=\lambda^{-n}\w$,
defining a measure. This is however {\em not} an ergodic measure,
as this sequence is {\em not} an extreme point
for the space of eigensequences.  But as we have seen, it is these that
 correspond to the ergodic invariant Borel measures. In this 
 irreducible case, this analysis will yield 
$p$ measures,  given by the $p$ shifts of an eigenvector {\em
  sequence} of period $p$.

So the point is that the eigensequences are necessary even in the
stationary case, when studying an irreducible but not primitive matrix.

To introduce the nonirreducible case,
let us recall 
 the
 classical (stationary) situation of a single nonnegative square
 integer matrix $M$ in
upper diagonal block form, addressed  in
 ~\cite{BezuglyiKwiatkowskiMedynetsSolomyak10}. See \S
 \ref{s:stationarycase}.
If $M$ has an  irreducible, nonprimitive
diagonal block,  then the ergodic measures  will correspond to
periodic 
eigenvector sequences (as above, for the block, but also for $M$), and our 
definition of  ``distinguished'' eigenvector
sequences must include this 
case. What Bezuglyi et al do is to eliminate the periodicity by taking
a power $M^n$, where $n$ is the least comon multiple of the block periods.
The
 diagonal blocks are now primitive, 
 the
 extreme
 nonnegative eigenvector sequences for $M$ correspond to
extreme
nonnegative eigenvectors for $M^n$. So  we can consider
the distinguished eigenvectors for those diagonal blocks, and apply 
the
Frobenius--Victory Theorem to this power.

 For the study of adic transformations the
matrices have integer entries, and for the stationary case of
~\cite{BezuglyiKwiatkowskiMedynetsSolomyak10}, the distinguished eigenvectors correspond to the finite invariant
measures for the full Bratteli diagram. For our extension 
to the nonstationary
case, we need to find an appropriate  definition of distinguished
sequence. 
Our  guiding principle in this will be 
to find  a condition 
which 
distinguishes the finite and 
infinite $\FC$-invariant Borel measures. The condition
should be necessary and sufficient, and should be
``checkable'', at least in nice cases. Furthermore,  
 our definition should reduce to the usual one in the stationary
 situation; we show this in Corollary \ref{c:stationarycase}.
 For the applications we have in mind,
 this machinery should apply not just to measures which are finite on
 some open subset, but are 
 finite for any subdiagram.

The final step of the present section, then, is to prove a 
nonstationary version
of the Frobenius--Victory theorem. Making use of the upper triangular
form of  Theorem \ref{t:FrobDecomp}, this extends to  subdiagrams of general Bratteli diagrams.

\medskip

\subsection{Simplified statement of nonstationary Frobenius--Victory Theorems}\label{ss:statement of main result}

Here is the gist of what we prove in Theorems \ref{t:FrobVictGeneral}
and \ref{t:BratFrobVict},
phrased a bit more simply:
\begin{theo}\label{t:FrobVictGeneral-simple} (Frobenius--Victory Theorem for
 stream decomposition of  Bratteli diagrams)
  The extreme rays of the convex cone of nonnegative eigenvector sequences of eigenvalue one for
  $N$ are in natural bijective correspondence with the
 distinguished extreme rays  for the primitive
  components given by the stream decomposition of Theorem \ref{t:StreamDecomp}.
  That is, for $A_i$ the primitive matrix sequence corresponding to
  the $i^{\text{th}}$ stream $\A_i$, then there exists a bijection $$\iota:
\cup_{k=1}^{\wh l} \ExtD  \V_{A_k}\to \Ext\V_N .$$
 The map $\iota$ is given by the limit of an iteration, while the
 inverse map $\pi$  is given by limiting iterations combined with projections.
\end{theo}

\begin{theo}\label{t:BratFrobVict-simple}
  (Frobenius--Victory Theorem for nested diagrams) Given a  reduced
  Bratteli diagram with
  bounded alphabet size, with matrix sequence $\wh M$, and given an extreme ray $\wh\w$  of the convex cone
  of nonnegative eigenvector sequences of eigenvalue one, then there
  exists an
  eventually unique maximal primitive subdiagram with matrix sequence $M$,
  with a unique distinguished extreme ray $\w$ which
  converges to $\wh\w$ under the iteration procedure. 
That is, $\wh \w=\lim \wh M^n(\w)$.

  Suppose we are  given nested Bratteli diagrams $M\leq \wh
  M$, with canonical cover $\wt
 M$. Then there exists a maximal primitive $A\leq M$
 with eigenvector sequence $\v$ which is $M/A$ distinguished. Furthermore, $\v$ is
 $\wh M/A$ distinguished iff it is $\wt M/A$ distinguished.
  
  \end{theo}

\subsection{Distinguished sequences: general definition}

Recalling Definition \ref{d:eigenvectorsequence}, given
an alphabet sequence $\A=(\A_i)_{i\geq 0}$ and writing
$V_i\equiv\r^{\A_i}$, we consider the topological vector space
$V_\A\equiv  \Pi_{i\geq 0}
V_i $. An 
$(\A_i\times \A_{i+1})$ 
real matrix sequence 
$N=(N_i)_{i\geq 0}$ defines a continuous  linear transformation,
sending
 $(\v_0, \v_1,\v_2,\dots)$ to $(N_0\v_1,
 N_1\v_2,\dots)$.  Note that  $N:V_\A\to V_\A$ is a product of the
 maps $N_i$, that is,
 $N= \Pi_0^\infty
N_i$.

We wish to find the nonnegative eigenvector sequences of eigenvalue one, and so
the fixed points of $N$. For this it is natural to
iterate, as:
\begin{lem} \label{l:generalfixedpoint} Let $f$ be a continuous map of a topological space $X$, and suppose
  that for $x_0\in X$ and $x_n=f^n(x_0)$, we have $x_n\to x$. Then $f(x)=x$.
\end{lem}

The proof is immediate from  continuity. Of course in the special case
of a contraction mapping of a complete metric space the fixed
point exists and is unique; this is the case for  the projective metric  proof of
the Perron-Frobenius theorem \cite{Samelson56}, \cite{Birkhoff57}, \cite{Birkhoff67}. But even for the weak contractions of the
sequence situation, which also
occur in \cite{Fisher09a}, this point of view can be useful,
as noted in the next lemma.

Recalling the notation of Definition \ref{d:partialproduct}, the product of the matrices from $i$ to $n$
is denoted $N_i^n= N_iN_{i+1}\cdots N_n$, so  $N_i^n: V_{n+1}\to V_i$. 
For $m\geq1$, the $m^{\text{th}}$ iterate $N^m$ of the product map $N$
on $V_\A$ is also a product:  $N^m= \Pi_0^\infty (N^m)_i$ where 
$(N^m)_i: V_{i+m}\to V_i$ is the map  $(N^m)_i(\v_{i+m})=N^{m+i-1}_i(\v_{i+m})= N_iN_{i+1}\dots
N_{i+m-1}(\v_{i+m})$. We have:

\begin{lem}\label{l:forall i}
Let 
$N=(N_i)_{i\geq 0}$ be  a   real $(l_i\times  l_{i+1})$ matrix
sequence. 
\item{(i)}  Given
and $\v= (\v_i)_{i\geq 0} \in V_\A$,
then if 
\begin{equation}\label{eq:generallimit}
 \lim_{n\to +\infty} N_j^n \v_{n+1}
\end{equation}
exists for some $j=i+1\geq 1$  it exists for  $j=i$. 
\item{(ii)} 
If this limit exists for all $i$, then defining a vector sequence 
$\v'$  by 
$(\v')_{i}= \lim_{n\to +\infty} N_i^n \v_{n+1}$, we have $N(\v')= \v'$.

 \end{lem}
\begin{proof}

\noindent
\item{(i)}: 
If $\lim_{n\to +\infty} N_{i+1}^n \v_{n+1}$ exists, then 
\begin{equation}
  \label{eq:limit}
  \lim_{n\to +\infty} N_i^n \v_{n+1}=N_i\lim_{n\to +\infty} N_{i+1}^n \v_{n+1}
\end{equation}
by continuity of matrix multiplication.

\noindent
\item{(ii)}
From the hypothesis
$\v'= \lim_{n\to +\infty} N^n \v$ exists, so from
 Lemma \ref{l:generalfixedpoint}, $N(\v')= \v'$. Or, directly,
taking the limits in \eqref{eq:limit} gives $\v'_i =N_i
\v'_{i+1}$ and the same conclusion.

\end{proof}

\begin{rem}\label{r:limitsequence}
 Summarizing Lemma \ref{l:forall i}, we proved
 first that convergence at one time implies
convergence for all earlier times, and second, if this limit exists, it is a fixed
point.

Specializing to a nonnegative sequence $N$, we recall from Definition
\ref{d:eigenvectorsequence} the following:
$V_i= \R^{\A_i}$,
$V_\A\equiv  \Pi_{i\geq 0}
V_i$; the zero element in  $V_\A$ is $\0= (\0_0, \0_1,\dots)$, 
$\V_N^\0$ denotes  the convex cone of  nonnegative fixed points for
the map $N$, while $\V_N$ denotes $\w\in \V_N^\0$ which
 are nonnegative eigenvector sequences hence such that $\w$ is {\em never zero},
 i.e.~ each projection is nonzero: $\w_i\neq\0_i$ for all $i$.
 
As
  we proved in $(i)$  of Lemma \ref{l:closedcone}, 
  $\V_N^\0$ is  a closed convex cone.
  From part $(ii)$ of that Lemma, when
$N$ is column--reduced then $\V_N^\0= \V_N\cup\{\0\}$, where 
$\0=(\0_k)_{k\geq 0}$ is the
{ identically zero} sequence. This is because when $N$ is
column--reduced, a nonnegative  fixed 
point which is zero at one time is zero for all times. If it is not column-reduced
there may exist a nonnegative  fixed
point which is {\em partially zero}, i.e.~there exists $k\geq 0$ such
that 
  $\w_i=\0_i$ for all $i\leq k$ and $\neq \0_i$ for all $i>k$.

For the case of $N=M$ with integer entries, we recall the twofold
importance
of the never zero fixed
points. First, from Lemma \ref{l:inductive eigenvectors},
$\V_M \neq \emptyset $ iff
$\Sigma_{M}^{0,+}\neq \emptyset $; second,
from Theorem \ref{t:basic_thm}, when $M$ is reduced, these fixed
points correspond to the 
$\FC-$
invariant probability measures on $\Sigma_{M}^{0,+}$.

In Theorem
  \ref{t:basic_thm}
  the matrix sequence $M$  is required to be reduced, which means that
  partially zero sequences do not occur, and an allowed finite
  sequence of symbols defines a (by definition nonempty) cylinder
  set.

  However
 we also shall want to allow for the nonreduced
case, see Theorems \ref{t:FrobVictFirstVersion}, \ref{t:FrobVictGeneral} and
Corollary \ref{c:TwoDisting}.

In fact, for  a  partially zero sequence, zero for $m\leq k$, the formula for measures in
Theorem
  \ref{t:basic_thm}
   still makes
sense, giving zero measure on the
  $0^{\text{th}}$ component
  $\Sigma_{\M}^{0,+}$,  while giving a positive measure on components
  $\Sigma_{\M}^{m,+}$ for $m>k$.
\end{rem}

\begin{exam}\label{exam:nonreduced} 
For an example, consider alphabets $\A_i=\{0,1\}$ for all $i\geq 0$, and matrices
$M_0= \left[ \begin{matrix}
1 &0\\
0&0\\
\end{matrix}\right],
M_i= \left[ \begin{matrix}
1 &0\\
0&1\\
\end{matrix}\right] $ for $i>0$.
Define sequences $\v^a$,  $\v^b$ by  for $i=0$, 
$\v^a_i
=\left[ \begin{matrix}
1 \\
0\\
\end{matrix}\right]$, 
$\v^b_i=\left[ \begin{matrix}
0 \\
0\\
\end{matrix}\right]$,
and for $i\geq 1$, 
$\v_i^a=\left[ \begin{matrix}
1 \\
0\\
\end{matrix}\right]$,
$\v_i^b=\left[ \begin{matrix}
0\\
1\\
\end{matrix}\right]$.
Then $\v^a$,  $\v^b$ are nonnegative fixed points for $N$.
The first gives an $\FC$-invariant 
probability measure on $\Sigma_{M}^{0,+}$ (point
mass on $(.00000\dots)$);
the second, partially zero,  gives the zero measure on
$\Sigma_{M}^{0,+}$, 
but on 
$\Sigma_{M}^{1,+}$ gives  point
mass on $(.11111\dots)$.

The word $0$ is allowed (the single letters of $\A$ are always allowed
in a vertex shift) but $[.0]$ is empty hence not a cylinder set. By
adding on one identity matrix  before $M_0$ one has a similar edge
shift example, where for $e$  with $e^-=1, e^+=1$, then  $e_0=e$ is
allowed, but  $[.e_0]$ is empty.

The cone $\V_N^\0$ has only  one nonzero extreme element up to multiplication by a
constant: $\v^a_0=\left[ \begin{matrix}
1 \\
0\\
\end{matrix}\right]$. For times $\geq 1$, there are two elements.
Note that $\v_b$ is a partially zero fixed point so
$\wt \0\equiv \V_{N}^\0\setminus \V_{N}=\{\0,  c\v_b:\,  c>0\}.$

 For an example with no point masses, replace the matrices for $i\geq 1$ by
$ \left[ \begin{matrix}
2 &0\\
0&2\\
\end{matrix}\right] $; this  gives a pair of dyadic  odometers 
and hence two nonatomic invariant ergodic probability measures on $\Sigma_{M}^{1,+}$.  
\end{exam}

\

Here is our general abstract definition for subdiagrams:

\begin{defi}\label{d:dist}(Distinguished eigenvector sequence)
 Given nonnegative real
 generalized matrix sequences 
 $N\leq \wh N$, with $\A\subseteq  \wh \A$ (see Definition \ref{d:generalizedmatrices}), for each $i$ we denote by $\mathtt{e}_i: \r^{\A_i} \to
 \r^{\wh \A_{i}}$ the natural
 embedding, with $\mathtt{e}: V_\A \to
 V_{\wh \A}$ the product map $\mathtt{e} = (\mathtt{e}_0,
 \mathtt{e}_1,\dots)$.
 We define $\V_{\wh N/N
    dist}^\0$ to be the following collection of vector sequences: 
 \item{(i)} $\w\in \V_N^\0$ and
\item{(ii)} $\iota(\w)\equiv \lim_{m\to\infty} \wh N^m (\mathtt{e}(\w))$ exists.

We  write
$\V_{\wh N/N
    dist}$ for those elements $\w$  of $\V_{\wh N/N
    dist}^\0$ which are never zero: $\w_k\neq \0_k$ for all
  $k\geq 0$. We say  $\w\in \V_{\wh N/N
    dist}$ is $\wh N/N$\,--\,{\em
  distinguished}.

\end{defi}

\begin{rem}
  Note that $(ii)$ is equivalent to:

  \noindent
  $(ii')$
  $\lim_{n\to +\infty} \wh N_i^n (\mathtt{e}(\w_{n+1}))$
  exists for $i$ arbitrarily  large.
These are equivalent since  (by Lemma \ref{l:forall i}) it is then true for all $i\geq 0$.
\end{rem}

\begin{lem}\label{l:neverzero}
  If $\w$ is a  distinguished eigenvector sequence then $\wh \w\equiv \iota(\w)$ is
  nonnegative and never zero.
\end{lem}
\begin{proof}Given $N\leq \wh N$ as above, let  $\w\in  \V_N$ be
  $\wh N/N$\,--\,{
    distinguished}. Thus
 $\iota(\w)\equiv \lim_{m\to\infty} \wh N^m (\mathtt{e}(\w))$ exists.
 Now $N\leq \wh N$ so
 $ \w= N \w \leq \wh N \w$. Thus 
since  $\lim_{n\to +\infty} \wh N_i^n (\mathtt{e}(\w_{n+1}))$
exists for all $i\geq 0$,   $\0_i\neq\w_i= \lim_{n\to +\infty}  N_i^n
(\mathtt{e}(\w_{n+1}))\leq \lim_{n\to +\infty} \wh N_i^n
(\mathtt{e}(\w_{n+1}))=\wh \w_i$ for all $i$. That is, while the vector
$\w_i\in V_i$ may have some zero entries, it also has some strictly
positive entries, and these same  entries are also strictly
positive for $\wh \w_i\in \wh V_i$.
\end{proof}

\begin{defi}(General definition, short version)
  In summary, a nonnegative vector sequence $\w$ is $\wh N/N$- distinguished iff it
is a nonnegative, never zero
fixed point
for $N$, which under iteration by $\wh N$ converges, to
a sequence $\wh \w$, which is necessarily  a nonnegative and  never
zero fixed point for $\wh N$.
\end{defi}

The content of the nonstationary Frobenius--Victory
theorem (matrix form, Theorem \ref{t:FrobVictGeneral}; diagram form, \ref{t:BratFrobVict}) will be that each  nonnegative,  never
zero fixed point $\wh \w$ 
for $\wh N$ 
determines, and
is determined by, a $\wh N/N$- distinguished
sequence $\w$ for some primitive
submatrix sequence $N$.

\begin{rem}
  What is not so clear is whether or not this cone is {\em closed}.
  That would be important as then it is generated by its extreme
  points. However, as we see below, there are counterexamples.

  Let us consider what could go wrong.
  
 Let $\v^{(k)}\in \V_{\wh N/N dist}\subseteq V_\A$, 
and suppose that $\v^{(k)}\to \v \in V_\A$.
By Lemma \ref{l:closedcone}, 
$\V_N^\0$ is a closed cone, so $\v\in \V_N^\0$. Suppose that $\v\neq \0$.

For each $k$, we know that $\lim_{n\to\infty} \wh
N^n(\v^{(k)})$  exists; call it $\wh \v^{(k)}$.
By Lemma \ref{l:forall i} each $\wh \v^{(k)}$ is a fixed point for
$\wh N$. By Lemma
\ref{l:neverzero},  each $\wh \v^{(k)}$
is never zero. 

Now if $\wh \v^{(k)}$ converges, say to $\wh \v$, then
from continuity of the linear transformation $\wh N$, $\wh \v$ is also a fixed point,
and again by Lemma
\ref{l:neverzero} is never zero.

The problem is that perhaps it blows up in the limit, and indeed this
can in fact happen. In that case, $\v$ is not distinguished.

\end{rem}

\subsection{Distinguished eigenvector sequences in the $(2\times 2)$ block case}
This case  will serve to both
illustrate the main ideas and  to provide the
inductive step for the proof of the general upper triangular block case of Theorem \ref{t:FrobVictGeneral}. We are 
given a
sequence $\A_i= \{1,\dots, l_i\}$ of  nonempty alphabets
and an $(l_i\times l_{i+1})$ sequence  
$(N_i)_{i\geq 0}$ of real matrices.
 The alphabet is partitioned
into nonempty streams $\alpha, \beta$ with 
$\alpha_i=\{1,\dots, l_i^\alpha\}$ and
$\beta_i= \{ l_i^\alpha +1,\dots,  l_i\}$, where $l_i=  l_i^\alpha+ l_i^\beta$.

Let us assume that
we have the
upper triangular  block decomposition of \eqref{eq:block_form_two}:
\begin{equation}
 N=\left[ \begin{matrix}
A& C\\
0& B
\end{matrix}  \right]; {\text{\, that is, for all $i$}},
 N_i=\left[ \begin{matrix}
A& C\\
0 & B
\end{matrix}  \right]_i =\left[ \begin{matrix}
A_i & C_i \\
0_i & B_i
\end{matrix}  \right] 
\end{equation}
(so
$A_i$ is $(\alpha_i\times \alpha_{i+1}), $
$B_i$ is $(\beta_i\times \beta_{i+1}), $
$C_i$ is $(\alpha_i\times \beta_{i+1}), $ and 
$0_i$  is the $(\beta_i\times \alpha_{i+1}) $
all-zero matrix.)

We have $V_\A\equiv  \Pi_{i\geq 0}
V_i $
where $V_i\equiv\r^{\A_i}=\r^{l_i}$. We write
$(V_\alpha)_i\equiv\r^{\alpha_i}$ and
$(V_\beta)_i=\r^{\beta_i}$, and define
$V_\alpha\equiv  \Pi_{i\geq 0}
(V_\alpha)_i$, $V_\beta\equiv  \Pi_{i\geq 0}
(V_\beta)_i$, so
$V_\A= V_\alpha\times V_\beta$; we
decompose a  vector 
$\v_i\in V_i$ as
$
\v_i=\left[ \begin{matrix}
\wh \v_i\\
 \w_i
\end{matrix}  \right] $
with 
$\wh \v_i\in 
(V_\alpha)_i$ 
and $\w_i\in 
(V_\beta)_i.$
 The canonical
embeddings are    both indicated now by the single map $\mathtt{e}$,  which
sends 
$ \wh\v$ to $ \left[ \begin{matrix}
\wh \v\\
 \0
\end{matrix}  \right] $ and
$\w$ to $ \left[ \begin{matrix}
\0\\
 \w
\end{matrix}  \right] $; thus $\mathtt{e}: V_\alpha\cup
V_\beta\hookrightarrow V_\A$. Setting

$$\mathring N=\left[ \begin{matrix}
A & 0\\
0 & 0
\end{matrix}  \right],$$

\noindent
we define $\pi \v,$   for $\v\in V_\A,$ whenever the following
limit exists:
$$\pi \v=\pi(\v)=\lim_{m\to +\infty}  \mathring N^m \v.$$

\noindent
For  
$\u\in V_\alpha\cup V_\beta$,
we define,  whenever the following
limit exists: 
$$\iota\u=\iota(\u)= \lim_{m\to +\infty}  N^m (\mathtt{e}(\u)).$$

For example, with $N$ nonnegative, if  $\u\in \V_A\subseteq V_\alpha$ then
the above limit  exists, since
$\mathtt{e}(\u)= \left[ \begin{matrix}
\u\\
\0
\end{matrix}  \right]\in \V_N $ whence
$ \mathtt{e}(\u)=\iota \u$ is just the embedding.

Note that when  $N$ is nonnegative, then by Definition \ref{d:subdiagram},
$A\leq N$ and also $B\leq
N$.

Our main goal in this section is to  show that nonnegative eigenvector
sequences of eigenvalue one for nonnegative $N$  come from those from either $A$ or
$B$, in a bidirectional  algorithmic way. These latter are 
the eigenvector
sequences which are 
$N/B$\,--\,distinguished.
This will be  a version of the Frobenius--Victory theorem (Theorem
\ref{t:FrobVictFirstVersion});  see the summary at the end
of this subsection.

We have $V_\A\equiv  \Pi_{i\geq 0}
V_i $
where $V_i\equiv\r^{\A_i}=\r^{l_i}$. We write
$(V_\alpha)_i\equiv\r^{\alpha_i}$ and
$(V_\beta)_i=\r^{\beta_i}$, and define
$V_\alpha\equiv  \Pi_{i\geq 0}
(V_\alpha)_i$, $V_\beta\equiv  \Pi_{i\geq 0}
(V_\beta)_i$, so
$V_\A= V_\alpha\times V_\beta$; we
decompose a  vector 
$\v_i\in V_i$ as
$
\v_i=\left[ \begin{matrix}
\wh \v_i\\
 \w_i
\end{matrix}  \right] $
with 
$\wh \v_i\in 
(V_\alpha)_i$ 
and $\w_i\in 
(V_\beta)_i.$
 The canonical
embeddings are    both indicated now by the single map $\mathtt{e}$,  which
sends 
$ \wh\v$ to $ \left[ \begin{matrix}
\wh \v\\
 \0
\end{matrix}  \right] $ and
$\w$ to $ \left[ \begin{matrix}
\0\\
 \w
\end{matrix}  \right] $; thus $\mathtt{e}: V_\alpha\cup
V_\beta\hookrightarrow V_\A$. Setting

$$\mathring N=\left[ \begin{matrix}
A & 0\\
0 & 0
\end{matrix}  \right],$$

\noindent
we define $\pi \v,$   for $\v\in V_\A,$ whenever the following
limit exists:
$$\pi \v=\pi(\v)=\lim_{m\to +\infty}  \mathring N^m \v.$$

\noindent
For  
$\u\in V_\alpha\cup V_\beta$,
we define,  whenever the following
limit exists: 
$$\iota\u=\iota(\u)= \lim_{m\to +\infty}  N^m (\mathtt{e}(\u)).$$

For example, with $N$ nonnegative, if  $\u\in \V_A\subseteq V_\alpha$ then
the above limit  exists, since
$\mathtt{e}(\u)= \left[ \begin{matrix}
\u\\
\0
\end{matrix}  \right]\in \V_N $ whence
$ \mathtt{e}(\u)=\iota \u$ is just the embedding.

Note that when  $N$ is nonnegative, then by Definition \ref{d:subdiagram},
$A\leq N$ and also $B\leq
N$.

\begin{rem}
  Although for the theory we are developing  $N$ will always be
nonnegative,  we state propositions here in more
generality, so as to highlight just where that assumption is really
needed.
\end{rem}

\

\begin{lem}\label{l:pi_iff_iota}
For $N$  a real matrix sequence in the block form of \eqref{eq:block_form_two},
then given
$\v\in V_\A$, let us write as above
$
\v=\left[ \begin{matrix}
\wh \v\\
\w
\end{matrix}  \right]$.
Then if $N\v= \v$ (whence $B\w=\w$), 
\item{(i)}
$\pi\v $ exists iff $\iota \w$ exists.
\item{(ii)} In this case,
$\v= \pi \v + \iota \w$.
\end{lem}
\begin{proof}

\noindent
\item{(i)}
Fixing $i\geq 0$,

\begin{align*}
\left[ \begin{matrix}
A& C\\
0 & B
\end{matrix}  \right]_i^n \left[ \begin{matrix}
\wh \v\\
\w
\end{matrix}  \right]_{n+1}
 &=
\left[ \begin{matrix}
A& C\\
0 & B
\end{matrix}  \right]_i^n 
\left[ \begin{matrix}
\wh \v\\
 0
\end{matrix}  \right]_{n+1}+ \left[ \begin{matrix}
A& C\\
0 & B
\end{matrix}  \right]_i^n
 \left[ \begin{matrix}
\0\\
\w
\end{matrix}  \right] _{n+1}\\
&=
\left[ \begin{matrix}
A& 0\\
0 & 0
\end{matrix}  \right]_i^n 
\left[ \begin{matrix}
\wh \v\\
 \w
\end{matrix}  \right]_{n+1}+ \left[ \begin{matrix}
A& C\\
0 & B
\end{matrix}  \right]_i^n
 \left[ \begin{matrix}
\0\\
\w
\end{matrix}  \right] _{n+1}
.\end{align*}

Thus 
\begin{equation}\label{eq:balance}
N_i^n \v_{n+1}=\mathring N_i^n \v_{n+1}+ N_i^n (\mathtt{e} \w)_{n+1}
\end{equation}

Since 
$N\v= \v$,
this is
\begin{equation}\label{eq:balance2}
\v_i=\mathring N_i^n \v_{n+1}+ N_i^n (\mathtt{e} \w)_{n+1}
\end{equation}
so one converges iff the other does.

\noindent
\item{(ii)} This  follows by taking the limit as $n\to\infty$ of equation
\eqref{eq:balance2}.
\end{proof}

See Lemma \ref{l:exampleconvergence}  regarding a 
condition for 
convergence, in the nonnegative $(2\times 2)$ case.

\begin{lem}\label{l: pi z + iota} Let $N$  be in the block
  form of
  \eqref{eq:block_form_two}. 
For $\z\in V_\A$ and $\w\in V_\beta$, 
if $\mathring N(\z)= \z$, $B\w= \w$ and $\iota\w$
exists, then
$N(\z+\iota\w)= \z+\iota\w$ and $\pi(\z+\iota\w)= \z.$
\end{lem}
\begin{proof}
Since $\iota\w$
exists, from $(ii)$ of  Lemma \ref{l:forall i}, 
$\left[ \begin{matrix}
\u\\
 \w
\end{matrix}  \right]\equiv \iota \w
$ is a fixed point for $N$. 
From $(i), (ii)$ of  Lemma  \ref{l:pi_iff_iota}, 
$\pi\left[ \begin{matrix}
\u\\
 \w
\end{matrix}  \right]$ exists and 
$$
\left[ \begin{matrix}
\u\\
 \w
\end{matrix}  \right]= \pi \left[ \begin{matrix}
\u\\
 \w
\end{matrix}  \right]+ \iota \w= 
\pi \left[ \begin{matrix}
\u\\
 \w
\end{matrix}  \right]+ 
\left[ \begin{matrix}
\u\\
 \w
\end{matrix}  \right].
$$
Hence $\pi(\iota \w)= \0.$
Since $\mathring N(\z)= \z$, then $N(\z)= \z$ so
$N(\z+\iota\w)= \z+\iota\w$, and 
 from the definition of the map $\pi$, $\pi\z= \z$, so $\pi(\z+\iota\w)= \z.$
\end{proof}

\noindent
It follows:
\begin{cor}\label{l: pi z + iota 2}
  For $N$ as in \eqref{eq:block_form_two} and nonnegative,
if $\z\in \V_{\mathring N}^\0$, $\w\in \V_B^\0$ and $\iota\w$
exists, then
$\z+\iota\w\in \V_N^\0$ and $\pi(\z+\iota\w)= \z.$
\end{cor}\ \ \qed

\begin{lem}\label{l:pi_v exists}
   For $N$  as in \eqref{eq:block_form_two} and nonnegative, if $\v\in \V_N$, then $\pi \v$ exists.
\end{lem}
\begin{proof}
We first claim that for $i$ fixed and $n\geq i$, 
$$
 \mathring N_i^n  \v_{n+1} \leq \v_i,$$
 in the partial order (i.e.~coordinate--by--coordinate) on $\r^{l_i}$:
 writing $
\v=\left[ \begin{matrix}
\wh \v\\
\w
\end{matrix}  \right] $, since $N$ and $\v$ are  nonnegative, 
$$
 \mathring N_i^n  \v_{n+1} =\mathring N_i^n  \left[ \begin{matrix}
\wh \v\\
\w
\end{matrix}  \right]_{n+1}=
\mathring N_i^n  \left[ \begin{matrix}
\wh \v\\
\0
\end{matrix}  \right]_{n+1}\leq 
\mathring N_i^n  \left[ \begin{matrix}
\wh \v\\
\0
\end{matrix}  \right]_{n+1}+
 N_i^n  \left[ \begin{matrix}
 \0\\
\w
\end{matrix}  \right]_{n+1}
=
 N_i^n  \left[ \begin{matrix}
 \wh \v\\
\w
\end{matrix}  \right]_{n+1}= \v_i,
$$
proving the claim.

Hence the 
sequence of vectors $ \mathring N_i^n  \v_{n+1} $
lies in a compact region of the positive cone $\r^{l_i+}$, so
 there exists an increasing subsequence $n_j$ such that 
$ \mathring N_i^{n_j }\v_{n_j+1}$ converges.
Using equation \eqref{eq:balance}, 
we see that therefore 
$N_i^{n_j} (\mathtt{e} \w)_{n_j+1}$ converges.
We claim that 
we can  deduce from this the convergence of 
$N_i^{n} (\mathtt{e}\w)_{n+1}$.

\noindent
Proof:
\begin{align*}
N_i^{n}  \left[ \begin{matrix}
 \0\\
\w
\end{matrix}  \right]_{n+1} 
&=
N_i^{n-1}  N_n\left[ \begin{matrix}
 \0\\
\w
\end{matrix}  \right]_{n+1} 
=
N_i^{n-1}  
\left[ \begin{matrix}
0& C_n\\
0 & B_n
\end{matrix}  \right]
\left[ \begin{matrix}
 \0\\
\w _{n+1} 
\end{matrix}  \right]
\\
&
=
N_i^{n-1}  
\biggl (\left[ \begin{matrix}
0& 0\\
0 & B_n
\end{matrix}  \right]+ 
\left[ \begin{matrix}
0& C_n\\
0 & 0
\end{matrix}  \right]\biggr)
\left[ \begin{matrix}
 \0\\
\w _{n+1} 
\end{matrix}  \right]
\\
&
=
N_i^{n-1}  
\left[ \begin{matrix}
 \0\\
\w _{n} 
\end{matrix}  \right]+
N_i^{n-1}  
\left[ \begin{matrix}
 C_n
\w _{n+1} \\
\0
\end{matrix}  \right]
\end{align*}
so 
$N_i^{n}  \left[ \begin{matrix}
 \0\\
\w
\end{matrix}  \right]_{n+1} $
is nondecreasing in $n$, and
therefore convergence along 
the subsequence implies convergence. 
Finally since this holds for all $i$,
then again by equation \eqref{eq:balance}, $ \lim_{n\to\infty}\mathring N_i^n
\v_{n+1}$ 
exists for all $i$, converging by definition to $\pi \v$.
\end{proof}

\begin{defi}\label{d:dist1}
  Given a convex cone $C$ in a vector space $V$, we call $\v$
  an {\em extreme vector} iff it is a nonzero element of an extreme ray
  of $C$. Equivalently,  $\v\neq \0$  and  if $\v= a\u+ b\w$ for
  $\u, \w\in C$ with $\u,\w$ linearly independent then either $a$ or
  $b$ is zero.
We say
two  extreme vectors are {\em distinct} iff they are linearly independent.

Given nonvirtual nonnegative real matrix
    sequences
 $B\leq N$, then 
 $\Ext \V_{B}$  denotes the
collection of extreme vectors
 of
the closed
convex cone $\V_{B}^\0$. 
$\V_{B \text{dist}}$ denotes the points in the convex cone $\V_{B }^\0$ that are
$N/B$\,--\,distinguished, and $\ExtD
\V_{B}$  the  collection of extreme vectors of
the convex cone $\V_{B \text{dist}}^\0$.

Recall that  a matrix sequence $B$ is {\em identically zero}  if it is all--zero for all times.
\end{defi}

We note that 
  $\V_{B \text{dist}}^\0$ is indeed a  convex cone; if there are no
  distinguished points it is $\{\0\}$.
  Recall from Remark \ref{r:limitsequence} that if $B$
  is column--reduced then $\V_{B \text{dist}}^\0= \V_{B
    \text{dist}}\cup\{\0\}$,
  as it contains no partially zero
sequences.

  \begin{theo}\label{t:FrobVictFirstVersion}
    (Frobenius--Victory
    Theorem, $(2\times 2)$ block case)
    
    Assume that we are given an $(l_i\times l_{i+1})$ nonnegative real matrix
    sequence   $N=(N_i)_{i\geq 0}$ in $(2\times 2)$ upper triangular
    block form as in \eqref{eq:block_form_two}, so
$N=\left[ \begin{matrix}
A& C\\
0& B
\end{matrix}  \right] $, with 
$A$ column--reduced while $B$ is either  column--reduced
 or identically zero. Then:
 \item{(i) } Consider the map
$\pi:\V_{N}^\0\to \V_{\mathring N}^\0\cong \V_A^\0$. The inverse image of $\z\in \V_{\mathring N}^\0$ is
the collection of 
points of the form $\v=\z+\iota(\w)$ for $\w\in \V_{B
  \text{dist}}^{\0}$.

\item{(ii)}
$\V_{N}^\0= \iota (\V_A^\0)+ \iota(\V_{B \text{dist}}^\0)$,  and in
fact,
defining $\Cal I(\v,\w)= \iota\v+\iota\w$, 
$$\Cal I:
\V_A^\0\times \V_{B \text{dist}}^{\0}\to \V_N^\0
\;\;\;\;\text{is\,a\,bijection.}$$

\item{(iii)}
$$\iota:
\Ext \V_{A}\cup \ExtD\V_{B}\to 
\Ext\V_{N}\;\;\;\;\text{is\,a\,bijection.}$$
\end{theo}

\begin{proof}

\noindent
\item{(i)}:
Let  $\z\in \V_{\mathring
  N}$;  thus $\z= 
\left[ \begin{matrix}
\u\\
 \0
\end{matrix}  \right] $, with $\u\in \V_A$. This is the bijection $\mathtt{e}:\,\V_A \to  \V_{\mathring N}$. 
Now given  $\v\in \V_N^\0$, by Lemma \ref{l:pi_v exists}, $\pi \v$ exists,
thus by $(ii)$ of Lemma \ref{l:forall i}, $\mathring
  N(\pi \v)= \pi \v$, 
whence  (since $A$ is column--reduced) $\pi\v\in \V_{\mathring
  N}^\0$.

Given  $\z\in \V_{\mathring
  N}$, then
as above, $\z= 
\left[ \begin{matrix}
\u\\
 \0
\end{matrix}  \right] $ is in $\V_N$ and $\pi \z= \z$. From Corollary \ref{l: pi z + iota 2}
we know that
for any $\w\in \V_{B \text{dist}}^{\0}$, $\pi(\z+\iota(\w))=\z$ as
well. This is the general solution: if $\pi(\wt \z)= \z$, then
defining $\x=\wt \z-\z$, we have $\pi\x=\0$; writing $
\x=\left[ \begin{matrix}
\wh \x\\
\w
\end{matrix}  \right]$, then by Lemma
\ref{l:pi_iff_iota}, $\x=\iota(\w)$, whence 
$\w$ is either $\0$ or distinguished, assuming that
$B$ is column--reduced. If $B=\0$, then $\w=\0$ as is $\iota(\w)$.

\medskip

\noindent
\item{(ii)}:
If $\w\in \V_{B \text{dist}}$, then $B$ must be column--reduced (since
it cannot be zero).
 By part $(ii)$ of Lemma \ref{l:forall
   i} $\iota\w=\left[ \begin{matrix}
\u\\
 \w
\end{matrix}  \right] $  is a fixed point for $N$; it is nonzero since
$\w$ is. Thus $\iota \w \in \V_N$. If $B_n$ is zero for all $n$, then
$\V_{B \text{dist}}=\emptyset$ and for $\w\in \V_{B \text{dist}}$,
$\iota\w\in\emptyset\subseteq \V_N$, trivially. 

For $\x\in \V_A$,
$\iota\x= \mathtt{e}\x\in \V_{\mathring N}\subseteq \V_N$.
This shows that for $\x\in \V_A$ and $\w\in \V_{B \text{dist}}$,
$\iota \x +\iota \w\in \V_N$.

We claim the map $\Cal I$  is onto.
Let $\v \in \V_N^\0$. 
 By Lemma \ref{l:pi_v exists}, $\pi \v$ exists,
whence by $(ii)$ of Lemma \ref{l:forall i}, $\mathring
  N(\pi \v)= \pi \v$, 
and writing 
$
\v=\left[ \begin{matrix}
\wh \v\\
\w
\end{matrix}  \right] $
and $\pi \v =\left[ \begin{matrix}
\x\\
\0
\end{matrix}  \right] $
,
by $(i)$ of Lemma \ref{l:pi_iff_iota} $\iota \w$ exists, 
while by part $(ii)$ of Lemma \ref{l:pi_iff_iota},
$\v= \pi \v + \iota \w= \iota\x+ \iota \w$. Note that 
$A\x=\x$ and $B\w=\w$, whence, by the assumptions on $A$ and $B$, 
 $\x\in \V_A^\0$ and 
$\w\in \V_{B \text{dist}}^\0$.

Next we check that this map is injective.
If $$\iota\x + \iota \w=
\left[ \begin{matrix}
\x\\
\0
\end{matrix}  \right] +
\left[ \begin{matrix}
\u\\
\w
\end{matrix}  \right] 
=
\iota\x'+ \iota \w'=
\left[ \begin{matrix}
\x'\\
\0
\end{matrix}  \right] +
\left[ \begin{matrix}
\u'\\
\w'
\end{matrix}  \right] 
,$$
then
$\w=\w'$, so $\u=\u'$, hence $\x=\x'$.

\noindent
\item{(iii)}
Let $\x\in \Ext\V_{A}$, so in particular $\iota \x
\in \V_{N};
$
we claim that $\iota \x
\in \Ext\V_{N}.
$
Now 
$\iota \x= \left[ \begin{matrix}
\x\\
\0
\end{matrix}  \right] 
;$
if it is not  an extreme vector, there exist non-$\0$
$\v_a, \v_b$ in distinct rays of $\V_N$, and 
$p,q>0$ with $p+q=1$ such that
$\left[ \begin{matrix}
\x\\
\0
\end{matrix}  \right] 
=
p\v_a+ q\v_b.$
Then 
$\v_a= \left[ \begin{matrix}
\x_a\\
\0
\end{matrix}  \right] , \v_b= \left[ \begin{matrix}
\x_b\\
\0
\end{matrix}  \right] 
$; note that $\x_a,\x_b$ are distinct extreme vectors of
$\V_A$,
with
$
\x= p\x_a+ q\x_b$, contradicting that $\x$ is extreme.

Let $\w\in \ExtD\V_{B}$, so from $(ii)$, $\iota \w
\in \V_{N}.$ 
If $\iota \w$ is not extreme for
$\V_N$, 
there exist non-$\0$ 
$\v_a, \v_b$ in distinct rays of $\V_N$, and 
$p,q>0$ with $p+q=1$ such that
$\iota \w
=
p\v_a+ q\v_b, $ with 
$\iota \w= \left[ \begin{matrix}
\u\\
\w
\end{matrix}  \right] 
,$
$\v_a= \left[ \begin{matrix}
\s_a\\
\w_a
\end{matrix}  \right] 
,$
$\v_b= \left[ \begin{matrix}
\s_b\\
\w_b
\end{matrix}  \right] 
.$ Then  $\w=p\w_a+ q\w_b$.
Since
$\v_a, \v_b\in\V_N$, we know $\w_a, \w_b\in
\V_B$.
Furthermore from Corollary \ref{l: pi z + iota 2}, $\pi(\iota\w)=\0$ whence
$\pi(\v_a), \pi(\v_b)=\0$. Thus, as in the proof of $(ii)$ above, $\v_a= \left[ \begin{matrix}
\s_a\\
\w_a
\end{matrix}  \right] = \iota(\w_a)$ and similarly for $\w_b$, and so
$\w_a, \w_b$ are  non-$\0$, distinguished points in distinct rays.
Since
$\w$ is extreme in $\V_{B \text{dist}}^\0$, 
this gives a contradiction.
Thus $\iota\w$ is extreme.

Next we show this map is onto $\Ext\V_{N}$.
Let $\v\in \Ext\V_{N}$.
Writing $
\v=\left[ \begin{matrix}
\wh \v\\
\w
\end{matrix}  \right] $, by Lemma  \ref{l:pi_v exists} and Lemma \ref{l:pi_iff_iota},  
$\v= 
\pi \v + \iota \w=
\left[ \begin{matrix}
\x\\
\0
\end{matrix}  \right] + 
\left[ \begin{matrix}
\u\\
\w
\end{matrix}  \right]
=
\iota \x + \iota \w
$
with $\x\in \V_A^\0$, 
$\w\in \V_B^\0$. Since $\v$ is extreme, either $\iota \x$ or  $\iota
\w$ is $\0$ . 
Suppose $\w=\0$. If $\x$ is not extreme for $\V_A$, then there exist
points
$\x_a,\x_b$ in distinct rays of  $\V_A$ and $p,q>0$,
$p+q=1$ such that
$\x= p\x_a+ q\x_b$, but then 
$\v= p\cdot\iota\x_a+ q\cdot\iota\x_b$, contradicting that $\v$ is extreme.
If on the other hand $\x=\0$ and $\w$ is not 
extreme in the convex
cone $\V_{B \text{dist}} $, then there exist
  points $\w_a,\w_b$ in distinct rays of $ \V_{B \text{dist}}$, and $p,q>0$,
$p+q=1$ such that
$\w= p\w_a+ q\w_b$. But then
$\v=  p\iota \w_a+ q\iota \w_b$ so is not extreme, a contradiction.
Thus $\v$
is indeed the image of a point in $\Ext \V_{A}\cup \ExtD\V_{B}$.

Injectivity of $\Cal I$ was proved in part $(ii)$, and this proves
injectivity here, as points $(\x,\0)$ and $(\0, \w)$ in $V_A^\0\times
\V_{B \text{dist}}^{\0}$ correspond to
$\x,\w$ in $\Ext \V_{A}\cup \ExtD\V_{B}$ respectively.
\end{proof}

To summarize, the convex cone $\V_N^\0$ exhibits a product structure: any
$\v$ in  $\V_N$ is decomposed uniquely as follows:
$
\v=\left[ \begin{matrix}
\wh \v\\
\w
\end{matrix}  \right] = \left[ \begin{matrix}
\z\\
\0
\end{matrix}  \right] + \left[ \begin{matrix}
\u\\
\w
\end{matrix}  \right] $ where $\left[ \begin{matrix}
\u\\
\w
\end{matrix}  \right] = \iota(\w)$ and $\z= \pi(\v)$, giving the
bidirectional correspondence detailed in the theorem; the extremity of rays
is preserved, and the
correspondence is algorithmic, in that the maps $\pi, \iota$ give
fixed points which are  limits of iterations.

\subsection{Distinguished eigenvector sequences for the upper triangular
  case}\label{ss:generalcase}
Next we extend from the $(2\times 2)$  to the general upper triangular
block case.
Let $N=(N_i) _{i\geq 0}$ be a
column--
reduced sequence   of  $(l_i\times
l_{i+1})$ nonnegative real
matrices with bounded alphabet size.
We write $V$ for the collection of vector sequences $\v= (\v_i)_{i\geq
  0}$ with $\v_i\in V_i=\r^{l_i}$; thus $V=\Pi_{i\geq 0} V_i$. Without loss of generality, by
taking a gathering, we
can achieve that $N$ is 
in fixed--size  Frobenius
normal form for times $\geq 1$, as in the conclusion of Theorem
\ref{t:FrobDecomp}.
Furthermore, by Remarks \ref{r:normalform}, and
\ref{r:fixedsize}, without loss of generality
we can assume this holds for times $\geq 0$.
Thus from now on, 
we assume that, for all times
$i\geq 0$,
the matrices $N_i$ are
$(l\times l)$ with $(\wh l\times \wh l)$ block
form of \eqref{eq:block_form},
with block sizes given by $l= \wt l_1+\dots + \wt l_{\wh l}$ :
\begin{equation*}
 N=
\left[ \begin{matrix}
A_1 &C_{12} &\dots  &C_{1\wh l}\\
0_{21} &A_2 & &\vdots\\
\vdots & &\ddots  &\\
0_{\wh l 1} &\dots  &  &A_{\wh l}
\end{matrix}  \right]
\end{equation*}
Furthermore, for each block index $1\leq j\leq
\wh l$, for  each time $i\geq 0$, the matrix sequence $A_j=
(A_j)_i$ is either
column--
reduced and primitive or identically zero. Moreover after a further
gathering we can assume,  for each time $i\geq 0$, that  $N_i$ has zero
blocks $C_*= \0_*$ in exactly the same above-diagonal locations.
Then since $A_j\leq N$ for each $j$,  it makes sense to examine the $N/A_j$\,--\,distinguished 
eigenvector sequences. We do this with the help of the $(2\times 2)$\,--\,
block inductive step just treated.

Defining for each time 
$i\geq 0$
 and for $1< j\leq \wh l$, 
$(V_j)_i= \Pi_{\wt l_1+\dots + \wt l_{j-1}+1}^{\wt l_1+\dots + \wt
  l_j} \r$, we then suppress this time index and use the same notation
for vector sequences, 
so $V= V_1\times \cdots \times V_{\wh l}$. We write the decomposition
of 
$\v\in V$ as $\v= (\v_1,\dots, \v_{\wh l})$ where $\v_j\in V_j$.
We have the natural embeddings of $V_j $ into $V$; we  unite these
into a single map,
$\mathtt{e}: V_1\cup\dots \cup V_{\wh l}\to V$; thus  
$\mathtt{e} (\w_j)= (\0_1,\dots, \0_{j-1},\w_j, \0_{j+1},\dots, \0_{\wh l})$, with $\0_j$  the zero
vector (sequence) in $V_j$.

It will be useful  to factor this map: we write for the natural
embeddings 
$$ V_j\xhookrightarrow{\varphi_j}  V_1\times \dots\times V_j
\xhookrightarrow{\psi_j}  V_1\times \dots\times V_{\wh l}= V,$$ so 
the restriction of 
$\mathtt{e}$ to  ${V_j}$ is
$\psi_j\circ \varphi_j$ for each $j$.

We adopt the notation
$V_{1,k}\equiv V_1\times \dots\times V_k$, 
 $\wh V_{1,k}\equiv \psi_k( V_{1,k})=V_1\times \dots\times V_k\times
 \{\0\}\times\dots \times\{\0\}\subseteq V$,
 so the above diagram becomes
 
$$
\xymatrix{
 V_k  \ar@{^{(}->}[r]^{\varphi_k} \ar@/_2pc/[rr]_{\mathtt{e}}&V_{1,k} \ar[r]^<<<<{\psi_k}&\wh V_{1,k}
 \subseteq V.
 }
$$

For $2\leq k\leq \wh l$, we define $N_k$ and $\mathring N_k$,  both 
$(\wt l_1+\dots+ \wt l_k)\times (\wt l_1+\dots+ \wt l_k)$, by
$$
N_k=
\left[ \begin{matrix}
A_1 &C_{12} &\dots  &C_{1 k}\\
0_{21} &A_2 & &\vdots\\
\vdots & &\ddots  &\\
0_{k1} &\dots  &  &A_{k}
\end{matrix}  \right], 
\mathring N_k=
\left[ \begin{matrix}
A_1 &C_{12} &\dots  &C_*&0_{1 k}\\
0_{21} &A_2 & & &\\
\vdots & &\ddots  & &\vdots\\
&  &  &A_{k-1} & 0_*\\
0_{k1} &\dots  &  & 0_*&0_{k k}\\
\end{matrix}  \right].
$$
We extend these to 
$(l\times l) $ matrices $\wh N_k,  {\mathring {\wh N_k}}$ by
filling   in with zero blocks; note that $ {\mathring {\wh N_k}} =\wh N_{k-1}$:
$$
\wh N_k =
\left[ \begin{matrix}
A_1 &C_{12} &\dots  &C_{1 k}&0_*&\dots & 0_*\\
0_{21} &A_2 & &C_{2k} &0_*& & \vdots\\
\vdots & &\ddots  & & & & \\
0_*& 0_* & \dots &A_{k}&0_*& & \\
0_* & 0_*& \dots & 0_*&0_*&\dots & \\
\vdots &  &  & & &\ddots & \vdots\\
0_{\wh l 1} &\dots  &  & & &\dots & 0_{\wh l\wh l}
\end{matrix}  \right], 
 {\mathring {\wh N_k}} =
\left[ \begin{matrix}
A_1 &C_{12} &\dots  &C_*&0_*&\dots & 0_*\\
0_{21} &A_2 & &C_* &0_*& & \vdots\\
\vdots & &\ddots  & & & & \\
0_*& 0_* & \dots &A_{k-1}&0_*& & \\
0_* & 0_*& \dots & 0_*&0_*&\dots & \\
\vdots &  &  & & &\ddots & \vdots\\
0_{\wh l 1} &\dots  &  & & &\dots & 0_{\wh l\wh l}
\end{matrix}  \right]
.$$

Since $A_k\leq N$, then
from Definition \ref{d:dist}, $\V_{A_k dist}\equiv \V_{N/A_k dist}$ denotes 
those $\w\in\V_{A_k}$ such that
$\lim_{n\to\infty} N^n \mathtt{e}\w$ exists. Note that $\w$ is $N/A_k$--distinguished iff it is $N_k/A_k$--distinguished.

\noindent
We define maps
$\iota :  \V_{A_1 dist}\cup\dots \cup\V_{{A_{\wh l} dist}}\to \V_N$
and
$\iota_k :  \V_{A_k dist}\to \V_{N_k}$ by 
$$\iota \w=\iota(\w)=\lim_{n\to\infty} N^n \mathtt{e}\w \;\;\;\text{and}$$ 

$$\iota_k\w=\lim_{n\to\infty} (N_k)^n \varphi_k\w.$$

\noindent
For  $\v\in \V_{N_k}$, we define  if the limit exists
$$\pi_k\v\equiv \lim_{n\to +\infty} ( {\mathring N_k} )^n \v$$

\noindent
Similarly, for $k\geq 2$, given $\v\in \V_{\wh N_k}$, we define, if the limit exists, 
$$\wh\pi_k\v\equiv \lim_{n\to +\infty} ( {\mathring {\wh N_k}} )^n
\v\in \V_ {\mathring {\wh N_k}}^\0$$

The main result of this section is:

\begin{theo}(a nonstationary Frobenius--Victory theorem for matrices)
  \label{t:FrobVictGeneral}
Assume that we are given a nonnegative column--reduced real matrix sequence $N=(N_i)_{i\geq 0}$ in fixed--size Frobenius
form of \eqref{eq:block_form} and with notations as above. Then:
\item{(i) } For each $2\leq k\leq \wh l$,
$\wh \pi_k$ defines a surjective map from $
\V_{\wh N_k}^\0$ to 
$\V_{\wh N_{k-1}}^\0$;
\item{(ii)}
$\V_{N}^\0= \iota(\V_{A_1 dist}^\0)+\dots+
\iota(\V_{{A_{\wh l}} dist}^\0)
;$ and in fact, defining $$\Cal I(\w_1,\dots, \w_{\wh l})= \iota\w_1+\dots+\iota
\w_{\wh l},$$ then 
$$\Cal I: \Pi_{k=1}^{\wh l}\V_{A_k dist}^\0
  \to \V_N^\0 \;\;\;\;\text{is\,a\,bijection.}$$
\item{(iii)}The map
$$\iota:
\cup_{k=1}^{\wh l} \ExtD  \V_{A_k}\to \Ext\V_N \;\;\;\;\text{is\,a\,bijection.}$$
\end{theo}

We note that the initial streams  by
construction  come before the pool collections in the  linear  order
$\preccurlyeq $, and that there is at least one, $\alpha_1$; they can be listed as
$\alpha_1,\dots,\alpha_s$, 
and
for the corresponding matrix sequences $A_k$ 
we have 
$\V_{A_k}= \V_{A_k dist}$.
So 
in part $(ii)$, the statement 
 can also be written as:
$$\Cal I: \V_{A_1 }^\0\times\dots\times \V_{A_s }^\0\times
\V_{A_{(s+1)} dist}^\0\times\dots\times \V_{A_{\wh l} dist}^\0\to
\V_N^\0\;\;\;\;\text\em {is\,a\,bijection.}$$

\begin{proof}

For the rest of this  subsection, we use the notation
$\V_k\equiv \V_{A_k}$,
$\V_{1,k}\equiv \V_{N_k}$,
$\wh\V_{1,k}\equiv \V_{\wh N_k}$, $ {\mathring\V}_{{1,k}}\equiv
\V_{\mathring N_k}$ and $ {\mathring{\wh \V}}_{{1,k}}\equiv
\V_ {\mathring {\wh N_k}} =\wh \V_{\wh N_{k-1}}=\wh \V_{1,k-1}$.

We have the conjugacies
indicated in the first two commutative diagrams;  taking limits giving the
third diagram:

$$
\xymatrix
{
(\wh V_{1,k})_{i+1}\ar[r]^{({\mathring {\wh N}}_k)_i} &(\wh V_{1,k})_{i}\\
(V_{1,k})_{i+1}\ar[u]_{(\psi_k)_i}  \ar[r]^{{(\mathring N_k)_i}} &( V_{1,k})_{i}\ar[u]_{(\psi_k)_{i+1}}\\
}
\hspace{1in}
\xymatrix
{
\wh V_{1,k} \ar[r]^{{\mathring {\wh N}}_k} &\wh V_{1,k}\\
V_{1,k}\ar[u]^{\psi_k} \ar[r]^{{\mathring N_k}}  &V_{1,k}\ar[u]_{\psi_k}\\
}
\hspace{1in}
\xymatrix
{
\wh\V_{1,k}^\0\ar[r]^{\wh\pi_k} & \wh\V_{1,k-1}^\0\\
\V_{1,k}^\0\ar[r]^{\pi_k}\ar[u]^{\psi_k}& {\mathring\V}_{{1,k}}^\0 \ar[u]_{\psi_k}\\
}
$$

To prove $(i)$, 
we write $N_k$ in block form as 
$
N_k= 
\left[ \begin{matrix}
N_{k-1}&  C\\
 0 & A_k
\end{matrix}  \right]$
which equals
$
\left[ \begin{matrix}
A& C\\
0 & B
\end{matrix}  \right]
$ in the notation of Theorem \ref{t:FrobVictFirstVersion}.
By assumption $N$ is column--reduced, whence so is $A=N_{k-1}$.
Also,
$B=A_k$ is either a primitive
or identically zero (Def. \ref{d:dist1}) matrix sequence, fitting the
hypothesis of Theorem \ref{t:FrobVictFirstVersion}.  Applying  $(i)$
of that theorem, we have that the map
$\pi_{k}: \V_{1,k}^\0 \to
\mathring\V_{1,k}^\0= \V_{1,k-1}^\0$
is surjective, and hence so is 
$\wh\pi_{k}$, by conjugation with the bijection
$\psi_k$ as in the diagram.

To prove $(ii)$, we note  first that from the definitions, for each $k$ with $1\leq
k\leq \wh l$ we have these commutative diagrams:
$$\xymatrix
{
&\wh V_{1,k}\ar[r]^{\wh{N}_k} & \wh V_{1,k}\\
V_k\ar[ru]^{{\mathtt{e}}}\ar[r]_{{\varphi_k}} &V_{1,k}\ar[r]^{{N_k}}
\ar[u]_
{\psi_k}&  V_{1,k}\ar[u]_{\psi_k}\\
}
\hspace{1in}
\xymatrix
{
&\wh\V_{1,k}\subseteq \V_N\\
\V_{A_k dist}\ar[ru]^{{\iota}}\ar[r]_{{\iota_k}} & \V_{1, k} \ar[u]_{\psi_k}\\
}
$$
The proof will be by induction on the number $\wh l$ of blocks on
the diagonal. 
Supposing the statement holds for $k-1$ blocks, we show it for $k$. Writing
$\Cal I_j$ for the map at level $j$, the induction hypothesis provides a bijection
$$\Cal I_{k-1}: \V_{A_1 dist}^\0\times\dots\times \V_{A_{k-1} dist}^\0\to
\V_{1,{k-1}}^\0$$
Using the 
block decomposition written above,
we have, from $(ii)$ of Theorem \ref{t:FrobVictFirstVersion}, a
bijection 
$\Cal I:\V_{1,k-1}^\0\times \V_{A_k dist}^\0\to
\V_{1,k}^\0$. Writing $\text{id}$ for the identity map on
$\V_{A_k dist}$,
we then have the composition of bijections
$$\Cal I_{k}= \Cal I\circ (\Cal I_{k-1}, \text{id}): 
(\V_{A_1 dist}^\0\times\dots\times \V_{A_{k-1} dist}^\0)\times \V_{A_k
  dist}^\0\to \V_{1,k}=
\V_{N_k}^\0.$$

This proves $(ii)$. The proof of $(iii)$ follows by induction in just the same way.
\end{proof}

To summarize,  the 
nonnegative eigenvector sequences with eigenvalue one  for $N$ have been classified in
terms of the much simpler primitive case  (for which, see \cite{Fisher09a}), as the extreme points for $N$
correspond bijectively to the extreme points for the 
distinguished eigenvector sequences of the primitive blocks
$A_k$. Specifically, such an eigenvector sequence
$\v=(\v_i)_{i\geq 0}$ has a
unique decomposition: there are unique $\w_k\in \V_{A_k
  dist}^\0=\V_{A_k dist}\cup{\0}$ (see Remark  \ref{r:limitsequence}) at
least one nonzero, such that 
$$\v= \iota(\w_1)+\iota(\w_2)+\dots +\iota(\w_{\wh l})=
\Cal I(\w_1, \dots, \w_{\wh l})=
 \left[ \begin{matrix}
\v_1 \\
\v_2\\
\v_3\\
\vdots\\
\v_{\wh l}
\end{matrix}  \right]= 
\left[ \begin{matrix}
\w_1 \\
\0_2\\
\0_3\\
\vdots\\
\0_{\wh l}
\end{matrix}  \right]+ 
\left[ \begin{matrix}
\u_{1,2} \\
\w_2\\
\0_3\\
\vdots\\
\0_{\wh l}
\end{matrix}  \right]+\, \cdots\, +
\left[ \begin{matrix}
\u_{1,\wh l}\\
\u_{2,\wh l}\\
\vdots\\
\u_{\wh l-1, \wh l} \\
\w_{\wh l}
\end{matrix}  \right]
$$

\noindent
 The  maps $\wh\pi_{k}$ project  along the filtration of cones
$
\V_N^\0=\wh\V_{1,\wh l}^\0\supseteq \cdots\supseteq 
\wh\V_{1,2}^\0\supseteq
\wh\V_{1,1}^\0
$, as follows:
\[
\xymatrix
{
&\wh\V_{1,\wh l}^\0 
\ar[r]^{\wh\pi_{\wh l}} & \wh\V_{1,\wh l-1}^\0 
\ar[r]^{\wh\pi_{{\wh l}-1}}&\cdots \ar[r]^{\wh\pi_{{3}}}&
\wh\V_{1,2}^\0 \ar[r]^{\wh\pi_2} &
\wh\V_{1,1}^\0 
}
\]
\noindent
with
 \begin{equation}
  \label{eq:ergodicdecomp}
 \wh\pi_{{\wh l}}\v= \iota(\w_1)+\iota(\w_2)+\dots \iota(\w_{\wh l-1}).
\end{equation}
So iterations determine successively $\w_{\wh l}, \w_{\wh l-1},\dots, \w_1$, 
with $\w_{\wh l}$ the last coordinate of $\v-\wh\pi_{\wh l}\v$, 
$\w_{\wh l-1}$ the last coordinate of $(\v-\wh\pi_{\wh
  l-1}\v)-\wh\pi_{\wh l}\v$,  
and so on. 
Some of the $\w_k$ may be  (identically) zero: this is
always the case for a
zero block at index $k$ on the diagonal, and 
furthermore if $\v$ is an extreme point then all but one of the $\w_k$
is zero.

We remark  that  \eqref{eq:ergodicdecomp} corresponds to
the ergodic decomposition of a central measure.

\subsection{Distinguished eigenvector sequences, general case  (subdiagrams)}

Now we are ready to consider the general case, with nested sequences
of nonnegative integer matrices $M\leq \wh M$ and corresponding 
nested Bratteli diagrams $\mathfrak B_{\A, \E}\leq \mathfrak B_{\wh \A,\wh\E}
.$  
 
As noted in $(ii)$ of Remark \ref{r:gen_matrices},
without loss of generality we can take $\A= \wh \A$.

We study this via the  the {\em canonical cover matrix}, see 
  Theorem \ref{t:towerring}
setting
$\wt M=
\left[ \begin{matrix}
\wh M& C \\
0 & M
\end{matrix}  \right] 
,$
with the sequence $C$ satisfying 
$C_i=\wh M_i- M_i$ for each $i\geq 0$.
Thus we shall need to compare two notions of distinguished: with respect to $\wh M$ and
$\wt M$; see Corollary \ref{c:TwoDisting}. 

For this, we begin with two lemmas. To show that the linear algebra in
these proofs remains valid in more
generality, the statements here are  for real entries. Below we
specialize to our main case of Bratteli diagrams, i.e.~for integer
entries, where the ``canonical cover'' is really a cover.

Suppose we are given an $(l_i\times l_{i+1})$ real nonnegative matrix
sequence $N=(N_i)_{i\geq 0}$ with upper triangular  block
decomposition 
as in \eqref{eq:block_form_two}. Thus, 
\begin{equation}
  \label{eq:twoblock}
N=\left[ \begin{matrix}
A& C\\
0 & B
\end{matrix}  \right]  
\end{equation} with subalphabet sequences denoted $\alpha, \beta$.
Recalling the partial product notation $N_i^n= N_i\cdots N_n$ of Definition
\ref{d:partialproduct}, we have:

\begin{lem}\label{l:conditions}
 For $N$ real nonnegative and upper triangular as in \eqref{eq:twoblock},  defining, for $i\geq 0$ and $n\geq i$,
 blocks 
$ \wh C_i^n$ so as to satisfy the equation
\begin{equation}
  \label{eq:C-hat}
 \left[ \begin{matrix}
A& C \\
0 & B
\end{matrix}  \right] _i^n=
\left[ \begin{matrix}
A_i^n & \wh C_i^n\\
0 & B_i^n
\end{matrix}  \right] , 
\end{equation}
\item{(i)}
then equivalently, 
\begin{equation}
  \label{eq:C-hat2}
 \wh C_i^n= \sum_{k=i}^n A_i^{k-1} C_k B_{k+1}^n \text{ where }
 A_j^{j-1}= I_{(\alpha_j\times \alpha_j)}, B_j^{j-1}= I_{(\beta_j\times \beta_j)}
\end{equation}
\item{(ii)}
For a vector sequence $\w\in V_\beta$, 
$\lim_{n\to\infty} N_i^n \left[ \begin{matrix}
\0\\
\w_{n+1}
\end{matrix}  \right]  $ exists iff both
$\lim_{n\to\infty} \wh C_i^n \w_{n+1}$ and  $\lim_{n\to\infty}  B_i^n \w_{n+1}$ 
exist.
\item{(iii)}
 In particular, 
$\w$ is $N/B$-distinguished 
iff 
$\w\in \V_B$ and for all $i\geq 0$,
$\lim_{n\to\infty} \wh C_i^n \w_{n+1}$
exists.
\end{lem} 
\begin{proof}
The definition of $ \wh C_i^n$ in \eqref{eq:C-hat2} is equivalent to:
$\wh C_i^i= C_i$, and, for $n\geq i$, 
\begin{equation}
  \label{eq:C-hat3}
 \wh C_i^{n+1}= A_i^n C_{n+1} + \wh C_i^n B_{n+1}
\end{equation}
We prove \eqref{eq:C-hat2}  from this: 
we verify that $\wh C_i^i=
C_i$ in \eqref{eq:C-hat2} as well, then assume as an induction
hypothesis that \eqref{eq:C-hat2} holds  for  $\wh C_i^n$, from which,
by way of  \eqref{eq:C-hat3}, this follows
for  $(n+1)$.

Now from \eqref{eq:C-hat},
$$
N_i^n 
\left[ \begin{matrix}
0\\
\w_{n+1}
\end{matrix}  \right] =
\left[ \begin{matrix}
 \wh C_i^n 
\w_{n+1}\\
 B_i^n 
\w_{n+1}
\end{matrix}  \right] 
,$$
proving the first claim.

From Definition \ref{d:eigenvectorsequence},
$\w\in\V_B$ means $\w$ is a nonnegative and never zero fixed point for
$B$. So  $ B_i^n 
\w_{n+1}=\w_i$, and thus
$\lim_{n\to\infty} N_i^n \left[ \begin{matrix}
\0\\
\w_{n+1}
\end{matrix}  \right]  $ exists iff 
$\lim_{n\to\infty} \wh C_i^n \w_{n+1}$ exists.
Therefore $\w$ is  $N/B$-distinguished iff $\w\in\V_B$ and this converges  for all $i\geq 0$.
\end{proof}

In an important   special case this condition
simplifies considerably: 
\begin{lem}\label{l:A-B}
For $N$ as in the previous lemma, let now $B_i\leq A_i$ be the same size, $(\wt l_i\times \wt l_{i+1})$ for $i\geq 0$.
Let $C_i= A_i-B_i$ for all $i$. Let $\wh C_i^n$ be defined by 
\eqref{eq:C-hat2}
above.
Then:
\item{(i)}
$\wh C_i^n= A_i^n-B_i^n.$
  \item{(ii)}
A vector sequence $\w$ is $N/B$-distinguished 
iff  $\w$ is $A/B$-distinguished.

For the next three statements we assume that  $D\leq B$, with both of the same size. 
 \item{(iii)} 
A vector sequence $\w$ is $N/D$-distinguished 
iff 
$\w$ is both $B/D$ and
$A/D$-distinguished.
\item{(iv)} If 
$\w$ is both $N/D$ and
$B/D$-distinguished, it  is $A/D$-distinguished.
\item{(v)} If 
$\w$ is 
$B/D$-distinguished, then it is $A/D$-distinguished  iff
it is $N/D$-distinguished.
\end{lem}

\begin{proof}
Assuming the statement is true for $n$, then 
from   \eqref{eq:C-hat3} we have that 
$$\wh C_i^{n+1}= A_i^nC_{n+1}+ \wh C_i^n B_{n+1}=
A_i^n(A_{n+1}-B_{n+1})+ (A_i^n-B_i^n) B_{n+1}=
A_i^{n+1} -B_i^{n+1},$$
so we are done by induction with $(i)$.
Then applying part $(iii)$ of Lemma \ref{l:conditions} proves $(ii)$.

To prove $(iii)$, we are to show that for $\w\in \V_D$, then
convergence of $\lim_{n\to\infty} N_i^n \w_{n+1}$ is equivalent to 
convergence of $\lim_{n\to\infty} B_i^{n}\w_{n+1}$ and $\lim_{n\to\infty} A_i^{n} \w_{n+1}$, for all 
$i\geq 0$. 

But  by
part $(ii)$ of  Lemma \ref{l:conditions}, convergence of 
$\lim_{n\to\infty} N_i^n \w_{n+1}$  is
equivalent to 
that  of $\lim_{n\to\infty} \wh C_i^n \w_{n+1}$ and $\lim_{n\to\infty} B_i^{n}\w_{n+1}$. By
$(i)$ above, convergence of $\lim_{n\to\infty} \wh C_i^n \w_{n+1}$ is
equivalent to that of 
$\lim_{n\to\infty} (A_i^{n} -B_i^{n})\w_{n+1}$. Thus
convergence of 
$\lim_{n\to\infty} N_i^n \w_{n+1}$  is
equivalent to 
that of $\lim_{n\to\infty} (A_i^{n} -B_i^{n})\w_{n+1}$ and
$\lim_{n\to\infty} B_i^{n}\w_{n+1}$,
and hence to 
$\w$ being both $B/D$ and  $A/D$-distinguished.

For $(iv)$,
we are to show that for $\w\in \V_D$, then
convergence of $\lim_{n\to\infty} A_i^{n} \w_{n+1}$
is equivalent to 
convergence of $\lim_{n\to\infty} N_i^n \w_{n+1}$ and $\lim_{n\to\infty} B_i^{n}\w_{n+1}$  for all $i\geq 0$. 
Now convergence of 
$\lim_{n\to\infty} N_i^n \w_{n+1}$ is
equivalent to 
that of $\lim_{n\to\infty} (A_i^{n} -B_i^{n})\w_{n+1}$ and
$\lim_{n\to\infty} B_i^{n}\w_{n+1}$, which in turn implies convergence of 
$\lim_{n\to\infty} A_i^{n} \w_{n+1}$.
Part $(v)$ follows logically from $(iii)$ and $(iv)$.
\end{proof}

\noindent

\begin{cor}\label{c:TwoDisting} (We number these so as to match the
  corresponding parts of the previous lemma).
  Given nested nonnegative real matrix
  sequences 
$A\leq N\leq \wh N$, with $\wt N$ the canonical cover matrix (defined
by the formula of  Theorem \ref{t:towerring}), then
\item {(ii)}a vector sequence $\w$ is $\wh N/N$\,--\,{
  distinguished} 
iff it is $\wt N/N$\,--\,{
  distinguished}.
\item{(iii)} 
A vector sequence $\w$ is $\wt N/A$\,--\,{
  distinguished} iff it is both $ N/A$
and $\wh  N/A$\,--\,{
  distinguished}.
  
\item{(iv)} If $\w$ is both $\wt N/A$\,--\,{
  distinguished} and 
$N/A$\,--\,{ distinguished}, it is $\wh N/A$\,--\,{
  distinguished}.
\item{(v)}
   If 
$\w$ is 
$N/A$-distinguished, then it is $\wh N/A$-distinguished  iff
it is $\wt N/A$-distinguished.
  
\end{cor}
\begin{proof}
To prove  $(ii)$: As noted in $(ii)$ of Remark \ref{r:gen_matrices}, we can without loss of generality assume $\wh
  \A=\A$,
so $\wh N$ and $N$ have the same size; we then directly apply part 
$(ii)$ of Lemma \ref{l:A-B}. To prove  parts $(iii),(iv),(v)$ we apply 
the corresponding parts of the
Lemma. 

\end{proof}

\begin{rem}
 Part $(v)$  of the Corollary is used in the proof of $(iii)$ of Theorem \ref{t:BratFrobVict}.
Parts $(ii)$ and $(iii)$ are used in the proof of Theorem
  \ref{t:inv-meas-subdiag} $(iii)$ below.
\end{rem}

\medbreak

Next we see that the definition of distinguished has 
a weaker formulation, which  will prove useful
in applications.

\begin{prop}\label{p:distsubsequence}
  For   $N\leq \wh N$, as above with $\wh
  \A=\A$, then a vector sequence $\w$ is $\wh N/N$ distinguished iff 
 $\w\in \V_N$ and $\liminf_{n} ||\wh N_i^n \w_{n+1}|| <\infty$, for
  infinitely many $i$. 
\end{prop}
\begin{proof}
For the canonical cover matrix
$$\wt N=
\left[ \begin{matrix}
\wh N& C\\
0 & N
\end{matrix}  \right]
$$
with $C= \wh N- N$, and $\w\in \V_N$, then as in the proof of Lemma \ref{l:pi_v
  exists}, 

  $$
\wt N_i^{n}  \left[ \begin{matrix}
 \0\\
\w
\end{matrix}  \right]_{n+1} 
=
\wt N_i^{n-1}  
\left[ \begin{matrix}
 \0\\
\w _{n} 
\end{matrix}  \right]+
\wt N_i^{n-1}  
\left[ \begin{matrix}
 C_n
\w _{n+1} \\
\0
\end{matrix}  \right]
$$
so 
$$\wt N_i^{n-1}  
\left[ \begin{matrix}
 \0\\
\w _{n} 
\end{matrix}  \right]
\leq \wt N_i^{n}  \left[ \begin{matrix}
 \0\\
\w
\end{matrix}  \right]_{n+1} .$$
Since by hypothesis the lim inf of the norms is bounded, we have that by  compactness of the
intersection of the positive cone with a ball, there exists
a subsequence $n_j$ such that $
\wt N_i^{n_j}  \left[ \begin{matrix}
 \0\\
\w
\end{matrix}  \right]_{n_j+1} $ converges. We have just seen that
$
\wt N_i^{n}  \left[ \begin{matrix}
\0\\
\w
\end{matrix}  \right]_{n+1} $
is nondecreasing in $n$, giving 
convergence for the sequence itself.
This is true for each $i\geq
0$, so $\w$ is $\wt N$\,--\, distinguished. Finally from Corollary
\ref{c:TwoDisting}, 
$\w$ is $\wh N/N$ distinguished as well.
\end{proof}

As a consequence we have:
\begin{cor}\label{c:gatheringdispersal}
 The notion of being $\wh N/N$\,--\,distinguished 
is preserved under the taking of gatherings and dispersals. Moreover
this notion only depends on the tail of the matrix sequence.
  \end{cor}
\begin{proof}
For $N= (N_i)_{i\geq 0}$, we recall that given     a
   subsequence $0=n_0<n_1<\dots$, the gathered 
 matrix  sequence $K$ is  
$K_i=
 N_{n_i}^{n_{i+1}-1}$. If $\w \in \V_N$, then $\wt \w_i=\w_{n_i}\in \V_K$,
since 
$K_i \wt \w_{i+1}= \wt \w_i$. The reverse holds, i.e.~a sequence in
$\V_K$ can be uniquely interpolated to one in $\V_N$, so this
correspondence is bijective.

 Now given  $N\leq \wh N$,  we claim that then for  $\wh K$, $K$ the
 gathered 
sequences, 
if $\w\in \V_N$, and if the
  corresponding sequence $\wt \w$ is 
  $\wh K/K$\,--\,distinguished, then $\w$ is  $\wh N/N$\,--\,distinguished.
But in that case
$\wh K_i^m \wt \w_{m+1}= \wh N_{n_i}^{n_{m+1}-1} \w_{n_{m+1}}$
converges, and hence by the Proposition 
$\lim_{j\to\infty}\wh N_{i}^{j} \w_{j+1}$ exists, for each $i\geq 0$.

It follows that the notion of distinguished is unchanged by erasing 
a finite initial part of the matrix sequence.
\end{proof}

\begin{rem}\label{r:fixedsize}
  In particular, passage from a fixed--size Frobenius block form
  $(\wh N_i)_{i\geq 0}$  to $(\wh N_i)_{i\geq 1}$, where $\wh N_0$ may not have the
  fixed size, does not affect the notion of distinguished. See  also Remark \ref{r:normalform}.
\end{rem}

Now we return to nested Bratteli diagrams and  hence nonnegative
integer entries for our matrix sequences $M, \wh M, \wt M$.

\begin{theo}\label{t:BratFrobVict}
  (Frobenius--Victory Theorem for nested Bratteli diagrams) Let $\mathfrak B_{\wh \A,\wh\E,\wh M}$ be a  reduced Bratteli diagram with
  bounded alphabet size.
\item{(i)}
  Let
  $\wh \w\in \Ext\V_{\wh M}. $ Then 
there exists  an eventually unique maximal primitive $M\leq \wh M$,
such that there exists $\w\in \V_{M}$ which
is $\wh M/M$\,--\,distinguished, with $\wh \w=\iota(\w)$.
\item{(ii)} Let $\wh \w\in \V_{\wh M}. $ Then there exist unique
  primitive submatrices $\A_k$ with disjoint streams and  a unique
  decomposition of $\wh \w$ as a sum $\wh \w=\iota\w_1+\dots+\iota
\w_{\wh l}$ where $\w_k\in \V_{A_k dist}$.

  \item{(iii)} 
  Suppose we are  given nested Bratteli diagrams, that is, $M\leq \wh
  M$, and
let $\wt
 M$ denote the canonical cover matrix.
  Suppose
 $\w\in \V_{M}$. Then there exists a maximal primitive $A\leq M$
 and $\v\in \V_A$ such that $\v $ is $M/A$ distinguished. Furthermore, $\v$ is
 $\wh M/A$ distinguished iff it is $\wt M/A$ distinguished.
  
  \end{theo}
  \begin{proof}
  \item{(i)}
    By the  nonstationary Frobenius Decompostion Theorem
\ref{t:FrobDecomp}, there exists a reordering of the alphabets and a
gathering such that the diagram has fixed-size Frobenius normal form
for times $\geq 1$. From Remark \ref{r:fixedsize} the change of
starting time from $0$
to $1$  does not affect the property of being distinguished.

Now, without loss of generality
we assume $\wh M$ is in this form.
Applying Theorem \ref{t:FrobVictGeneral}, since $\wh \w\in
\Ext\V_{\wh M}$,
there exists a unique $k$ and $\w\in \ExtD  \V_{A_k}$ such that
$\iota(\w)=\wh \w$. $M\equiv A_k\leq \wh M$ is the unique maximal primitive
submatrix sequence.

\item{(ii)} Next we apply part $(ii)$ of
  Theorem \ref{t:FrobVictGeneral}. 

\item{(iii)} We apply part $(v)$ of Corollary
  \ref{c:TwoDisting}. (Note that this works because, as in the Corollary, we can assume the alphabets
  for $M$, $A$ have been extended to that of $\wh M$.) Then by $(i)$ applied to $M$,  
  we are done.
\end{proof}

\medbreak

\medskip

\section{The classification of invariant Borel measures}\label{s:Measures}

Here we come to our main goal: to identify the
ergodic invariant
measures for $\FC$ (or equivalently 
from $(ii), (iii)$ of Proposition \ref{p:conserg} for adic
transformations) under the assumption that
the measure is 
finite on 
some subdiagram. The notion of distinguished  eigenvector sequence 
(Def.~ \ref{d:dist})
provides a
  necessary and sufficient condition for the measure to be finite. We
  shall begin with the stronger assumption that the measure is finite
  on some nonempty open subset; this will  lead us to the general
  case. After treating
  adic transformations, we carry the analysis over to adic towers.

We harvest the work of the previous sections, first in
\S \ref{t:inv-meas} assuming the upper
triangular form of  \S \ref{ss:generalcase} (i.e.~ fixed--size  Frobenius
normal form).
Our approach is inspired by~\cite{BezuglyiKwiatkowskiMedynetsSolomyak10}; in particular, 
the proof of $(i)$ follows the line of reasoning  for the stationary case in  Lemma 4.2 in
that paper. Then, in \S \ref{ss:measuresfinite} we
address the general case of invariant Borel measures 
which are positive finite on some sub-Bratteli diagram. Lastly we
apply these results to the study of some simple
examples.

\subsection{The upper triangular block case }

In this and the following subsection, $\nu_\w$ denotes the measure  on
$\Sigma_{\M}^{0,+}$ which is  defined from $\w\in
\V_M$  in
\eqref{eq:centralmeasure} of Theorem \ref{t:basic_thm}. For
$A\leq M$,  we denote by $\nu^A_\w$ the measure on $\Sigma_{A}^{0,+}$ defined
from $\w\in \V_A$, as in  Definition
\ref{d:measure}. 
 
With $A_k\leq M$ the diagonal blocks of $M$ from the fixed--size  Frobenius
normal form of Definition  \ref{d:fixedsizeFrob}, see Fig.~
\eqref{eq:block_form},
then $\A^k=
(\A^k_i)_{i\in \N}$ denotes  the subalphabet sequence associated
to the matrix sequence $A_k$.

In the next theorem we show there is a  bijective correspondence of the 
  $\FC$\,--\,invariant conservative ergodic Borel measures
  on $\Sigma_M^{0,+}$ with those on ${\Sigma_{A_k}^{0,+}}$, given by restriction.
  Via this correspondence we identify 
  the measures which are  finite on some open subset, and those which have  finite total
  mass.

\begin{theo}\label{t:inv-meas}
  Assume we are given a 
  sequence   of  nonnegative integer
matrices $(M_i)_{i\geq 0}$  in fixed--size  Frobenius
normal form with diagonal
blocks $A_k$ for $1\leq
k\leq \wh l$ (so in particular these  blocks are either
identically zero or
reduced primitive), and 
 a nonzero  ergodic $\FC_M$\,--\,invariant Borel measure $\nu$  on
$\Sigma_M^{0,+}$. Then:
\item{(i)}  
 There exists 
a unique
$k$, $1\leq k\leq \wh l$, such that
the restriction 
$\nu_{{A_k}}$ of $\nu $ to ${\Sigma_{A_k}^{0,+}}$ is
positive.  
 This measure  is $\FC_{A_k}$\,--\,invariant  and
ergodic, with  $\nu$ its invariant extension  to
$\Sigma_M^{0,+}$. It  
is  either positive infinite, or
positive finite,  on all nonempty open
subsets of $\Sigma_{A_k}^{0,+}$. This correspondence, between ergodic
Borel measures on $\Sigma_M^{0,+}$ and on  ${\Sigma_{A_k}^{0,+}}$ for some
$k$, is bijective. Also, one is conservative iff the other is.
Note that here ${\Sigma_{A_k}^{0,+}}$ is in general not open
in $\Sigma_M^{0,+}$.  
\item{$(ii)$}
    $\nu$ is positive
finite on some open set of $\Sigma_M^{0,+}$  if and only if 
$\nu_{{A_k}}$ from $(i)$ is positive finite, if and only if 
 $\nu_{{A_k}}$  equals $ \nu^{A_k}_\w$ for some 
point $\w$ in an extreme ray of 
$\V_{A_k}$.

\noindent
The measure $\nu$ is   itself finite iff  $\w$ is
$M/A_k$--distinguished, i.e.~
 $\nu= \nu_{\iota(\w)}$.
\end{theo}

\begin{proof}\item{(i)}:
From the hypothesis, there exists an edge cylinder set $[f_0\dots f_m]$  of $\Sigma_M^{0,+}$ with 
$\nu ([f_0\dots f_m])$ positive finite or infinite. Define $k$ to be
the maximal index 
from $1$ to $\wh l$  such that there exists an allowed extension $f_0\dots f_m
f_{m+1}\dots f_q$ of
this word  satisfying $\nu([f_0\dots f_m
f_{m+1}\dots f_q])>0$, with $f_q^+\in \A^k$. 
Then since 
$(A_k)_0^\infty$ is reduced, there exists an allowed word $e_0\dots e_q$ within the stream
$\alpha_k$ such that $e_q^+ =f_q^+$.
By invariance  
$\nu([e_0\dots e_q])= \nu([f_0\dots f_q])>0$. 
Now $\nu([e_0\dots e_q])=\nu([e_0\dots e_q]\cap \Sigma_{A_k}^{0,+})$,
since for any subcylinder $[e_0\dots
e_qg_{q+1}\dots g_t]$
of positive measure then $g_{q+1}, \dots, g_t$ must be in
$\alpha_k$, for if $g_{i}\in \alpha_j$ with $j<k$ this string would
not be allowed by the upper triangular form, while if $j>k$ the measure is zero since $ k$ is
maximal.
Since $\nu$ is $\FC_M$-invariant,  it follows that, a fortiori, the  restricted measure
$\nu_{{A_k}}$ is $\FC_{A_k}$\,--\,invariant.
We
denote by $\wt \nu_{{A_k}}$ the tower
measure on  $\Sigma_{M/A_k}^{0,+}\subseteq \Sigma_M^{0,+}$. This is
the $\FC_M$-invariant extension of  $\nu_{{A_k}}$ from Theorem
\ref{t:towermeasure}. By $(i)$, $(ii)$ of that theorem the invariant ergodic Borel measures on base
$\Sigma_{A_k}^{0,+}$ and adic tower $\Sigma_{M/A_k}^{0,+}$
correspond bijectively. Now
$\nu_{{A_k}}<<\nu$
whence 
$\wt\nu_{{A_k}}<<\nu$, and both are $\FC_M$\,--\,invariant, thus
$\wt\nu_{{A_k}}=\nu$:
they are equal on the base $ \Sigma_{A_k}^{0,+}$ of the
tower, hence on the tower,
which is invariant and thus must be a.s.~ the whole space $\Sigma_M^{0,+}$ by ergodicity of $\nu$.

By Lemma \ref{l:finite_on_all},  because of the
primitivity of $A_k$,  the measure $\nu_{{A_k}}$
 is either positive finite or infinite on all open subsets of $\Sigma_{A_k}^{0,+}$.

By $(i)$ of Theorem
\ref{t:towermeasure} the invariant Borel measures on base
$\Sigma_{A_k}^{0,+}$ and adic tower $\Sigma_{M/A_k}^{0,+}$
correspond bijectively, and by 
$(ii)$ of that theorem the  ergodicity of 
$\nu=\wt\nu_{{A_k}}$ and $\nu_{{A_k}}$  correspond,  as well.

To complete the proof of part $(i)$, we show the base and tower have been uniquely identified in the
above construction, 
by verifying the uniqueness of this index $k$. Suppose that for some
$j\neq k$, the restricted measure
$\nu_{{A_j}}$ is positive; without loss of generality $j>k$. Let $j$
be the maximal such index. Then there exists an edge cylinder set $[f_0\dots f_m]$  with 
$\nu_{{A_j}} ([f_0\dots f_m])$ positive. Define $l$ to be
the largest index 
 such that there exists an allowed extension $f_0\dots f_m
f_{m+1}\dots f_q$ of
this word  satisfying $\nu([f_0\dots f_m
f_{m+1}\dots f_q])>0$, with $f_q^+\in (\A_l)_{q+1}$. Then as above there is a 
cylinder $[e_0\dots e_q]$ in $\Sigma_{A_l}^{0,+}$ with 
$e_q^+ =f_q^+$ and with positive $\nu$-measure. But then $l=j$ since
$j$ was maximal. This implies in particular that the
stream $\alpha_j$ can have no front-connections to $\alpha_k$. But then by
the same reasoning as given above for $k$,  the tower
measure $\wt \nu_{{A_j}}$ is equal to $\nu$ and hence to
$\wt\nu_{{A_k}}$, which implies that there are front-connections to
$\alpha_k$, since then $\Sigma_{A_j}^{0,+}$ belongs to the tower
$\Sigma_{M/A_k}^{0,+}$, a contradiction.

 From $(ii)$ of Theorem \ref{t:towermeasure}, a  measure on the tower
  is conservative iff that holds for the restriction to the base.

We have noted in the statement of $(i)$ that $\Sigma_{A_k}^{0,+}$ itself is in general not an open
subset of $\Sigma_M^{0,+}$;  see Proposition \ref{p:opengeneral} and
Cor.~\ref{c:uppertriangularopen}.

\noindent
\item{(ii)}:  
If $\nu$ is finite positive on some open subset of
$\Sigma_M^{0,+}$, equivalently we have $0<\nu
([f_0\dots f_m])<\infty$ for
some thin edge cylinder of $\Sigma_M^{0,+}$. Then by the above reasoning there
exists an extension $f_0\dots f_m
f_{m+1}\dots f_q$  and $\e$ in the stream $\alpha_k$ with $e_q^+ =f_q^+$ so that $0<\nu_{{A_k}}([e_0\dots e_q])=\nu ([f_0\dots f_q])
<\infty$. Therefore by part $(i)$,  $\nu_{{A_k}}$ is
finite on all
open sets of $\Sigma_{A_k}^{0,+}$ and hence in particular $\nu_{{A_k}}$ has finite
total mass.

Equivalently 
by Theorem \ref{t:basic_thm},   
$\nu_{{A_k}}=  \nu^{A_k}_\w$ for some
$\w$  in an extreme ray of 
$\V_{A_k}$.

\

Lastly we consider the case where $\nu$ itself is finite.
Now by Theorem \ref{t:basic_thm},   there is a bijection $\Phi$
from the collection  $\Ext \CCM$ of 
finite ergodic central measures $\nu$  to
$\Ext \V_{\M}^\Delta$. We recall the definition of this map. For $  \Phi(\nu)=
\w'=
(\w'_0,\w'_1,\dots, \w'_i,\dots )$, then 
for any $s\in
\A_i$,  where $
e_{i-1}^+=s$, we have that  $\nu([.e_0\dots e_{i-1}]) = (\w'_i)_s.$

On the other hand, from part $(iii)$ of Theorem
\ref{t:FrobVictGeneral} there is a bijection $\iota$ from
$
\cup_{k=1}^{\wh l} \ExtD  \V_{A_k}$ to $\Ext\V_M $.
Thus there is a unique 
$j$ and $\w^j\in \ExtD  \V_{A_j}$ with $\w'= \iota(\w^j)=
(\u^{1} , \u^{2} , \dots, \u^{j-1} , 
\w^j, 
\0^{j+1}, \dots, 
\0^{\wh l})$, where the upper index is used for for the stream
decomposition and the  lower index will be reserved for time.

Now since $\w^j\in \V_{A_j}, $ it determines a measure 
$\nu^{A_j}_{\w^j}$ on $\Sigma_{A_j}^{0,+}$ by the similar formula: for each  cylinder set 
$[.e_0\dots e_{i-1}]$ of  $\Sigma_{A_j}^{0,+}$ with $e_{i-1}^+=s\in
\A_i^j$, $\nu^{A_j}_{\w^j}([.e_0\dots e_{i-1}])= (\w^j_i)_s$. But
this is exactly equal to the previous value, for these $s\in
\A_i^j\subseteq \A_i$.
 Hence 
$\nu^{A_j}_{\w^j}= \nu_{{A_j}}$, i.e.~ it is  the restriction of $\nu$
to
$\Sigma_{A_j}^{0,+}$. Since $\w^j\in \Ext \V_{A_j}$, this measure is
ergodic for the action of $\FC$ on $\Sigma_{A_j}^{0,+}$. 

From part $(i)$, the index $j$ is unique, so $j=k$ from parts $(i)$
and  the first part of $(ii)$ above. That is, $\w$ as above equals $\w^j$, and so
since $\iota(\w^j)= \w'$, $\w$ is $M/A_k$--distinguished and
 $\nu= \nu_{\w'}=\nu_{\iota(\w)}$.

\end{proof}
\begin{rem}\label{r:reducedremark}
 
 We note that in case $(ii)$ above both measures are conservative, since
 the base measure  is finite hence  this holds by the Poincar\'e
 Recurrence Theorem, see Remark \ref{r:ergodicmeasures}.

  We comment on an important but subtle technical point.
  At some points in this paper we have used
  reduced matrix sequences and at others, only
  column--reduced. In the above theorem, we only needed the primitive diagonal
  blocks reduced, but did not assume this for $M$; this flexibility
  proves useful in the proof of part $(iii)$ of Theorem \ref{t:inv-meas-subdiag}.

  For other examples, in the Frobenius decomposition
  theorem,
  starting with a reduced matrix sequence 
  we showed we could achieve diagonal blocks which are either zero or
  reduced
  primitive.
 However in the  
Frobenius--Victory  Theorem, both the  inductive $(2\times
2)$ block case of Theorem \ref{t:FrobVictFirstVersion} and the general case
\ref{t:FrobVictGeneral},we were careful to prove everything under the
weaker requirement of the matrices and diagonal block submatrices
being column--reduced. The reason for this
is that otherwise we could have a problem in the inductive step of the
general case, as the property of 
$N$ being reduced may not be inherited by  the
submatrices $N_k$ to which we apply the $(2\times 2)$ block step. By
contrast, being 
column--reduced {\em is} passed on from $N$ to $N_k$.

Of course our ultimate focus is on the measure theory, where the basic result of
Theorem
\ref{t:basic_thm} relates the invariant Borel measures to the nonnegative eigenvector
sequences of eigenvalue one. And for that theorem we again need the
condition of being reduced: both for the full matrix sequence and for
the primitive diagonal blocks. Now fortunately that much is guaranteed
by the Frobenius Decomposition Theorem.

In summary, we use reduced at the beginning (for the Decomposition
Theorem), and at the end (to conclude about the measure theory), while
in the middle, to prove the Frobenius--Victory Theorem, we need to
to work with the weaker condition of being column-reduced.
\end{rem}

\medskip

\subsection{The general case}\label{ss:measuresfinite}
In this section we bring together  the ingredients developed
throughout the paper -- the
stream and 
Frobenius decompositions, the definition of distinguished eigenvector
sequence for a subdiagram, the notion of the
canonical cover, and the nonstationary Frobenius-Victory theorem -- to
prove  in Theorem
\ref{t:inv-meas-subdiag} our main result: a characterization of  the invariant measures for   a
Bratteli diagram
in terms of measures on a subdiagram or sub-subdiagram and related
distinguished eigenvector sequences. See Example \ref{ex:nestedodometer}
and the further examples in \S \ref{ss:examples}.

Beginning with an 
ergodic
$\FC$--invariant Borel measure $\nu$ on the path space $\Sigma_{\wh M}^{0,+}$
of a Bratteli diagram $\wh{\mathfrak
B}$ with matrix sequence $\wh M$, we  show 
in part $(i)$ that $\nu$ determines an eventually unique maximal
primitive reduced subdiagram $\mathfrak B$ with matrix sequence
$M$.  Its path space $\Sigma_{M}^{0,+}$ is the base of an adic tower
on which the restriction $\nu_M$  of $\nu$ to $\Sigma_{M}^{0,+}$ can
be finite or infinite. In part $(ii)$ we consider
 measures which are positive finite on some open set of $\Sigma_{\wh M}^{0,+}$,
with finiteness of $\nu$ characterized
in terms of $\wh M/M$ distinguished eigenvector sequences.
This extends to general diagrams the case of fixed-size (upper triangular) Frobenius
form from Theorem \ref{t:inv-meas}.

For the proof we first 
     use  the Frobenius decomposition to find an eventually unique maximal (ordered by
     containment)  primitive subdiagram. Then we use the
    fixed--size Frobenius normal
form proved in  Theorem
\ref{t:FrobDecomp} together with 
Theorem \ref{t:inv-meas} to
describe the invariant Borel measures. 

Now to achieve the fixed-size form from the streams we needed to permute
the alphabets and gather the matrix sequence; so now we must transfer
 this analysis
back to the streams and hence to 
the orginal diagram. For this we have
from Remarks \ref{r:fixed--size_notunique},
that a change of order on the alphabets, or a gathering, induces a
topological conjugacy of the shift spaces and $\FC$-actions.
Furthermore,  by Remarks \ref{r:normalform},  \ref{r:fixedsize}
Proposition
\ref{p:gathermeasure},
   and 
Corollary \ref{c:gatheringdispersal} these operations and also a
change of the starting time  does not affect 
 the  collections of invariant Borel measures or 
distinguished eigenvector sequences. We also make use of the eventual
uniqueness of the stream collection and Frobenius decomposition from
Theorem
\ref{t:StreamDecomp} and Theorem
\ref{t:FrobDecomp}.  That does not affect the
analysis of invariant measures, also because these edge spaces are
identical after some fixed time, using  Corollary \ref{c:zeroand_k}.

The main subtlety  of the theorem comes in the proof of
 part $(iii)$ where we consider the case much more general than $(ii)$
 of a measure perhaps not finite on any open subset but finite on  a
 subdiagram.  For such a subdiagram   defined by $M'\leq
\wh M$, and further for a sub-subdiagram defined by $M\leq M'\leq \wh
M$, we specify which of  these measures are finite or infinite in
terms of distinguished eigenvector
sequences. 

For the proof we apply part $(ii)$ of Theorem
\ref{t:inv-meas} twice, once for
$M'\leq \wh M$
and once for a further subdiagram $M\leq M'$. To carry this out  we first
achieve an upper triangular form,  but this requires several steps. The first step, 
in the case of a general subdiagram, is to build the canonical cover. That
allows us to
place the submatrix $M'\leq \wh M$ as a
diagonal block in the cover matrix $\wt M$. If $M'$ happens to be primitive we are then
back in the situation of $(ii)$, and are done. 
If not, we find an upper triangular form for $M'$,  reordering the subalphabet for $M'$ by
the Frobenius Decomposition Theorem, and finding a primitive diagonal
block $M$ with positive measure.  Here we make use of Proposition
\ref{p:gateringtower} that the canonical cover is measure isomorphic
to the original path space.
We  keep track of this reordering by 
a permutation
matrix,  so that we can carefully follow the changed eigenvector sequence.
That the resulting eigenvector sequence is indeed
distinguished then makes full use of our analysis of the $(2\times
2)$-block case, specifically via two applications of Corollary  \ref{c:TwoDisting}.

This leads in Corollary \ref{c:secondcount} to an improved measure count.
In  Theorem \ref{t:invariant-measures-towers} we describe, similarly,  measures for 
adic towers.

\begin{theo}\label{t:inv-meas-subdiag}
  (Measure classification for Bratteli diagrams  of bounded rank)
Let $\wh{\mathfrak B}\equiv \wh{\mathfrak B}_{\wh \A,\wh\E,\wh M}$ be a  reduced Bratteli diagram with
bounded alphabet size, with alphabet sequence $\wh \A$, and let $ \nu
$ be an ergodic $\FC_{\wh M}$\,--\,invariant Borel
measure.
Then:

\item{(i)}
There exists  an eventually unique maximal primitive reduced $M\leq \wh M$,
with alphabet sequence $\A \subseteq\wh \A$,  such that the
restriction $\nu_M$ of $\nu$ to the path space ${\Sigma_{M}^{0,+}}$ is positive.
This  measure  is 
$\FC_M$\,--\,invariant  and ergodic, and is positive infinite, or
positive finite,  on all open
subsets of ${\Sigma_{M}^{0,+}}$, and $\nu=  \wh\nu_M$ is the invariant
extension of $\nu_M$ on the base $\Sigma_{M}^{0,+}$ to its tower
 $\Sigma_{\wh
  M/M}^{0,+}$.

This correspondence, between ergodic Borel
measures on $\Sigma_{\wh M}^{0,+}$ and   on ${\Sigma_{M}^{0,+}}$ for some
maximal primitive reduced $M\leq \wh M$,  equivalently on a tower $\Sigma_{\wh
  M/M}^{0,+}$,
is bijective,  up to  eventual uniqueness of the sequence $M$.
Also, one is conservative iff the other is.

\item{(ii)}
    $\nu$ is positive
finite on some open set of $\Sigma_{\wh M}^{0,+}$  if and only if
$\nu$ is positive finite on the maximal primitive subdiagram
 $\Sigma_{M}^{0,+}$ from $(i)$,
if and only if
 $\nu_M$  equals $ \nu^M_\w$ for some 
point $\w$ in an extreme ray of 
$\V_{M}$.

\noindent
The total measure $\nu(\Sigma_{\wh M}^{0,+})$ is   finite iff   $\w$ is
$\wh M/M$--distinguished, iff
$\nu= \nu_{\overline \w} $ for $\overline \w={\iota_{\wh M}(\w)}$.

\item{(iii)} 
Suppose that $\nu$ is  positive finite
for some subdiagram with matrix sequence $M'\leq
  \wh M$, thus 
 $0<\nu(\Sigma_{M'}^{0,+})<\infty$. (This includes the case in $(ii)$
 of being finite on some open subset).
Then $\nu$ is the tower measure over $(\Sigma_{M'}^{0,+}, \nu_{M'})$,
which is   ergodic. 
 We know there exists a unique $\w'\in \Ext\V_{M'}$ such that  $\nu_{M'}= \nu_{\w'}^{M'}$.
There exists an eventually unique maximal primitive reduced
sub-subdiagram with matrix sequence $M\leq M'$ and $\w\in 
\Ext\V_M$ such that $\w'= \iota_{M'}(\w)$, and  
 $\nu$ also is the tower
measure over 
$(\Sigma_{M}^{0,+},\nu_M)$. The measure $\nu$
is finite iff $\w'$ is $\wh {M}/{M'}$\,--\,distinguished,
iff $\w$ is $\wh M/M$\,--\,distinguished.

\item{(iv)}
Lastly, suppose that $\nu$ is ergodic and infinite on a subdiagram defined by $ M'\leq
  \wh M$. Such  measures are analyzed further as in
 $(i)-(iii)$. That is, there exists an eventually unique maximal primitive
 reduced $\wt M\leq  M'$,
such that  $\nu_{\wt M}$ is positive. If $\nu_{\wt M}$ is finite, we proceed as
in $(ii)$. If it is finite on some further subdiagram,  for $\wt M$, we proceed as in $(iii)$.
\end{theo}

\begin{proof} of Theorem \ref{t:inv-meas-subdiag}:
 
{\em {(i), (ii)}}:
If $\wh M$ happens to be in fixed--size Frobenius normal form, then
  parts $(i)$, $(ii)$ are just those parts from Theorem
  \ref{t:inv-meas}. That is to say, from that theorem there is a unique $k$ such that
  ${\Sigma_{A_k}^{0,+}}$ has positive measure;  and then
 $M=A_k\leq \wh M$ is the eventually unique maximal primitive reduced
 submatrix sequence claimed in $(i)$. Part $(ii)$ follows.

  Now we consider general $\wh M$.
 To prove $(i)$, first,
  if $\wh M$ is primitive, then $M= \wh M$. If not, then
let $\{\alpha(1)\dots, \alpha(d),\P\} $
be the  stream decomposition guaranteed by
Theorem \ref{t:StreamDecomp}. By permuting each alphabet we can
place the matrix sequence in Frobenius normal form. This stream
decomposition is eventually unique. By performing a (non-unique)
gathering, we can achieve fixed--size Frobenius normal form, with primitive reduced
diagonal blocks $A_l$.

Writing for simplicity $\nu$ also for the measure on
this new path space, then from part
$(i)$ of Theorem \ref{t:inv-meas}, there is a unique
$k$ such that the restriction of $\nu$ to ${\Sigma_{A_k}^{0,+}}$ is
positive, the restricted measure $\nu_{{A_k}}$  is invariant and 
ergodic,  and the tower measure $\wh \nu_{{A_k}}$ over $\nu_{{A_k}}$
equals $\nu$. Further,
$\nu_{{A_k}}$ is positive infinite, or
positive finite,  on all open
subsets of ${\Sigma_{A_k}^{0,+}}$.

Let $\{\alpha(1)\dots, \alpha(d),\P\} $
be the  list of ungathered streams associated to the diagonal blocks
$A_1,\dots A_d$ and pool stream for the Frobenius normal form, so
$\alpha(k)$ denotes the stream associated to the (gathered)
matrix sequence $A_k$. We define  $M$ to be the ungathered matrix
sequence for this stream.

By Remarks \ref{r:normalform},  \ref{r:fixedsize}
Proposition
\ref{p:gathermeasure},
  Corollary \ref{c:zeroand_k},   and 
Corollary \ref{c:gatheringdispersal} the operations of gathering,
alphabet permutation, and changing of the starting time  do not affect 
 the  collections of invariant Borel measures or 
distinguished eigenvector sequences. 

We claim that $M$ 
is  the eventually unique such  matrix
sequence (that is maximal, primitive, reduced and with positive
measure). 
Let 
$\alpha'(j)$ be a maximal primitive stream not in the list; then by 
Theorem \ref{t:StreamDecomp} there exists a stream
decomposition $\{\alpha'(1)\dots, \alpha'(d),\P\} $ including
$\alpha'(j)$ for some $1\leq j\leq d$,  and by Theorem
\ref{t:FrobDecomp}, the two stream collections $\alpha$ and $\alpha'$
for the two Frobenius normal  forms are eventually
equal. Hence if a stream $\alpha'(j)$ has positive measure for its
path space, then 
the streams $\alpha'(j)$ and $ \alpha(k)$ are eventually equal. 

The last part of $(i)$ is the bijection between measures. But due to
the uniqueness just proved, this now follows from  $(i)$ of Theorem
\ref{t:inv-meas}.

Part $(ii)$ is proved  using now part $(ii)$ of Theorem
\ref{t:inv-meas}, as everything remains valid for our choice $M\equiv A_k$.

Proof of $(iii)$:
 By taking
smaller alphabets $\A=(\A_i)_{i\geq 0}$,  we can assure that the matrix sequence ${M'}$ is reduced, see part $(i)$ of Remark
\ref{r:gen_matrices}.
We apply Theorem
\ref{t:FrobDecomp} to this
reduced sequence: we 
reorder $\A= (\A_i) _{i\geq 0}$ and
gather so that ${M'}$
is in 
fixed--size Frobenius normal form
with diagonal blocks $A_{1},\dots A_{\check l}$.

Since by hypothesis $\nu_{M'}$ is positive finite, by $(ii)$ of Theorem
\ref{t:inv-meas} there exists a unique primitive reduced  subblock
$A_k$ on the diagonal of ${M'}$
and a unique 
$\w\in\Ext \V_{A_k}$, with
$\nu_{M'}= \nu_{\w'}^{M'}$, where $\w'= \iota_{M'}(\w)$. 

Now we would like to be able to choose $M=A_k$ as the matrix $M\leq
M'$ with   $M'\leq \wh M.$ However, $M'$  does not itself appear as a diagonal block of $\wh M$, so
neither do its subblocks $A_l$.

This is where the use of the canonical cover matrix $\wt M$ will be
crucial.
We note that by Proposition \ref{p:gateringtower} the canonical cover
introduces a measure isomorphism.
A version of $M'$ will 
appear as a diagonal block of the cover matrix $\wt M$. By doing this
carefully we can make use of the fixed-size normal form already achieved for
$M'$ with $A_k$ as one of the diagonal blocks.

To do this, 
we   augment the alphabet sequence for ${M'}$ from $\A$ 
to $\wh \A$, as in part $(ii)$ of Remark
\ref{r:gen_matrices}. We then order $\wh \A$  by placing the symbols of $\wh
\A\setminus \A$ last. Now this augmented version, also called ${M'}$,
may no longer be reduced, but it is still in fixed-size
normal (upper triangular) form,  as these last symbols only add zero matrix elements,
including a zero block on
the diagonal.

Then we construct the
canonical cover matrix $\wt M=
\left[ \begin{matrix}
\wh M& C \\
0 & {M'}
\end{matrix}  \right] 
$ where $C= \wh M-{M'}$,
as in Theorem \ref{t:towerring}. The alphabet $\wh\A$ remains ordered
as above. We use this alphabet for the 
block $\wh M$ (that is, the original $\wh M$ has been conjugated by a sequence of
permutation matrices), but now we use a copy $\wh \A'$ of $\wh \A$ for
$M'$, so the alphabet for the cover matrix is 
$\wh
A\cup\wh \A'$. Thus the subblock 
${M'}$  is still in 
fixed--size Frobenius normal form.

We define $M$ to be the diagonal subblock  of $M'$, now  with augmented
alphabet $\wh A'$,  which corresponds to $A_k$; this  will be the claimed eventually unique
primitive reduced subsequence. As above, we have the unique
$\w\in\Ext \V_{M}$, with
$\nu_{M'}= \nu_{\w'}^{M'}$, where $\w'= \iota_{M'}(\w)$.
Of course this implies that $\w$ is $M'/M$--distinguished.

We claim that  $\nu$ itself is finite iff $\w'$ is $\wh M/{M'}$\,--\,
distinguished, iff $\w$ is $\wh M/M$\,--\, distinguished.
Now we know from $(ii)$ of Theorem \ref{t:basic_thm} that  $\nu$  being
finite ergodic is
equivalent to
$\nu= \nu_{\w''}$ for some $\w''\in \Ext \V_{\wh M}$. The proof will
be complete if we show that $\w''= \iota_{\wh M}(\w')= \iota_{\wh M}(\w)$.

At this point, we have   reordered the alphabets
$\wh \A$ and its copy $\wh \A'$ in the same way, so as to put the
subblock ${M'}$ of $\wt M$ in
fixed--size Frobenius normal form.
We next  apply Theorem
\ref{t:FrobDecomp}  to the subblock $\wh M$ so as to  put it in fixed--size Frobenius
normal form as well,  however this time 
reordering
$\wh \A$ {\em without} changing the order of $\wh \A'$,  so as to keep the work
already done there.

The result is that (after a second gathering) this second reordering
has   conjugated $\wt M$  via $Q$ 
of the form 
$Q= \left[ \begin{matrix}
P& 0 \\
0 & I
\end{matrix}  \right] $, where $P$ is a permutation matrix sequence, 
to
 $$\wt M'= Q\wt MQ^{-1}=\left[ \begin{matrix}
P\wh MP^{-1}& PC\\
0 & {M'}
\end{matrix}  \right] =\left[ \begin{matrix}
\wh M'& C' \\
0 & {M'}
\end{matrix}  \right] 
$$ in  fixed--size Frobenius normal form with
diagonal blocks $\wh A_0,\dots, \wh A_{\wh
  l}, A_{1},\dots A_{\check l}$. (These are all sequences, so this
means that for all $i\geq 0$,
$\wt M'_i= Q_i\wt M_iQ^{-1}_{i+1}$.)

Since we have not reordered $\wh \A'$, we no longer have that
$C'$ equals $\wh M'- {M'}$. Nonetheless, the nonnegative eigenvector sequence $\w''$ for
$\wt M$ determines that for $\wt M'$:  it 
is $Q \w''$. Since $\wt M'$ is in
fixed--size Frobenius normal form, we can apply $(ii)$ of Theorem
\ref{t:inv-meas}
to conclude that $Q\w''=  \iota_{\wt M'}(\w)$, where $\w$ is a
nonnegative eigenvector sequence for one of the diagonal subblocks  of  $\wh A_0,\dots, \wh A_{\wh
  l}, A_{1},\dots A_{\check l}$.  But this must be the subblock
$M=A_k$ of $M'$ already found above, by the uniqueness in part $(i)$ of Theorem
\ref{t:inv-meas}, applied now to $\wt M'$. (This only
required the diagonal primitive blocks being reduced, not the full
matrix; see Remark \ref{r:reducedremark}).

We write $\wt\A, \wt \A', \A_M$ for the alphabet sequences of $\wt M, \wt
M', M=A_k$ respectively.
Recalling from Definition \ref{d:dist} that $\mathtt{e}$ denotes  the
embedding of vector sequences, so  $\mathtt{e}_{\wt M'}: V_{\A_M} \to
 V_{\wt \A'}$ and $\mathtt{e}_{\wt M}: V_{\A_M} \to
 V_{\wt \A}$, 
we note that $\mathtt{e}_{\wt M'}= Q\mathtt{e}_{\wt M}$ since we are just permuting
the $\wh \A$-- coordinates of the embedding. Therefore
$$Q\w''=\iota_{\wt M'}(\w)\equiv \lim_{n\to\infty} (\wt M')^n \; \wt
\iota_{\wt M'} (\w)= \lim_{n\to\infty}  Q\wt M^n Q^{-1}Q
\mathtt{e}_{\wt M}(\w)=  Q\iota_{\wt M}(\w)$$ whence
$\w''=  \iota_{\wt M}(\w)$. So in particular, $\w$ is $\wt M/M$\,--\, distinguished.

We claim that also 
$\w''=  \iota_{\wt M}(\w')$. 
But the inverse of
$\iota$ is a projection, and (as in the proof of $(ii)$ of Theorem
\ref{t:inv-meas}), 
since $\nu$ restricts to $\nu_{M'}$ which in turn restricts to $\nu_{M}$,
also $\w''$ projects to $\w'$ on $\Sigma_{{M'}}^{0,+}$, which projects to
$\w$ on $\Sigma_{M}^{0,+}$, verifying the
claim. 

This shows  that 
 $\w'$ is $\wt M/{M'}$\,--\,
distinguished. But according to part $(ii)$ of Corollary  \ref{c:TwoDisting}, $\w'$ is
$\wt M/{M'}$\,--\,distinguished iff it is $\wh M/{M'}$\,--\,distinguished. 

Now we know that $\w$ is
$ {M'}/M$\,--\,distinguished and $\wt M/M$\,--\,distinguished. 
Therefore, by part $(iii)$ of Corollary  \ref{c:TwoDisting}, $\w$ is
$\wh M/M$\,--\,distinguished. 

The  gatherings we have employed do not affect these conclusions: from Corollary
\ref{c:gatheringdispersal}, $\w'$ and $\w$ are
distinguished for the original (non-gathered) sequence.

This completes the proof of $(iii)$. Part $(iv)$ follows the previous parts.

\end{proof}

\begin{exam}Integer Cantor sets (Nested odometers)\label{ex:nestedodometer}

  Let us consider our two models for the  Integer Cantor Set 
(see  Example \ref{exam:Cantor}) in the light of this theorem. The first model is the infinite measure $\nu$ on the
  triadic odometer, with 
  the constant matrix sequence $\wh M_k=
\left[ \begin{matrix}
3\\
\end{matrix}  \right] 
$ for all $k\geq 0, $ and its subset of the embedded dyadic
odometer, with matrix sequence $M\leq \wh M$ where
$M_k=
\left[ \begin{matrix}
2\\
\end{matrix}  \right] 
$. The second is the canonical cover of this, with matrix
sequence
$\wt M=
\left[ \begin{matrix}
\wh M& C \\
0 & M
\end{matrix}  \right] 
$
with $C$ satisfying 
$C_i=\wh M_i- M_i=[2],$ so $\wt M=\left[ \begin{matrix}
3 & 1\\
0 & 2
\end{matrix}  \right]$; 
see 
  Fig. ~\ref{F:ICSNew}.
For the first model, both $\wh M$ and $M$ are primitive matrix
sequences and hence give primitive path spaces. But the embeddings of the path space
$\Sigma_M ^{0,+}$ are completely different in the two larger spaces
$\Sigma_{\wh M}^{0,+}$ and $\Sigma_{\wt
  M}^{0,+}$.
In the cover space $\Sigma_{\wt
  M}^{0,+}$ it is open, indeed the tower $\Sigma_{\wt M/M}^{0,+}$ is
open dense in  $\Sigma_{\wt
  M}^{0,+}$, and $\Sigma_M ^{0,+}$ itself is the maximal primitive subset containing
$\Sigma_M ^{0,+}$,  while in $\Sigma_{\wh M}^{0,+}$ it is neither open
(its tower $\Sigma_{\wh M/M}^{0,+}$ is  a dense set with empty
interior in $\Sigma_{\wh M}^{0,+}$, see Examples \ref{exam:Cantor} and \ref{exam:Cantor_b}.)
nor the maximal primitive path space containing $\Sigma_M ^{0,+}$ (as
that is $\Sigma_{\wh M}^{0,+}$).

For part $(ii)$, considering the cover $\wt M$, then the subdiagram for $M$ is a maximal reduced
primitive subdiagram with {finite} measure (Bernoulli measure of the dyadic
odometer), and the eigenvector sequence for this is {\em not}
$\wt M/M$- distinguished, whence the tower measure is infinite.

Regarding the first model (with $M\leq \wh M$), part $(i)$ tells us that since the measure
is infinite on the open set $\Sigma_{\wh M}^{0,+}$, it is 
locally infinite. For the first model, part $(ii)$ does not say much
as $\wh M$ is
the maximal primitive subdiagram and the measure is infinite
there.

Now for this model part $(iii)$
is more appropriate and interesting: 
there is a subdiagram (that given by
the embedding of $M$) on which the measure is finite.
Thus part $(iii)$ for the first model corresponds to part $(ii)$ for
the second.

For an example of nested odometers which exhibits finite total measure, let $0<b_k$ such that $\sum_0^\infty b_k=
c<\infty$. Then for $a_k= e^{-b_k}$
define $n_k= \lceil 1/a_k \rceil$,  where this denotes the least integer greater than or
equal to $1/a_k$.

Then $1/n_k\leq a_k$, so 
we  have 
$0< \ep \equiv  e^{-c}=\Pi_0^\infty a_k\leq \Pi_0^\infty 1/n_k <1.$ 

Let $\wh M$ be 
 the $(1\times 1)$ matrix sequence defined for  $k\geq 0$ by $\wh M_k=
\left[ \begin{matrix}
 n_k\\
\end{matrix}  \right] 
$. The unique $\FC-$ invariant Borel measure on
$\Sigma_{\wh M}^{0,+}$ is nonstationary Bernoulli measure. That is,
take $\nu_k(\{e\})= 1/n_k$ for each edge $e\in \E_k= \{1,\dots n_k\}$ and set
$\nu= \otimes_0^\infty  \nu_k$, the product measure. Define $M\leq \wh M$ by removing one edge from
each matrix; thus, $ M_k=
\left[ \begin{matrix}
 n_k-1\\
\end{matrix}  \right] 
$.
We have removed the set with edge $n_k$ for each $k$, which by independence  has measure 
$\Pi_0^\infty1/n_k>\ep$.
Thus $0<\nu(\Sigma_M ^{0,+})<1-\ep $.

The unique $\FC_M$-invariant  measure on $\nu(\Sigma_M^{0,+})$ is a constant times  the restriction
$\nu_M$, and is strictly positive and $<1$. This gives a nonmaximal primitive stream
different from that of the unique maximal  stream $\wh M$.

We note that a similar construction produces a nested sequence
$\dots M^{(2)}\leq \dots M^{(1)}\leq M^{(0)}\leq \wh M$ such that
$\nu(\Sigma_{ M^{(k)}}^{0,+})$ decreases to any desired $c\geq
0$. These are
nested odometers on
 nested closed subsets whose intersection is given by a
subdiagram
which may have measure $0$.

\end{exam}

\begin{rem}\label{r:classification}

We note that statement $(ii)$ applies  immediately to the following more general situation: that there exists a gathering of $\wh M'$
of $ \wh M$, 
such that there exists $\M'\leq \wh M'$ as in  $(ii)$. This is more
general because e.g.~erasing a single edge in the diagram for $\wh M'$
gathered along the subsequence $(n_k)_{k\geq 0}$ 
at time $k$ 
corresponds to removing an edge cylinder set $[e_{n_k}\dots
e_{n_{k+1}-1}]$, not  a single edge (which is a larger set). A related point was made in the proof of Proposition
\ref{p:gathermeasure}. A concrete example is given by the nested rotations of
Example \ref{exam:nestedrot}, where the subshift defined by 
removing  edges from the diagram for the 
multiplicative continued fraction cannot be realized by removing edges
for the additive continued fraction.

By removing edges from nested subsequences which define gatherings
(i.e.~ by nested telescoping of the Bratteli diagrams) one can
produce  a sequence of gathered subdiagrams
such that their measure 
decreases to some $c\geq 0$, but
such that the intersection is a closed set which is {\em not} itself
given by a subdiagram. We mention that it could be interesting to further investigate
such examples, and to consider measures which are finite on some
closed subset.

Summarizing, we have completely analyzed the measures which are finite on some
subdiagram. Then,  
 if it is infinite, we can, by part $(iv)$ of the theorem,  look for a further subdiagram with 
$M'\leq
M$ such that $0<\nu(\Sigma_{M'}^{0,+})<\infty$.
\end{rem}

\

\medskip

We have as a corollary of Theorem \ref{t:inv-meas-subdiag}:

\begin{theo}\label{t:invariant-measures-towers}(Measure classification
  for adic towers)
Given nested  Bratteli diagrams $\mathfrak B_{ \A,\E,M}\leq \mathfrak B_{\wh \A,\wh\E,\wh M}$ with
bounded alphabet size, let
 $ \nu $ be an ergodic $\FC$\,--\,invariant
Borel measure on the adic tower ${\Sigma_{\wh M/M}^{0,+}}$. Then:

\item{(i)}The restriction $\nu_M$ of $\nu$ to ${\Sigma_{M}^{0,+}}$ is a 
positive ergodic invariant Borel measure, whose tower extension $\wh \nu_M$ equals $\nu$.
The measure $\nu$ is
 positive finite on some open set of the tower iff 
 $\nu_M$ is 
positive finite on some open set of $\Sigma_{M}^{0,+}$.
In this case there exists $ A $ primitive with $A\leq M\leq \wh M$
such that $\nu_A$ is positive finite, with $\nu$  the tower extension $\wh \nu_A$.

\item{(ii)} If $\nu $ (ergodic invariant on the tower) is positive finite on ${\Sigma_{A}^{0,+}}$ for some $A\leq
  \wh M$, then $A\leq M$ and 
 there exists a unique $\w\in \Ext\V_A$ such that $\nu_A= \nu_{\w}$ with
$\nu $  the tower measure over $\nu_{\w}$ on ${\Sigma_{A}^{0,+}}$. 
Furthermore $\nu$ has finite total mass iff $\w$ is $\wh M/A$\,--\,distinguished.

There exists $A_0\leq A$ primitive and $\w\in 
\Ext\V_{A_0}$ such that $\nu$ is the tower measure over $\nu_{\w}$ on 
$\Sigma_{A_0}^{0,+}$, with $\nu$ finite iff $\w$ is $\wh M/A_0$\,--\,distinguished.
\end{theo}

\begin{proof}
  
From $(ii)$ of Theorem \ref{t:covermeasure}, since any invariant
measure $\nu$ on the tower gives positive mass to ${\Sigma_{M}^{0,+}}$, $\nu$
is the tower measure over  $\nu_M$.
From $(iii)$ of Theorem \ref{t:covermeasure}, $\nu$ is
 positive finite on some open set of ${\Sigma_{\wh M/M}^{0,+}}$ iff 
 $\nu_M$ is 
positive finite on some open set of ${\Sigma_{M}^{0,+}}$. The remaining
statements follow from 
Theorem \ref{t:inv-meas-subdiag}.
\end{proof}

Now we return to the problem of estimating the number of ergodic
measures. Making use of the Frobenius decomposition, we get a new
proof of Proposition \ref{p:extremepts} with a generally better upper
bound, and now can also include the infinite measures.

We then have from Theorems \ref{t:inv-meas} and \ref{t:inv-meas-subdiag}:
\begin{cor}\label{c:secondcount}
(Counting the finite and infinite ergodic central measures) Let $(M_i)_{i\geq 0}$
  be a sequence   of  nonnegative integer
matrices with bounded alphabet size. Without loss of generality, 
assume these are in fixed--size  Frobenius
normal form, with (zero or primitive) diagonal blocks $A_k$ of size
$\wt l_k$, for $1\leq
k\leq \wh l$.

Then the number of  finite or infinite central measures 
(i.e.~ the conservative ergodic $\FC_M-$invariant 
   measures which are positive on some open subset),  determined up to
   multiplication by a positive constant, is equal to $\sum_{A_k\neq \0}\#\Ext\V_{A_k}^\Delta\leq
 \sum_{k=1}^{\wh l} \wt l_k\leq \liminf l_n$. The number of (finite)
 central
measures  equals the 
number of distinguished extreme points $\#\Ext\V_M=\sum_{A_k\neq \0}\#\Ext\V_{A_k
  dist}^\Delta$. The number of infinite central measures equals the 
number of nondistinguished extreme points.
\ \ \qed \end{cor}

\subsection{Examples}\label{ss:examples}
 
Statement \eqref{eq:1b} below (that the associated measure is finite
iff this series converges) is independently presented in equation (6.10) of
\cite{BezuglyiKwiatkowskiMedynetsSolomyak13}. (Their series is
equivalent to ours after a change of starting time; transposed matrices
are used throughout that paper, so the order of multiplication is
reversed, and one has lower rather than  upper
triangular Frobenius form). We give two
proofs, the first an application of our definition of distinguished
eigenvector sequences and  the general result Theorem
\ref{t:inv-meas}, the second geometric, in Remark
\ref{r:geometric}.  The geometry is simple because we are in the 
$(2\times 2)$ case. In fact this geometric
argument, once reformulated abstractly,  led to the general approach
including the above notion of distinguished.

 \begin{lem}\label{l:exampleconvergence}
Let $
N_i=
\left[ \begin{matrix}
a_i& c_i \\
0 & b_i
\end{matrix}  \right] $ for $i\geq 0$ with real entries $a_i,
b_i>0$ and $c_i\geq 0$.
Writing $a_0^k= a_0 a_1 \cdots a_k $ and
$b_0^k= b_0 b_1 \cdots b_k $, we
define $\w_0=1$, $\w_{k+1}= (b_0^k)^{-1}$; this is the unique normalized nonnegative eigenvector sequence with eigenvalue $1$ for 
$[b_i]_{i\geq0}$. 
Then $\iota(\w)$ exists (i.e.~ for
$
 N=\left[ \begin{matrix}
A& C\\
0& B
\end{matrix}  \right]
$,
$\w$ is an $N/B$- distinguished
eigenvector sequence) iff
\begin{equation}
  \label{eq:1a}
  \sum_{k\geq 0}\frac{a_0^k}{b_0^k}\frac{c_k}{a_k}<\infty.
\end{equation}

In the special case where $a_i\geq b_i$ and $c_i= a_i-b_i$, then
\eqref{eq:1a}
converges iff \[\limsup \frac{a_0^n}{b_0^n}<\infty.\]
 \end{lem}
  \begin{proof}

From Definition \ref{d:dist},
$\w= (\w_i)_{i\geq 0}$ is a
distinguished
eigenvector sequence iff
 for all ${i\geq 0}$, 
$\iota(\w)_i= \lim_{n\to +\infty}  N_i^n (\mathtt{e}_k \w)_{n+1}$ exists.
 Defining 
$\wh c_i^n$ by
$$N_i^n=
\left[ \begin{matrix}
a_i^n& \wh c_i^n \\
0 & b_i^n
\end{matrix}  \right],
$$
then equivalently $\wh c_i^i= c_i$, and for $n>i$
\begin{equation}
  \label{eq:induction}
  \wh c_i^{n+1} =a_i^n c_{n+1} + \wh c_i^n b_{n+1}
\end{equation}
 (compare \eqref{eq:C-hat3}). For $i=0$ we have
$$
N_0^n 
\left[ \begin{matrix}
0\\
\w_{n+1}
\end{matrix}  \right]=
\left[ \begin{matrix}
\wh c_0^n ( b_0^n)^{-1}\\
1
\end{matrix}  \right]
$$
and by induction, using \eqref{eq:induction}, we get 
\begin{equation}
  \label{eq:partialsum}
 \wh c_0^n ( b_0^n)^{-1}= \sum_{k=
  0}^n\frac{a_0^k}{b_0^k}\frac{c_k}{a_k}.
\end{equation}

Similarly, for $i\geq 0$
\begin{equation*}
    N_i^n \left[ \begin{matrix}
0\\
\w_{n+1}
\end{matrix}  \right]=
\left[ \begin{matrix}
\wh c_i^n ( b_i^n)^{-1}\\
(b_0^{i-1})^{-1}
\end{matrix}  \right]
\end{equation*}
and now
\begin{equation}
  \label{eq:3}
  \wh c_i^n ( b_i^n)^{-1}=
\sum_{k=i}
 ^n\frac{a_i^k}{b_i^k}\frac{c_k}{a_k},
\end{equation}
and since this converges as $n\to \infty$ iff
the sum for $i=0$ does, we indeed need only check \eqref{eq:1a}.

For the special case, by $(ii)$ of Lemma \ref{l:A-B}, $\w$ is distinguished iff 
$\lim {a_0^n}/{b_0^n}$ exists (and is finite), but since
$a_i/b_i\geq 1$ this sequence is nondecreasing so the statement for
the $\limsup$ is equivalent to this.

\end{proof}
 
That was the statement for real matrices; we next draw the consequence
for integer matrices and hence for adic transformations:

\begin{prop}\label{p:invariantmeasure2x2}
  Let $
M_i=
\left[ \begin{matrix}
a_i& c_i \\
0 & b_i
\end{matrix}  \right] $ for $i\geq 0$ with integer entries $a_i,
b_i>0$ and $c_i\geq 0$. We consider the 
$\FC$\,--\,invariant
Borel 
measures which are finite positive on some open subset of $\Sigma^{0,+}_M$.
Writing $a_0^k= a_0 a_1 \cdots a_k $ and
$b_0^k= b_0 b_1 \cdots b_k $,
then if 
\begin{equation}
  \label{eq:1b}
  \sum_{k\geq 0}\frac{a_0^k}{b_0^k}\frac{c_k}{a_k}<\infty,
\end{equation}
there are  exactly two such ergodic
invariant probability measures;
if the sum is infinite, there is  
one such ergodic invariant probability measure and one such (up to
multiplication by a constant)
$\sigma$\,--\,finite  infinite conservative ergodic invariant
measure.

In the special case where $a_i\geq b_i$ and $c_i= a_i-b_i$, then
\eqref{eq:1b}
converges iff \[\limsup \frac{a_0^n}{b_0^n}<\infty.\]
 \end{prop}
  \begin{proof}
     Considering first the $(1\times 1)$ matrix sequence $[a]=[a_i]_{i\geq
      0}$, with associated nonstationary edge shift space 
$\Sigma_{[a]}^{0,+}= \Pi_0^\infty \{1,\dots,a_i\}$,
then there is a unique ergodic $\FC_{[a]}-$ invariant
probability measure 
$\nu_{[a]}$.

This is 
nonstationary Bernoulli (product)  measure
with equal weights $1/a_i$.
There are two cases here: if $a_i=1$ except for finitely many $1$,
then $\Sigma_{[a]}^{0,+}$ is a finite set, with $\Pi_0^\infty a_i$ point masses
permuted by $\FC_{[a]}$; 
otherwise $\Sigma_{[a]}^{0,+}$ is a Cantor set (a nonstationary
odometer, see Example 1 of~\cite{Fisher09a}).

Considering  $M_i=\left[ \begin{matrix}
a_i& c_i \\
0 & b_i
\end{matrix}  \right] $,  since $\nu_{[a]}$ extended to $\Sigma_M^{0,+}$ is also $\FC_M$\,--\, invariant, and is ergodic, we have our
first measure. 

By the same reasoning, for the $(1\times 1)$ matrix sequence $[b]=[b_i]_{i\geq0}$
  there is a unique ergodic 
$\FC_{[b]}-$ invariant
probability measure 
$\nu_{[b]}$ on $\Sigma_{[b]}^{0,+}$. By part $(ii)$ of Theorem  \ref{t:inv-meas},  the second conservative ergodic $\FC_M$\,--\, invariant 
measure on $\Sigma_M^{0,+}$ which is positive on some open subset is
the tower measure $\wh\nu_{[b]}$ over the clopen set and tower base
base $(\Sigma_{[b]}^{0,+}, \nu_{[b]})$.

The unique normalized nonnegative eigenvector sequence with eigenvalue $1$ for
$[b]$ 
is $\w= (\w_i) _{i\geq0}$ with $\w_0=1$, $\w_{k+1}= (b_0^k)^{-1}$. 
From Theorem \ref{t:inv-meas}, 
$\wh\nu_{[b]}$ is a finite measure iff $\w= (\w_i)_{i\geq 0}$ is a
distinguished
eigenvector sequence.
Applying the Lemma concludes the first proof.
\end{proof}
 
  \begin{rem}\label{r:geometric}
   Here is the promised geometric proof of  Proposition \ref{p:invariantmeasure2x2}. Recalling that $C_n$ denotes
   the positive cone of $\R^2$ at time $n$, the extreme rays of the nested cones 
$M_0^n C_n$ are generated by the vectors 
$
M_0^n
\left[ \begin{matrix}
1\\
0
\end{matrix}  \right] 
$ and
$
M_0^n
\left[ \begin{matrix}
0\\
1
\end{matrix}  \right] 
;$
the first gives, projectively, the vector
$\left[ \begin{matrix}
1\\
0
\end{matrix}  \right] 
$ while the second is projectively the same as
$
M_0^n
\left[ \begin{matrix}
0\\
\w_{n+1}
\end{matrix}  \right] 
.$
So the nested cones have as their intersection a single ray
iff
the slope of this last vector approaches
zero. This slope is the inverse of the sum in \eqref{eq:partialsum},
so goes to $0$ iff the sum is infinite.
We have to check this also starting at time $i$, in which case the sum
is that in \eqref{eq:3}, and as above this converges iff the sum
starting at $0$ does.
Hence, by Proposition \ref{p:extremepts}, we have a single nonnegative eigenvector
sequence of eigenvalue one, and so a single probability measure,
leading to the same conclusion as for the previous proof.
  \end{rem}
Here are some simple examples:
  \begin{cor} If
in  Proposition \ref{p:invariantmeasure2x2} for all $k$ we have 
$a_k=a, b_k=b$ and $0<\kappa_1<c_k<\kappa_2<\infty$,
then if $a<b $ there exist  exactly two 
   ergodic invariant probability measures which are finite positive on
   some open subset of $\Sigma^{0,+}_M$; if $a\geq b$
  there exists  exactly one such probability measure 
 and  (up to normalization) one such $\sigma$\,--\,finite infinite conservative ergodic
  invariant Borel measure on  $\Sigma^{0,+}_{M}$.

  If $a_k=a=b_k=b$ (but not assuming the bounds on $c_k$), 
  then this second  measure is finite iff $\sum c_k<\infty$.
  \end{cor}

  \begin{rem}
     The second possibility above can be visualized as two
  systems 
  (take say $a=b=2$, giving two odometers) hooked together by an arrow
  with
  nonstationary probabilities
  $c_k$ of the mass ``leaking'' over from the second to the first
  system.

One can imagine generalizing this example to model two stationary dynamical
systems with nonstationary  communication from the first to the second.
 \end{rem}

  As a next example we reprove a result of ~\cite{Fisher92}: the   Integer Cantor Set
transformation satisfies the following:

\begin{defi} We say a homeomorphism of a Polish space is  {\em infinite measure
uniquely ergodic} if
  there is  up to normalization a unique infinite  invariant Borel measure
  which is positive on
some open set.
\end{defi}

\begin{exam}(the Integer Cantor Set inside the  3-adic odometer)\label{exam:Cantor_b}
As in Example \ref{exam:Cantor}, see Fig.~\ref{F:ICSNew}, we take 
 $\wh M_i= [3]$ for all $i\geq 0$, with edge alphabet $\wh \E_i= 
\{a,b,c\} $ and $M_i=[2]$ with $\E_i= \{a,c\}$.  The canonical cover
matrix is 
$\wt M=\left[
\begin{matrix}
\wh M& C \\
0 & M
\end{matrix}  \right] 
=\left[
\begin{matrix}
3& 1 \\
0 & 2
\end{matrix}  \right] 
.$ 
Now $\V_M^\Delta$ has the single element, $\w= (\w_0\w_1\dots)$ with $\w_n=
2^{-n}$; then by Lemma \ref{l:exampleconvergence}, or directly by checking the criterion of Lemma \ref{l:A-B} part $(ii)$, $\w$ is not $\wh M/M$-
distinguished; indeed, 
$\lim_{n\to\infty} \wh M_i^n \w_{n+1}=
3^{-i}\lim_{n\to\infty} 3^n
2^{-n}= +\infty$.

The $\FC_{\wh M}$-
orbit of $\Sigma_M^{0,+}$ is the tower $\Sigma_{\wh  M/M}^{0,+}$,
a dense subset of $\Sigma_{\wh  M}^{0,+}$. 

Since $\w$ is not distinguished, by  Theorem \ref{t:inv-meas}, and as in Proposition \ref{p:invariantmeasure2x2},   the tower  measure $\nu= \nu_{\iota(\w)}$ on $\Sigma_{\wh  M/M}^{0,+}$
is infinite on every open subset of the triadic odometer, $\Sigma_{\wh  M}^{0,+}$, and is finite on the  compact 
subset $\Sigma_M^{0,+}$ (the dyadic odometer).

By part $(iii)$ of Theorem \ref{t:inv-meas-subdiag}, $\nu$  is up to normalization the unique
invariant Borel measure which is positive finite on an open set for the adic
cover tower $\Sigma_{\wt  M/M}^{0,+}$,
and by Theorem \ref{t:covermeasure} also for the adic tower
$\Sigma_{\wh  M/M}^{0,+}$. These towers are infinite measure uniquely ergodic.

\end{exam}

By the special case in the Proposition, the ICS example  generalizes to:

\begin{cor}(Nested nonstationary odometers)
Consider two nested odometers, the first the $(1\times 1)$
nonnegative integer 
matrix sequence $\wh M$ with $\wh M_i= [a_i]$ and the second $M\leq \wh M$  with 
  $M_i= [b_i]$ (so $b_i\leq a_i)$.
Then the tower  $\Sigma_{\wh M/M}^{0,+}$ has up to normalization a single
invariant ergodic measure which is positive on some (in fact all) open
sets; this is locally finite and inner regular, and is infinite $\sigma$\,--\,finite
iff  \[\limsup \frac{a_0^n}{b_0^n}=\infty.\]
On the space $\Sigma_{\wh M}^{0,+}$ the corresponding  measure is
infinite  on every open subset.
\end{cor}

\begin{rem}\label{rem:Bruin}
  Henk Bruin asked us in a conversation (in 2007, after we gave a talk
about the ICS example at  CIRM) about the existence of 
other infinite invariant Borel measures. Afterwards, we realized there are many, as shown by the
next construction; by 
Lemma \ref{l:finite_on_all}, these are infinite on 
every open subset of the space. 
\end{rem}

\begin{exam}\label{exam:bruin}
  We take $\wh M_i= [a_i]$ with $a_i=2$ for all $i\geq 0$ (so $\Sigma_{\wh M}^{0,+}$ is the
  dyadic odometer space) and $b_i= 2$ for $i$ even$, = 1$ for $i$
  odd. 
  Then  ${a_0^{2n}}/{b_0^{2n}}=2^n\to\infty$, so the tower measure is
  infinite.

Next, let us consider $ a_i, b_i$ as just defined, and $\wh a_i= 3$. Then for
$\wh M= [\wh a]$, $M=[a]$ and $N= [b]$,  the adic tower map on  $\Sigma_{\wh M/M}^{0,+}$
gives  the Integer Cantor Set transformation of the previous example, with measure $\nu$ the  unique (up to scaling) infinite
invariant Borel measure which is positive on open sets. Now  $\Sigma_{\wh
  M/N}^{0,+}$ has the infinite measure $\wt\nu$ which is
positive on open sets of that tower $\Sigma_{\wh M/M}^{0,+}$, but
which is infinite on every open subset of $\Sigma_{\wh M}^{0,+}$ (since $\wh M=[2]$ is
primitive, by
Lemma \ref{l:finite_on_all}).

\end{exam}

Geometrically interesting examples of subdiagrams can be constructed
within 
circle rotations, as follows:

\begin{exam}\label{exam:nestedrot}
  (Nested circle rotations) This example produces
  interesting conservative ergodic measures for irrational circle
  rotations, which are infinite on every nonempty open subset.

From 
Example 3 of \cite{Fisher09a}, also see \cite{ArnouxFisher00}, we can
  code an irrational circle rotation by an adic transformation, as follows.

Taking as alphabet sequence $\A_i= \{A, B\}$ for all $i\geq 0$, 
we define 
a pair  of substitutions $\rho_+,\rho_-$ by
\begin{align*}
&\rho_+ (A)= AB, &&\rho_-(A)= A,\\
& \rho_+(B)= B, &&\rho_-(B)=A B
\end{align*}  
The associated matrices (the abelianizations of the substitutions) are
$$
P 
= \left[ \begin{matrix}
1 & 0 \\
1 & 1
\end{matrix}  \right],
\;
Q= \left[ \begin{matrix}
1 & 1 \\
0 & 1
\end{matrix}  \right]$$
Given now  a sequence $(n_i)_{i\geq 0}$ of positive integers, we 
define 
substitution and matrix sequences $(\rho_i)_{i\geq 0}$, $(N_i)_{i\geq 0}$ by
$N_i=P$ or $ Q$, with 
the first occuring $n_0$ times, followed by the other $n_1$ times and
so on and similarly for $\rho_i$. We call this a one-sided {\em additive family} as it is related
to the additive mapping family on the torus of \cite{ArnouxFisher05} 
and to the additive continued fraction.  

Now we gather the sequence $N$ to form the {\em multiplicative} family
$M$ with

 $$
 M_i =
\left[ \begin{matrix}
1 & 0 \\
 n_i& 1
\end{matrix}  \right] \text{for $i$ even,} 
 \left[ \begin{matrix}
1 &  n_i \\
0 & 1
\end{matrix}  \right]\text{ for $i$ odd. } 
$$

The resulting adic
transformation (on the path space of the Bratteli diagram  order given by the substitution sequence)
is isomorphic to an irrational circle rotation $R_{\theta}$ of angle
$\theta$ 
defined  by the exchange of the two intervals of lengths $1$ and $\alpha$
with 
$$\alpha= [n_0\dots n_k \dots] \equiv \cfrac{1}{n_0+
 \cfrac{1}{n_{1}+\dotsb 
}}$$ Here $\theta= \alpha/(1+\alpha)$ (and would be
$1/(1+\alpha)$ if we had started with parity $(-)$ instead of $(+)$).

Now suppose we are given  $(\wh n_i)_{i\geq 0}$ with $1\leq
  n_i\leq \wh n_i$. Defining  the sequence $ \wh M$, note that indeed $M\leq \wh M$;
we call the resulting pair of nested adic transformations  {\em nested circle
  rotations}. This has geometrical meaning: indeed,
if $n_i<\wh n_i$ infinitely often, then the rotation for $n$ naturally
embeds
in an order-preserving way as a rotation on a Cantor subset of the circle
rotation for $\wh n$. And regarding the measures, building on a result of
\cite{ArnouxFisher05} together with this paper, we have:
\begin{theo}\label{t:rotationsmeasures}
Given two nested circle rotations with 
multiplicative 
matrix sequences  $M\leq \wh M$, that is, 
we have
$n_i\leq \wh n_i$ for all ${i\geq 0}$,
let $\nu$ denote the unique
central measure for $\Sigma^{0,+}_M$, writing  $\wh \nu$ for its  extension to the
adic tower
 $\Sigma_{\wh M/M}^{0,+}$ and for the corresponding $R_{\wh \theta}$
 -invariant Borel measure on the circle.

Writing $\lambda_i= ([n_in_{i+1}\dots])^{-1}$ and similarly for $\wh
\lambda$, 
then 
the tower measure $\wh \nu$ is inner regular; it is finite iff
$\liminf ||\wh M_0^n \w_{n+1}|| <\infty$, and that holds iff
 $\liminf({\wh \lambda_0^n})/({\lambda_0^n})<\infty.$

\

When the measure  is infinite, it  is
infinite  on every open subset of $\Sigma_{\wh M}^{0,+}$ (and of the
circle),  hence is not (recall Definition \ref{d:Radon}) Radon.
\end{theo}

 We mention regarding the infinite measure ergodic theory that at least
in the case of periodic combinatorics (nested quadratic irrational
rotations) one can go farther: following methods of
\cite{MedynetsSolomyak14}, one can prove an
order-two ergodic theorem as shown in \cite{Fisher92} for the Integer Cantor Set.
We thank
Solomyak for conversations regarding this, and the case of
nonperiodic combinatorics, which remains an intriguing question.
Note that \cite{MedynetsSolomyak14} extends the discussion to
graph-directed sets and self-similar
tilings of $\R^d$. 

\end{exam}

\begin{rem}
  Similar examples can be constructed for minimal 
interval exchange transformations, by considering an induced map on a
subset of measure zero defined by a subdiagram, and the adic tower over that.
Explicitly, take a
path in the Rauzy graph which has iterated loops, gather along the
returns to a node of that loop, and define a
subdiagram with fewer iterates. The above is the simplest case, a minimal exchange
of two intervals, i.e.~an irrational circle rotation.
\end{rem}
\medskip

\section{Appendix: Comparison with the classical (stationary) 
theorems}\label{s:stationarycase}
Here we compare our nonstationary Perron-Frobenius, Frobenius Decomposition
and Frobenius--Victory theorems to the stationary case. In particular
we discuss Victory's definition of distinguished
eigenvector sequence, relating that to
the construction of Parry measures.

We recall:
  for a real or complex $(d\times d)$ matrix $M$, having a right {\em eigenvector} $\bfw$
  with {\em eigenvalue } $\lambda$ means by definition that for
  $\bfw\in \C^d\setminus \{\0\}$ and $\lambda\in \C$ we have
  $M\bfw= \lambda \bfw$. (The reason the zero vector is excluded is
  that otherwise {\em every} scalar would be an eigenvalue; this
  became  a
  subtle point when we defined eigenvector sequence in Def.~
  \ref{d:eigenvectorsequence}, see Remark
  \ref{r:limitsequence}). For an example of a real matrix where
  complex $\lambda$ and $\bfw$ come up naturally, see the permutation example
  below.

  Recall that a $(d\times d)$ nonnegative real matrix  $M$ is
  {\em irreducible} iff for all $a,b\in \A$, there
exists $n= n(a,b)\geq 1$ such that $(M^n)_{ab}>0$. It is 
{\em  primitive} iff $\exists n\geq 1$ which does not depend on the
entries, i.e.~ such that for all $a,b\in \A$,
$ (M^n)_{ab}>0$.
A basic example of a
matrix which is irreducible but not primitive is a permutation matrix
with a single cycle, i.e.~ with one communicating class.

  Now the
 Perron-Frobenius Theorem as generally stated addresses the primitive
 case, saying that there exists a 
nonnegative right eigenvector $\w$, called the {\em Perron-Frobenius
eigenvector} and unique  up to multiples; $\w$ is
strictly positive; its
eigenvalue $\lambda $ is (strictly) positive, and for any other 
eigenvalue $\alpha\in \C$, we have $|\alpha|<\lambda.$

Historically however, this part of the theorem is due to Perron, while
the full Perron-Frobenius theorem includes Frobenius' contribution to
the wider irreducible case, which we next describe.

\begin{defi}\label{d:period}
  For  irreducible  $M$ one defines 
  the {\em period} of a state $a\in \A$ to be  the greatest common
  divisor (\gcd) \,  of $\{n:
  M^n_{aa}>0\}$; the period of $M$ is the \gcd\, of the periods of its
  states.
\end{defi}
  
One shows that 
  the period is the same for all states in a communicating class, see
  the clear proof in \cite{LindMarcus95} and also below.

The statement of Frobenius' theorem is the same as Perron's except that now $|\alpha|\leq
\lambda$ (it may not be  {\em strictly } less in modulus); and moreover 
if $p$ is the period of $M$, then collection of eigenvalues is invariant by rotation in the complex plane by angle
$2\pi /p$, whence there are 
$p$ nonnegative eigenvectors equal to $\lambda$ times the $p^{\text{th}}$ roots of
unity, with the rest being smaller in modulus.
See   \cite{Gantmacher59},
Theorem XIII, 2.1 and 2.2.

  \begin{theo}(Stationary Perron-Frobenius Theorem)\label{t:stationaryirreducible}
 Let $M$ be a $(d\times d)$ nonnegative
    real matrix.

    \noindent
${(I)}$ {\em Primitive case} (Perron): Let $M$ be a $(d\times d)$ nonnegative
  real matrix. Assume that $M$ is primitive. Then:
  \item{(i)} there exists (up to multiplication by a constant) a
    unique nonnegative right 
eigenvector; it is strictly positive. Its eigenvalue $\lambda$ is strictly
positive, and is equal to that for the 
 left nonnegative eigenvector.
\item{(ii)}
 Any other eigenvalue $\alpha $ has strictly smaller modulus:
$|\alpha|<\lambda$.

\noindent
${(II)}$ {\em Irreducible case, eigenvectors.} (Frobenius)
    Suppose that $M$ is irreducible, with period $p$. Then:
     \item{(i)} 
    there exists a nonnegative eigenvector, unique up to multiplication by
    a positive constant. This is strictly
    positive, and has a positive eigenvalue $\lambda^+$.
     \item{(ii)}
    There are another 
$p-1$ eigenvalues,  equal to $\lambda^+$ times the $p^{\text{th}}$ roots of
unity. All other eigenvalues have smaller modulus.

\noindent
${(III)}$ {\em Irreducible case, eigenvector sequences.} 
 There are 
    $p$ nonnegative eigenvector sequences of eigenvalue one. These
    correspond to the $p$ eigenvectors for $M^p$, one for each of the
    diagonal primitive blocks.
  \end{theo}

For our purpose of studying $\FC$-invariant Borel measures for
adic transformations, as we know from Theorem 2.9 of
~\cite{BezuglyiKwiatkowskiMedynetsSolomyak10}, see
 Theorem \ref{t:basic_thm} above, what is important are the nonnegative eigenvector
 {\em sequences} of eigenvalue one. Thus
in the above theorem, part $II$ which treats the eigen{\em vectors}
plays no role: all we care about is parts $I$ and  $III$.

We can see this clearly with
 a simple
example.  For $a,b,c$ positive real numbers, let
$M=\left[ \begin{matrix}
0& a & 0\\
0&0&b\\
c&0&0
\end{matrix}  \right] $. The characteristic polynomial is
$p(\lambda)=
\left|\begin{matrix}
-\lambda& a & 0\\
0&-\lambda&b\\
c&0&-\lambda
\end{matrix}  \right| =-\lambda^3+ a b c$ so $M$ has
three eigenvalues equal to the three
complex solutions of $\lambda^3=abc$.  Only for one of these,  $\lambda= \lambda^+$, the {\em
  Perron-Frobenius eigenvalue} is positive real. 
The   right eigenvector corresponding to each $\lambda$ is
a multiple of $\v=(1, \lambda/a,
\lambda^2/ab)$.The eigenvector  $\v_{\lambda^+}$ corresponding  to
$\lambda^+$ is 
called the {\em Perron-Frobenius
eigenvector}.  Note that the other eigenvalues and eigenvectors are
given by multiplication of $\lambda^+, \v_{\lambda^+}$ by the three complex roots of $1$.

Now we can conjugate the stationary sequence $M, M, M,\dots$ via a periodic sequence of  permutation
matrices to the
periodic sequence 
$$\left[ \begin{matrix}
a& 0 & 0\\
0&b&0 \\
0&0&c
\end{matrix}  \right], 
\left[ \begin{matrix}
b& 0 & 0\\
0&c&0 \\
0&0&a
\end{matrix}  \right], 
\left[ \begin{matrix}
c& 0 & 0\\
0&b&0 \\
0&0&a
\end{matrix}  \right]\dots.$$ 
Note that $M^3= 
\left[ \begin{matrix}
abc& 0 & 0\\
0&bca&0 \\
0&0&cba
\end{matrix}  \right]= 
abc \cdot I.$

Choosing  $a,b,c\in \N^*$,  we draw the Bratteli diagram for $M$, and
examine the $\FC-$invariant Borel measures. Let us consider for
example
$M=\left[ \begin{matrix}
0& 1 & 0\\
0&0&2 \\
3&0&0
\end{matrix}  \right] $.
We know that the central measures are
preserved by taking a gathering, e.g.~in the stationary case by taking
powers of the matrix. For the Bratteli diagram of $M^3= 6 \cdot I$ 
we see there are three embedded $6$-adic odometers, each giving an ergodic
measure. But how  do we see this for the original matrix $M$? The
answer is given by the insight of Theorem 2.9 of
~\cite{BezuglyiKwiatkowskiMedynetsSolomyak10}, see
 Theorem \ref{t:basic_thm} above:  invariant Borel measures correspond (even in
the stationary case) not to the eigen{\em vectors} but to the nonnegative eigenvector
{\em sequences} of eigenvalue one;
there is one associated to each symbol, and so there are three
extreme sequences, and three ergodic measures, one for each odometer. 

Further, as above we can conjugate the sequence $M, M, M,\dots$ to the
periodic sequence 

$\left[ \begin{matrix}
1& 0 & 0\\
0&2&0 \\
0&0&3
\end{matrix}  \right], 
\left[ \begin{matrix}
2& 0 & 0\\
0&3&0 \\
0&0&1
\end{matrix}  \right], 
\left[ \begin{matrix}
3& 0 & 0\\
0&2&0 \\
0&0&1
\end{matrix}  \right]\dots$ and now in the conjugate Bratteli diagram we clearly
see the three periodic nonstationary odometers, each a shift of the next.

Now the invariant Borel measures correspond, from Theorem \ref{t:basic_thm} (even in a stationary
case, as here!) to the eigenvector {\em sequences} of eigenvalue one, and the
ergodic measures to the extreme points of this collection of
sequences. Defining $(\w_0, \w_1\dots)$ to be $(a_0^{-1}(1,0,0), a_1^{-1}((0,1,0),
a_2^{-1}(0,0,1),\dots)$ where $\underline a=(a_0, a_1, a_2\dots ) =( 1,2, 6;
6\cdot(1,2,6); 6^2\cdot (1,2,6); \dots)$
then  in this example the three extreme eigensequences of eigenvalue
one are (up to multiples) the three shifts of the sequence $\w$. For
the untwisted matrix $\wt M$, the three eigenvector sequences are 
the three shifts of $\underline a$ (up to a constant multiple) times
the standard basis vectors.
Note that although for $M$ there are three eigenvectors, as above, $(1, \lambda,
\lambda^2/2)$ where $\lambda$ are the three solutions of
$\lambda^3=6$, these do {\em not} play a role in finding the 
three ergodic invariant Borel measures. Indeed, for the simplest case
$a,b,c=1$, then the Perron-Frobenius eigenvector is $\v=(1,1,1)$, but as
an eigenvector sequence this is not an extreme point- it is in the
middle of the unit simplex; in this case since $3$ is odd it is also the
unique real eigenvector. The three extreme points of the simplex are
permuted by $M$, and these give the three eigenvector sequences: 
$(\e_1, \e_2, \e_3, \e_1, \e_2, \e_3,\dots)$ and its shifts, each
defining an invariant Borel measure (each a point mass, in this case). The
Perron-Frobenius eigenvector $\v$ does give an eigenvector sequence, 
the constant sequence $(\v,\v,\dots)$ and hence a measure, but this is
nonergodic being a convex combination of the point masses. So, once
again, the Perron-Frobenius eigenvector plays no special role in the
measure theory, and the right way to understand things is in terms of the
eigenvector sequences.

Something very similar to this  example occurs  for a general
irreducible 
nonnegative $(d\times d)$ matrix $M$. The Frobenius theory then says
the following. The alphabet can be partitioned into $p$ subsets called
{\em period classes} (see Proposition 4.5.6 of 
\cite{LindMarcus95}) with $l_i\geq 1 $ elements for $1\leq i\leq p$,
so $l_1 + l_2 + \dots + l_p= d$;
reordering $\A$ accordingly, $M$ has the $(p\times p)$ block form of a
permutation matrix 
with nonzero blocks $B_i$ which are  $(l_{i}\times l_{i+1})$ for $i<p$ and $(l_{i}\times
l_1)$ for $i=p$; taking for example $p= 3$,
$M=\left[ \begin{matrix}
0 & B_1&0 \\
0 & 0& B_2\\
B_3& 0& 0
\end{matrix}  \right]$.

Frobenius' theorem then identifies the eigenvalues and eigenvectors, saying that the spectrum of $M$ is invariant 
with respect to multiplication by the 
$p^{\text{th}}$ roots of unity. In particular, $M$ has exactly $p$
eigenvalues of modulus $\lambda$, multiples 
by the $p^{\text{th}}$ roots of unity.
For more on this theory see especially \cite{LindMarcus95}, also  \cite{Gantmacher59} and
\cite{BezuglyiKwiatkowskiMedynetsSolomyak10}.

We can understand this as follows. But moreover, we can find the nonnegative
eigenvector sequences of eigenvalue one, which are what we need for
the invariant Borel measures; as above, the eigenvectors themselves are of
no help there.

Thus, 
for $A_1= B_1B_2B_3,  A_2= B_2 B_3B_1, A_3=B_3 B_1 B_2$ we have 
$M^3=\left[ \begin{matrix}
A_1 & 0&0 \\
0 & A_2& 0\\
0& 0& A_3
\end{matrix}  \right]$. Since the diagonal blocks are
primitive, there is a further power  with diagonal blocks all strictly
positive.  Once again, we can untwist the Bratteli diagram by
conjugation with a 
periodic sequence of permutation matrices, to an untwisted matrix $\wt
M$, with three diagonal primitive subblocks $B_i$. As before, the three 
eigenvectors for these blocks correspond, in the
original diagram, to three eigenvector sequences 
which differ by a time shift.

For the conclusion of Frobenius' theorem, 
note that each of the
$(l_i\times l_i)$ matrices $A_i$ 
has the same collection of eigenvalues. (This is easy to prove directly
by seeing how the eigenvectors correspond.)

Again, the ergodic measures are given not by these
$p$ eigenvectors for $M$, but rather by the $p$ nonnegative eigenvectors for the
untwisted matrix $\wt M$.

Note that from the cone point of view, as in the projective metric
proof of the Perron-Frobenius Theorem for the primitive case, there
are three subcones of the positive cone which are permuted and mapped
into each other, nesting down to the extreme sequences.

\subsection{Comparison with Parry measures}\label{ss:ParryMeasures}
Given a $(d\times d)$ primitive matrix $A$, for simplicity with
entries $0,1$,  the
adic-invariant 
central measure $\nu$ and the shift-invariant Borel measure of maximal entropy (the
Parry measure) $\mu$ are both unique and have closely related formulas. For
the irreducible case this situation changes dramatically, as we
explain. 

It will be convenient to use the following matrix formalism in
describing these measures. Let $\v^t,\w $ be the left and  right 
Perron-Frobenius eigenvectors of $A$, with eigenvalue $\lambda.$ We
normalize the vectors so that $\v\in\Delta$
and $\v\cdot \w=1$. We write $\1$ for the $(d\times 1)$ column
vector of all $1's$, and $\bpi^t $ for the probability row vector with
entries $v_iw_i$.  We define $W$ to be the diagonal matrix with entries
$W_{ii}= w_i$. 
Then $W \1= \w$ and so 
\begin{equation}
  \label{eq:Parry1}
 P\equiv \frac{1}{\lambda} W^{-1} A W
\end{equation}
satisfies $P\1= \1$, i.e.~ $P$ is row--stochastic; moreover, 
$\bpi^t P= \bpi^t$. 

The formula \eqref{eq:Parry1} is just a matrix version of the familiar
Parry-Shannon  formula

\begin{equation}
  \label{eq:Parry2}
P_{ij}= \frac{1}{\lambda}\frac{w_j}{w_i} A_{ij}.
 \end{equation}

We define the {\em Parry measure} $\mu$ on the vertex shift space $\Sigma_A^{0,+}$ by
$$\mu[x_0\dots x_{n}]= \pi_{x_0} P_{x_0 x_1}\cdots P_{x_{n-1}x_n}= 
\lambda^{-n}v_{x_0}w_{x_n}$$
where we use \eqref{eq:Parry2} to calculate the collapsing product.  
It follows from $\bpi^t P= \bpi^t$ that this definition is
shift-invariant, so $\mu$ extends to an invariant probability measure on the bilateral
shift space  
$\Sigma_A$. 

By contrast, the central measure only depends on the right nonnegative eigenvector: 
$$\nu[x_0\dots x_{n}]= 
\lambda^{-n}w_{x_n}.$$ This agrees with our definition in 
\eqref{eq:basictheorem} since $(\w)=(\w_k)_{k\geq 0} $ with $\w_k\equiv \lambda^{-k} \w$ is an
nonnegative eigenvector sequence of eigenvalue $1$.

Now consider the irreducible case.
There the Perron-Frobenius eigenvectors $\v^t, \w$ are still unique,
and the same formula gives the Parry measure $\mu$, again the unique
measure of maximal entropy (equal to $\log\lambda$).

However as we have seen, for ergodic central measures the Perron-Frobenius eigenvector itself plays
no role, and there are now $p$  such measures, each governed by 
one of the $p$ extreme nonnegative eigenvector sequences.

We offer two  explanations for this striking contrast to the primitive case.
First, there are in fact $p$  {\em nonstationary}
Parry measure {\em sequences}, see \cite{Fisher09a}, periodic of
period $p$, 
and each of entropy 
$\log\lambda$. In fact, in the above example,
each lives on a {\em sub}-subshift, given by the $p$
periodic matrix sequences such as $N_0, N_1, N_2,\dots$ where 
$N_0=\left[ \begin{matrix}
0 & B_1&0 \\
0 & 0& 0\\
0& 0& 0
\end{matrix}  \right], N_1=\left[ \begin{matrix}
0 & 0&0 \\
0 & 0& B_2\\
0& 0& 0
\end{matrix}  \right], N_2=\left[ \begin{matrix}
0 & 0&0 \\
0 & 0& 0\\
B_3& 0& 0
\end{matrix}  \right]$.  Again, conjugation by a periodic sequence of
matrices straightens out $(N_i)$ so we are actually studying the periodic
\sft\,  given by the 
primitive periodic
sequence $B_1, B_2, B_3,\dots$.

A second explanation 
 comes via Lemma
2.4 of Bowen and Marcus in 
\cite{BowenMarcus77}, where  the uniqueness of the central
measure for a
primitive \sft\,  was proved via the mixing of the Parry measure,
which could be summarized as {\em ``mixing of the hyperbolic dynamics implies unique
ergodicity of the transverse dynamics''}. 
But it  is exactly in  the irreducible nonprimitive case that the Parry measure 
is {\em not} mixing, and indeed, as we have seen, unique ergodicity
fails as there are $p$
central measures where $p$ is the period of the matrix, see Def.~\ref{d:period}.

\subsection{Comparison with the classical Frobenius--Victory Theorem}

Frobenius  went on from his study of the  irreducible case to analyze
the stationary reducible case.
This second case, as we now explain, is much more involved.
Here the tools are the Frobenius decomposition and
the Frobenius--Victory theorem. 

The Frobenius decomposition theorem, see  \S 4.4 of
\cite{LindMarcus95}, equation (69) \S XIII.4 of \cite{Gantmacher59}, and
equation (4) of \cite{BezuglyiKwiatkowskiMedynetsSolomyak10}, 
 states that a nonnegative square matrix
$N$ can be put
in upper triangular block form 
\begin{equation}\label{eq:block_form_stationary}
N=
\left[ \begin{matrix}
A_1 &C_{12} &\dots  &C_{1\wh l}\\
0_{21} &A_2 & &\vdots\\
\vdots & &\ddots  &\\
0_{\wh l 1} &\dots  &  &A_{\wh l}
\end{matrix}  \right]
\end{equation}
with $A_j$ square matrices that are zero or irreducible. Using what we have just seen, 
by taking a power, we can achieve that the diagonal blocks 
(now possibly larger in number) are  zero or primitive.

We recall the statement of 
the Frobenius--Victory theorem, so named e.g.~ in
\cite{TamSchneider00}.
See Proposition 1 of \cite{Victory85}, 
 Theorem 6 of \S XIII.4 of \cite{Gantmacher59}, and Theorem 3.1 of
\cite{BezuglyiKwiatkowskiMedynetsSolomyak10}; apparently the theorem
is actually due to Frobenius.   We shall explain how this
agrees with  the
 nonstationary version above in Theorem \ref{t:FrobVictGeneral}.

Here is the standard definition (following  Victory) of distinguished
eigenvalue, eigenvector and communicating class:
  
\begin{defi}\label{defi:distinguishedvector} Given a $(\wh l\times\wh l)$ nonnegative real matrix
 $N$, then an eigenvector $\w$, with eigenvalue $\lambda$, is 
 {\em distinguished} iff it is 
nonnegative, and in this case $\lambda$ is  a  {\em
  distinguished eigenvalue}. If $\beta$ is the class of collection of
communicating states
corresponding to $\w$, then  $\beta$ is a {\em distinguished class} iff its
eigenvalue 
$\lambda_\beta$ is strictly greater than $\lambda_\alpha$ for any
collection of communicating states $\alpha$ 
such that $\alpha $ communicates to $\beta$.
\end{defi}

\

Note that the eigenvector for $\beta$ is a Perron-Frobenius
eigenvector for an irreducible matrix corresponding to $\beta$, since
this is a communicating class. 

In terms of the upper triangular form of the single matrix $N$
in upper triangular block form of \eqref{eq:block_form_stationary},
 this
means the following: letting $C_{kj}^{(n)}$ denote the $kj$-block of
$A^n$, then 
an eigenvector $\w$ for $A_j$ is distinguished 
iff
$||A_j||>||A_k||$
for all $k$ such that for some $n\geq 0$ (hence for all larger $n$)
the block $C_{kj}^{(n)}$ is nonzero. 

Note that since for an irreducible matrix the Perron-Frobenius
eigenvalue
 is the same for 
all 

Thus for example if all the upper blocks $C_{kj} $ are nonzero, then
the Perron-Frobenius
eigenvector for 
each  diagonal block is distinguished iff
$||A_1||<||A_2||<\dots ||A_{\wh l}||.$

\begin{prop}\label{c:stationarycase}
Given an $(\wh l\times\wh l)$ nonnegative real matrix $N_0$ 
then a nonnegative  eigenvector $\w_0$ for a communicating class is distinguished in
the sense of Definition \ref{defi:distinguishedvector} iff, defining
$\w= (\w_n)_{n\in \N}$ for $n\geq 0$ where 
$\w_n= \lambda^{-n}\w_0$, then  this nonnegative eigenvector sequence of eigenvalue one  
is distinguished in the sense of
Definition \ref{d:dist}, for the stationary sequence  sequence $N=
(N_i)_{i\geq 0}$ with $N_i\equiv N_0$.  
\end{prop}

\begin{proof}We cut the matrix down to the upper-left $(j\times
  j)$\,--\,block submatrix, written as
\begin{equation}\label{eq:block_form_a}
\left[ \begin{matrix}
A&C\\
0 &B\\
\end{matrix}  \right]
\end{equation}
where $B= A_j$. Now $B$ is a primitive matrix, so by the
Perron-Frobenius theorem it has a unique nonnegative eigenvector $\w_0$
with eigenvalue $\lambda>0$; moreover, $\lambda= ||B||$. 
We define from this a nonnegative eigenvector sequence $\w$ of eigenvalue one,
with $\w= (\w_n)_{n\in \N}$ where $\w_n= \lambda^{-n}\w_0$ and $B^n
\w_n= \w_0$.
From 
Lemma \ref{l:conditions}, $\w$ is distinguished iff 
$\lim_{n\to\infty} \wh C_0^{n} \w_{n+1}$
exists, 
and we have
$$
 \wh C_0^{n} \w_{n+1}= \sum_{k=0}^{n} A^kC B^n_{k+1}\w_{n+1}=
\sum_{k=0}^{n} A^kC \w_{k}=\sum_{k=0}^{n} A^kC \w_0
\lambda^{-k}.$$
Now
 $||A^kC \lambda^{-k}\w_0||= ||A^kC \w_0||\cdot ||B||^{-k}\leq
 ||C||\bigl(\frac{||A||}{||B||}\bigr)^k $
so the sum will converge if $||A||<||B||$. 
We want a necessary and sufficient condition however, 
for this note first that since each $A_j$ is primitive, all
 vectors in the corresponding  nonnegative cone  grow asymptotically at the same rate
 as the Perron-Frobenius eigenvector $\wt \w_j$ with eigenvalue $\lambda_j$; in other words, given $A_j$
 $(l\times l)$ with $\v\in \R^{l,+}$ then there exists $a>0$ and
 $0<c_1<c_2$ such that $c_1\lambda_j^n<||A_j^n(a\v)||<c_2\lambda_j^n$
 for $n$ large; 
see e.g.~ Theorem 8 of \S XVI of
 \cite{Birkhoff67}. 
We consider 
 $J=\{ k:\,
A_k\text{ communicates to } A_j\}$. Set $K=  \max_{1\leq k\in J\leq
  j-1} ||A_k||^n$. 
Then 
$
 \wh C_0^{n} \w_{n+1}= \sum_{k=0}^{n} A^kC B^n_{k+1}\w_{n+1}=
\sum_{k=0}^{n} A^kC \w_{k}=\sum_{k=0}^{n} A^kC \w_0
\lambda^{-k}.$
so if for some $k\in J$, $||A_k||= ||A_j||$ then the sum
will explode.
\end{proof}

\begin{rem}\label{r:distinguished}
 With this, we can at last explain the choice of the term ``distinguished'' in
  Definition \ref{d:dist1}. 
For the nonstationary
case, the communicating classes 
get replaced by streams $\beta,\alpha$ and the 
condition is that $\alpha\leq \beta.$ The term ``distinguished'' is
being used in the sense of {\em distinguished from} i.e.~{\em
  separated} from;
if this holds then the nonnegative eigenvector sequence from $\beta$ generates an
nonnegative eigenvector sequence for the larger matrix $N$ which remains distinct 
from that of $\alpha$ alone, whereas if it is not distinguished, the
generated eigenvector gets attracted to and swallowed up by that for
$\alpha$, under the iteration defined here by $(\w)$.
\end{rem}

\

\begin{rem}
We mention that,
regarding Parry measures for the reducible case, the analysis is
identical to that above; thus the relevant ``Parry measures'' are now
built from the central measures, whether finite or infinite, by
multiplying by a left nonnegative eigenvector sequence. 
The same relationship between Parry and central measures holds in the
nonstationary setting, see \S 4 of 
\cite{Fisher09a}.
\end{rem}

\end{document}